\documentclass[11pt,a4paper]{article}
\usepackage[utf8]{inputenc}
\usepackage[T1]{fontenc}
\usepackage[normalem]{ulem} 
\usepackage{amsmath}
\usepackage{amsfonts}
\usepackage{amssymb}
\usepackage{amsthm}
\usepackage{mathtools}
\usepackage{bbm}
\usepackage{setspace}
\usepackage{graphicx}
\usepackage{hyperref}
\hypersetup{
        citecolor=black,
	colorlinks=true,
	linkcolor=black,
	filecolor=black,
	urlcolor=black,
}
\usepackage{cleveref}
\usepackage{natbib}
\usepackage{soul}
\usepackage{float}
\usepackage{verbatim}
\usepackage{longtable,booktabs}
\usepackage[font=footnotesize,labelfont=bf]{caption}
\usepackage{subcaption} 
\usepackage{makecell}
\usepackage{xcolor}
\usepackage[margin=0.75in]{geometry}
\usepackage{xr}
\author{Stefan Stein\footnote{Stein is PhD student, Department of Statistics, University of Warwick, Email: s.stein@warwick.ac.uk.}, \and Rui Feng\footnote{Feng is PhD student, Department of Statistics, University of Warwick, Email: rui.feng.1@warwick.ac.uk}, 
 \and Chenlei Leng\footnote{Leng is Professor, Department of Statistics, University of Warwick, Email: c.leng@warwick.ac.uk. Corresponding author.}}

\title{A Sparse Beta Regression Model for Network Analysis}

\providecommand{\keywords}[1]
{
	\small	
	\textbf{\textit{Key words:}} #1
}
\newcommand{\R}{\mathbb{R}}
\newcommand\sbmc{S$\beta$RM }
\def\serc{ERC }
\newcommand{\E}{\mathbb{E}}
\DeclareMathOperator*{\argmin}{arg\,min}

	\newtheorem{Satz}{Theorem}
	\newtheorem{Prop}{Proposition}
	\newtheorem{Lem}{Lemma}
	\newtheorem{Kor}{Corollary}
	
	\theoremstyle{definition}
	\newtheorem{Def}[Prop]{Definition}
	\newtheorem{Assum}{Assumption}
	\newtheorem*{Rem}{Remark}

	\newtheorem*{Prop*}{Proposition}
	\newtheorem*{Satz*}{Theorem}
        \newtheorem{Assum*}{Assumption}[section]

\begin{document}
	\maketitle
	
	\begin{abstract}
	\begin{singlespace}
For statistical analysis of  {network} data, the $\beta$-model has emerged as a useful tool, thanks to its flexibility in incorporating nodewise heterogeneity and theoretical tractability. To generalize the $\beta$-model, this paper proposes the Sparse $\beta$-Regression Model (S$\beta$RM) that unites two research themes developed recently in modelling homophily and sparsity.  In particular, we employ differential heterogeneity that assigns weights only to important nodes and propose penalized  {likelihood} with an $\ell_1$ penalty for parameter estimation. While our estimation method is closely related to the LASSO method for logistic regression, we develop new {theory} emphasizing the use of our model for dealing with a parameter regime that can handle sparse networks usually seen in practice. More interestingly, the resulting inference on the homophily parameter demands no debiasing normally employed in LASSO type estimation. We provide extensive simulation and data analysis to illustrate the use of the model. As a special case of our model, we extend the Erd\H{o}s-R\'{e}nyi model by including covariates and develop the associated statistical inference for sparse networks, which may be of independent interest.
		\end{singlespace}
	\end{abstract}

\normalsize
\section{Introduction}
Network data are ubiquitous in today's society. Although they exhibit many characteristics, there are a few stylized features that most real-life networks share \citep{Kolaczyk:2009, Newman:2018}. First, nodes in a real-life network have different, sometimes drastically different, tendency to make  {connections}, leading to degree heterogeneity. Second, nodes similar in their attributes or nodal covariates are more likely to attach to each other than dissimilar ones, resulting in what is named homophily in the literature. Above all, it is known that most real-life networks are sparse, in the sense that the total number of connections scales sub-quadratically in the order of $o(n^2)$ with respect to $n$, the number of nodes. On {the} one hand, degree heterogeneity calls for models that are flexible in reflecting nodewise differences and homophily calls regression-type of models that can handle covariates. On the other, to model sparse networks, these models should refrain from over-parametrizing due to {the} scarcity of connections. This paper is about a new model that aims to balance both needs.

To fix ideas, assume that we have observed data organized as $\{A_{ij}, Z_{ij}\}_{i, j =1, i\not= j}^n$, where $A=(A_{ij})_{i,j=1}^n$ is the adjacency matrix with $A_{ij}=1$ if nodes $i$ and $j$ are connected and $A_{ij}=0$ otherwise, and $Z_{ij} \in \R^p$ are $p$-dimensional covariates associated with these two nodes. Given the covariates, undirected links are independently formed with the probability of a connection between nodes $i$ and $j$ being
\begin{equation}\label{SBetaM covariates}
P(A_{ij} = 1|Z_{ij}) = p_{ij} = \frac{\exp(\beta_i + \beta_j + \mu + Z_{ij}^T\gamma)}{1+\exp(\beta_i + \beta_j + \mu + Z_{ij}^T\gamma)},
\end{equation}
where $\beta=(\beta_1, ..., \beta_n)^T \in \R^n$  is the heterogeneity parameter, $\gamma\in \R^p$ is the homophily parameter, and $\mu \in \R$ is a global density parameter, for which we allow $\mu \rightarrow -\infty$ as $n \rightarrow \infty$. For identifiability, we assume $\min_i \beta_i = 0$, so that $\beta \in \R^n_+$ with $\R_+=[0, +\infty)$, because otherwise $\mu$ can be absorbed into $\beta_i$. While it may seem appealing to impose $\min_i \vert \beta_i \vert = 0$ instead of $\min_i \beta_i = 0$, restricting the degree heterogeneity parameters only in absolute value would result in an unidentifiable parameter. We study our model where only a single undirected network is observed with its number of nodes growing to infinity. This is arguably the most interesting setup for network models \citep{Kolaczyk:2009,Goldenberg:etal:2009,Fienberg:2012,Kolaczyk:2017}.

Central to our model is the assumption that $\beta$ is sparse. As such, we shall name our model the Sparse $\beta$-Regression Model (S$\beta$RM). In this model, $\beta_i$ specifies how node $i$ participates in network formation and thus reflects nodewise heterogeneity directly. We interpret $\beta_i=0$ as if node $i$ is a background node, with its propensity of making  {connections} only depending on $\mu$ and $\gamma$, the two global parameters in the S$\beta$RM. If $\beta_i\not= 0$, we say that node $i$ has its own characteristic of establishing ties. The sparsity assumption on $\beta$ makes sense intuitively, since the focus in modelling networks is usually on those hub or popular nodes having relatively many connections. Consigning other less important nodes to having zero heterogeneity parameters will reduce the dimensionality of the model and, as a result,  { allow} statistical inference for a wider range of networks. In particular, in our asymptotic analysis, we allow $\mu$ and $\beta$ to vary with $n$ so that the model handles sparse networks. 
The parameter $\gamma$  captures the effect of the covariate $Z_{ij}$ for initiating connections, where $Z_{ij}$ either represents node-similarity or encodes edge-covariates.

There are several models that are closely related to the S$\beta$RM.  When $\beta=0$ and $\gamma=0$, $p_{ij}=\exp(\mu)/(1+\exp(\mu))$ and the model becomes the Erd\H{o}s-R\'{e}nyi model, a foundational probabilistic model that has been extensively studied \citep{erdds1959random,erdos1960evolution,Gilbert:1959}. 
 When $\beta\not=0$ and $\gamma=0$, without the sparsity assumption on $\beta$, it becomes the $\beta$-model with the consistency of its maximum likelihood estimator (MLE) proved in \cite{Chatterjee:etal:2011}  and asymptotic normality in \cite{Yan:Xu:2013}. See also 
\cite{Rinaldo:etal:2013}, \cite{Karwa:Slavkovic:2016} and \cite{yan2016asymptotics} for further results, and 
\cite{Yan:etal:2016} for a directed version of the $\beta$-model. Since the $\beta$-model associates each node with its own parameter, it can only fit networks that are relatively dense  \citep{Yan:Xu:2013}. To overcome this, \cite{Chen:etal:19} proposed the sparse $\beta$-model (S$\beta$M) by making a similar parameter sparsity assumption to this paper, while \cite{shao20212} applied a ridge penalty on the parameters. These aforementioned papers did not consider covariates. The first study on the $\beta$-model accounting for covariates effects, that is, when $\beta\not= 0$ and $\gamma\not=0$,  was conducted by \cite{Graham:2017} with a dense $\beta$. See \cite{jochmans} for further results and \cite{Yan:etal:2019} for a generalization to directed networks. We note that in a parallel line of research, there are many efforts made in incorporating covariates especially in another popular class of models called the stochastic block model. We refer to \cite{zhang:2016}, \cite{Binkiewicz:2017}, \cite{Huang:Feng:2018},  \cite{zhao2019logistic}, and \cite{Yan:Sarkar:2020},  \cite{weng2022community}, among many others. {In addition, \cite{mama:2020} considered a latent space model with covariates and proposed two universal fitting algorithms.} 

Thus, in a certain sense, the model in \eqref{SBetaM covariates} can be seen as an attempt to unite the ideas in \cite{Chen:etal:19} in modelling sparse networks and \cite{Graham:2017} in modelling homophily. However, our work differs substantially from these two papers. Specifically, we employ a penalized likelihood method with an $\ell_1$ penalty on the heterogeneity parameter for estimating the parameters, in contrast to the $\ell_0$ penalized method developed in \cite{Chen:etal:19}. The use of the $\ell_1$ penalty connects our methodology to the LASSO framework \citep{Tibshirani:96,vandegeer2011}, enabling us to draw upon the vast literature on high-dimensional data analysis, especially for logistic regression. Despite the somewhat superficial similarity of our estimator to the penalized logistic regression with an $\ell_1$ penalty,  great care needs to be taken when applying LASSO theory to our estimator. Firstly, the design matrix of our model associated with $\beta$ is deterministic while that with $\gamma$ is random, making the common assumptions made on the eigenvalues of the design matrix typically seen in LASSO not applicable. 
 Furthermore, our approach differs from classical LASSO theory for logistic regression in that the linking probabilities $p_{ij}$ are not assumed to be uniformly bounded away from zero, because otherwise the network will be dense. This assumption is often made in LASSO theory; see \cite{vandegeer2011}, Theorem 6.4; \cite{Buena:2008}, Theorem 2.4; or \cite{vandegeer2008}, Theorem 2.1, for example, among many others. 
 {To the best of our knowledge}, we are not aware of similar conditions explicitly stated in the literature, at least not to a model similar to ours. Importantly, our analysis reveals an interesting insight about the effective sample sizes of different parameters. For each heterogeneity parameter $\beta_i$, its effective sample size depends on the number of connections that node $i$ has, while that of $\mu$ and $\gamma$ depends on the total number of edges. We find that the rate of convergence of our estimator for excess risk and $\ell_1$-error differ from that of the classical LASSO estimator only in an additional factor having an explicit relation to the expected edge density of a network. This delineates the role that the sparsity of a network plays in determining the rate of convergence.
 
 This paper contributes an innovation to the development of statistical inference not previously seen in the literature, by providing a central limit theorem for $\gamma$ in the face of vanishing link probabilities. Remarkably, we show that this theorem holds without the need to apply the kind of debiasing usually required for LASSO estimators due to shrinkage  \citep{zhang:zhang:2014,vandegeer2014} or the need to deal with the incidental parameter problem due to over-parametrization \citep{Graham:2017,Yan:etal:2019}. Crucially, inference for LASSO type estimators relies on finding a good approximation to the inverse of the population Gram matrix whose minimum eigenvalue is routinely assumed to be bounded away from zero, uniformly in $n$ \citep[e.g.]{vandegeer2014}.  In our case, however, this matrix depends on the link probabilities $p_{ij}$ and since we allow $p_{ij} \rightarrow 0$ for many $i$ and $j$, such a uniform lower bound assumption becomes invalid. We demonstrate how to overcome this difficulty as long as rates are chosen carefully. The ability to conduct inference with an asymptotically non-invertible Gram matrix and vanishing link probabilities is a significant improvement over many existing methods and a prerequisite for dealing with sparse networks. \color{black}{In addition to the inference of the homophily parameter $\gamma$, we also provide a debiased estimator for each heterogeneity parameter $\beta_i$ and obtain its asymptotic normality under mild conditions.}

Another contribution of the paper comes from the study of a simplified form of the \sbmc  when degree heterogeneity does not exist such that $\beta=0$ but homophily does in that $\gamma\not=0$. For this model, we allow the density parameter to diverge to $-\infty$ to model sparse networks, which distinguishes it from the usual logistic regression. 
We name it the Erd\H{o}s-R\'{e}nyi model with covariates (ERC) for obvious reasons. The implication of a model being able to handle sparse networks is revealed in \cite{Krivitsky:Kolaczyk:2015} that provided an insightful answer to the question of the effective sample size. The focus of \cite{Krivitsky:Kolaczyk:2015} is on the Erd\H{o}s-R\'{e}nyi model without covariates, a simpler model compared to ours. For the ERC, we develop the theory for the properties of the estimators of $\mu$ and $\gamma$ that can be used for statistical inference.

The rest of the paper is structured as follows. In Section \ref{Sec: SBetaM covariates}, we present the S$\beta$RM and derive the consistency of its penalized likelihood estimator in terms of excess risk, $\ell_1$-norm and $\ell_q$-norm. \color{black}{In Section \ref{Sec: Inference}, we derive a central limit theorem for our estimator of the homophily parameter $\gamma$ and  a debiased estimator of the heterogeneity parameter $\beta$}.  We then present the Erd\H{o}s-R\'{e}nyi model with covariates in Section \ref{subsection: generalized ER model} and provide the theory for its estimator. We present extensive simulation results in Section \ref{Sec: Simulation} and apply our model to a friendship network of a corporate law firm and the world trade network in Section \ref{Sec: Data Analysis}. Conclusion remarks are presented in Section \ref{section: conclusion}. An extensive comparison of our model with the work of \cite{Chen:etal:19} can be found in the Supplementary Material, together with all of our proofs and additional simulations. The code implementing the approach in the paper can be found on \href{https://github.com/ChrisFeng1998/Sparse-Beta-Regression-Model}{https://github.com/ChrisFeng1998/Sparse-Beta-Regression-Model}.

\subsection{Notation}
A network on $n$ nodes is represented as an undirected graph $G_n = (V, E)$, consisting of a node set $V=\{1, \dots, n\}$ and an edge set $E$. A graph $G_n$ is represented as a binary adjacency matrix $A \in \R^{n \times n}$, where $A_{i,j} = A_{j,i} = 1$, if $\{i,j\} \in E$ and $A_{i,j} = A_{j,i} = 0$ otherwise.  We write $d_i = \sum_{j = 1}^n A_{ij}$ as the degree of node $i$, $d=(d_1, ..., d_n)^T$ as the degree sequence, and $d_+ = \sum_{i = 1}^{n} d_i/2 = \sum_{ i < j} A_{ij}$ as the total number of edges. By $a_n\sim b_n$  we mean $0< \liminf_{n\to \infty} a_n/b_n \le \limsup_{n \to \infty} a_n/b_n <\infty$ for two sequences of positive numbers $a_n$ and $b_n$. We call a network sparse if $\mathbb{E}[d_+]\sim n^\zeta$ for some $\zeta \in (0,2)$, where $\mathbb{E}$ is the expectation with regard to the data generating process. A network is dense if $\mathbb{E}[d_+]\sim n^2$.  

For a vector $v \in \R^n$, we use $S(v)=\{i: v_i \not=0 \}$ to denote its support and $\|v\|_0=|S(v)|$ as the cardinality of $S(v)$.  Let $\Vert \, . \, \Vert_1, \Vert \, . \, \Vert_q, \Vert \, . \, \Vert_\infty$ denote the vector $\ell_1$-, $\ell_q$- and $\ell_\infty$-norm respectively where $q>1$.  For any subset $S \subset \{1, \dots, n\}$, denote $v_S \in \R^n$ such that $(v_S)_i=v_i$ if $i\in S$ and $(v_S)_i=0$ if $i \not\in S$.   When denoting a vector $v=(v_{i,j})_{1\le i<j \le n} \in \R^{\binom{n}{2}}$, we number its elements as  $v=(v_{12}, v_{13}, \ldots, v_{n-1,n})$. Likewise for a matrix $B \in \R^{\binom{n}{2}\times p}$, we number its rows as $B_{ij}^T, i<j$ in a similar manner. 
Thus, we can define $Z=(Z_{ij})_{i<j \leq n}$ with its $(ij)$th row being $Z_{ij}^T$. 
For brevity, we denote the set of parameters collectively as $\theta = (\beta^T, \mu, \gamma^{T})^T $  and its true value as $\theta_0 = (\beta^{T}_0, \mu_0, \gamma^{T}_0)^T$. We write $S_0=S(\beta_0)$ as the support of $\beta_0$. For ease of presentation, we introduce the shorthand notation $s_0 = \vert S_0 \vert$ and $S_{0,+} \coloneqq S_0 \cup \{n+1, n+2, \dots n+1+p\}$ with cardinality $s_{0,+} = \vert S_{0,+} \vert = s_0 + p+1$ to refer to all active indices including $\mu_0$ and $\gamma_0$. Finally, we denote $e_i\in \R^n$ as the $i$th basis vector with its $i$th element being one and zero elsewhere.

\section{Sparse \texorpdfstring{$\beta$}{beta}-Regression Model}\label{Sec: SBetaM covariates}

Given an observed adjacency matrix $A$ and the associated covariates $\{Z_{ij} \}$, the negative log-likelihood of the \sbmc  is
\begin{equation}\label{negative log-likelihood}
\mathcal{L}(\theta)=\mathcal{L}(\beta, \mu, \gamma) =  \sum_{i <j } -A_{ij}[(e_i+e_j)^T\beta+\mu+Z_{ij}^T\gamma] + \log(1 + e^{(e_i+e_j)^T\beta + \mu + Z_{ij}^T\gamma)}).
\end{equation}

It is easily seen by differentiating that
$\theta_0 = \argmin_{\theta \in \Theta} \mathbb{E}[\mathcal{L}(\theta)]$. 
Since $\beta$ is assumed sparse, one approach for estimation is to minimize the loss in \eqref{negative log-likelihood} subject to an $\ell_0$ penalty on $\beta$. For the sparse $\beta$-model without covariates, \cite{Chen:etal:19} indeed found that this non-convex optimization problem is computationally tractable, thanks to a key monotonicity lemma. Roughly speaking, they showed that in their setting, nodes with the same degree can be treated as equivalent, reducing the number of heterogeneity parameters that have to be estimated to the number of distinct observed degrees. This no longer holds once there are covariates attached to each node, meaning this lemma does not extend to the current setting.  This simple observation motivates the use of an $\ell_1$ penalty on $\beta$ to encourage a sparse solution, immediately connecting our approach to the LASSO methodology \citep{Tibshirani:96} developed for variable selection. This connection enables us to draw upon the vast literature on high-dimensional data analysis, especially for logistic regression. In particular, we can leverage existing algorithms developed for LASSO. For this work, we use the functions in the \texttt{glmnet} \textsf{R} package \citep{glmnet} by properly setting up the design matrix and the constraints on $\beta$.
  
The design matrix corresponding to $\theta$ in \eqref{negative log-likelihood}, denoted as $D\in\R^{\binom{n}{2}\times(n+1+p)}$ for the moment, admits a simple form in that its $(ij)$th row is $(X_{ij}^{T},1,Z_{ij}^T)$, where $X_{ij}^{T}:=e_i^T+e_j^T$. Here we see a crucial feature of this design matrix: While the parameters $\mu$ and $\gamma$ appear in the link probability of all $\binom{n}{2}$ node pairs, each $\beta_i$ only appears in $(n-1)$ such probabilities. 
That means, while the effective sample size for $\mu$ and $\gamma$ is $\binom{n}{2}$, it is only $n-1$ for each entry of $\beta$, i.e. it is of order $n$ smaller. This is also reflected in the different rates of convergence we obtain in Theorem \ref{Cor: no approximation error} below. Since the Gram matrix $D^TD$ plays a pivotal role in studying the estimation of $\theta$ as in logistic regression, we scale the columns of $D$ such that the effective sample size of $\beta$ is comparable to that of $\mu$ and $\gamma$. As we will see later, this scaling has the effect of making the population Gram matrix of the re-scaled design matrix well behaved in that its eigenvalues are bounded away from zero and infinity after normalization. In particular, we write our scaled design matrix as 
\begin{equation}\label{Eq: Def D}
	\bar{D} = [\bar{X}, \textbf{1}, Z]
	 \in \R^{\binom{n}{2}\times (n+p+1)},
\end{equation}
where $\textbf{1} \in \R^{\binom{n}{2}}$ is the vector containing only ones and the $(ij)$th row of $\bar{X}$ is $\bar{X}_{ij}^T=\frac{\sqrt{n}}{\sqrt{2}}X_{ij}^T$. 

Our proposed sparse $\beta$-regression model simply solves the following 
\begin{equation}\label{eq:l1pl}
 \min_{\beta\in \R_+^n,\mu\in \R, \gamma\in \R^p} ~\frac{1}{\binom{n}{2}}\mathcal{L}(\beta, \mu, \gamma) + \lambda \Vert \beta \Vert_1,
\end{equation}
where $\lambda$ is a tuning parameter.

\subsection{Theory}\label{Sec: Theory}
We focus on the finite-dimensional covariate case by assuming that $p$, the dimension of the covariates $Z_{ij}$, is fixed. 
We assume that $Z_{ij}$ are independent realizations from centered, uniformly bounded random variables. The random design assumption of $Z_{ij}$ is somewhat more interesting than a fixed design one and our results can be readily extended to the latter. 
We do not require $Z_{ij}$  to be i.i.d. and $Z_{ij}$ may have correlated entries. These assumptions imply in particular, that there exist constants $\kappa, c > 0$ such that $\vert Z_{ij}^T\gamma_0 \vert \le \kappa$ for all $1 \le i < j \le n$ and $\vert Z_{ij,k} \vert \le c$ for all $1 \le i < j \le n, k = 1, \dots, p$.
We assume further that $\gamma_0$, the homophily parameter associated with $Z_{ij}$, lies in a compact, convex set $\Gamma \subset \R^p$, which means we may choose a universal $\kappa$ independent of $\gamma_0$. Recalling the notation $\theta = (\beta^T, \mu, \gamma^{T})^T$, we let $\Theta \coloneqq \R_+^n \times \R \times \Gamma$ denote the parameter space. 

Since we aim to develop a theory for sparse networks, we allow $\mu_0 \rightarrow -\infty$ as $n \rightarrow \infty$. As a result, some link probabilities may go to zero as $n \rightarrow \infty$. In order to perform consistent estimation, it is clear that we need to restrict the rate at which this may happen. Therefore, we assume there is a non-random sequence $1/2 \ge \rho_{n,0} > 0$, $\rho_{n,0} \rightarrow 0$, as $n \rightarrow \infty$, such that almost surely for all $i, j$: $1 - \rho_{n,0} \ge p_{ij} \ge \rho_{n,0}.$ 
Since a smaller $\rho_{n,0}$ allows sparser networks, we  refer to $\rho_{n,0}$ as the \textit{network sparsity parameter}. It effectively characterizes the maximum permissible sparsity of our network.
Applying $\text{logit}(x) = \log(x/(1-x))$ to the inequality above we get for all $i,j$
\[- \text{logit}(\rho_{n,0}) = \text{logit}(1 - \rho_{n,0}) \ge \beta_{0,i} + \beta_{0,j} + \mu_0 + \gamma_0^TZ_{ij} \ge \text{logit}(\rho_{n,0}),\]
which is equivalent to 
$\vert  \beta_{0,i} + \beta_{0,j} + \mu_0 + \gamma_0^TZ_{ij} \vert \le - \text{logit}(\rho_{n,0}) \eqqcolon r_{n,0}, ~ \forall i,j.$ 
Note that since $\rho_n \le 1/2$, we have $r_{n,0} \ge 0$.
The previous inequality can also be expressed in terms of the design matrix $D$ associated with the corresponding logistic regression problem as $
\Vert D\theta_0 \Vert_\infty \le r_{n,0}$. 
This motivates the following procedure: Given a sufficiently large constant $r_n$, we define the local parameter space
$\Theta_{\text{loc}} = \Theta_{\text{loc}}(r_n) \coloneqq \left\{ \theta \in \Theta : \Vert D\theta \Vert_\infty \le r_n  \right\}$ 
and perform estimation via
\begin{equation}\label{Eq: Penalized llhd with covariates}
\hat{\theta} = (\hat{\beta}^T, \hat{\mu}, \hat{\gamma}^T)^T = \argmin_{\theta = (\beta^T,\mu, \gamma^T)^T \in \Theta_{\text{loc}}} \frac{1}{\binom{n}{2}}\mathcal{L}(\beta, \mu, \gamma) + \lambda \Vert \beta \Vert_1.
\end{equation}
We remark that the formulation above is needed for technical reasons only when it comes to prove the existence of the estimator as shown in the lemma below. In practice, \eqref{eq:l1pl} is used for computing. 
In \eqref{Eq: Penalized llhd with covariates}, we have replaced the condition $\min_i \beta_i = 0$ by the less strict condition $\beta \in \R^n_+$. The following Lemma shows that as long as the observed graph is neither empty nor complete, for any $\lambda > 0$, a solution $\hat \beta$ to (\ref{Eq: Penalized llhd with covariates}) always exists and automatically fulfills $\min_{1 \le i \le n} \hat{\beta}_i = 0$.
\begin{Lem}\label{Lem: Existence of solution and identifiability}
	Assume that $0 < d_+ < \binom{n}{2}$. Then, for any $0 < \lambda < \infty$ there exists a minimizer for the  optimization problem (\ref{Eq: Penalized llhd with covariates}) and any solution $\hat{\theta} = (\hat{\beta}^T, \hat{\mu}, \hat{\gamma}^T)^T$ of (\ref{Eq: Penalized llhd with covariates}) must satisfy $\min_{1 \le i \le n} \hat{\beta}_i = 0$.
\end{Lem}
Following the empirical risk literature (cf. \cite{GreenshteinRitov2004}, \cite{koltchinskii2011}) we will analyze the performance of our estimator in terms of {excess risk} which is defined as
$	\mathcal{E}(\theta) \coloneqq \frac{1}{\binom{n}{2}} \E[\mathcal{L}(\theta) - \mathcal{L}(\theta_0)]. $
For now, we will assume $r_n \ge r_{n,0}$ so that $\theta_0 = \argmin_{\theta \in \Theta_{\text{loc}}} \frac{1}{\binom{n}{2}}\mathbb{E}[\mathcal{L}(\theta)]$, which is the most interesting scenario. We make the following standard assumption of the random design covariates first. 

\begin{Assum}\label{Assum: minimum EW}\label{Assum: maximum EW}
	There is a universal constant $c_{\min}> 0$ such that for all $n \in \mathbb{N}$, 
	the minimum eigenvalue $\lambda_{\text{min}}$
	and the maximum eigenvalue $\lambda_{\max}$ of $\frac{1}{\binom{n}{2}}\E[Z^T Z]$ fulfil
	$c_{\min} \le \lambda_{\text{min}} \le \lambda_{\max} \le 1/c_{\min} < \infty$. 
\end{Assum}
This assumption is standard as it effectively states that the population covariance matrix of $Z$, which is fixed dimensional, is positive definite. A crucial assumption needed in LASSO theory is the so called {compatibility condition} \citep{vandegeer2011, vandegeer2014} by relating the quantities $\Vert (\hat{\theta} - \theta_0)_{S_{0,+}}\Vert_1$ and
\[
(\hat{\theta} - \theta_0)^T \left(\frac{1}{\binom{n}{2}} \E[\bar{D}^T\bar{D}] \right) (\hat \theta - \theta_0)
\]
in a suitable sense made precise below. Define 
\begin{equation}
\Sigma \coloneqq \frac{1}{\binom{n}{2}} \E [\bar{D}^T\bar{D}] = \frac{1}{\binom{n}{2}} \begin{bmatrix}
{ \bar{X}^T\bar{X} } &{  \bar{X}^T\textbf{1} } & \textbf{0} \\
{\textbf{1}^T\bar{X}} & {\textbf{1}^T\textbf{1} }& \textbf{0} \\
\textbf{0} & \textbf{0} & \E[{Z^TZ}].
\end{bmatrix}.\label{Eq: adjusted Gram fraction outside}
\end{equation}
We present the compatibility condition for our model in the following proposition.
\begin{Prop}\label{Prop: compatibility condition Sigma}
	Under Assumption \ref{Assum: minimum EW}, for $s_{0,+} = o(\sqrt{n})$ and $n$ large enough,
	it holds that for every ${\theta} \in \R^{n+1+p}$ with $\Vert {\theta}_{S^{c}_{0,+}} \Vert_1 \le 3 \Vert {\theta}_{S_{0,+}} \Vert_1$,
	\[
		\Vert {\theta}_{S_{0,+}} \Vert_1^2 \le \frac{2s_{0,+}}{c_{\min}} {\theta}^T \Sigma { \theta}.
	\]
\end{Prop}
Proposition \ref{Prop: compatibility condition Sigma} requires $s_{0,+} = o(\sqrt{n})$. The ``$n$ large enough''-condition is made precise in the proof and requires that $n$ be such that $1/ \sqrt{n} < 1/s_{0,+}$, which is implied by $s_{0,+} = o(\sqrt{n})$ and sufficiently large $n$.  Let us put this in the context of general LASSO theory in which to show the $\ell_1$-error going to zero in probability, it is imposed that the sparsity $s$ of the true parameter fulfils
\[
s \cdot \sqrt{\frac{\log(\text{ number of columns of design matrix })}{\text{effective sample size}}} ~{\rightarrow} ~0
\]
where $n \rightarrow \infty$; 
see for example \cite{vandegeer2011}, Chapter 6.
In our case the sparsity refers to $\beta$ and we thus should expect that the restrictions we have to impose on $s_0$  are based on the sample size associated with $\beta$.
We make the following assumption on $s_{0,+}$, the sparsity of $\theta_0$ when $r_{n, 0}\le r_n$.

\begin{Assum}\label{Assum: rate of s_0}
$s_{0,+}\frac{\sqrt{\log(n)}}{\sqrt{n}\rho_n} \rightarrow 0, n \rightarrow \infty$.
\end{Assum}
\noindent This assumption implies that, up to an additional factor $\rho_n^{-1}$ -- which is the price we have to pay for allowing our link probabilities to go to zero -- the permissible sparsity for $\beta_0$
is the permissible sparsity in classical LASSO theory for an effective sample size of order $n$. 
We state our first main theorem. 
\begin{Satz}\label{Cor: no approximation error}
	Assume Assumptions \ref{Assum: minimum EW} and \ref{Assum: rate of s_0}. 
	Fix a confidence level $t$ and let 
	\[
	a_n \coloneqq \sqrt{\frac{2\log(2(n+p+1))}{\binom{n}{2}}} (1 \vee c).
	\]
	Let $\bar{\lambda} =  \frac{\sqrt{n}}{\sqrt{2}}\lambda$ such that
	\[
		\bar{\lambda}  \ge 8 \cdot \left( 8a_n + 2 \sqrt{\frac{t}{\binom{n}{2}}( 11 (1 \vee (c^2p) ) + 8\sqrt{2}(1 \vee c) \sqrt{n} a_n  )} + \frac{2\sqrt{2}t(1 \vee c) \sqrt{n}}{3\binom{n}{2}}\right).
	\]
	Then, with probability at least $1 - \exp(-t)$ we have
	\begin{equation*}
	\mathcal{E}(\hat{\theta}) + \bar{{\lambda}}\left( \frac{\sqrt{2}}{\sqrt{n}} \Vert \hat{\beta} - \beta_0  \Vert_1 + \vert \hat{\mu} - \mu_0 \vert + \Vert \hat{\gamma} - \gamma_0 \Vert_1  \right) \le C  \frac{s_{0,+}{\bar{\lambda}}^2}{\rho_{n}}
	\end{equation*}
	with constant $C = 128/c_{\min}$.
	\end{Satz}

Theorem \ref{Cor: no approximation error} gives us an explicit formula for how the sparsity of our network will affect our rate of convergence, which is particularly insightful, since in many related works the conditions on network density enter the rate of convergence only indirectly as assumptions on the norm of the true parameter vector, see for example \cite{Chatterjee:etal:2011, Yan:Xu:2013}. Also, this is essentially the rate of convergence we would expect in the classical LASSO setting for logistic regression up to an additional factor {$\rho_{n}^{-1}$} (cf. \cite{vandegeer2011}). Recall that in the classical LASSO setting, when the model is correctly specified, probabilities stay bounded away from zero, and we have the same effective sample size for each parameter, we obtain the rates
\[
O_P\left( \text{sparsity} \cdot \frac{\log(\text{number of columns of design matrix})}{\text{effective sample size}} \right)
\]
for the excess risk and 
\[
O_P\left( \text{sparsity} \cdot \sqrt{\frac{\log(\text{number of columns of design matrix})}{\text{effective sample size}} }\right)
\]
for the $\ell_1$-error.
In the setting of Theorem \ref{Cor: no approximation error}, choosing  $\bar{\lambda}$ of the order $\sqrt{{\log(n)}/{\binom{n}{2}}}$, we obtain
\begin{align*}
	\mathcal{E}(\hat{\theta}) &= O_P\left(s_{0,+} \cdot \frac{1}{\rho_{n} } \cdot \frac{\log(n)}{\binom{n}{2}}\right), \\
	\vert \hat{\mu} - \mu_0 \vert + \Vert \hat{\gamma} - \gamma_0 \Vert_1 &= O_P\left(  s_{0,+} \cdot \frac{1}{\rho_{n} } \cdot \sqrt{\frac{\log(n)}{\binom{n}{2}}} \right), \\
	\Vert \hat{\beta} - \beta_0  \Vert_1 &=  O_P\left(  s_{0,+} \cdot \frac{1}{\rho_{n} } \cdot \frac{\sqrt{\log(n)}}{\sqrt{n-1}} \right).
\end{align*} 
That is, up to an additional factor {$\rho_{n}^{-1}$}, we obtain the LASSO rate of convergence for sample size $\binom{n}{2}$ for the global excess risk. The excess risk $\mathcal{E}(\hat{\theta})$ measures the predictive performance of an estimator which may be more meaningful in case estimating individual parameters is inaccurate due to collinearity. By the second line of the display above, we have immediately $\hat{\mu} \overset{P}{\rightarrow} \mu_0$ and $\hat{\gamma} \overset{P}{\rightarrow} \gamma_0$ at the rate expected from a LASSO type estimator with effective sample size $\binom{n}{2}$ (up to an additional factor). 
Furthermore, the third line implies that, again, up to an additional factor, for the error of $\hat{\beta}$, we obtain the rate of convergence we would expect for a LASSO type estimator with sample size $n-1$. 

In particular, the assumptions we have to impose to obtain $\ell_1$-consistency include the case $\Vert \beta_0 \Vert_\infty = o(\log(\log(n)))$, which is the condition that had to be imposed in the original $\beta$-model for their strong consistency result (cf. \cite{Yan:Xu:2013}, Theorem 1). Note that by setting $\gamma=0$, our proposed $\ell_1$ penalized likelihood method can also handle the model in \cite{Chen:etal:19} where they used the $\ell_0$ penalty for estimation. The comparison of our estimator with the one in \cite{Chen:etal:19} in the Supplementary Material indicates that the estimator proposed in this paper is preferable when the network is relatively dense. \color{black}{\cite{shao20212} considered an $\ell_2$ penalized MLE in the $\beta$-model, establishing the first estimation optimality results in the $\beta$-model literature that only require the network sparsity $\rho_n \gg n^{-1}$. We refer to \cite{shao20212} for a detailed comparison of different penalization methods for the $\beta$-model.} Next we give the $\ell_q$-error bound ($1 < q \leq 2$) for $\hat{\theta}$.

\begin{Prop}\label{Prop: q error bound}
	Under conditions of Theorem \ref{Cor: no approximation error}, for $1 < q \leq 2$ and a confidence level $t$, with probability at least $1 - \exp(-t)$ we have
	\begin{align*}
	 \Vert \frac{\sqrt{2}}{\sqrt{n}}(\hat{\beta} - \beta_0)  \Vert^{q}_q + \vert \hat{\mu} - \mu_0 \vert^{q} + \Vert \hat{\gamma} - \gamma_0 \Vert^{q}_q  \le \left(4^q+2^{q+1}\right)s_{0,+}\left(\frac{C{\bar{\lambda}}}{\rho_n}\right)^q,
	\end{align*}	
where $C =64/c_{\min}.$
\end{Prop}

We have discussed the scenario when $r_n \ge r_{n, 0}$. In practice however, it may happen that $r_n < r_{n,0}$. In this case, we define the {best local approximation} $\theta^*$ of the true $\theta_0$ as
\begin{equation*}
	\theta^* = (\beta^{*T}, \mu^*, \gamma^{*T})^T=\argmin_{\theta \in \Theta_{\text{loc}}} \frac{1}{\binom{n}{2}}\mathbb{E}[\mathcal{L}(\theta)].
\end{equation*}
If $r_{n} < r_{n,0}$, it may happen $\theta^* \not= \theta_0$ and thus estimating $\theta^*$ is the best we can achieve when solving (\ref{Eq: Penalized llhd with covariates}). To analyze the corresponding estimator, we resort to the notion of {local excess risk} as in \cite{Chen:etal:19}, which measures how close a parameter $\theta$ is to the best local approximation $\theta^*$ in terms of excess risk:
\begin{equation*}
	\mathcal{E}_{\text{loc}}(\theta) \coloneqq \mathcal{E}(\theta) - \mathcal{E}(\theta^*).
\end{equation*}
Clearly, $\theta^*$ also fulfils
$
\theta^* = \argmin_{\theta \in \Theta_{\text{loc}}} \mathcal{E}(\theta)
$
and we may consider the excess risk of the best local approximation, $\mathcal{E}(\theta^*)$, as the approximation error of our model. It accounts for the fact that our model might be misspecified, in the sense that the parameter $r_n$ is not large enough. As is usual in LASSO theory (cf. \cite{vandegeer2011}, Chapter 6), it is tacitly assumed that this approximation error is small, i.e. $r_n$ is sufficiently large. Note that the global excess risk of our estimator $\hat{\theta}$ decomposes as 
$
\mathcal{E}(\hat{\theta}) = \mathcal{E}(\theta^*) + \mathcal{E}_{\text{loc}}(\hat{\theta}),
$
where we can consider the approximation error $\mathcal{E}(\theta^*)$ as a deterministic bias. Define
\begin{equation}\label{Eq: Def K_n}
	K_n = \frac{2(1 + \exp( \max\{r_{n,0}  , r_{n} \}))^2}{\exp(\max\{r_{n,0}  ,  r_{n}   \})}.
\end{equation}
\begin{Assum}\label{Assum: rate of s^*}
$s^*_+\frac{\sqrt{\log(n)}K_n}{\sqrt{n}} \rightarrow 0$ as $n \rightarrow \infty$ where $s^*_+=|S(\beta^*)|+p+1$.
\end{Assum}
Direct calculation shows that when $r_n \ge r_{n, 0}$, $K_n$ is of the order $\rho_{n}^{-1}$, which means Assumption \ref{Assum: rate of s^*} can be seen as a general version of Assumption \ref{Assum: rate of s_0} under model misspecification. We have the following theorem in this scenario.
\begin{Satz}\label{Thm: consistency}
	Under Assumptions \ref{Assum: minimum EW} and \ref{Assum: rate of s^*}, with probability at least $1 - \exp(-t)$ we have
	\begin{align*}
	\mathcal{E}(\hat{\theta}) + {\bar{\lambda}} \left( \frac{\sqrt{2}}{\sqrt{n}} \Vert \hat{\beta} - \beta^*  \Vert_1 + \vert \hat{\mu} - \mu^* \vert + \Vert \hat{\gamma} - \gamma^* \Vert_1  \right) \le 6 \mathcal{E}(\theta^*) + 32  \frac{s^*_+K_n{\bar{\lambda}}^2}{c_{\min}}.
	\end{align*}
\end{Satz}

\section{\color{black}Asymptotic Normality}\label{Sec: Inference}

\color{black}{In this section, we consider the statistical inference of the homophily parameter $\gamma$ and the degree heterogeneity parameter $\beta$ when $r_n \ge r_{n,0}$ and thus $\theta^* = \theta_0$. We first study the limiting distribution for $\hat \gamma$.} We will see that the same arguments used for deriving the limiting distribution for $\hat \gamma$ also work for $\hat \mu$ and as a by-product of our proofs we also obtain an analogous limiting result for $\hat \mu$. To ease notation a little we will use $\vartheta = (\mu, \gamma^T)^T$ to refer to the unpenalized parameter subvector of $\theta$. 
Denote by $H(\hat{\theta}) \coloneqq \left.H_{\vartheta \times \vartheta}(\theta)\right|_{\theta = \hat{\theta}}$ the Hessian of $\frac{1}{\binom{n}{2}}\mathcal{L}(\theta)$ with respect to $\vartheta$ only, evaluated at $\hat{\theta}$. Let $\mathbb{E}[H(\theta_0)] $ be the corresponding population version. To be consistent with commonly used notation, call $\hat{\Sigma}_\vartheta = H(\hat \theta)$ and $\Sigma_\vartheta = \mathbb{E}[H(\theta_0)]$ and
$
\hat \Theta_\vartheta \coloneqq \hat \Sigma_\vartheta^{-1}, \Theta_\vartheta \coloneqq \Sigma_\vartheta^{-1}.
$

We will need to invert $\hat \Sigma_\vartheta$ and $\Sigma_\vartheta$ and show that these inverses are close to each other in an appropriate sense. It is commonly assumed in LASSO theory (cf.~\cite{vandegeer2014}) that the minimum eigenvalues of these matrices stay bounded away from zero. In our case, however, such an assumption is invalid, as we demonstrate in the Supplementary Material. Therefore, a careful argument is needed and we have to impose stricter assumptions than for our consistency result alone.
\begin{Assum}\label{Assum: new rate of s and rho_n}
	$s_{0,+}\frac{\sqrt{\log(n)}}{\sqrt{n}\rho_n^2} \rightarrow 0, n \rightarrow \infty$.
\end{Assum}

Assumption \ref{Assum: new rate of s and rho_n} is a slightly stricter version of the previously imposed Assumption \ref{Assum: rate of s_0}. Previously we only needed a factor of $\rho_n^{-1}$ to ensure that the $\ell_1$-error for $\hat \beta$ goes to zero. 
Notice, though, that these assumptions still allow sparsity rates for {$\rho_{n}$} of small polynomial order. More precisely, up to a $\log$-factor and depending on the speed of $s_{0,+}$, { $\rho_{n}$} may still go to zero at a speed of order up to $n^{-1/4}$. 

\begin{Satz}\label{Thm: inference}
	Under Assumptions \ref{Assum: minimum EW} and \ref{Assum: new rate of s and rho_n}, when $\theta^* = \theta_0$, we have for any $k = 1, \dots, p$, as $n \rightarrow \infty$,
	\begin{equation*}
	\sqrt{\binom{n}{2}}\frac{\hat \gamma_k - \gamma_{0,k}}{\sqrt{\hat \Theta_{\vartheta,k+1,k+1}}} \overset{d}{\longrightarrow} \mathcal{N}(0,1).
	\end{equation*}
	We also have for our estimator of the global sparsity parameter, $\hat \mu$, as $n \rightarrow \infty$,
	\begin{equation*}
		\sqrt{\binom{n}{2}}\frac{\hat \mu - \mu_0}{\sqrt{\hat \Theta_{\vartheta,1,1}}} \overset{d}{\longrightarrow} \mathcal{N}(0,1).
	\end{equation*} 
\end{Satz}

Theorem \ref{Thm: inference} states that inference on $\hat \mu$ and $\hat \gamma$ can be conducted directly after fitting our model. It is in stark contrast to the usual LASSO estimates, where a separate debiasing step must be carried out for the correct inference of model parameters due to the bias incurred by shrinkage  \citep{zhang:zhang:2014,vandegeer2014}. 
This bias is made explicit in equation \eqref{Eq: subdifferential first order equation} in the Supplementary Material: The penalized parameter values do not fulfill the first-order estimating equations exactly, but rather a bias of the form $\lambda v$ is incurred as prescribed by subdifferential calculus. While the unpenalized parameter estimates $(\hat \mu, \hat \gamma^T)^T$ do fulfill the first-order estimating equations exactly, in standard settings, this alone would still not be enough to ensure the asymptotic normality of $\hat \vartheta$. \color{black}{However, in our special case, due to the differing sample sizes between $\beta$ and $\vartheta$,  the bias incurred from the part of the likelihood relating to $\beta$ vanishes in probability, which allows us to derive a limiting distribution for $(\hat \mu, \hat \gamma^T)^T$ without a debiasing step. We note that Assumption 4 in Theorem \ref{Thm: inference} can be relaxed to Assumption \ref{Assum: rate of s_0}, thus permitting network sparsity to be of the order $n^{-1/2}$, if we allow $p_{ij}$ to be of the same order.}

{\color{black} On the other hand, for the inference of the penalized parameter $\beta$, a debiasing step is required. Let $\hat{W} = \text{diag}(\sqrt{p_{ij}(\hat{\theta})(1-p_{ij}(\hat{\theta}))}, i < j)$ and $\hat{V}_{\beta} = X^{T} \hat{W}^2X$. Due to the specific structure of $\hat{V}_{\beta}$, our debiasing procedure is more straightforward and does not require nodewise regression. Let $\hat{U}_\beta=\text{diag}(1/\hat{V}_{\beta,1,1},1/\hat{V}_{\beta,2,2}, ..., 1/\hat{V}_{\beta,n,n})$. Our debiased estimator for $\beta$ is defined as $\hat{b} := \hat{\beta}-\hat{U}_\beta \nabla_{\beta} \mathcal{L}(\hat{\theta})$. 

\begin{Assum}\label{inference_beta} $s_{0,+}\frac{\log n}{\sqrt{n}\rho^{2}_n} \rightarrow 0$, $n \rightarrow \infty$.
\end{Assum}

Assumption \ref{inference_beta} is slightly stricter than Assumption \ref{Assum: new rate of s and rho_n} due to an additional factor of $\sqrt{\log n}$. Despite this, it still permits the network sparsity $\rho_n$ to approach zero at a rate up to $n^{-1/4}$.

\begin{Satz}\label{beta_AN} Under Assumptions \ref{Assum: minimum EW} and  \ref{inference_beta}, we have, for any 
$k = 1, \dots, n$, as $n \rightarrow \infty$,
\[
	\sqrt{\hat{V}_{\beta,k,k}}(\hat{b}_{k} - \beta_{0,k})\overset{d}{\longrightarrow} N(0,1).
\]
\end{Satz}
Since $(n-1)\rho_n/4 \le \hat{V}_{\beta,k,k} \le (n-1)/4$, Theorem \ref{beta_AN} implies that the convergence rate of each $\hat{b}_k$ is between $n^{-1/2}$ and $(n\rho_n)^{-1/2}$.}


\section{Erd\H{o}s-R\'{e}nyi with Covariates}\label{subsection: generalized ER model}
When $\beta=0$, that is, when we do not consider degree heterogeneity, the linking probability in \sbmc becomes 
\begin{equation}\label{eq:nullmodel2}
	P(A_{ij} = 1|Z_{ij}) = p_{ij} = \frac{\exp(\mu + Z_{ij}^T\gamma)}{1+\exp( \mu + Z_{ij}^T\gamma)},
\end{equation}
which can be seen as a generalized Erd\H{o}s-R\'{e}nyi model when covariates are incorporated. For this reason, we will abbreviate this model as ERC. 

Of course, the properties of the MLE of $\mu$ and $\gamma$ in \eqref{eq:nullmodel2} are standard if both parameters are fixed, but this regime gives rise to dense networks and thus is not very interesting. Instead, we study them under the sparse network regime by reparametrizing $\mu$ as
\[
\mu = -\xi \log(n) + \mu^\dagger,
\]
where $\xi \in [0,2)$ effectively takes the role of $\rho_{n,0}$ from the previous section and $\mu^\dagger \in [-M, M]$ for a fixed $M < \infty$ independent of $n$. This parametrization first appeared in \cite{Krivitsky:Kolaczyk:2015} when the notion of effective sample sizes for network models was discussed, and was further studied in  \cite{Chen:etal:19}. 
To appreciate this reformulation, we see that the expected total number of edges of ERC is of the order $O(n^{2-\xi})$. When $\xi=0$, ERC becomes a standard logistic regression model with fixed parameters.  It can generate arbitrarily sparse networks when $\xi>0$. To the best of our knowledge, a model of this type that accounts for covariates has not been studied in the literature before and thus the results below may be of independent interest.

Denote $\mu_0^\dagger$ and $\gamma_0$ as the true parameters  of $\mu^\dagger$ and $\gamma$ respectively. To present a consistent notation with the other sections, we abuse notation slightly and denote a generic parameter as $\theta = (\mu^\dagger, \gamma)$, the true parameter as $\theta_0 = (\mu_0^\dagger, \gamma_0)$ and our estimator (defined below) as $\hat \theta = (\hat \mu^\dagger, \hat \gamma)$. We make the following assumptions.
\begin{Assum}\label{Assum: MLE in interior}
	The true parameter $\theta_0 = (\mu_0^\dagger, \gamma_0^T)^T$ lies in the interior of $[-M,M] \times \Gamma$.
\end{Assum}

\begin{Assum}\label{Assum: min eval} The $Z_{ij}$ are i.i.d. realizations of the same random variable. 
	The covariance matrix of $Z_{12}$, that is the matrix $\E[Z_{12}Z_{12}^T]$, is strictly positive definite with minimum eigenvalue $\lambda_{\min} > 0$. 
\end{Assum}
Assumption \ref{Assum: min eval} is analogous to Assumption \ref{Assum: minimum EW} in the case with non-zero $\beta$. We remark that the i.i.d. condition is used to simplify parts of the proofs and can be relaxed at the expense of lengthier proofs.

We consider the following function which is proportional to the negative log-likelihood of the \serc up to a summand independent of the parameter 
\begin{equation}\label{Eq: ER-C neg llhd}
	\mathcal{L}^\dagger(\mu^\dagger, \gamma) = -d_+\mu^\dagger - \sum_{i < j} (\gamma^TZ_{ij})A_{ij} + \sum_{i < j} \log\left( 1 + n^{-\xi} \exp(\mu^\dagger + \gamma^TZ_{ij}) \right).
\end{equation}
In the ERC, the dimension of the parameter is fixed. Therefore, it is not necessary to employ a penalized likelihood approach as in the \sbmc and we estimate $\theta$ via maximum likelihood  
\begin{equation}\label{Eq: ER-C max llhd}
	\hat \theta = (\hat \mu^\dagger, \hat \gamma^T)^T = \argmin_{\theta = (\mu^\dagger, \gamma^T)^T} \mathcal{L}^\dagger(\mu^\dagger, \gamma),
\end{equation}
where the argmin is taken over $[-M, M] \times \Gamma$.
The design matrix $D$  takes the simplified form where its $(ij)$th row is $(1,Z_{ij}^T)$. Define the matrix $\Sigma \in \R^{(p+1) \times (p+1)}$ as
\[
\Sigma \coloneqq \E\left[  (D_{12}D_{12}^T) \exp(\mu_0^\dagger) \exp(\gamma_0^TZ_{12}) \right],
\]
which is invertible by Assumption \ref{Assum: min eval}.
We have the following central limit theorem for $\hat \theta$. Denote by $\mathcal{N}(0, \Sigma^{-1})$ the law of the multivariate normal distribution with zero mean vector and covariance matrix $\Sigma^{-1}$.
\begin{Satz}\label{Thm: asymptotic normality ER-C}
	Under Assumptions \ref{Assum: MLE in interior} and \ref{Assum: min eval}, it holds, as $n \rightarrow \infty$,
	\[
	\sqrt{\frac{\binom{n}{2}}{n^{\xi}}} (\hat \theta - \theta_0) \overset{d}{\longrightarrow} \mathcal{N}(0, \Sigma^{-1}).
	\]
\end{Satz}

Since the expected number of observed edges in the \serc is of order $n^{2-\xi}$, the factor $\sqrt{\binom{n}{2}/n^{\xi}}$ in Theorem \ref{Thm: asymptotic normality ER-C} corresponds to the square root of the effective sample size. This means, having the link probabilities go to zero  reduces the information we gain about $\theta_0$ and this information loss is made explicit in a rate of convergence slower than what we would obtain in a classical parametric setting. 

While we consider Theorem \ref{Thm: asymptotic normality ER-C} to be interesting from a theoretical point of view, in practice, the sparsity rate parameter $\xi$ will not be known, which makes solving \eqref{Eq: ER-C max llhd} and finding the MLE $(\hat \mu^\dagger, \hat \gamma)$ impossible. Remarkably, it is possible, though, to circumvent this problem with the following argument.  Let $\hat \mu$ be the MLE of the global sparsity parameter before reparametrization and hence does not require knowledge of $\xi$. Define the matrix
\[
\hat \Sigma = \frac{1}{\binom{n}{2}} D^T\text{diag}\left( \frac{\exp(\hat \mu + \hat \gamma^TZ_{ij})}{(1 + \exp(\hat \mu + \hat \gamma^TZ_{ij}))^2} , i < j \right) D.
\]
Then without having to know $\xi$, we have the following corollary.

\begin{Kor}\label{Cor: Corollary ER-C}
	Under Assumptions \ref{Assum: MLE in interior} and \ref{Assum: min eval}, the following componentwise asymptotic normality results hold as $n \rightarrow \infty$ for $k = 1, \dots, p$,
	\[
	\sqrt{\binom{n}{2}} \cdot \frac{\hat \gamma_k - \gamma_{0,k}}{\sqrt{\hat\Sigma^{-1}_{k+1,k+1}}} \overset{d}{\longrightarrow} \mathcal{N}(0,1).
	\]
\end{Kor}

\noindent Simulation results corroborating the claims in Corollary \ref{Cor: Corollary ER-C} are shown in Section \ref{Sec: serc simulation}. {We remark that all the theoretical results in this section can be generalized to the scenario where there are a diverging number of covariates. See the Supplementary Material for more details.}

\section{Simulation}\label{Sec: Simulation}

\subsection{S\texorpdfstring{$\beta$}{Beta}RM: Sparse \texorpdfstring{$\beta$}{beta}-regression model}

We illustrate the finite sample performance of our penalized likelihood estimator with an extensive set of Monte Carlo simulations. We only show results for \sbmc with the estimator defined in (\ref{eq:l1pl}), as where applicable the results in the case without covariates are very similar. We check both the $\ell_1$-convergence of our parameter estimates to the true parameter, \color{black}{as well as the asymptotic normality of $\hat \gamma$ and $\hat{b}$} by documenting the empirical coverage of $95\%$ confidence intervals.

Since our estimation involves the choice of a tuning parameter, we explored the use of the Bayesian Information Criterion (BIC) for model selection as well as a heuristic based on the theory developed in the previous sections to specify its value. While the former criterion is purely data-driven, the use of the latter is to ensure that our theoretical results are about right in terms of the rates. 
To make the dependence of our estimator (\ref{eq:l1pl}) on the penalty parameter explicit, we denote the solution of (\ref{eq:l1pl}) when using penalty $\lambda$ by $\hat{ \theta}(\lambda) = (\hat \beta(\lambda)^T, \hat \mu (\lambda), \hat \gamma (\lambda)^T)^T$ and write $s(\lambda) = \vert \{ i : \hat \beta_i (\lambda)\} > 0 \vert$ for its sparsity. The value of the BIC at $\lambda$ is given by
\[\text{BIC} = 2 \mathcal{L}(\hat \theta (\lambda)) + s(\lambda) \log(n(n-1)/2)  \]
and the penalty $\lambda$ was chosen to minimize BIC. For the heuristic approach for tuning parameter selection, based on Theorem \ref{Cor: no approximation error}, we set $c$ to the maximum observed covariate value and $t = 2$. With the lower bound on $\bar{\lambda}$ in Theorem \ref{Cor: no approximation error}, we choose 
\[\bar{\lambda} = 8 \cdot \left(8a_n + 2 \sqrt{\frac{t}{\binom{n}{2}}( 11 (1 \vee (c^2p) ) + 8\sqrt{2}(1 \vee c) \sqrt{n} a_n  )} + \frac{2\sqrt{2}t(1 \vee c) \sqrt{n}}{3\binom{n}{2}}\right),\] and let $\lambda = \frac{\sqrt{2}}{\sqrt{n}} \bar{\lambda}$.

For our simulation, we fixed $p=2$ by setting the covariate weights as $\gamma_0 = (1, 0.8)^T$
and generated the covariates from a centered $\text{ Beta }(2,2)$ distribution as $Z_{ij, k} \sim \text{Beta}(2,2) - 1/2$. We consider networks of sizes $n=300, 500, 800$ and $1000$ in which the sparsity of \(\beta_0\)  is set as 7, 9, 10, and 12 respectively. We tested our estimator on three different model configurations with different combinations of $\beta_0$ and $\mu_0$, resulting in networks with varying degrees of sparsity. For each simulation configuration, 1000 data sets are simulated. Specifically,

\noindent
\textbf{Model 1}: We pick \(\beta_0 = (1.2, 0.8, 1, \dots, 1, 0, \dots, 0)^T\), where the number of ones increases with the network size to match the aforementioned sparsity level, and set $\mu_0 = -0.5 \cdot \log(\log(n))$;

\noindent
\textbf{Model 2}: We pick \(\beta_0 = \log(\log(n)) \cdot (1.2, 0.8, 1, \dots, 1, 0, \dots, 0)^T\) and set $\mu_0 = -1.2\cdot\log(\log(n))$;

\noindent
\textbf{Model 3}: We pick \(\beta_0 = \log(\log(n)) \cdot (2, 0.8, 1, \dots, 1, 0, \dots, 0)^T\) and set {\(\mu_0 = -0.2\cdot\log(n)\)}. 

In these three models, we allow $\mu_0$ to get progressively more negative to generate networks that are increasingly sparse, and allow the sparsity of $\beta$ to increase with network size $n$. 
All three models get progressively sparser with increasing $n$ {and at the same time satisfy Assumption \ref{Assum: new rate of s and rho_n}. We also conducted additional simulations for sparser networks beyond the limit of Assumption \ref{Assum: new rate of s and rho_n}. The results can be found in the Supplementary Material.}  

\noindent\textbf{Consistency:}
We calculated the mean absolute error (MAE) for estimating $\beta_0$, the absolute error for estimating $\mu_0$ and the $\ell_1$-error for estimating $\gamma_0$.
For Model 1 the results are shown in Figures \ref{Fig: beta MAE}--\ref{Fig: l1 gamma}. While BIC performs slightly better for estimating $\beta_0$ and $\mu_0$ for smaller network sizes, our heuristic performs better for larger network sizes. The $\ell_1$-error for estimating $\gamma$ is almost the same between both model selection schemes across all network sizes. For both methods we can see that the various errors decrease with increasing network size. {Model 2 and Model 3 give similar results as Model 1. The various errors for parameter estimation are shown in Figures \ref{Fig: beta MAE model 2}--\ref{Fig: l1 gamma model 2} and Figures \ref{Fig: beta MAE model 3}--\ref{Fig: l1 gamma model 3} respectively.}

\begin{figure}[H]
	\centering
	\begin{subfigure}{0.32\textwidth}
		\centering
		\includegraphics[scale=0.22]{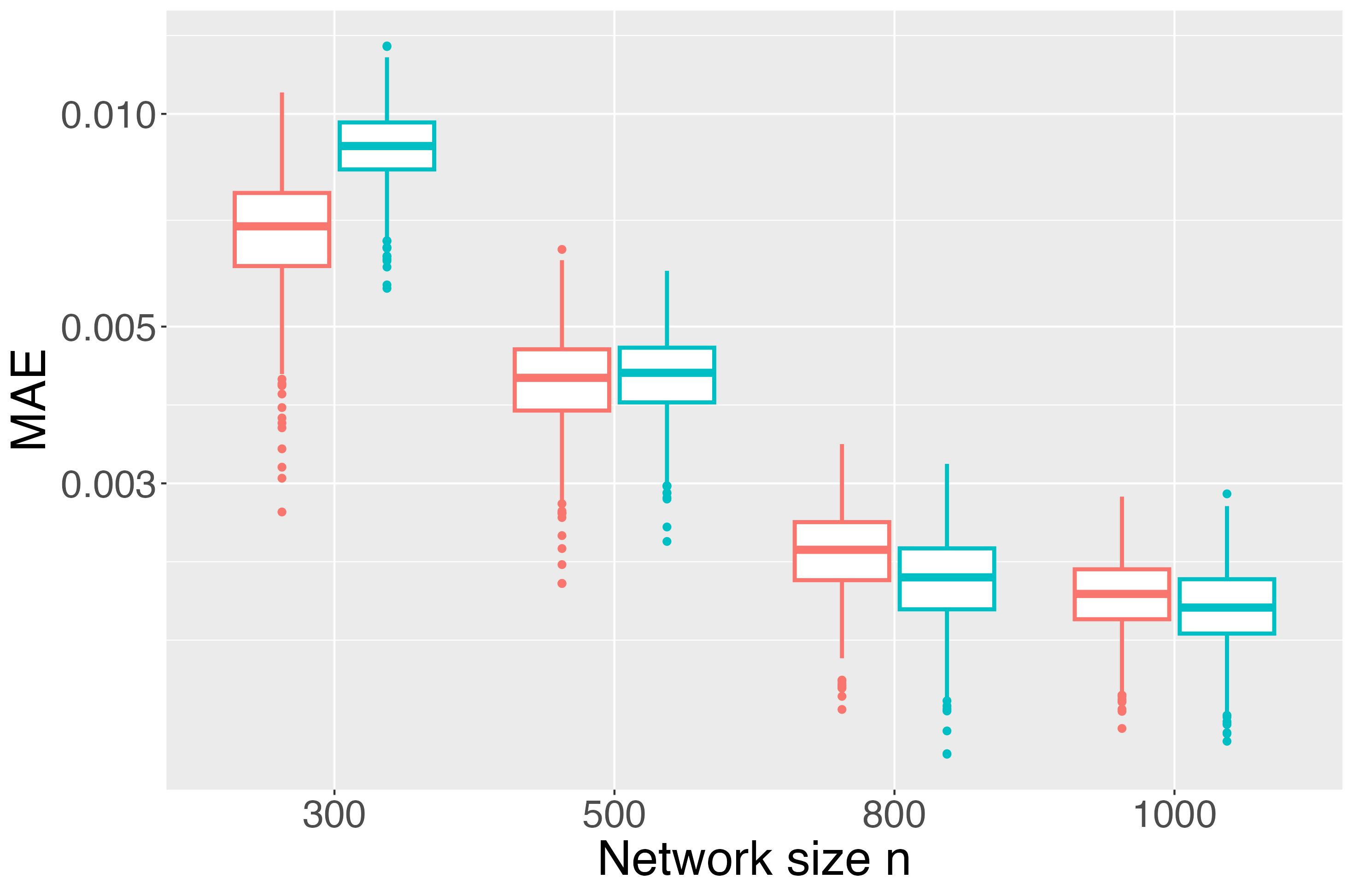}
		\caption{MAE for $\beta_0$}
		\label{Fig: beta MAE}
	\end{subfigure}%
	\begin{subfigure}{0.32\textwidth}
		\centering
		\includegraphics[scale=0.22]{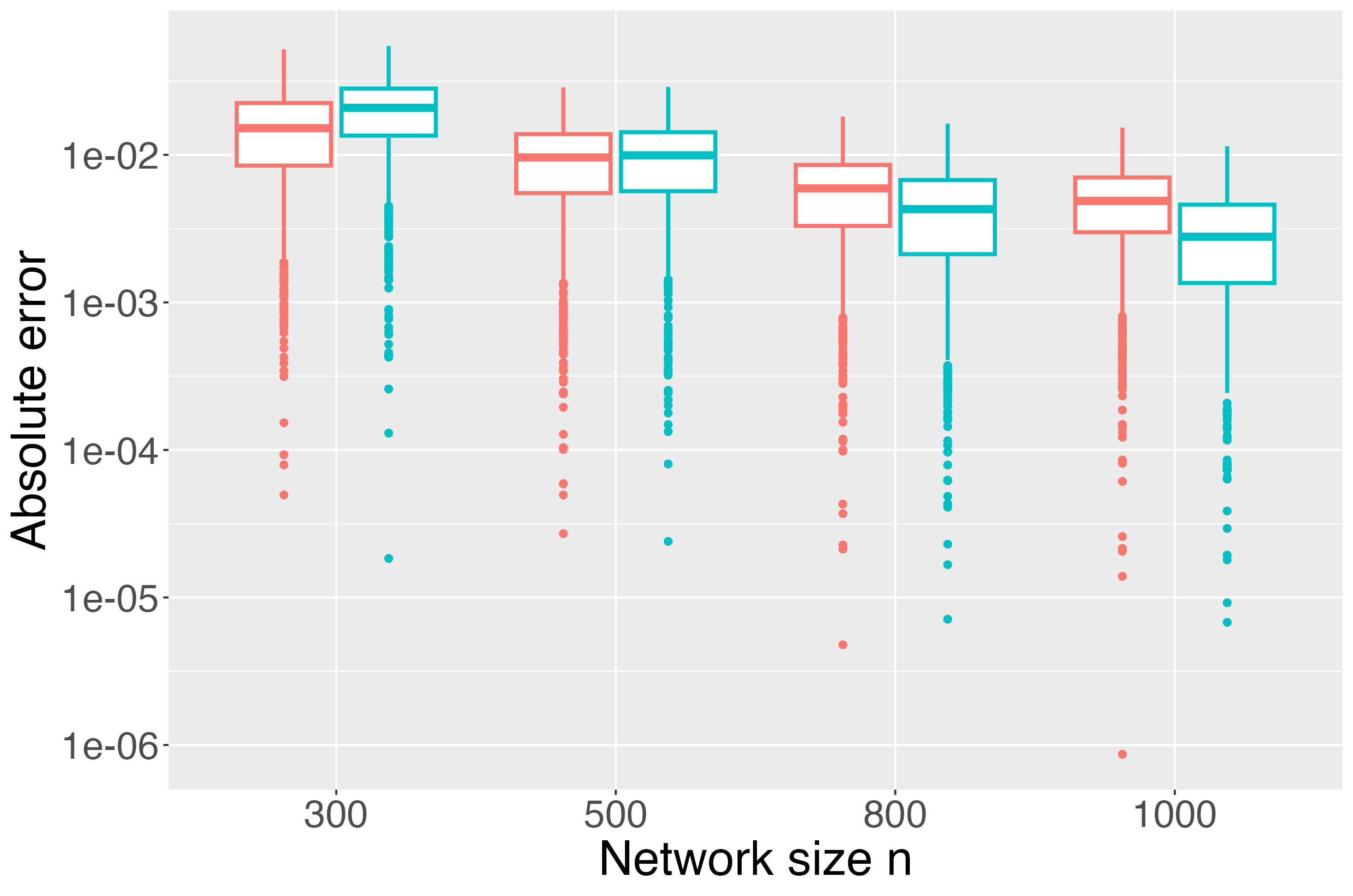}
		\caption{Absolute error for  $\mu_0$.}
		\label{Fig: abs mu}
	\end{subfigure}
	\begin{subfigure}{0.32\textwidth}
		\centering
		\includegraphics[scale=0.22]{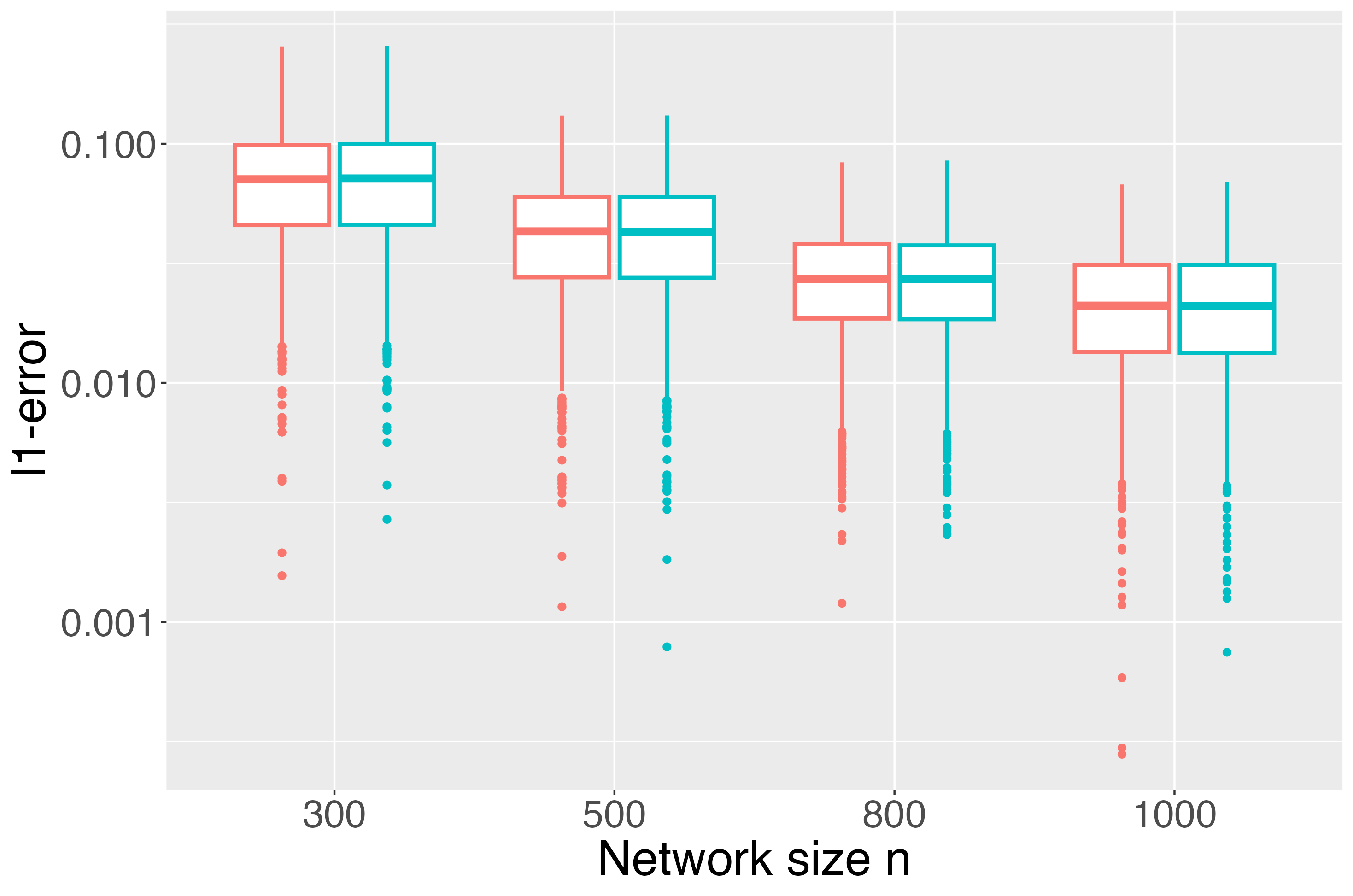}
		\caption{$\ell_1$-error for  $\gamma_0$.}
		\label{Fig: l1 gamma}
	\end{subfigure}
	\caption{Errors for estimating the true parameter $\theta_0$ in Model 1 across various network sizes and 1000 repetitions. Comparison between model selection via BIC and a heuristic approach. The results when model selection is done with BIC are displayed in red (left boxes), those for the pre-determined $\lambda$ in green (right boxes). The $y$-axis uses a log scale.}
\end{figure}

\begin{figure}[!htbp]
	\centering
	\begin{subfigure}{0.32\textwidth}
		\centering
		\includegraphics[scale=0.22]{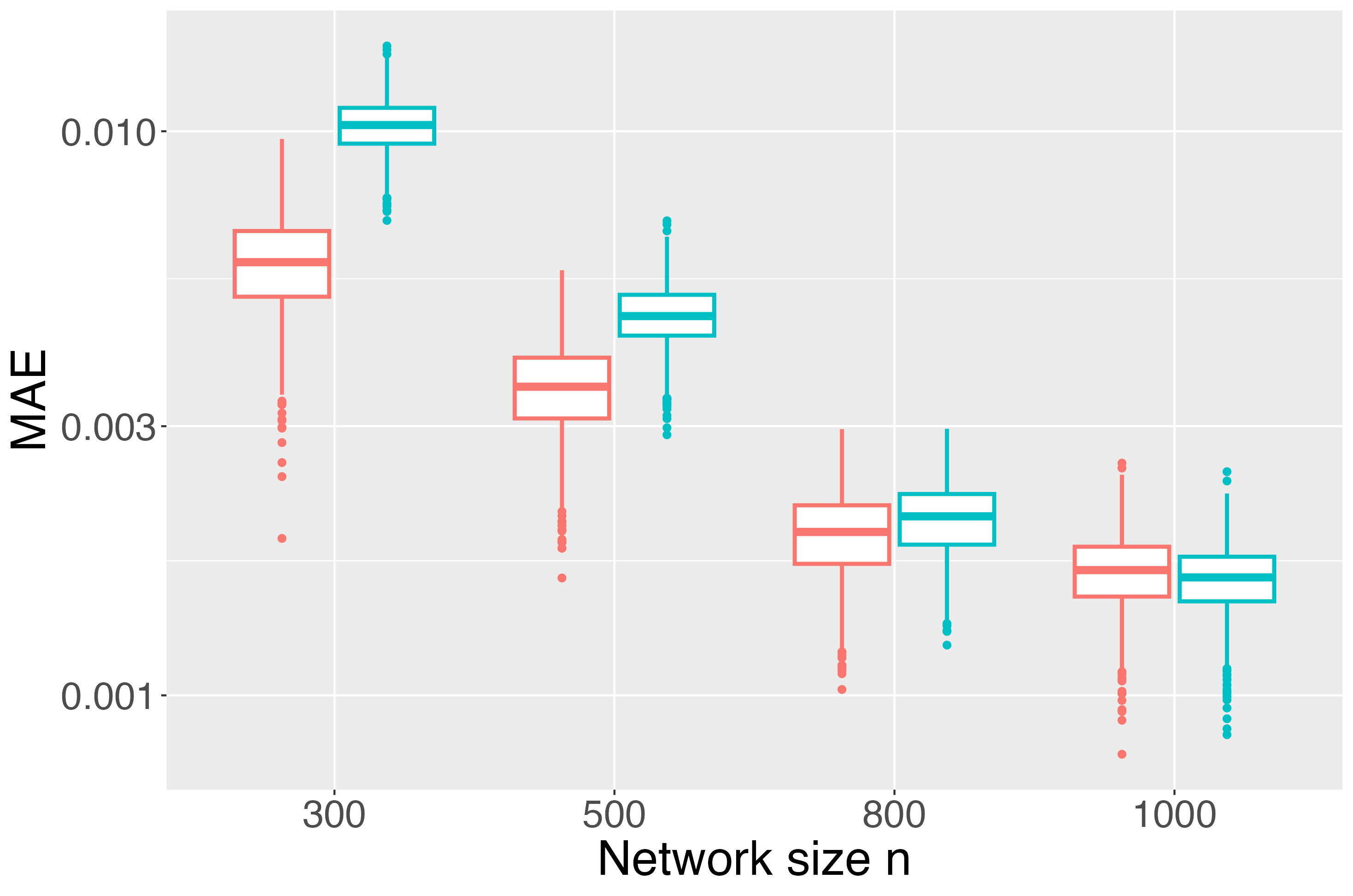}
		\caption{MAE for $\beta_0$}
		\label{Fig: beta MAE model 2}
	\end{subfigure}%
	\begin{subfigure}{0.32\textwidth}
		\centering
		\includegraphics[scale=0.22]{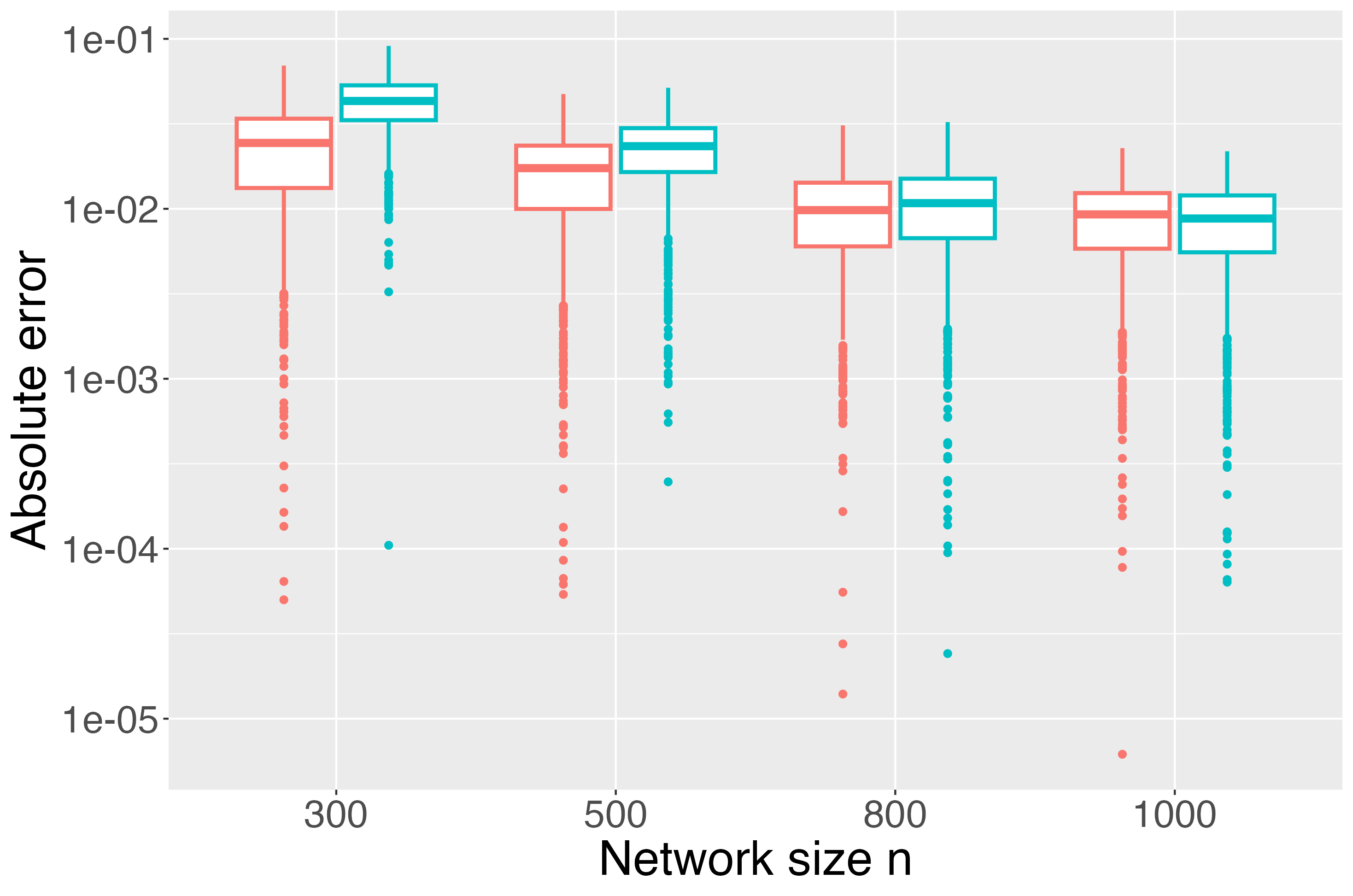}
		\caption{Absolute error for $\mu_0$.}
		\label{Fig: abs mu model 2}
	\end{subfigure}
	\begin{subfigure}{0.32\textwidth}
		\centering
		\includegraphics[scale=0.22]{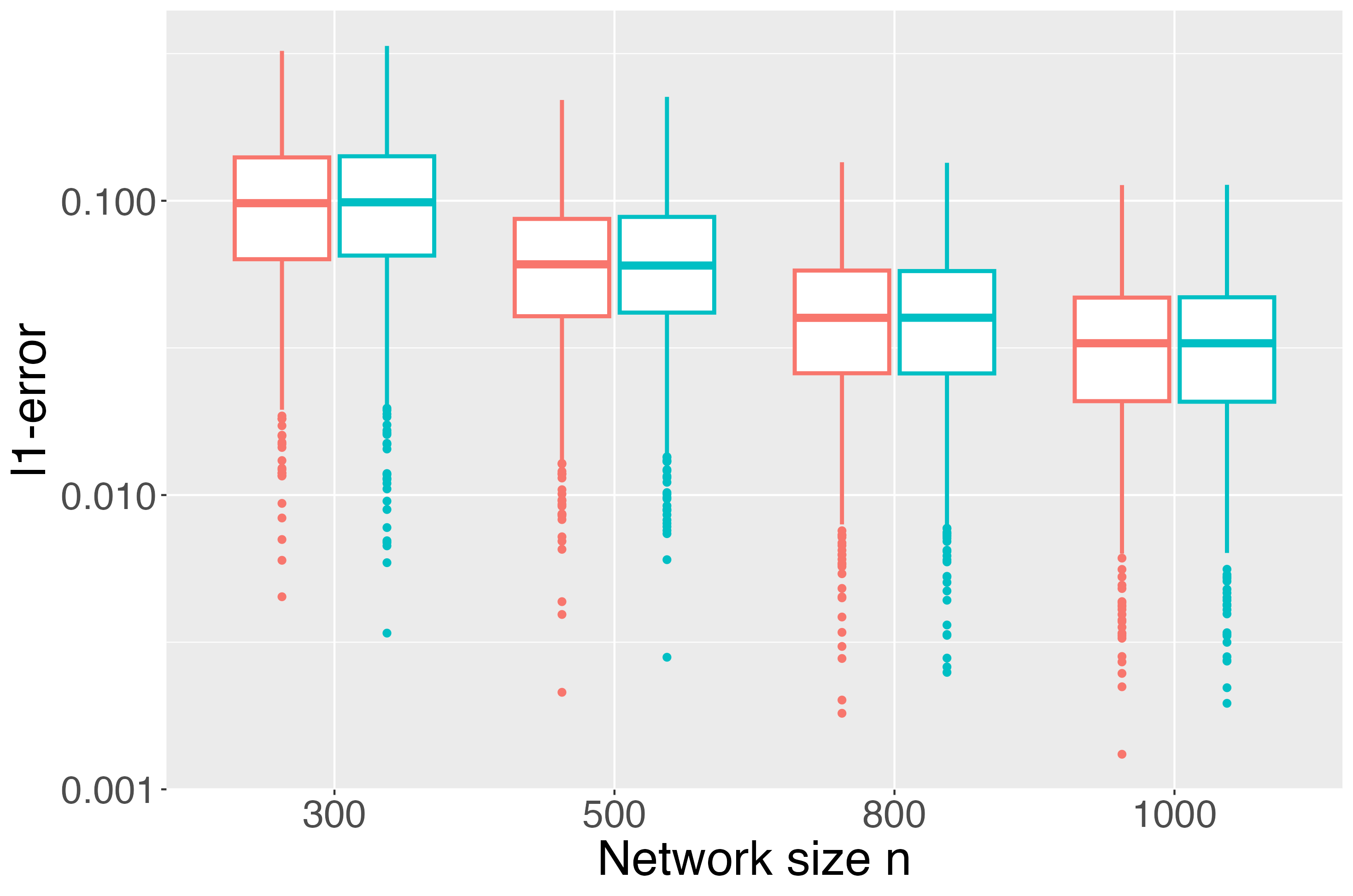}
		\caption{$\ell_1$-error for $\gamma_0$.}
		\label{Fig: l1 gamma model 2}
	\end{subfigure}
	\caption{Errors for estimating the true parameter $\theta_0$ in Model 2.}
\end{figure}

\begin{figure}[!htbp]
	\centering
	\begin{subfigure}{0.32\textwidth}
		\centering
		\includegraphics[scale=0.22]{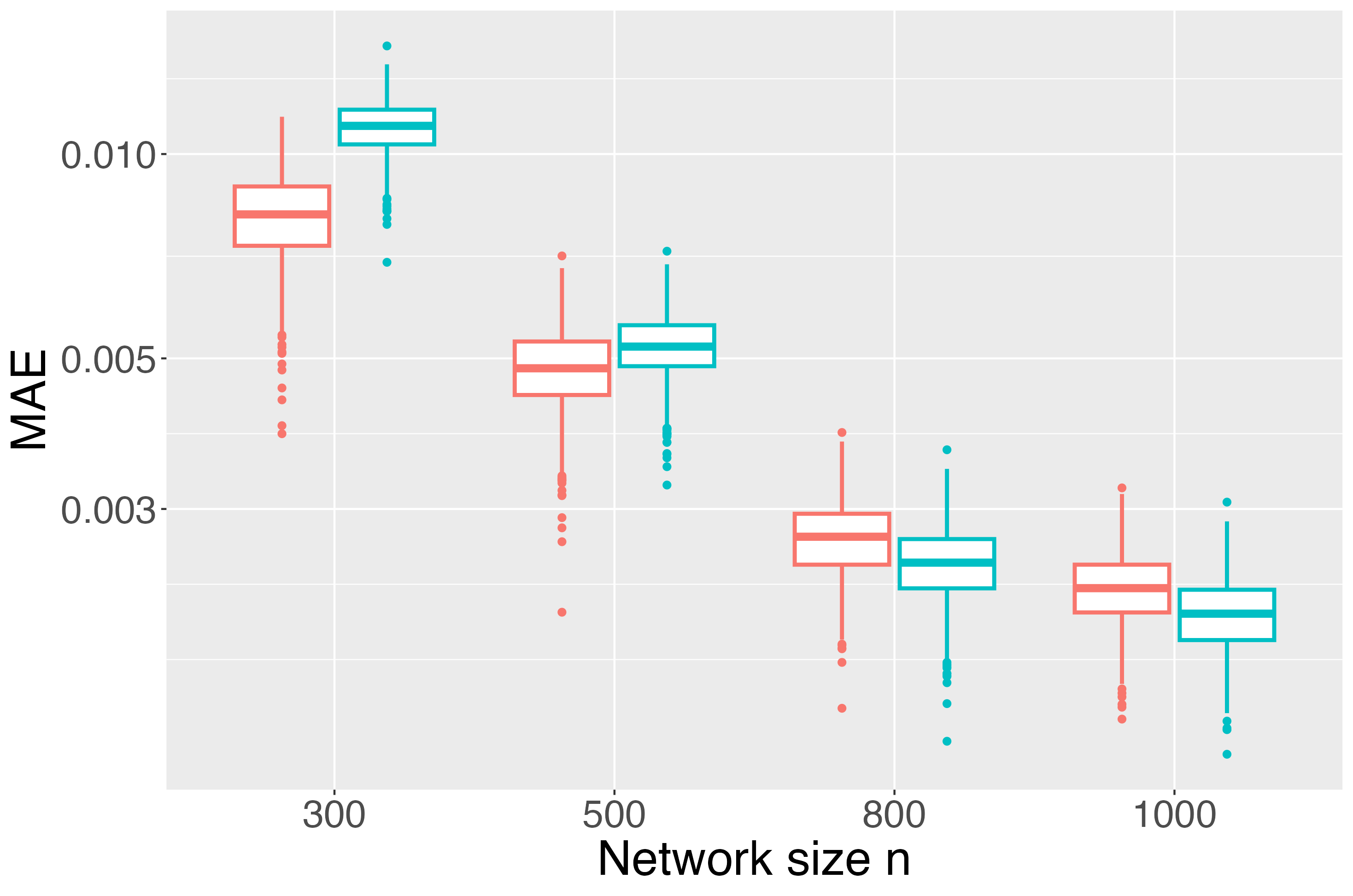}
		\caption{MAE for  $\beta_0$}
		\label{Fig: beta MAE model 3}
	\end{subfigure}%
	\begin{subfigure}{0.32\textwidth}
		\centering
		\includegraphics[scale=0.22]{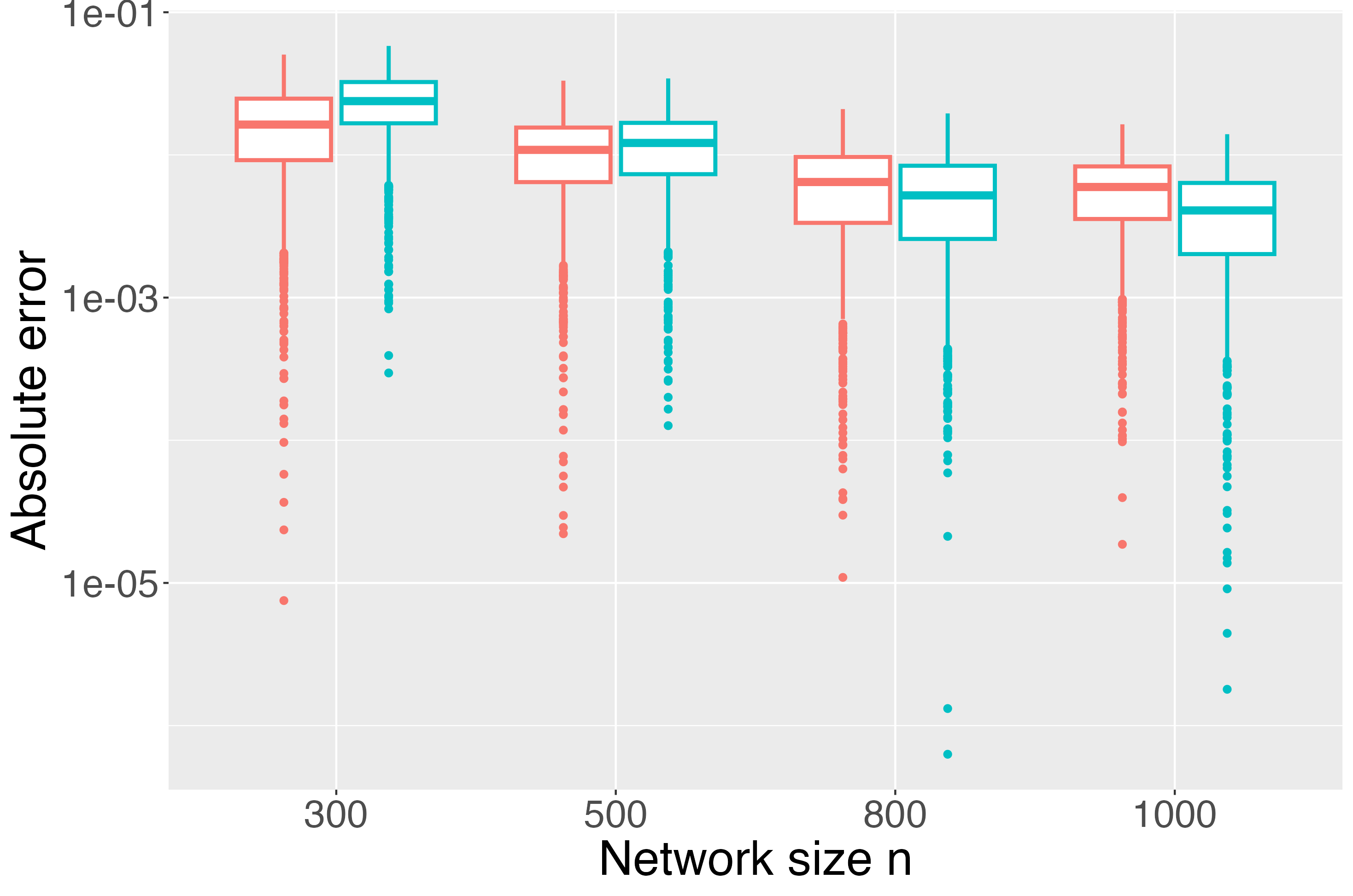}
		\caption{Absolute error for  $\mu_0$.}
		\label{Fig: abs mu model 3}
	\end{subfigure}
	\begin{subfigure}{0.32\textwidth}
		\centering
		\includegraphics[scale=0.22]{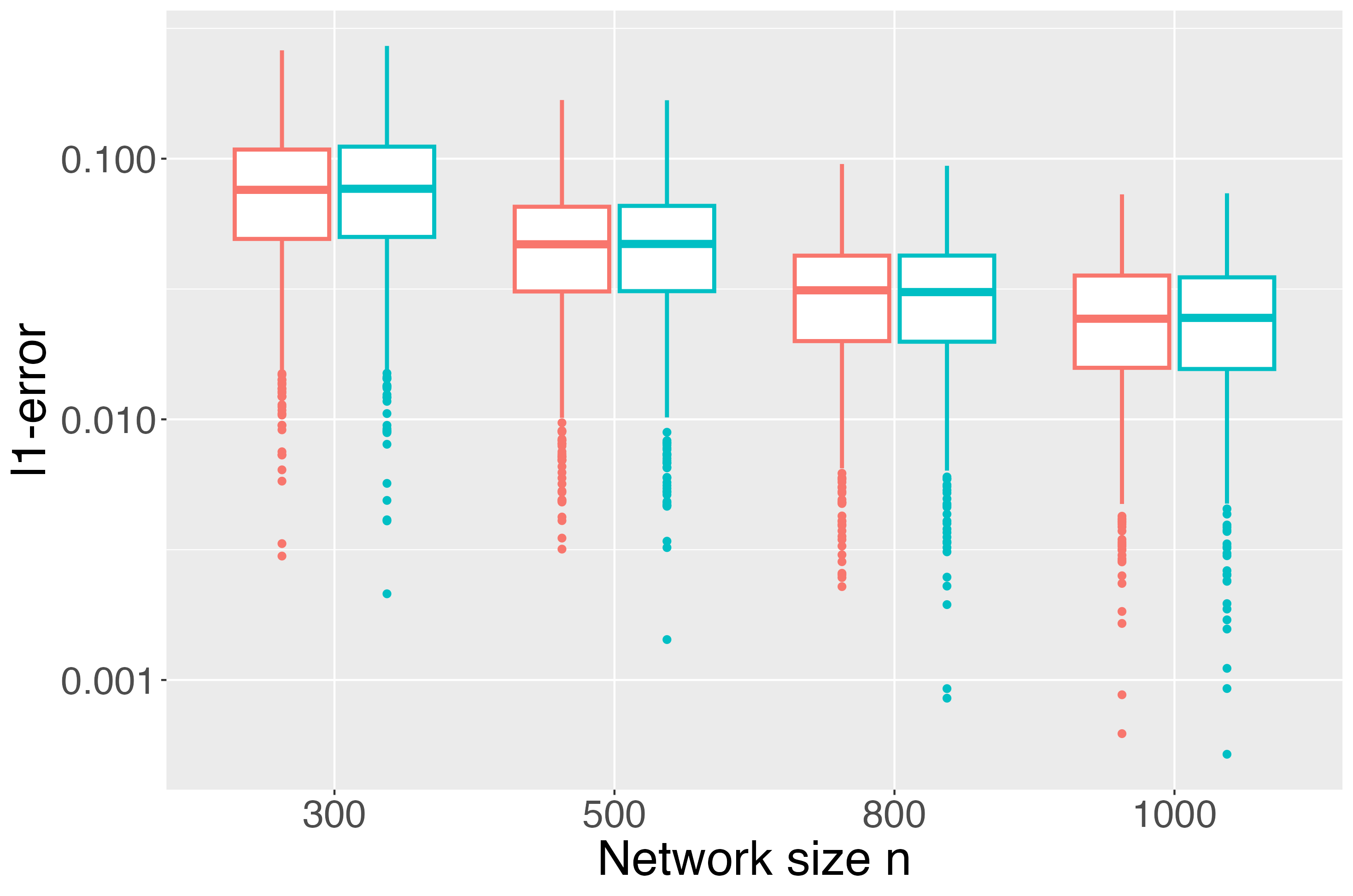}
		\caption{$\ell_1$-error for $\gamma_0$.}
		\label{Fig: l1 gamma model 3}
	\end{subfigure}
	\caption{Errors for estimating the true parameter $\theta_0$ in Model 3. 
	}
\end{figure}

\noindent\textbf{Asymptotic normality:}
Next, we consider the normal approximation for our estimator. We calculate the standardized \(\gamma\)-values 
\[
\sqrt{\binom{n}{2}}\frac{\hat \gamma_k - \gamma_{0,k}}{\sqrt{\hat \Theta_{\vartheta,k+1,k+1}}}, \ k = 1,2,
\] 
which by Theorem \ref{Thm: inference} asymptotically follow a \(\mathcal{N}(0,1)\) distribution. This allows us to construct approximate $95\%$-confidence intervals for $\gamma_{0,k}$ as
\[
CI_k = \left(\hat \gamma_k - z_{1 - \alpha/2} \cdot \sqrt{\frac{\hat \Theta_{\vartheta,k+1,k+1}}{\binom{n}{2}}}, \hat \gamma_k + z_{1 - \alpha/2} \cdot \sqrt{\frac{\hat \Theta_{\vartheta,k+1,k+1}}{\binom{n}{2}}} \right), \ k = 1,2,
\]
where $z_{1 - \alpha/2}$ is the $1 - \alpha/2$ quantile of the standard-normal distribution and we use $\alpha = 0.05$. \color{black}{Similarly we can construct approximate $95\%$-confidence intervals for $\beta_{0,k}$ ($k = 1, \dots, n$) by Theorem \ref{beta_AN}.
We present the empirical coverage of these intervals and their median length for the different network sizes.
Table \ref{table: CV and CI} shows the results for $\gamma_{0,1}$ and $\beta_0$ (where we take an average of all $\beta_i$'s) across the different models and sample sizes. The results for $\gamma_{0,2}$ and $\mu_0$ are omitted to save space.
The coverage is very close to the $95\%$-level across all network sizes and all models and independent of which model selection criterion we use. This empirically illustrates the validity of the asymptotic results derived in Theorem \ref{Thm: inference} and \ref{beta_AN}. The median length of the confidence interval decreases with increasing network size and is similar between BIC and the heuristic. This is what we would expect since the estimates are very similar between both methods as shown in Figures \ref{Fig: l1 gamma}, \ref{Fig: l1 gamma model 2}, and \ref{Fig: l1 gamma model 3}. In each scenario, the median length of the confidence intervals for $\beta_0$ is relatively larger than $\gamma_{0,1}$ since the effective sample size of $\beta_0$ is smaller. Comparing the length of the confidence intervals between Models 1 and 2, we see that as the models become sparser, the median length increases, which is also to be expected. }

\begin{table}[!htbp]
\begin{center}
\begin{tabular}{rrrrrrrrr}
			\toprule
			& \multicolumn{2}{c}{{Pre-determined $\lambda$}} & \multicolumn{2}{c}{{BIC}}& \multicolumn{2}{c}{{Pre-determined $\lambda$}} & \multicolumn{2}{c}{{BIC}}\\
			\midrule
		\multicolumn{1}{c}{$n$} & Coverage & Width & Coverage & Width& Coverage & Width & Coverage & Width\\
			\midrule
			\multicolumn{9}{c}{Model 1}\\ 
			300 & 0.949 & 0.182 &  0.950 & 0.182 &0.952 & 0.496 &  0.952 & 0.496\\
			500 & 0.952 & 0.110 & 0.949 & 0.110& 0.951 & 0.388 & 0.951 & 0.388\\
			800 & 0.949 & 0.069&0.949 & 0.069& 0.951 & 0.310&0.951 & 0.309\\
			1000 & 0.945 & 0.056 & 0.943 & 0.056& 0.951 & 0.278 & 0.950 & 0.278\\
			\addlinespace[0.3em]
			\multicolumn{9}{c}{Model 2}\\
			300 & 0.928 & 0.251 & 0.934 & 0.252& 0.954 & 0.693 & 0.953 & 0.697\\
		500 & 0.943 & 0.158 & 0.945 & 0.158& 0.952 & 0.563 & 0.952 & 0.564\\
		800 & 0.947 & 0.103 & 0.948 & 0.103& 0.952 & 0.464 & 0.952 & 0.464\\
		1000 & 0.943 & 0.083& 0.943 & 0.083& 0.951 & 0.422& 0.951 & 0.422\\
		\addlinespace[0.3em]
			\multicolumn{9}{c}{Model 3}\\
			300 & 0.931 & 0.192 & 0.940 & 0.192& 0.952 & 0.525 & 0.952 & 0.526\\
		500 & 0.940 & 0.118 & 0.943 & 0.118& 0.950 & 0.418 & 0.950 & 0.418\\
		800 & 0.936 & 0.076 & 0.933 & 0.076& 0.951 & 0.340 & 0.951 & 0.340\\
		1000 & 0.947 & 0.062& 0.944& 0.061& 0.951 & 0.308 & 0.951 & 0.308\\
			\bottomrule

\end{tabular}
\end{center}
\caption{Empirical coverage under nominal 95\% coverage and median lengths of confidence intervals for $\gamma_{0,1}$ (columns 2-5) and $\beta_0$ (last 4 columns). }
\label{table: CV and CI}
\end{table}%

\subsection{ERC: The Erd\H{o}s-R\'{e}nyi model with covariates}\label{Sec: serc simulation}

In this section, we illustrate the finite sample performance of the MLE in \eqref{Eq: ER-C max llhd} in the ERC \eqref{eq:nullmodel2}. We focus on inference for the covariate weights, $\gamma$, in the more realistic case of unknown $\xi$, that is, we only estimate $\mu_0$ rather than $\mu_0^\dagger$. Our emphasis is on illustrating that the MLE can be used to perform inference in extremely sparse network settings. To that end, we fixed the covariate dimension $p$ and a true parameter vector $(\mu_0^\dagger, \gamma_0^T)^T$ and varied the sparsity parameter $\xi$. The exact model setup was as follows. We set $p=20$ and sampled the covariate values $Z_{ij,k}, k = 1, \dots, p, i < j$ from a centered $\text{Beta}(2,2)$ distribution. We used $\mu_0^\dagger = 1$ and  \(\gamma_0 = (1.5, 1.2, 0.8, 1, \dots, 1)^T\). For the sparsity parameter $\xi$ we used the values \(\xi = 0.3, 1.0,\) or  \(1.5\). Note that the larger $\xi$, the sparser the resulting network. 
As before, we sampled networks of sizes $n = 300, 500, 800, 1000$, and for each configuration we drew $1000$ realizations of the \serc and analyzed the performance of the MLE defined in \eqref{Eq: ER-C max llhd}. 
The sparsest case $\xi = 1.5$ is close to the maximum theoretically permissible sparsity and results in extremely sparse networks. For example, when $n = 1000$, on average, only $73$ out of the almost half million possible edges are observed in this setting.

The asymptotic normality for each component of $\hat \gamma$ allows us to construct confidence intervals at the 95\%-level as prescribed by Corollary \ref{Cor: Corollary ER-C} and we assess the performance of our MLE by calculating the empirical coverage for each component. There is no significant difference in the empirical coverage or the average length of the confidence intervals between the various components of $\gamma$, which is why we only present them for $\gamma_1$ in Table \ref{Table: ER-C inference gamma}. As we can see, coverage is very close to the nominal confidence level of 95\% and the length of the confidence intervals decreases with increasing network size. As expected, confidence intervals are larger for sparse networks. For $\xi = 1.5$ we observe very wide confidence intervals, which is due to the very low effective sample size.

\begin{table}[!htbp]
	\centering
	\begin{tabular}[t]{rrrrrrrrr}
		\toprule
			$n$ & \multicolumn{2}{c}{{$\xi = 0.3$}} && \multicolumn{2}{c}{{$\xi = 1.0$}} && \multicolumn{2}{c}{{$\xi = 1.5$}}\\
		\midrule
		& Coverage & Width && Coverage & Width && Coverage & Width\\
		\midrule
		\addlinespace[0.3em]
		300 & 0.941 & 0.193 && 0.956 & 0.711 && 0.944 & 2.892 \\
		500 & 0.955 & 0.118 && 0.938 & 0.541 && 0.967 & 2.505\\
		800 & 0.943 & 0.075 && 0.950 & 0.424 && 0.951 & 2.235 \\
		1000 & 0.949 & 0.061 && 0.935 & 0.379 && 0.948 & 2.107 \\
		\bottomrule
	\end{tabular}
	\caption{Empirical coverage under nominal 95\% coverage and median lengths of confidence intervals for $\gamma_1$. The results are similar for the other components of $\gamma$.}
	\label{Table: ER-C inference gamma}
\end{table}

\section{Data Analysis}\label{Sec: Data Analysis}

We illustrate our results further by applying our method to two real-world data sets.

\noindent
\textbf{Lazega's lawyer friendship data}. In this data set, 
 71 lawyers of a New England Law Firm were asked
to indicate with whom in the firm they regularly socialized outside of
work \citep{lawyernetwork}. This is a frequently used network data set that was also analyzed, for example,
in \cite{Yan:etal:2019}, \cite{jochmans} and \cite{Snijders:etal:2006}. For our analysis we focus on mutual friendships
between lawyers as in \cite{Snijders:etal:2006}, 
 that is, we consider the network in which an undirected edge
is placed between two lawyers when they both indicated to socialize with one another. 
The degrees of the resulting network range from \(0\) to \(16\), with
eight isolated nodes. The average degree is \(4.96\) and the edge
density is $7\%$. 
It is important to note that we did not remove the isolated nodes before conducting
inference, as opposed to some existing exercises \citep[cf.]{Yan:etal:2019}. Omitting nodes prior to model fitting in the latter suffers from the issue of data selective inference as discussed in \cite{stein2022fallacy} leading to biased estimators. 
 Alongside the network, the following variables were
collected: The status of the lawyer (partner or associate), their gender
(man or woman), which of three offices they worked in, the years they
had spent with the firm, their age, their practice (litigation or
corporate) and the law school they had visited (Harvard and Yale, UConn
or other).

We fitted the \sbmc to this data set, by using the positive absolute difference of the nodewise covariates as $Z_{ij}$, where for categorical variables the
difference is defined as the indicator whether the values are equal. Since our simulation studies suggest that BIC performs better for smaller networks, we only present its results. Model selection with the heuristic gives in a slightly larger penalty and slightly different estimates, but overall very similar results. In both cases, \sbmc identified six lawyers with non-zero $\beta$-value, four partners and two associates, with degrees ranging from 11 to 16.  While those non-zero $\beta$-values generally do correspond to lawyers with larger than average degrees, it is interesting to note that there is one other lawyer with degree 11 and two more lawyers with degree 10 with zero $\beta$, suggesting that \sbmc is able to pick up subtleties in network formation that go beyond simply assigning non-zero $\beta$s to the nodes with the highest degree. We elaborate on this further in our second example below.

\begin{table}[!htp]
	\centering
	\begin{tabular}{lrr}
		\toprule
		Covariate & Estimate & Confidence Interval\\
		\midrule
		Same status & $0.91$ & $(0.54, 1.28)$ \\
		Same gender & $0.46$ & $(0.12, 0.81)$ \\
		Same office & $2.21$ & $(1.81, 2.60)$\\
		Years with firm difference & $-0.073$ & $(-0.11, -0.040)$\\
		Age difference & $-0.031$ & $(-0.060, -0.0023)$ \\
		Same practice & $0.57$ & $(0.25, 0.89)$\\
		Same law school & $0.30$ & $(-0.030, 0.62)$\\
		\bottomrule
	\end{tabular}
	\caption{\label{tab:lawyer results} Covariate weights for Lazega's Lawyer friendship network and $95\%$ confidence intervals.}
\end{table}

We constructed confidence intervals for the estimated parameters at the $95\%$-level as shown in Table \ref{tab:lawyer results}. The findings in this table are in line, both in terms of magnitude of
estimated paramters as well as, more importantly, the sign of each parameter, 
with what we would expect and with the results in the aforementioned
papers. In order of importance, working in the same office, having the
same status, being of the same practice and having the same gender have a positive effect on friendship
formations, whereas a big difference in age or tenure has a negative
effect on friendship formation. While our point estimate for having gone to the same law school is positive, its confidence interval extends to the negative real line and we thus cannot make a definite statement about its effect on friendship formation. This effect is also present when doing model selection with our heuristic. To appreciate how the covariates influence the connection pattern, we visualize the network in Figure \ref{Fig: Lawyer} by examining the effect of office in Figure \ref{Fig: Lawyer office} and that of status in Figure \ref{Fig: Lawyer status} respectively. We can see indeed that these two covariates have played important roles in shaping how connections were made.
\begin{figure}[!htbp]
	\centering
	\begin{subfigure}{0.45\textwidth}
		\centering
		\includegraphics[scale=0.3]{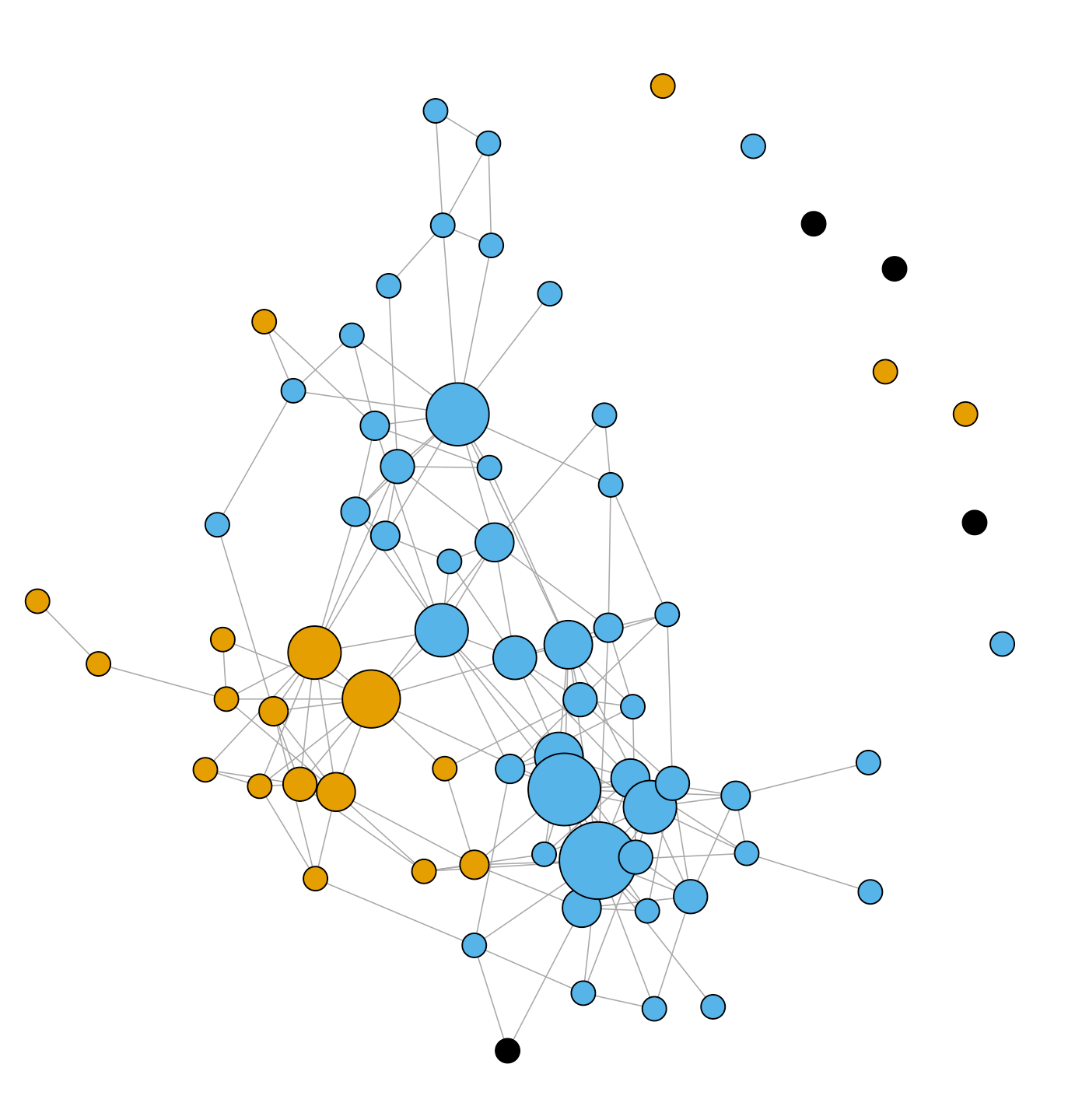}
		\caption{Lawyer network by office.}
		\label{Fig: Lawyer office}
	\end{subfigure}%
	\begin{subfigure}{0.45\textwidth}
		\centering
		\includegraphics[scale=0.3]{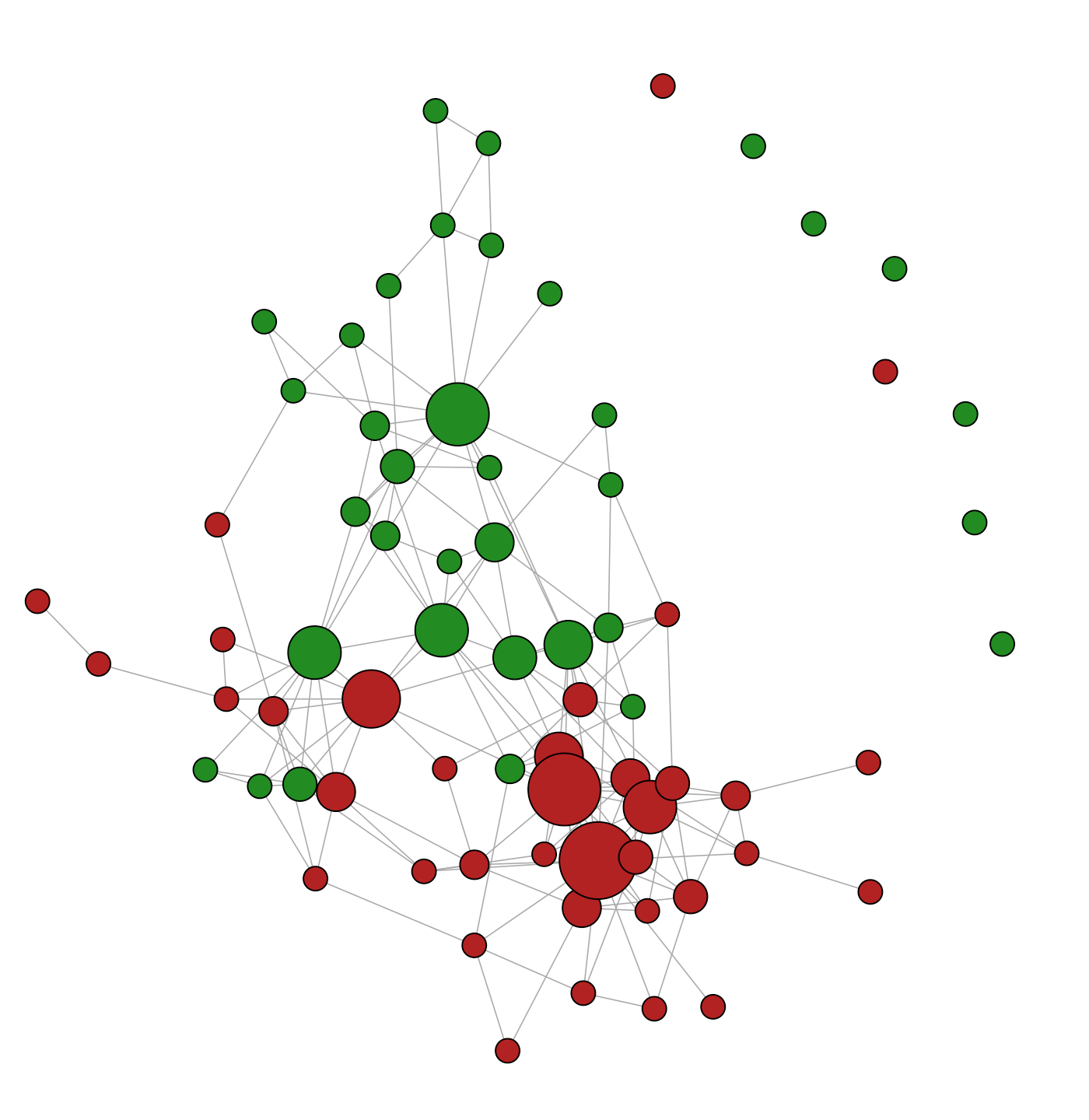}
		\caption{Lawyer network by status.}
		\label{Fig: Lawyer status}
	\end{subfigure}
	\caption{Lazega's friendship network among 71 lawyers. The size of the nodes is proportional to their degree. For better visibility we set the size of all nodes with a degree of five or lower to the size corresponding to a degree of five. In \ref{Fig: Lawyer office} the different colors indicate different offices (blue: Boston, yellow: Hartford, black: Providence; notice that only four lawyers are based in the small Providence office) and in \ref{Fig: Lawyer status} different statuses (red: partner, green: associate). The positions of the vertices are the same in both plots.}
	\label{Fig: Lawyer}
\end{figure}

\noindent
\textbf{Trade partnerships network}. 
For our second data set, we analyzed mutually important trade partnerships between 136 countries/regions in 1990. This data was originally analyzed by \cite{Silva:Tenreyro:2006} and further analyzed in \cite{jochmans}. Even back in 1990 almost every country would trade with every other country, resulting in a very dense network. To be able to make the underlying network formation mechanisms visible, we decided to only focus on important trade partnerships in which the trade volume exceeds a certain limit. More precisely, we place an undirected edge between two countries if the trade volume makes up at least \(3\%\) of the importing {countries'} total imports or if it makes up at least \(3\%\) of the exporting { countries'} total exports.
This leaves us with an undirected network with 136 nodes and 1279 edges, meaning that we have an edge density of \(13.9\%\). The minimum degree of the resulting network was \(3\) (Dominican Republic), the maximum degree was \(126\) (USA), and the median degree was \(13\). 

We analyze the same covariates as \cite{jochmans}. That is, we have indicator variables {common language} and {common border} that take the value one if countries \(i\) and \(j\) share a common language or border and zero otherwise, {log distance} which is the log of the geographic distance between the countries, {colonial ties} which is one if at some point \(i\) colonized \(j\) or vice versa and zero otherwise, and {preferential trade agreement} which is an indicator whether or not a preferential trade agreement exists between the countries. 
Again, we chose BIC for model selection for the reasons outlined above. The results are summarized in Table \ref{tab:trade}. These results are in line with what one would expect. Having a preferential trade agreement has the strongest positive effect on mutual trade between countries. Speaking the same language, sharing a border or having colonial ties also has a positive effect, while a large geographical distance has a strong negative effect.

\begin{table}[!htbp]
	\centering
	\begin{tabular}{lrr}
		\toprule
		Covariate & Estimate & Confidence Interval \\
		\midrule
		Log distance & $-1.03$ & $(-1.04, -1.02)$\\
		Common border & $0.45$ & $(0.10, 0.79)$ \\
		Common language & $0.31$ & $(0.086, 0.54)$ \\
		Colonial ties & $0.42$  & $(0.17, 0.66)$ \\
		Preferential trade agreement & $0.81$  & $(0.36, 1.27)$\\
		\bottomrule
	\end{tabular}
	\caption{\label{tab:trade}Covariate estimation for world trade data and $95\%$ confidence intervals.}
\end{table}

Notice that the confidence intervals for the categorical variables are all much larger than the one for the continuous variable log distance between countries. This is due to the fact that all the columns corresponding to categorical covariates are quite sparse, while the column corresponding to log distance contains only non-zero entries. Only 142 dyads are part of a preferential trade agreement and only 180 share a common border. Consequently, the confidence intervals corresponding to these covariates are largest. Note that 1565 node pairs have colonial ties with one another and 1925 speak a common language. While the columns corresponding to these covariates are thus much more populated, they are still relatively sparse when compared to the total number of dyads.

BIC selected \(32\) active \(\beta\)-entries, which are visualized on a map in Figure \ref{Fig: world map}. We presented the top half of these countries/regions with their degree and GDP in Table \ref{tab:betas}. The ranking of the $\beta$ values correlates with our intuition of the economic power of the countries. However, we also pick up underlying network formation mechanisms that go beyond sheer economic power and that are neither explainable by only looking at network summary statistics (such as degree of a node) nor by only looking at economic metrics such as a country's GDP.
More precisely, we note that the top six positions are occupied by six of the seven G7 countries, which serves to show that the \sbmc works well for identifying the most important nodes in a network. Note however, that Japan has the largest \(\beta\), albeit having a smaller degree (122) and a significantly smaller GDP than the USA (degree = 126), which comes in second place. In general, the order of degrees no longer aligns exactly with the order of the \(\beta\)-values as would have been predicted by the S\(\beta\)M without covariates in \cite{Chen:etal:19}. 
Norway, for example, has a \(\beta\)-value of zero, even though its degree of \(17\) and GDP of US\$\ensuremath{1.22\times 10^{11}} exceeds the degree and the GDP of several nodes with an active \(\beta\)-value. An examination of Norway's neighboring nodes reveals that it was trading mostly with countries that either are close  geographically or have a large \(\beta\)-value themselves (such as USA and Japan), meaning that the observed covariates are sufficient to explain the linking behavior of Norway.
This illustrates that the S\(\beta\)M with covariates is able to pick up subtleties in network formation that one might miss if one relied solely on network summary statistics such as the degree of a node or solely on non-relational summary statistics such as a country's GDP. 

\begin{figure}[!htbp]
	\centering
	\includegraphics[scale=0.5]{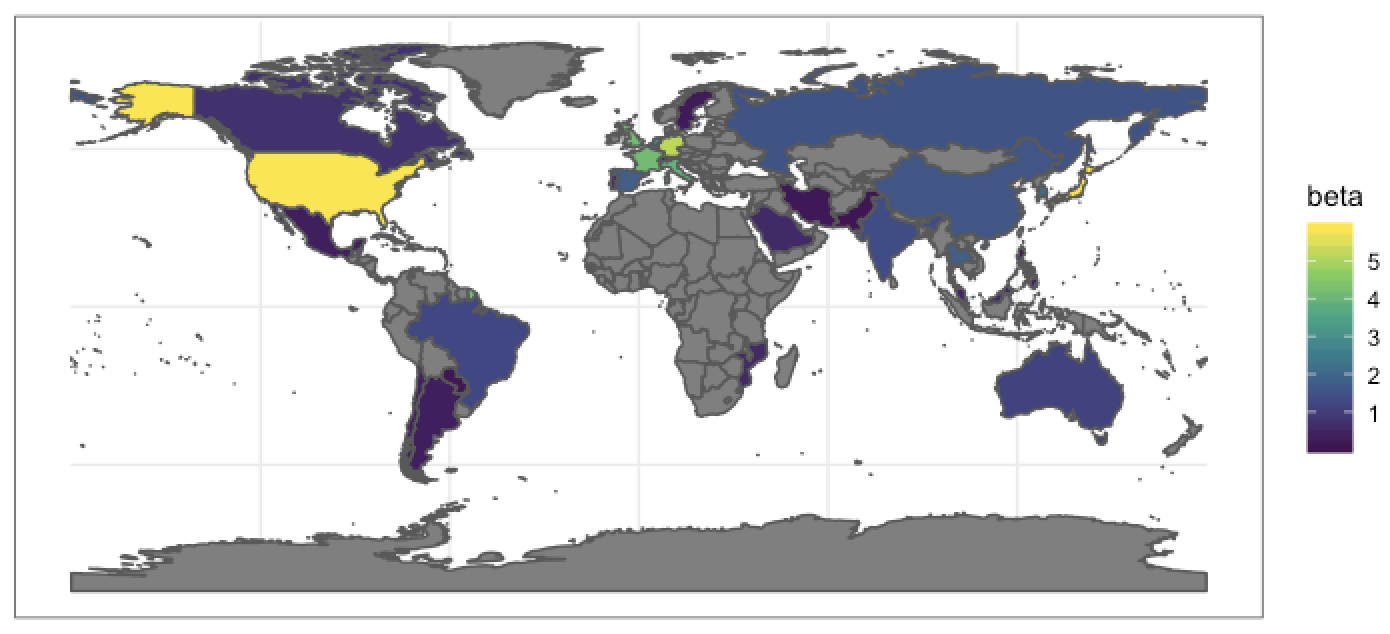}
	\caption{Visualization of the estimated $\beta$ values in the world trade network in 1990 between 136 countries/regions. The color of the country/region corresponds to the magnitude of the estimated $\beta$. Countries in grey either have an estimated $\beta$ value of zero or were not present in the data set}
	\label{Fig: world map}
\end{figure}

\begin{table}[!htbp]
	\centering
	\begin{tabular}{lrrrrrrr}
		\toprule
		 & $\hat{\beta}$ & Degree & GDP (US\$)&& $\hat{\beta}$ & Degree & GDP (US\$)\\
		\midrule
		Japan & 5.85 & 122 & 4.95e+12 & Korea & 2.11 & 34 & 3.42e+11\\
		USA & 5.82 & 126 & 6.51e+12 &Singapore & 2.06 & 37 & 5.39e+10\\
		Germany & 5.17 & 120 & 2.27e+12&Hong Kong & 2.05 & 40 & 1.07e+11\\
		France & 4.16 & 103 & 1.47e+12&Spain & 1.80 & 41 & 5.46e+11\\
		UK & 4.15 & 104 & 1.04e+12&Thailand & 1.78 & 33 & 1.11e+11\\
		Italy & 3.92 & 95 & 1.03e+12&China & 1.58 & 30 & 3.98e+11\\
		Netherlands & 3.15 & 73 & 3.75e+11&Russia & 1.53 & 28 & 5.43e+11\\
		Belgium-Lux & 2.60 & 59 & 2.56e+11&India & 1.33& 32 & 2.75e+11\\
		\bottomrule
	\end{tabular}
	\caption{The top 16 active $\beta$-values for the world trade network.}
	\label{tab:betas}
\end{table}
Our fitted model indicates that it may be interesting to explore the interplay between hub and background countries in a formally way, by for example treating hub and background nodes differently according to their degrees. Towards this, we have fitted an alternative model encoding the interaction between $Z_{ij}$ and a dummy variable whether the corresponding link is between two background nodes. The resulting fitted model gave qualitatively similar parameter estimates of the covariate effect, while having a much sparser estimate for the heterogeneity parameter $\beta$. The nonzero fitted $\beta$ parameters in this alternative model nevertheless corresponding to the most important countries in Table \ref{tab:betas}, implying that our model is able to capture heterogeneity effects. Note that in the alternative model, the network information is used as part of the covariates as the dummy variable is indicative of the degree of a node. Thus, our proposed model can reveal similar findings to this alternative without invoking network information as covariates.

\section{Conclusion}\label{section: conclusion}

We have presented a new model named \sbmc that simultaneously captures homophily and degree heterogeneity in a network. We have shown that \sbmc is well suited to model sparse networks, thanks to the sparsity assumption on the nodal parameter that can effectively reduce the dimensionality of the model. We have presented {a} theory for the penalized likelihood estimator based on an $\ell_1$ penalty on the nodal parameter, including consistency of the excess risk and \color{black}{the central limit theorem for the estimators.} Our theoretical contributions go beyond existing theory for LASSO as we must deal with a parameter regime where networks are sparse. The computation of our estimator leverages the recent vast algorithmic development on solving LASSO type problems. Thus, \sbmc represents an attractive model for networks with statistical guarantees and computational feasibility. Along this line, \cite{stein2021sparse} has extended the S$\beta$RM in this article to directed networks but with a theoretical emphasis on the selection consistency of the estimation of the heterogeneity parameter.

There are many important issues for future research. 
First, it will be interesting to incorporate a {low-rank} component in \sbmc in order to capture transitivity, the phenomenon that nodes with common neighbors are more likely to connect, as is done in \cite{mama:2020}. Equally interestingly, it will be  {important}  to relax the independence assumption on the dyads, similar in spirit to the progress made in \cite{stewart2020pseudo}. Second, it will be interesting to see how \sbmc can be used to model networked data under privacy  {constraints}, along the line of research initiated by  \cite{Karwa:Slavkovic:2016} for the $\beta$-model. Third, we note that our results still require the network to be relatively dense. Although promising numerical results shown in the Supplementary Material demonstrate that our approach may still work for very sparse networks,  investigating this theoretically remains an important problem. 
 Lastly, a growing list of networked data  {is} observed along a temporal dimension  \citep{jiang2020autoregressive} and it will be interesting to extend 
 our model to a time series context. 
These issues are beyond the scope of the current paper and will be explored elsewhere.

 \bibliographystyle{agsm}
\bibliography{bib}

 \newpage

 \begin{center}
 Supplementary Materials
 \end{center}	

The supplementary material contains the main proofs, a comparison of the S$\beta$RM without covariates using the $\ell_1$ penalty with that in \cite{Chen:etal:19} using the $\ell_0$ penalty, and some additional simulations for sparser networks. Without loss of generality, in Assumption \ref{Assum: minimum EW}, we assume that $c_{\min}<1/2$. 
\appendix

\section{Proofs of Lemma \ref{Lem: Existence of solution and identifiability} and Proposition \ref{Prop: compatibility condition Sigma}}

\subsection{Proof of Lemma \ref{Lem: Existence of solution and identifiability}}\label{Appendix: Proof of Lemma 1}

\begin{proof}[Proof of Lemma \ref{Lem: Existence of solution and identifiability}]
	We first show that a solution exists. Using duality theory from convex optimization (cf. \cite{Bertsekas:1995}, Chapter 5), we know that for any $\lambda > 0$ there exists a finite $s > 0$ such that the penalized likelihood problem is equivalent to the primal optimization problem
	\begin{align}\label{primal problem}
	\begin{split}
	&\min_{\beta, \mu, \gamma} \frac{1}{\binom{n}{2}}\mathcal{L}(\beta, \mu, \gamma), \\
	& \text{subject to:} \; (\beta^T, \mu, \gamma^T)^T \in \Theta_{\text{loc}}, \sum_{i = 1}^n \vert \beta_i \vert \leq s.	
	\end{split}
	\end{align}
	Let $\beta = (\beta_1, \dots, \beta_n)^T, \gamma = (\gamma_1, \dots, \gamma_p)^T$ be fixed. To obtain an estimate for $\mu$, we minimize the function
	\begin{align*}
	g_{\beta, \gamma}(\mu) &= \frac{1}{\binom{n}{2}}\mathcal{L}(\beta, \mu, \gamma) \\
	& = \frac{1}{\binom{n}{2}}\left(- \sum_{i = 1}^n \beta_id_i - d_+\mu - \sum_{i < j} (Z_{ij}^T\gamma)A_{ij} + \sum_{i < j} \log(1 + \exp(\beta_i + \beta_j + \mu + Z_{ij}^T\gamma))\right).
	\end{align*}
	It has derivative
	\[
	g_{\beta, \gamma}'(\mu) = \frac{1}{\binom{n}{2}} \left(- d_+ +  \sum_{i < j} \frac{e^{\beta_i + \beta_j + \mu + Z_{ij}^T\gamma}}{1 + e^{\beta_i + \beta_j + \mu + Z_{ij}^T\gamma}}\right).
	\]
	We observe that
	\[
	\lim_{\mu \rightarrow \infty} g'_{\beta, \gamma}(\mu) = \frac{1}{\binom{n}{2}} \left(- d_+ + \binom{n}{2}\right) > 0
	\]
	and
	\[
	\lim_{\mu \to -\infty} g'_{\beta, \gamma} = -d_+\frac{1}{\binom{n}{2}} < 0.
	\]
	Furthermore, $g_{\beta, \gamma}'$ is continuous and strictly increasing in $\mu$. Hence, there exists a unique value $\mu^* = \mu^*(\beta, \gamma)$, such that $g_{\beta, \gamma}'(\mu^*) = 0$. Since 
	\[
	g_{\beta,\gamma}''(\mu) = \sum_{i < j} \frac{e^{\beta_i + \beta_j + \mu + Z_{ij}^T\gamma}}{(1 + e^{\beta_i + \beta_j + \mu + Z_{ij}^T\gamma})^2} > 0
	\]
	for all $\mu$, $\mu^*$ is a minimizer of $g_{\beta, \gamma}$. Since $g_{\beta, \gamma}''$
	is invertible, we can apply the implicit function theorem with function $F(\beta,\gamma, \mu) = g_{\beta, \gamma}'(\mu)$, which gives us that the corresponding function $\mu^* = \mu^*(\beta, \gamma)$ is continuously differentiable.
	Plugging in $\mu^*(\beta, \gamma)$ for $\mu$ in (\ref{primal problem}), we are left with the minimization problem
	\begin{align}\label{primal problem with mu0}
	\begin{split}
	&\min_{\beta, \gamma} \mathcal{L}(\beta, \mu^*(\beta, \gamma), \gamma), \\
	& \text{s.t.:} \; (\beta^T, \mu^*(\beta, \mu), \gamma^T)^T \in \Theta_{\text{loc}}, \sum_{i = 1}^n \vert \beta_i \vert \leq s.	
	\end{split}
	\end{align}
	Since $\Gamma$ is compact, we are minimizing a continuous function over a compact set in (\ref{primal problem with mu0}). Hence it attains a minimum $\mathcal{L}^*$. By the definition of $\mu^*$, $\mathcal{L}^*$ must also be a solution of (\ref{primal problem}). 
	
	For the second claim of the Lemma,
	suppose there is an $1 \le i_0 \le n$ such that $\hat{\beta}_{i_0} = \min_{1 \le i \le n} \hat{\beta}_i > 0$. Consider the following vector $\tilde{\theta} = (\tilde{\beta}^T, \tilde{\mu}, \tilde{\gamma}^T)^T$: for all $k$ let $\tilde{\beta}_k = \hat{\beta}_k - \hat{\beta}_{i_0}$ and $\tilde{\mu} = \hat{\mu} + 2 \hat{\beta}_{i_0}$, while keeping $\tilde{\gamma} = \hat{\gamma}$. Then, $\tilde{\beta}_k \ge 0$ for all $k$, i.e. $\tilde{\theta}$ is a feasible point for the penalized likelihood problem (\ref{Eq: Penalized llhd with covariates}). Furthermore $\min_k \tilde{\beta}_k = 0$ and $\mathcal{L}(\tilde{\theta}) = \mathcal{L}(\hat{\theta}).$ 
	However,
	\[
	\Vert \tilde{\beta} \Vert_1 = \sum_{ i = 1}^n \vert \hat{\beta}_i - \hat{\beta}_{i_0} \vert < \Vert \hat{\beta} \Vert_1,
	\]
	where the inequality follows from the minimality of $\hat\beta_{i_0}$. This gives 
	\[
	\frac{1}{\binom{n}{2}}\mathcal{L}(\tilde{\theta}) + \lambda\Vert \tilde{\beta} \Vert_1 < \frac{1}{\binom{n}{2}}\mathcal{L}(\hat{\theta}) + \lambda\Vert \hat{\beta} \Vert_1.
	\]
	A contradiction to the optimality of $\hat{\theta}$.
\end{proof}

\subsection{Proof of Proposition \ref{Prop: compatibility condition Sigma}}\label{Sec: Proof of compatibility condition}

Notice that the compatibility condition is clearly equivalent to the condition that
\[
\kappa^2(\Sigma, s_{0,+}) \coloneqq \min_{\substack{\theta \in \R^{n+1+p} \backslash\{0\} \\ \Vert \theta_{S^{c}_{0,+}} \Vert_1 \le 3 \Vert \theta_{S_{0,+}} \Vert_1}} \frac{\theta^T\Sigma\theta}{\frac{1}{s_{0,+}} \Vert \theta_{S_{0,+}} \Vert_1^2}
\]
stays uniformly bounded away from zero. To prove that $\Sigma$ fulfills the compatibility condition, we generalize the techniques used in \cite{kock_tang_2019}. More precisely, we will first show that the compatibility condition holds for the matrix
\begin{equation*}
	\Sigma_A \coloneqq \begin{bmatrix}
		\frac{1}{n-1}X^TX & \textbf{0} & \textbf{0} \\
		\textbf{0} & 1 & \textbf{0} \\
		\textbf{0} & \textbf{0} &  \mathbb{E} [ Z^TZ/{\binom{n}{2}}]
	\end{bmatrix} \in \R^{(n+1+p) \times (n+1+p)},
\end{equation*}
that is, we will show that $\Sigma_A$ fulfills $\kappa^2(\Sigma_A, s_{0,+}) \ge C > 0$ for all $n$ and some universal $C > 0$.
We then show that $\Sigma$ and $\Sigma_A$ are close to each other in an appropriate sense and that $\kappa^2(\Sigma, s_{0,+})$ is bounded away from zero. 
Let us analyze the top left block matrix of $\Sigma_A$, i.e.~$1/(n-1) \cdot X^TX$, first:
\[
\frac{1}{n-1}X^TX = \begin{bmatrix}
1 & \frac{1}{n-1} & \frac{1}{n-1} & \dots & \frac{1}{n-1} \\
\frac{1}{n-1} & 1 & \frac{1}{n-1}  & \dots & \frac{1}{n-1} \\
\vdots & \ddots & \ddots & \dots & \vdots \\
\frac{1}{n-1} & \frac{1}{n-1} & \dots & \dots  & 1
\end{bmatrix},
\]
that is, $1/(n-1) \cdot X^TX$ has all ones on the diagonal and $1/(n-1)$ everywhere else. This is a special kind of Toeplitz matrix; a circulant matrix to be precise. It is known (see for example \cite{kra_circulantmatrices}), that every circulant matrix $M$ has an associated polynomial $p$ and that the eigenvalues of $M$ are given by $p(\xi_j), j = 0, \dots, n-1$, where $\xi_j, j = 0, \dots, n-1$, denote the $n$th roots of unity, i.e.
$
	\xi_j = \exp\left( \iota {2\pi j}/{n}\right),
$
where $\iota$ is the imaginary unit and $\xi_0 = 1$.
The associated polynomial of the matrix $1/(n-1)X^TX$ is 
\[
	p(x) = 1 + \frac{1}{n-1} (x + x^2 + \dots + x^{n-1})
\]
and thus the eigenvalues of $1/(n-1)X^TX$ are
\begin{align*}
p(1) = 2, \quad p(\xi_j) = 1 + \frac{1}{n-1} (-1) = \frac{n-2}{n-1}, \; j = 1, \dots, n-1,
\end{align*}
where the eigenvalue $(n-2)/(n-1)$ has multiplicity $n-1$.
Hence, we observe the following: For any vector $\theta = (\beta^T, \mu, \gamma^T)^T$,
\[
\theta^T\Sigma_A\theta = \beta^T\left(\frac{1}{n-1}X^TX\right)\beta + \mu^2 + \frac{1}{\binom{n}{2}} \gamma^T \mathbb{E} [Z^TZ]\gamma \ge \frac{n-2}{n-1} \beta^T\beta + \mu^2 + \frac{1}{\binom{n}{2}} \gamma^T \mathbb{E}[Z^TZ]\gamma,
\]
where for the inequality we have used that for any semi-positive definite, symmetric matrix $M$ with smallest eigenvalue $\lambda$ and any vector $x \neq 0$ of appropriate dimension, we have $x^TMx \ge \lambda x^Tx$.
Thus, for any $\theta = (\beta^T, \mu, \gamma^T)^T$,
\begin{align*}
\frac{\theta^T\Sigma_A\theta}{\frac{1}{s_{0,+}} \Vert \theta_{S_{0,+}} \Vert_1^2} &\ge \frac{ \frac{n-2}{n-1} \beta^T\beta + \mu^2 + \frac{1}{\binom{n}{2}} \gamma^T \mathbb{E}[Z^TZ]\gamma}{\frac{1}{s_{0,+}} \Vert \theta_{S_{0,+}} \Vert_1^2} \\
&\ge \frac{ \frac{n-2}{n-1} \Vert \beta \Vert_2^2 + \mu^2 + \frac{1}{\binom{n}{2}} \gamma^T \mathbb{E}[Z^TZ]\gamma}{\Vert \beta \Vert_2^2 + \mu^2 + \Vert \gamma \Vert_2^2}, \quad \text{ by Cauchy-Schwarz} \\
&\ge \frac{\frac{n-2}{n-1} (\Vert \beta \Vert_2^2 + \mu^2) + \frac{1}{\binom{n}{2}} \gamma^T \mathbb{E}[Z^TZ]\gamma}{\Vert \beta \Vert_2^2 + \mu^2 + \Vert \gamma \Vert_2^2}, \quad \text{ since } 1 \ge (n-2)/(n-1) \\
&= \frac{n-2}{n-1} \cdot \frac{\Vert \beta \Vert_2^2 + \mu^2 + \frac{n-1}{n-2} \frac{1}{\binom{n}{2}} \gamma^T \mathbb{E}[Z^TZ]\gamma}{\Vert \beta \Vert_2^2 + \mu^2 + \Vert \gamma \Vert_2^2}.
\end{align*}
Now, notice that for any $a,b,c \in \R$, we have $\frac{a + b}{a+c} \ge \min\{1, b/c\}$. This is easily seen by considering the cases $\min\{1, b/c\} = 1$ and $\min\{1, b/c\} = b/c$ separately and rearranging. Thus,
\begin{align*}
\frac{\theta^T\Sigma_A\theta}{\frac{1}{s_{0,+}} \Vert \theta_{S_{0,+}} \Vert_1^2} &\ge \frac{n-2}{n-1} \min \left\{ 1 , \frac{n-1}{n-2} \frac{1}{\binom{n}{2}} \frac{\gamma^T \mathbb{E}[Z^TZ]\gamma}{\Vert \gamma \Vert^2_2} \right\} = \min \left\{ \frac{n-2}{n-1} ,  \frac{\gamma^T \left(\frac{1}{\binom{n}{2}} \mathbb{E}[Z^TZ]\right)\gamma}{\Vert \gamma \Vert_2^2} \right\} \\
& \ge \min \left\{ \frac{n-2}{n-1} ,  \lambda_{\text{min}} \right\},
\end{align*}
where $\lambda_{\text{min}}$ is the minimum eigenvalue of $\frac{1}{\binom{n}{2}} \mathbb{E}[Z^TZ]$. By Assumption \ref{Assum: minimum EW}, we now have for $n \ge 3$,
\begin{equation}
\kappa^2(\Sigma_A, s_{0,+}) = \min_{\substack{\theta \in \R^{n+1+p} \backslash\{0\} \\ \Vert \theta_{S^{c}_{0,+}} \Vert_1 \le 3 \Vert \theta_{S_{0,+}} \Vert_1}} \frac{\theta^T\Sigma_A\theta}{\frac{1}{s_{0,+}} \Vert \theta_{S_{0,+}} \Vert_1^2} \ge c_{\text{min}} > 0.
\end{equation}
Now, we need to show that with high probability $\kappa(\Sigma, s_{0,+}) \ge \kappa(\Sigma_A, s_{0,+})$, which would imply that the compatibility condition holds with high probability for the sample size adjusted Gram matrix $\Sigma$ and the associated sample size adjusted design matrix. To that end, we have the following auxiliary Lemma found in \cite{kock_tang_2019}. For completeness, we give the short proof of it. The notation is adapted to our setting.

\begin{Lem}[Lemma 6 in \cite{kock_tang_2019}]\label{Lem: Lemma 6 in KockTang}
	Let $A$ and $B$ be two positive semi-definite $(n+1+p) \times (n+1+p)$ matrices and $\delta = \max_{ij} \vert A_{ij} - B_{ij} \vert$. For any set $S_{0,+} \subset \{1, \dots, n+1+p\}$ with cardinality $s_{0,+}$, one has
	\[
	\kappa^2(B, s_{0,+} ) \ge \kappa^2(A, s_{0,+}) - 16\delta s_{0,+}.
	\]
\end{Lem}
\begin{proof}
	Let $\theta = (\beta^T, \mu, \gamma^T)^T \in \R^{n+1+p}\backslash\{0\}$, with $\Vert \theta_{S^{c}_{0,+}} \Vert_1 \le 3\Vert \theta_{S_{0,+}} \Vert_1$. Then,
	\begin{align*}
	\vert \theta^TA\theta - \theta^TB\theta \vert &= \vert \theta^T(A-B) \theta \vert \le \Vert \theta \Vert_1 \Vert (A-B) \theta \Vert_\infty \le \delta \Vert \theta \Vert_1^2 \\
	&= \delta (\Vert \theta_{S_{0,+}} \Vert_1 + \Vert \theta_{S^{c}_{0,+}} \Vert_1)^2 \le \delta (\Vert \theta_{S_{0,+}} \Vert_1 + 3 \Vert \theta_{S_{0,+}} \Vert_1)^2 \\
	&\le 16\delta\Vert \theta_{S_{0,+}} \Vert_1^2.
	\end{align*}
	Hence, $\theta^TB\theta \ge \theta^TA\theta - 16\delta\Vert \theta_{S_{0,+}} \Vert_1^2$ and thus
	\begin{align*}
	\frac{\theta^TB\theta }{\frac{1}{s_{0,+}}\Vert \theta_{S_{0,+}} \Vert_1^2} \ge \frac{\theta^TA\theta }{\frac{1}{s_{0,+}}\Vert \theta_{S_{0,+}} \Vert_1^2} - 16\delta s_{0,+} \ge \kappa^2(A, s_{0,+}) - 16\delta s_{0,+}.
	\end{align*}
	Minimizing the left-hand side over all $\theta \neq 0$ with $\Vert \theta_{S^{c}_{0,+}} \Vert_1 \le 3 \Vert \theta_{S_{0,+}} \Vert_1$ proves the claim.
\end{proof}

This shows that to control $\kappa^2(\Sigma, s_{0,+})$, we need to control the maximum element-wise distance between $\Sigma$ and $\Sigma_A$: $\max_{ij} \vert \Sigma_{ij} - \Sigma_{A,ij} \vert$.
We will now show that in the setting of Proposition \ref{Prop: compatibility condition Sigma},
\[
	\max_{ij} \vert \Sigma_{ij} - \Sigma_{A,ij} \vert \le \frac{c_{\text{min}}}{32s_{0,+}},
\]
and thus, by Lemma \ref{Lem: Lemma 6 in KockTang}, we have 
$
\kappa^2(\Sigma, s_{0,+}) \ge \kappa^2(\Sigma_A, s_{0,+}) - \frac{c_{\text{min}}}{2} \ge \frac{c_{\text{min}}}{2} > 0
$
and i.e.~the compatibility condition holds for $\Sigma$.

\begin{proof}[Proof of Proposition \ref{Prop: compatibility condition Sigma}]
	To make referencing of sections of $\Sigma$ easier, we number its blocks as follows
	\begin{equation*}
		\Sigma = \frac{1}{\binom{n}{2}} \begin{bmatrix}
			\underbrace{ \frac{n}{2} X^TX }_{\text{\textcircled{1}}} & \underbrace{ \frac{\sqrt{n}}{\sqrt{2}} X^T\textbf{1} }_{\text{\textcircled{2}}} & \underbrace{\textbf{0} }_{\text{\textcircled{3}}} \\
			\underbrace{\frac{\sqrt{n}}{\sqrt{2}}\textbf{1}^TX}_{\text{\textcircled{4}}} & \underbrace{\textbf{1}^T\textbf{1} }_{\text{\textcircled{5}}}& \underbrace{\textbf{0}}_{\text{\textcircled{6}}} \\
			\underbrace{\textbf{0}}_{\text{\textcircled{7}}} & \underbrace{ \textbf{0} }_{\text{\textcircled{8}}}& \underbrace{\E [Z^TZ]}_{\text{\textcircled{9}}}
		\end{bmatrix}
	\end{equation*}
	For $i,j = 1, \dots, n$, we have $\Sigma_{ij} = \Sigma_{A,ij}$ (block \textcircled{1}). The entry at position $(n+1),(n+1)$ (block \textcircled{5}) is also equal and so are blocks \textcircled{3}, \textcircled{6}, \textcircled{7}, \textcircled{8} and \textcircled{9}.
	For the entries at positions $i,j$ with $i = n+1$ and $j = 1, \dots, n$ as well as positions with $i = 1, \dots, n$ and $j = n+1$ (blocks \textcircled{2} and \textcircled{4}), we have:
	\[
	\Sigma_{ij} - \Sigma_{A,ij}= \Sigma_{ij}= \frac{(n-1)\sqrt{2}}{(n-1) \sqrt{n}} = \frac{\sqrt{2}}{\sqrt{n}} \le \frac{c_{\min}}{32s_{0,+}}
	\]
	for $n \gg 0$, since we assume that $s_{0,+} = o(\sqrt{n})$. The claim now follows from Lemma \ref{Lem: Lemma 6 in KockTang}.
\end{proof}

Analog to Proposition \ref{Prop: compatibility condition Sigma}, when our model might be misspecified, we have

\begin{Prop}\label{Misspecification: compatibility condition Sigma}
	Under Assumption \ref{Assum: minimum EW}, for $s_{+}^* = o(\sqrt{n})$ and $n$ large enough,
	it holds that for every ${\theta} \in \R^{n+1+p}$ with $\Vert {\theta}_{S^{*c}_+} \Vert_1 \le 3 \Vert {\theta}_{S^*_+} \Vert_1$,
	\[
		\Vert {\theta}_{S^*_+} \Vert_1^2 \le \frac{2s^*_+}{c_{\min}} {\theta}^T \Sigma { \theta}.
	\]
\end{Prop}

\begin{proof}[Proof of Proposition \ref{Misspecification: compatibility condition Sigma}] The proof follows step by step as above.
\end{proof}

\section{Consistency with covariates}\label{Appendix: Consistency with covariates}

In this section, we don't assume $r_n \ge r_{n,0}$ in the beginning and prove Theorem \ref{Thm: consistency} at first. Then Theorem \ref{Cor: no approximation error} follows. Finally, we prove Proposition \ref{Prop: q error bound}.

\subsection{A rescaled penalized likelihood problem}\label{Sec: rescaled problem}

We already mentioned in Section \ref{Sec: SBetaM covariates} that it is possible to present an equivalent formulation of the problem \eqref{Eq: Penalized llhd with covariates} in terms of a rescaled likelihood problem using the sample-size adjusted design matrix $\bar D$.
We will rely heavily on this formulation which we now make precise.

Recall that in the definition of $\bar D$ we effectively blew up the entries belonging to $\beta$. The blow-up factor was chosen precisely such that we can now reformulate our problem as a problem in which each parameter effectively has sample size $\binom{n}{2}$. That is, our original penalized likelihood problem can be rewritten as
\begin{align}\label{Eq: Penalized llhd with covariates bar}
\begin{split}
\hat{\bar{ \theta}} = (\hat{\bar{\beta}}, \hat{\mu}, \hat{\gamma}) = \argmin_{\bar{\beta},\mu, \gamma} &\frac{1}{\binom{n}{2}} \Bigg( - \sum_{i = 1}^n  \frac{\sqrt{n}}{\sqrt{2}}\bar{\beta_i} d_i - d_+\mu - \sum_{i < j} (Z_{ij}^T\gamma)A_{ij} \\
&+ \sum_{i < j} \log\left(1 + \exp\left( \frac{\sqrt{n}}{\sqrt{2}}\bar{\beta_i} +  \frac{\sqrt{n}}{\sqrt{2}}\bar{\beta_j} + \mu + Z_{ij}^T\gamma\right)\right) \Bigg) \\
&+ \bar{\lambda} \Vert \bar{\beta} \Vert_1,
\end{split}
\end{align}
where $\bar{\lambda} =  \frac{\sqrt{n}}{\sqrt{2}}\lambda$ and the argmin is taken over $\bar{\Theta}_{\text{loc}} = \{\bar{\theta} \in \Theta: \Vert \bar{ D}\bar{\theta} \Vert_\infty \le r_n \}$. Note that by the same arguments as before, $\bar{\Theta}_{\text{loc}}$ is convex.
Then, given a solution $(\hat{\bar{\beta}}, \hat{\mu}, \hat{\gamma})$ for a given penalty parameter $\bar{\lambda}$ to this modified problem (\ref{Eq: Penalized llhd with covariates bar}), we can obtain a solution to our original problem (\ref{Eq: Penalized llhd with covariates}) with penalty parameter $\lambda = \bar{\lambda}\sqrt{2}/\sqrt{n}$, by setting
\[
(\hat{\beta}, \hat{\mu}, \hat{\gamma}) = \left( \frac{\sqrt{n}}{\sqrt{2}}\hat{\bar{\beta}}, \hat{\mu}, \hat{\gamma} \right).
\]
For a compacter way of writing, introduce the following notation:
For any parameter $\theta = (\beta^T, \mu, \gamma^T)^T \in \Theta$, we introduce the notation
\[
\bar{\theta} = \left( \frac{\sqrt{2}}{\sqrt{n}} \beta, \mu, \gamma \right)
\]
and also write $\bar{\beta} =  \frac{\sqrt{2}}{\sqrt{n}} \beta$. In particular we use the notation $\bar{\theta}_0 = (\bar{\beta}_0^T, \mu_0, \gamma_0^T)^T, \bar{\beta}_0 =  \frac{\sqrt{2}}{\sqrt{n}} \beta_0 $, to denote the re-parametrized truth and $\bar{\theta}^* = (\bar{\beta}^{*T}, \mu^*, \gamma^{*T})^T, \bar{\beta}^* =  \frac{\sqrt{2}}{\sqrt{n}} \beta^*$ to denote the re-parametrized best local approximation. Note that for any $\theta \in \Theta$, $D\theta = \bar{ D}\bar{\theta}$ and hence the bound $r_n$ is the same in the definitions of $\Theta_{\text{loc}}$ and $\bar{ \Theta}_{\text{loc}}$.
Also, since rescaling the set $\R_+^n$ still results in $\R_+^n$, there is no need to introduce a set $\bar{ \Theta}$. Note that $\theta \in \Theta_{\text{loc}}$ if and only if $\bar{\theta} \in \bar{ \Theta}_{\text{loc}}$.

For any $\bar{\theta} = (\bar{\beta}^T, \mu, \gamma)^T$, denote the negative log-likelihood function corresponding to the rescaled problem (\ref{Eq: Penalized llhd with covariates bar}) as
\begin{align*}
	\bar{\mathcal{L}}(\bar{\theta}) &= - \sum_{i = 1}^n  \frac{\sqrt{n}}{\sqrt{2}}\bar{\beta_i} d_i - d_+\mu - \sum_{i < j} (Z_{ij}^T\gamma)A_{ij} \\
	&\quad + \sum_{i < j} \log\left(1 + \exp\left( \frac{\sqrt{n}}{\sqrt{2}}\bar{\beta_i} +  \frac{\sqrt{n}}{\sqrt{2}}\bar{\beta_j} + \mu + Z_{ij}^T\gamma\right)\right).
\end{align*}
Then, clearly $\bar{\mathcal{L}}(\bar{\theta}) = \mathcal{L}(\theta)$ and
\begin{equation*}
\bar{\mathcal{E}}(\bar{\theta}) \coloneqq \frac{1}{\binom{n}{2}} (\mathbb{E}[\bar{\mathcal{L}}(\bar{\theta})] - \mathbb{E}[\bar{\mathcal{L}}(\bar{\theta}^*)] ) = \mathcal{E}(\theta).
\end{equation*}
Thus, $\bar{\theta}^*$ fulfills
\begin{equation*}
\bar{\theta}^* = \argmin_{\theta \in \bar{\Theta}_{\text{loc}}} \bar{\mathcal{E}}(\bar{\theta}), 
\end{equation*}
i.e. $\bar{\theta}^*$ is the best local re-parametrized solution.
\newline
\newline
\indent
To give us a more compact way of writing, for any $\theta \in \Theta$ we introduce functions $f_\theta: \R^{n+1+p} \rightarrow \R, f_\theta(v) = v^T\theta$ and denote the function space of all such $f_\theta$ by $\mathbb{F} \coloneqq \{ f_{\theta} : \theta \in \Theta \}$. We endow  $\mathbb{F}$ with two norms as follows. 
Denote the law of the rows of $\bar{D}$ on $\R^{n+1+p}$,~i.e. the probability measure induced by $(\bar{X}_{ij}^T, 1, Z_{ij}^T)^T, i <j$, by $\bar{Q}$. That is, for a measurable set $A = A_1 \times A_2 \subset \R^{n+1} \times \R^p$, 
\[
	\bar Q(A) = \frac{1}{\binom{n}{2}} \sum_{ i < j} P(\bar D_{ij} \in A) = \frac{1}{\binom{n}{2}} \sum_{ i < j} \bar{\delta}_{ij}(A_1) \cdot P(Z_{ij} \in A_2),
\]
where $\bar{\delta}_{ij}(A_1) = 1$ if $(\bar{X}_{ij}^T,1)^T \in A_1$ and zero otherwise, is the Dirac-measure. We are interested in the $L_2$ and $L_\infty$ norm on $\mathbb{F}$ with respect to the measure $\bar{Q}$ on $\R^{n+1}\times\R^p$. Denote the $L_2(\bar{Q})$-norm of $f \in \mathbb{F}$ simply by $\Vert \, . \, \Vert_{\bar{Q}}$ and let $\E_Z$ be the expectation with respect to $Z$:
\[
\Vert f \Vert^2_{\bar{Q}} \coloneqq \Vert f \Vert_{L_2(\bar{Q} )}^2 = \int_{\R^{n+1}\times\R^p} f(v)^2 \bar{Q}(dv) = \frac{1}{\binom{n}{2}}\sum_{ i < j} \E_Z[ f((\bar{X}_{ij}^T, 1, Z_{ij}^T)^T)^2]
\]
and define the $L_\infty(\bar{Q})$-norm as usual as the $\bar{Q}$-a.s. smallest upper bound of $f$:
\begin{align*}
\Vert f \Vert_{\bar{Q}, \infty} &= \inf\{ C \ge 0 : \vert f(v) \vert \le C \text{ for } \bar{Q} \text{-almost every } v \in \R^{n+1+p} \}.
\end{align*}
Notice in particular, that for any $f_\theta \in \mathbb{F}, \theta = (\beta^T, \mu, \gamma^T)^T \in \Theta_{\text{loc}}$: $\Vert f_\theta \Vert_\infty \le \sup_{Z_{ij}}\Vert D\theta \Vert_\infty \le r_n$.

We make the analogous definitions for the unscaled design matrix. 
Define the probability measure induced by the rows of $D$ on $\R^{n+1+p}$ as $Q$.
It is easy to see that we can switch between these norms as follows. Given a parameter $\theta$ and its rescaled version $\bar{\theta}$, then clearly
\[
\Vert f_{\bar{\theta}} \Vert_{\bar{Q}} = \Vert f_\theta \Vert_Q, \quad	\Vert f_{\bar{\theta}} \Vert_{\bar{Q}, \infty} = \Vert f_\theta \Vert_{Q, \infty}.
\]
Also note that for any $\bar{\theta}$
\begin{equation}\label{Eq: Q bat norm identity}
	\Vert f_{\bar{\theta}} \Vert_{\bar{Q}}^2 = \E_Z \left[\frac{1}{\binom{n}{2}}\sum_{ i < j}  (\bar D_{ij}^T\bar \theta)^2\right] = \bar \theta^T \Sigma \bar \theta.
\end{equation}
Recall that we want to apply the compatibility condition to vectors of the form $\bar{\theta} = \bar{\theta}_1 - \bar{\theta}_2, \bar{\theta}_1, \bar{\theta}_2 \in \bar{\Theta}_{\text{loc}}$. We have the following corollary which follows immediately from Proposition \ref{Misspecification: compatibility condition Sigma}.

\begin{Kor}\label{Lem: compatibility condition}
	Under Assumption \ref{Assum: minimum EW}, for $s_{+}^* = o(\sqrt{n})$ and $n$ large enough,
	it holds that for every $\bar{\theta} = \bar{\theta}_1 - \bar{\theta}_2, \bar{\theta}_1, \bar{\theta}_2 \in \bar{\Theta}_{\text{loc}}$ with $\Vert \bar{\theta}_{S^{*c}_+} \Vert_1 \le 3 \Vert \bar{\theta}_{S^*_+} \Vert_1$,
	\[
	\Vert \bar{\theta}_{S^*_+} \Vert_1^2 \le \frac{2s^*_+}{c_{\min}} \Vert f_{\bar{\theta}_1} - f_{\bar{\theta}_2} \Vert_{\bar{Q}}^2.
	\]
\end{Kor}

\begin{proof}
		The proof follows from Proposition \ref{Misspecification: compatibility condition Sigma} and identity \eqref{Eq: Q bat norm identity}.
\end{proof}

\subsection{Two basic inequalities}

A key result in the consistency proofs in classical LASSO settings is the so called \textit{basic inequality} (cf. \cite{vandegeer2011}, Chapter 6). We give two formulations of it, one for the original penalized likelihood problem (\ref{Eq: Penalized llhd with covariates}) and one, completely analogous result, for the rescaled problem (\ref{Eq: Penalized llhd with covariates bar}). To that end, let $P_n$ denote the empirical measure with respect to our observations $(A_{ij}, Z_{ij})$, that is, for any suitable function $g$,
$
P_ng \coloneqq \sum_{i < j} g(A_{ij}, Z_{ij})/\binom{n}{2}.
$
In particular, if we let for each $\theta \in \Theta$, 
$$l_\theta(A_{ij}, Z_{ij}) = -A_{ij} (\beta_i + \beta_j + \mu + \gamma^TZ_{ij}) + \log(1 + \exp(\beta_i + \beta_j + \mu + \gamma^TZ_{ij})),$$
then
$
P_nl_{\theta} =  \mathcal{L}(\theta)/\binom{n}{2}.
$
Similarly, we define  $P = \mathbb{E}P_n$. In particular,
$
Pl_{\theta} = \mathbb{E}P_n l_\theta = \mathbb{E}[\mathcal{L}(\theta)/\binom{n}{2}],
$
where we suppress the dependence on $n$ in our notation. 
We define the \textit{empirical process} as
\[
\left\{ v_n(\theta) = (P_n - P)l_{\theta} : \theta \in \Theta \right\},
\]
which can also be written in re-parametrized form as
\[
\bar{v}_n(\bar{\theta}) \coloneqq \frac{1}{\binom{n}{2}} (\bar{\mathcal{L}}(\bar{\theta}) - \mathbb{E}[\bar{\mathcal{L}}(\bar{\theta})] ) = v_n(\theta).
\]

\begin{Lem}\label{Lem: basic inequality}
	For any $\theta = (\beta^T, \mu, \gamma^T)^T \in \Theta_{\textup{loc}}$, it holds
	\[
	\mathcal{E}(\hat{\theta}) + \lambda \Vert \hat{\beta} \Vert_1 \le - [v_n(\hat{\theta}) - v_n(\theta)] + \mathcal{E}(\theta) + \lambda \Vert \beta \Vert_1.
	\]
\end{Lem}

\begin{proof}
	Plugging in the definitions, the above equation is equivalent to
	\begin{align*}
	\frac{1}{\binom{n}{2}}\left(\mathbb{E}[\mathcal{L}(\hat{\theta})] - \mathbb{E}[\mathcal{L}(\theta^*)]\right) + \lambda \Vert \hat{\beta} \Vert_1 &\le 
	-\frac{1}{\binom{n}{2}} \mathcal{L}(\hat{\theta}) + \frac{1}{\binom{n}{2}} \mathbb{E}[\mathcal{L}(\hat{\theta})] +\frac{1}{\binom{n}{2}} \mathcal{L}(\theta) - \frac{1}{\binom{n}{2}} \mathbb{E}[\mathcal{L}(\theta)] + \lambda \Vert \beta \Vert_1 \\
	&\quad \quad + \frac{1}{\binom{n}{2}} \left(\mathbb{E}[\mathcal{L}(\theta)] - \mathbb{E}[\mathcal{L}(\theta^*)]\right).
	\end{align*}
	Rearranging shows that this is true if and only if
	\[
	\frac{1}{\binom{n}{2}} \mathcal{L}(\hat{\theta}) + \lambda \Vert \hat{\beta} \Vert_1
	\le \frac{1}{\binom{n}{2}} \mathcal{L}(\theta) + \lambda \Vert \beta \Vert_1,
	\]
	which is true by definition of $\hat{\theta}$.
\end{proof}

\begin{Rem}
	For any $0 < t < 1$ and $\theta \in \Theta_{\text{loc}}$, let $\tilde{\theta} = t\hat{\theta} + (1 - t)\theta$. Since $\Gamma$ is convex, $\tilde{\theta} \in \Theta_{\text{loc}}$ and since $\theta \rightarrow l_{\theta}$ and $\Vert \, . \, \Vert_1$ are convex functions, we can replace $\hat{\theta}$ by $\tilde{\theta}$ in the basic inequality and still obtain the same result. Plugging in the definitions, we see that the basic inequality is equivalent to the following:
	\begin{gather*}
	\mathcal{E}(\tilde{\theta}) + \lambda \Vert \tilde{\beta} \Vert_1 \le - [v_n(\tilde{\theta}) - v_n(\theta)] + \lambda \Vert \beta \Vert_1 + \mathcal{E}(\theta)\\
	\iff \frac{1}{\binom{n}{2}} \mathcal{L}(\tilde{\theta}) + \lambda \Vert \tilde{\beta} \Vert_1
	\le \frac{1}{\binom{n}{2}} \mathcal{L}(\theta) + \lambda \Vert \beta \Vert_1
	\end{gather*}
	and by convexity
	\[
	\frac{1}{\binom{n}{2}} \mathcal{L}(\tilde{\theta}) + \lambda \Vert \tilde{\beta} \Vert_1
	\le  \frac{1}{\binom{n}{2}} t \mathcal{L}(\hat{\theta}) + \frac{1}{\binom{n}{2}} (1-t) \mathcal{L}(\theta) + t \lambda \Vert \hat{\beta} \Vert_1 + (1-t) \lambda \Vert \beta \Vert_1 \le \frac{1}{\binom{n}{2}} \mathcal{L}(\theta) + \lambda \Vert \beta \Vert_1,
	\]
	where the last inequality follows by definition of $\hat{\theta}$. In particular, for any $M > 0$, choosing
	\[
	t = \frac{M}{M + \Vert \hat{\theta} - \theta \Vert_1},
	\]
	gives $\Vert \tilde{\theta} - \theta \Vert_1 \le M$.
\end{Rem}

\begin{Lem}\label{Lem: basic inequality bar}
	For any $\bar{\theta} \in \bar{\Theta}_{\textup{loc}}$ it holds
	\[
	\bar{\mathcal{E}}(\hat{\bar{\theta}}) + \bar{\lambda} \Vert \hat{\bar{\beta}} \Vert_1 \le - [\bar{v}_n(\hat{\bar{\theta}}) - \bar{v}_n(\bar{\theta})] + \bar{\mathcal{E}}(\bar{\theta}) + \bar{\lambda} \Vert \bar{\beta} \Vert_1.
	\]
\end{Lem}
Since the proof of Lemma \ref{Lem: basic inequality} only relies on the argmin property of $\hat{\theta}$, the proof of Lemma \ref{Lem: basic inequality bar} is line by line the same as for Lemma \ref{Lem: basic inequality}. We also get the same property for convex combinations of $\hat{\bar{\theta}}$ and $\bar{ \theta}$: For any $t \in (0,1)$ the rescaled basic inequality Lemma \ref{Lem: basic inequality bar} holds for $\hat{\bar{\theta}}$ replaced by $\tilde{\theta} = t\hat{\bar{\theta}} + (1-t) \bar{\theta}$. Note in particular, that $\tilde{\theta} \in \bar{ \Theta}_{\text{loc}}$.
Take note that in the basic inequalities we are controlling the \textit{global} excess risk of any \textit{local} parameters $\theta \in \Theta_{\text{loc}}$.

\subsection{Lower quadratic margin for \texorpdfstring{$\mathcal{E}$}{E}}\label{Sec: quadratic margin covariates}

In this section, we will derive a lower quadratic bound on the excess risk $\mathcal{E}(\theta)$ if the parameter $\theta$ is close to the truth $\theta_0$. This is a necessary property for the proof to come and is referred to as the \textit{margin condition} in classical LASSO theory (cf. \cite{vandegeer2011}). We will conduct our derivations for the original parameter space $\Theta_{\text{loc}}$. Since $\mathcal{L}(\theta) = \bar{\mathcal{L}}(\bar{ \theta})$ and $\mathcal{E}(\theta) = \bar{\mathcal{E}}(\bar{ \theta})$, we will find that the same results hold in the rescaled model.

The proof mainly relies on a second order Taylor expansion of the function $l_\theta$ of introduced in Section \ref{Sec: Theory}.
Given a fixed $\theta$, we treat $l_\theta$ as a function in $\theta^Tx$ and define new functions $l_{ij}: \R \rightarrow \R, i < j,$
\[
l_{ij}(a) = \mathbb{E}[l_\theta(A_{ij}, a) \vert Z_{ij}] = -p_{ij}a + \log(1+\exp(a)),
\]
where $p_{ij} = P(A_{ij} = 1 \vert Z_{ij})$ and by slight abuse of notation we use $l_\theta(A_{ij}, a) \coloneqq -A_{ij} a + \log(1 + \exp(a))$. Taking derivations, it is easy to see that
\[
f_{\theta_0} ((X_{ij}^T, 1, Z_{ij}^T)^T) \in \arg\min_a l_{ij}(a).
\]
Note that we are using the actual truth $\theta_0$ in the above equation, not the best local approximation $\theta^*$. Write $f_0 = f_{\theta_0}$.

All $l_{ij}$ are clearly twice continuously differentiable with derivative
\[
\frac{\partial^2}{\partial a^2}l_{ij}(a) = \frac{\exp(a)}{(1 + \exp(a))^2} > 0, \forall a \in \R.
\]
Using a second order Taylor expansion around $a_0 = f_0((X_{ij}^T, 1, Z_{ij}^T)^T)$ we get
\[
l_{ij}(a) = l_{ij}(a_0) + l'(a_0) (a - a_0) + \frac{l''(\bar{a})}{2}(a - a_0)^2 =  l_{ij}(a_0) + \frac{l''(\bar{a})}{2}(a - a_0)^2,
\]
with an $\bar{a}$ between $a$ and $a_0$. Note that $\frac{\exp(a)}{(1 + \exp(a))^2}$ is symmetric and monotone decreasing for $a \ge 0$.

\noindent Then we have
\begin{align}\label{Eq: derivation lower bound}
\begin{split}
l_{ij}(a) - l_{ij}(a_0) &= \frac{\exp(\bar{a})}{(1 + \exp(\bar{a}))^2} \frac{(a - a_0)^2}{2} \\
&= \frac{\exp(\vert\bar{a}\vert)}{(1 + \exp(\vert\bar{a}\vert))^2} \frac{(a - a_0)^2}{2} \\
&\ge \frac{\exp(\max\{\vert a_0 \vert, \vert a \vert \})}{(1 + \exp(\max\{\vert a_0 \vert, \vert a \vert \}))^2} \frac{(a - a_0)^2}{2} \quad (\text{since } \vert \bar{a} \vert \le \max\{\vert a_0 \vert, \vert a \vert \})
\end{split}
\end{align}
In particular, for any $\theta \in \Theta_{\text{loc}}$, $\vert f_\theta((X_{ij}^T, 1, Z_{ij}^T)^T) \vert \le r_{n}$, we have
\begin{align*}
l_{ij}(f_{\theta} ((X_{ij}^T, 1, Z_{ij}^T)^T)) &- l_{ij}(f_0((X_{ij}^T, 1, Z_{ij}^T)^T)) \\
\begin{split}
		 &\ge \frac{\exp(\max\{\vert f_0((X_{ij}^T, 1, Z_{ij}^T)^T\vert ,\vert f_{\theta}((X_{ij}^T, 1, Z_{ij}^T)^T\vert\})}{(1 + \exp(\max\{\vert f_0((X_{ij}^T, 1, Z_{ij}^T)^T\vert ,\vert f_{\theta}((X_{ij}^T, 1, Z_{ij}^T)^T\vert\}))^2} \\ 
		 &\quad \quad \quad \cdot \frac{(f_\theta((X_{ij}^T, 1, Z_{ij}^T)^T) - f_0((X_{ij}^T, 1, Z_{ij}^T)^T))^2}{2} 
\end{split} \\
&\ge  \frac{\exp(\max\{ r_{n,0}  , r_{n}   \})}{(1 + \exp( \max\{ r_{n,0}  ,  r_{n}  \}))^2} \frac{(f_\theta((X_{ij}^T, 1, Z_{ij}^T)^T) - f_0((X_{ij}^T, 1, Z_{ij}^T)^T))^2}{2}.
\end{align*}
Define the function
\[
	\tau =  \frac{\exp(\max\{r_{n,0}  ,  r_{n}   \})}{2(1 + \exp( \max\{r_{n,0}  , r_{n} \}))^2} 
\]
and notice that for $K_n$ defined in \eqref{Eq: Def K_n} we now have
\begin{equation*}
K_n =  \frac{1}{\tau}.
\end{equation*}
Now, for any $\theta \in \Theta_{\text{loc}}$
\begin{align*}
\mathcal{E}(\theta)
&= \frac{1}{\binom{n}{2}} \sum_{ i < j} \mathbb{E} [ l_{\theta}(A_{ij}, (X_{ij}^T, Z_{ij}^T)^T) - l_{\theta_0}(A_{ij},  (X_{ij}^T, Z_{ij}^T)^T) ] \\
&= \frac{1}{\binom{n}{2}} \sum_{ i < j} \E[( l_{ij}(f_\theta ( (X_{ij}^T, Z_{ij}^T)^T)) - l_{ij}(f_0( (X_{ij}^T, Z_{ij}^T)^T)))] \\
&\ge \frac{1}{\binom{n}{2}} \sum_{ i < j} \tau \E[(f_{\theta}( (X_{ij}^T, Z_{ij}^T)^T) - f_0( (X_{ij}^T, Z_{ij}^T)^T))^2] \\
&\ge \frac{1}{K_n} \Vert f_{\theta} - f_0 \Vert_Q^2.
\end{align*}
Thus, we have obtained a lower bound for the excess risk given by the quadratic function $G_n(\Vert f_{\theta} - f_0 \Vert)$ where $G_n(u) = 1/K_n \cdot u^2$.
Since $\mathcal{E}(\theta) = \bar{\mathcal{E}}(\bar{ \theta})$ and $\Vert f_{\theta} \Vert_Q^2 = \Vert f_{\bar{\theta}} \Vert_{\bar{Q}}^2$ and $\Vert f_{\theta} \Vert_{Q, \infty} = \Vert f_{\bar{\theta}} \Vert_{\bar{Q}, \infty}$, we obtain the same result for the rescaled problem (\ref{Eq: Penalized llhd with covariates bar}): For any $\bar{ \theta} \in \bar{ \Theta}_{\text{loc}}$, we have
\begin{equation*}
\bar{\mathcal{E}}(\bar{ \theta}) \ge \frac{1}{K_n} \Vert f_{\bar{\theta}} - f_{\bar{\theta}_0} \Vert_{\bar{Q}}^2.
\end{equation*}
That is, the quadratic margin condition holds for any $\bar \theta \in \bar{\Theta}_{\text{loc}}$.

Recall that the convex conjugate of a strictly convex function $G$ on $[0, \infty)$ with $G(0) = 0$ is defined as the function
\[
H(v) = \sup_{u} \{ uv - G(u) \}, \quad v > 0
\]
and in particular, if $G(u) = cu^2$ for a positive constant $c$, we have $H(v) = v^2/(4c)$. 
Hence, the convex conjugate of $G_n$ is
\begin{equation*}
H_n(v) = \frac{v^2 K_n}{4}.
\end{equation*}
Keep in mind that by definition for any $u,v$
\[
uv \le G(u) + H(v).
\]

\subsection{Consistency on a special set}\label{Sec: Consistency on a special set}

In this section we will show that the penalized likelihood estimator is consistent in the sense that it converges to the best possible approximation $\theta^*$. We will first define a set $\mathcal{I}$ and show that consistency holds on $\mathcal{I}$. It will then suffice to show that the probability of $\mathcal{I}$ tends to one as well. The proof follows in spirit \cite{vandegeer2011}, Theorem 6.4.

Define
\[
\epsilon^* =  \frac{3}{2}\bar{ \mathcal{E}}(\bar{\theta}^*) + H_n\left(   \frac{4\sqrt{2s^*_+}\bar{\lambda}}{\sqrt{c_{\min}}}\right).
\]
Remember that $\bar{ \mathcal{E}}(\bar{\theta}^*) =  \mathcal{E}(\theta^*)$ corresponds to the approximation error of our model. 
Let for any $M > 0$
\[
Z_M \coloneqq \sup_{\substack{\bar{\theta} \in \bar{\Theta}_{\text{loc}}, \\  \Vert \bar{\theta} - \bar{\theta}^* \Vert_1 \le M}} \vert \bar{v}_n(\bar{\theta}) - \bar{v}_n(\bar{\theta}^*) \vert,
\]
where $\bar{v}_n$ denotes the re-parametrized empirical process. Recall that for any rescaled $\bar{\theta}$ we have 
$
\bar{v}_n(\bar{\theta}) = v_n(\theta).
$
Also, by construction $\bar{\theta} \in \bar{\Theta}_{\text{loc}}$ if and only if $\theta \in \Theta_{\text{loc}}$. Hence, the set over which we are maximizing in the definition of $Z_M$ can be expressed in terms of parameters $\theta$ on the original scale as
\[
\left\{ \theta = (\beta^T, \mu, \gamma^T)^T \in \Theta_{\text{loc}}: \frac{\sqrt{2}}{\sqrt{n}} \Vert \beta - \beta^* \Vert_1 + \vert \mu - \mu^* \vert + \Vert \gamma - \gamma^* \Vert_1 \le M \right\}.
\]
To ease notation, define
\[
 \lambda_0 = 8a_n + 2 \sqrt{\frac{t}{\binom{n}{2}}( 11 (1 \vee (c^2p) ) + 8\sqrt{2}(1 \vee c) \sqrt{n} a_n  )} + \frac{2\sqrt{2}t(1 \vee c) \sqrt{n}}{3\binom{n}{2}}
\]
and set
\[
M^* \coloneqq \epsilon^* / \lambda_0,
\]
where $\lambda_0$ is a lower bound on $\bar{\lambda}$ that will be made precise in the proof showing that $\mathcal{I}$ has large probability. Define
\begin{equation}\label{Def: I covariates}
\mathcal{I} \coloneqq \{ Z_{M^*} \le \lambda_0 M^* \} = \{ Z_{M^*} \le \epsilon^* \}.
\end{equation}

\begin{Satz}\label{Thm: Consistency covariates}
	Assume that assumptions \ref{Assum: minimum EW} and \ref{Assum: rate of s^*} hold and that
	$
	\bar{\lambda} \ge 8 \lambda_0.
	$
	Then, on the set $\mathcal{I}$, we have
	\begin{align*}
	\mathcal{E}(\hat{\theta}) + \bar{ \lambda} \left( \frac{\sqrt{2}}{\sqrt{n}} \Vert \hat{\beta} - \beta^*  \Vert_1 + \vert \hat{\mu} - \mu^* \vert + \Vert \hat{\gamma} - \gamma^* \Vert_1  \right) \le 4\epsilon^*
	 = 6 \mathcal{E}(\theta^*) + 4 H_n\left(   \frac{4\sqrt{2s^*_+}\bar{ \lambda}}{\sqrt{c_{\min}}}\right).
	\end{align*}
\end{Satz}

\begin{proof}[Proof of Theorem \ref{Thm: Consistency covariates}]
	We assume that we are on the set $\mathcal{I}$ throughout. Set
	\[
	t = \frac{M^*}{M^* + \Vert \hat{\bar{ \theta}} - \bar{ \theta}^* \Vert_1}
	\]
	and $\tilde{\theta} = (\tilde{\beta}^T, \tilde{\mu}, \tilde{\gamma}^T)^T = t \hat{\bar{ \theta}} + (1 - t)\bar{\theta}^*$. Then, 
	\[
	\Vert \tilde{\theta} - \bar{\theta}^* \Vert_1 = t \Vert \hat{\bar{ \theta}} - \bar{\theta}^* \Vert \le M^*.
	\]
	Since $\hat{\bar{ \theta}}, \bar{ \theta}^* \in \bar{\Theta}_{\text{loc}}$ and by the convexity of $\bar{\Theta}_{\text{loc}}$, $\tilde{ \theta} \in \bar{\Theta}_{\text{loc}}$, and by the remark after Lemma \ref{Lem: basic inequality bar}, the basic inequality holds for $\tilde{\theta}$:
	\begin{align*}
	\bar{ \mathcal{E}}(\tilde{ \theta}) + \bar{ \lambda} \Vert \tilde{\beta} \Vert_1 &\le - (\bar{v}_n(\tilde{ \theta}) - \bar{v}_n(\bar{\theta}^*) ) + \bar{ \mathcal{E}}(\bar{ \theta}^*) + \bar{ \lambda} \Vert \bar{\beta}^* \Vert_1 \\
	&\le Z_{M^*} + \bar{ \lambda} \Vert \bar{\beta}^* \Vert_1 + \bar{ \mathcal{E}}(\bar{ \theta}^*) \\
	&\le \epsilon^* + \bar{ \lambda} \Vert \bar{\beta}^* \Vert_1 + \bar{ \mathcal{E}}(\bar{ \theta}^*).
	\end{align*}
	From now on write $\tilde{\mathcal{E}} = \bar{\mathcal{E}}(\tilde{\theta})$ and $\mathcal{E}^* = \bar{ \mathcal{E}}(\bar{ \theta}^*)$.
	Note, that $\Vert \tilde{\beta} \Vert_1 = \Vert \tilde{\beta}_{S^{*c}} \Vert_1 + \Vert \tilde{\beta}_{S^*} \Vert_1$ and thus, by the triangle inequality,
	\begin{align}\label{Eq: 6.29 covariates}
	\begin{split}
	\tilde{\mathcal{E}} + \bar{\lambda} \Vert \tilde{\beta}_{S^{*c}} \Vert_1 &\le \epsilon^* + \bar{ \lambda} (\Vert \bar{\beta}^* \Vert_1 - \Vert \tilde{\beta}_{S^{*}} \Vert_1) + \mathcal{E}^*\\
	&\le \epsilon^* + \bar{ \lambda} (\Vert \bar{\beta}^* - \tilde{\beta}_{S^{*}} \Vert_1) + \mathcal{E}^* \\
	&\le \epsilon^* + \bar{ \lambda} (\Vert \bar{\beta}^* - \tilde{\beta}_{S^{*}} \Vert_1 + \Vert (\mu^*, \gamma^{*T})^T - (\tilde{\mu}, \tilde{\gamma}^T)^T \Vert_1) + \mathcal{E}^* \\
	&= \epsilon^* + \bar{ \lambda} \Vert (\tilde{\theta} - \bar{ \theta}^*)_{S^*_+} \Vert_1 + \mathcal{E}^* \\
	&\le 2\epsilon^* + \bar{ \lambda} \Vert (\tilde{\theta} - \bar{\theta}^*)_{S^*_+} \Vert_1.
	\end{split}
	\end{align}
	Where for the equality we have used that by assumption $\mu^*$ and $\gamma^*$ are active and hence $\Vert (\tilde{\theta} - \bar{\theta}^*)_{S^*_+} \Vert_1 = \Vert \bar{\beta}^* - \tilde{\beta}_{S^*} \Vert_1 + \Vert (\mu^*, \gamma^{*T})^T - (\tilde{\mu}, \tilde{\gamma}^T)^T \Vert_1$.
	\,
	\newline
	\newline
	\textbf{Case i)} If $\bar{ \lambda} \Vert (\tilde{\theta} - \bar{\theta}^*)_{S^*_+} \Vert_1 \ge \epsilon^*$, then
	\begin{equation}\label{Eq: 6.30 covariates}
	\bar{ \lambda} \Vert \tilde{\beta}_{S^{*c}} \Vert_1 \le \tilde{\mathcal{E}} + \bar{ \lambda} \Vert \tilde{\beta}_{S^{*c}} \Vert_1 \le 3\bar{ \lambda} \Vert (\tilde{\theta} - \bar{\theta}^*)_{S^*_+} \Vert_1.
	\end{equation}
	Since $\Vert (\tilde{\theta} - \bar{\theta}^*)_{S^{*c}_+} \Vert_1 = \Vert \tilde{\beta}_{S^{*c}} \Vert_1$, we may thus apply the compatibility condition corollary \ref{Lem: compatibility condition} (note that $\bar{\beta}^* = \bar{\beta}^*_{S^*}$) to obtain
	\[
	\Vert (\tilde{\theta} - \bar{ \theta}^*)_{S^*_+} \Vert_1 \le \frac{\sqrt{2s^*_+}}{\sqrt{c_{\min}}} \Vert f_{\tilde{ \theta}} - f_{\bar{ \theta}^*} \Vert_{\bar{Q}},
	\]
	where we have used that $\theta \mapsto f_\theta$ is linear and hence $f_{\tilde{ \theta} - \bar{\theta}^* } = f_{\tilde{ \theta}} - f_{\bar{ \theta}^*} $
	Observe that
	\begin{equation}\label{Eq: theta norm decomposition}
	\Vert \tilde{\theta} - \theta^*\Vert_1 = \Vert \tilde{\beta}_{S^{*c}} \Vert_1 + \Vert (\tilde{\theta} - \theta^*)_{S^*_+} \Vert_1.
	\end{equation}
	Hence, 
	\begin{align*}
	\tilde{\mathcal{E}} + \bar{ \lambda} \Vert \tilde{\theta} - \bar{ \theta}^* \Vert_1  &= \tilde{\mathcal{E}} + \bar{ \lambda} ( \Vert \tilde{\beta}_{S^{*c}} \Vert_1 + \Vert (\tilde{\theta} - \bar{\theta}^*)_{S^*_+} \Vert_1)\\
	&\le \epsilon^* + 2\bar{ \lambda} \Vert (\tilde{\theta} - \bar{\theta}^*)_{S^*_+} \Vert_1 + \mathcal{E}^*\\
	&\le \epsilon^* + \mathcal{E}^* + 2\bar{ \lambda}\frac{\sqrt{2s^*_+}}{\sqrt{c_{\min}}} \Vert f_{\tilde{\theta}} - f_{\bar{ \theta}^*} \Vert_{\bar{Q}}.
	\end{align*}
	Recall that for a convex function $G$ and its convex conjugate $H$ we have $uv \le G(u) + H(v)$. Since $\tilde{ \theta}, \bar{ \theta}^* \in \bar{ \Theta}_{\text{loc}}$, we obtain
	\begin{align*}
	2\bar{\lambda}\frac{\sqrt{s^*_+}}{\sqrt{c_{\min}}} \Vert f_{\tilde{\theta}} - f_{\bar{ \theta}^*} \Vert_{\bar{Q}} &= 4 \bar{\lambda} \frac{\sqrt{2s^*_+}}{\sqrt{c_{\min}}} \frac{\Vert f_{\tilde{\theta}} - f_{\bar{ \theta}^*} \Vert_{\bar{Q}}}{2} \\
	&\le 4 \bar{\lambda} \frac{\sqrt{2s^*_+}}{\sqrt{c_{\min}}} \frac{\Vert f_{\tilde{ \theta}} - f_{\bar{ \theta}_0} \Vert_{\bar{Q}} + \Vert f_{\bar{ \theta}^*} - f_{\bar{ \theta}_0} \Vert_{\bar{Q}}}{2} \\
	&\le H_n\left(4 \bar{\lambda} \frac{\sqrt{2s^*_+}}{\sqrt{c_{\min}}}\right) + G_n\left( \frac{\Vert f_{\tilde{ \theta}} - f_{\bar{ \theta}_0} \Vert_{\bar{Q}} + \Vert f_{\bar{ \theta}^*} - f_{\bar{ \theta}_0} \Vert_{\bar{Q}}}{2}\right) \\
	&\overset{G_n \text{ convex}}{\le}H_n\left(4 \bar{ \lambda} \frac{\sqrt{2s^*_+}}{\sqrt{c_{\min}}}\right)+ \frac{G_n(\Vert f_{\tilde{\theta}} - f_{\bar{ \theta}_0} \Vert_{\bar{Q}})}{2} + \frac{G_n(\Vert f_{\bar{\theta}^*} - f_{\bar{ \theta}_0} \Vert_{\bar{Q}})}{2}  \\
	&\overset{\text{margin condition}}{\le} H_n\left(4 \bar{ \lambda} \frac{\sqrt{2s^*_+}}{\sqrt{c_{\min}}}\right)+ \frac{\tilde{\mathcal{E}}}{2} + \frac{\mathcal{E}^*}{2}.
	\end{align*}
	It follows
	\begin{equation*}
	\tilde{\mathcal{E}} + \bar{ \lambda} \Vert\tilde{\theta} - \bar{\theta}^*\Vert_1 \le \epsilon^* + \frac{3}{2}\mathcal{E}^*+ H_n\left(4 \bar{\lambda} \frac{\sqrt{2s^*_+}}{\sqrt{c_{\min}}}\right)+ \frac{\tilde{\mathcal{E}}}{2} = 2 \epsilon^* + \frac{\tilde{\mathcal{E}}}{2}
	\end{equation*}
	and therefore
	\begin{equation}\label{Eq: 6.31 covariates}
	\frac{\tilde{\mathcal{E}}}{2} +\bar{\lambda} \Vert\tilde{\theta} - \bar{\theta}^*\Vert_1 \le 2 \epsilon^*.
	\end{equation}
	Finally, this gives
	\[
	\Vert\tilde{\theta} - \bar{\theta}^*\Vert_1 \le \frac{2 \epsilon^*}{\bar{\lambda}} = \frac{2 \lambda_0 M^*}{\bar{\lambda}} \underbrace{\le}_{\bar{ \lambda} \ge 4 \lambda_0} \frac{M^*}{2}.
	\]
	From this, by using the definition of $\tilde{\theta}$, we obtain
	\begin{align*}
	\Vert \tilde{\theta} - \bar{\theta}^* \Vert_1 = t \Vert \hat{\bar{\theta}} - \bar{\theta}^* \Vert_1 = \frac{M^*}{M^* + \Vert \hat{\bar{\theta}} - \bar{\theta}^* \Vert_1} \Vert \hat{\bar{\theta}} - \bar{\theta}^* \Vert_1 \le \frac{M^*}{2}.
	\end{align*}
	Rearranging gives
	\[
	\Vert \hat{\bar{\theta}} - \bar{\theta}^* \Vert_1 \le M^*.
	\]
	\,
	\newline
	\textbf{Case ii)} If $\bar{\lambda} \Vert (\bar{\theta}^* - \tilde{\theta})_{S^*_+} \Vert_1 \le \epsilon^*$, then from (\ref{Eq: 6.29 covariates})
	\[
	\tilde{\mathcal{E}} + \bar{\lambda} \Vert \tilde{\beta}_{S^{*c}} \Vert_1 \le 3\epsilon^*.
	\]
	Using once more (\ref{Eq: theta norm decomposition}), we get
	\begin{equation}\label{Eq: 6.32 covariates}
	\tilde{\mathcal{E}} + \bar{\lambda} \Vert \tilde{\theta} - \bar{\theta}^* \Vert_1 = \tilde{\mathcal{E}} + \bar{\lambda} \Vert \tilde{\beta}_{S^{*c}} \Vert_1 + \bar{\lambda}\Vert (\tilde{\theta} - \bar{\theta}^*)_{S^*_+} \Vert_1 \le 4\epsilon^*.
	\end{equation}
	Thus,
	\[
	\Vert \tilde{\theta} - \bar{\theta}^* \Vert_1 \le 4 \frac{\epsilon^*}{\bar{\lambda}} = 4 \frac{\lambda_0}{\bar{\lambda}}M^* \le \frac{M^*}{2}
	\]
	by choice of $\lambda \ge 8 \lambda_0$. Again, plugging in the definition of $\tilde{\theta}$, we obtain
	\[
	\Vert \hat{\bar{\theta}} - \bar{\theta}^* \Vert_1 \le M^*.
	\]
	Hence, in either case we have $\Vert \hat{\bar{\theta}} - \bar{\theta}^* \Vert_1 \le M^*$. That means, we can repeat the above steps with $\hat{\bar{\theta}}$ instead of $\tilde{ \theta}$: Writing $\hat{\mathcal{E}} \coloneqq \bar{\mathcal{E}}(\hat{\bar{\theta}})$, following the same reasoning as above we arrive once more at (\ref{Eq: 6.29 covariates}):
	\begin{equation}\label{Eq: 2.14}
	\hat{\mathcal{E}} + \bar{\lambda} \Vert \hat{\bar{\beta}}_{S^{*c}} \Vert_1 \le 2\epsilon^* + \bar{\lambda} \Vert \bar{\beta}^* - \hat{\bar{\beta}}_{S^*} \Vert_1 \le 2\epsilon^* + \bar{\lambda}\Vert (\hat{\bar{\theta}} - \bar{\theta}^*)_{S^*_+} \Vert_1.
	\end{equation}
	From this, in \textbf{case i)} we obtain (\ref{Eq: 6.30 covariates}) which allows us to use the compatibility assumption to arrive at (\ref{Eq: 6.31 covariates}):
	\[
	\frac{\hat{\mathcal{E}}}{2} + \bar{\lambda} \Vert \hat{\bar{\theta}} - \bar{\theta}^* \Vert_1 \le 2 \epsilon^*,
	\]
	resulting in
	\[
	\hat{\mathcal{E}} + \bar{\lambda} \Vert \hat{\bar{\theta}} - \bar{\theta}^* \Vert_1 \le 4 \epsilon^*.
	\]
	In \textbf{case ii)} on the other hand, we arrive directly at (\ref{Eq: 6.32 covariates}), and hence
	\[
	\hat{\mathcal{E}} + \bar{\lambda} \Vert \hat{\bar{\theta}} - \bar{\theta}^* \Vert_1 \le 3 \epsilon^*.
	\]
	Plugging in the definitions of $\hat{\bar{\theta}}$ and $\bar{\theta}^*$ and using the fact that $\hat{\mathcal{E}} = \bar{\mathcal{E}}(\hat{\bar{\theta}}) = \mathcal{E}(\hat{\theta})$ proves the claim.
\end{proof}

\subsection{Controlling the special set \texorpdfstring{$\mathcal{I}$}{I}}\label{Sec: Controlling I}

We now show that $\mathcal{I}$ has probability tending to one. Recall some results on concentration inequalities.

\subsubsection{Concentration inequalities}
We first recall some probability inequalities that we will need. This is based on Chapter 14 in \cite{vandegeer2011}. Throughout let $Z_1, \dots, Z_n$ be a sequence of independent random variables in some space $\mathcal{Z}$ and $\mathcal{G}$ be a class of real valued functions on $\mathcal{Z}$.
\begin{Def}
	A \textit{Rademacher sequence} is a sequence $\epsilon_1, \dots, \epsilon_n$ of i.i.d. random variables with $P(\epsilon_i = 1) = P(\epsilon_i = -1) = 1/2$ for all $i$.
\end{Def}
\begin{Satz}[Symmetrization theorem as in \cite{vandervaartwellner1996}, abridged]\label{Thm: symmetrization theorem}
	Let $\epsilon_1, \dots, \epsilon_n$ be a Rademacher sequence independent of $Z_1, \dots, Z_n$. Then
	\[
	\mathbb{E}\left( \sup_{g \in \mathcal{G}} \left| \sum_{ i = 1}^n \{ g(Z_i) - \mathbb{E}[g(Z_i)] \} \right| \right) \le 2 \mathbb{E}\left( \sup_{g \in \mathcal{G}} \left| \sum_{ i = 1}^n \epsilon_i g(Z_i) \right|  \right).
	\]
\end{Satz}

\begin{Satz}[Contraction theorem as in \cite{Ledoux:Talagrand:1991}]\label{Thm: contraction theorem Ledoux Talagrand}
	Let $z_1, \dots, z_n$ be non-random elements of $\mathcal{Z}$ and let $\mathcal{F}$ be a class of real-valued functions on $\mathcal{Z}$. Consider Lipschitz functions $g_i: \R \rightarrow \R$ with Lipschitz constant $L =1$, i.e. for all $i$
	\[
	\vert g_i(s) - g_i(s') \vert \le \vert s - s' \vert, \forall s,s' \in \R.
	\]
	Let $\epsilon_1, \dots, \epsilon_n$ be a Rademacher sequence. Then for any function $f^*: \mathcal{Z} \rightarrow \R$ we have
	\[
	\mathbb{E}\left( \sup_{f \in \mathcal{F}} \left| \sum_{ i = 1}^n \epsilon_i \{ g_i(f(z_i)) - g_i(f^*(z_i)) \} \right| \right) \le 2 \mathbb{E}\left( \sup_{f \in \mathcal{F}} \left| \sum_{ i = 1}^n \epsilon_i \{ f(z_i) - f^*(z_i) \} \right| \right).
	\]
\end{Satz}

The last theorem we need is a concentration inequality due to \cite{Bousquet:2002}. We give a version as presented in \cite{vandegeer2008}.

\begin{Satz}[Bousequet's concentration theorem]\label{Thm: concentration theorem Bousquet}
	Suppose $Z_1, \dots, Z_n$ and all $g \in \mathcal{G}$ satisfy the following conditions for some real valued constants $\eta_n$ and $\tau_n$
	\[
	\Vert g \Vert_{\infty} \le \eta_n, \; \forall g \in \mathcal{G}
	\]
	and
	\[
	\frac{1}{n} \sum_{ i = 1}^n \textnormal{Var}(g(Z_i)) \le \tau_n^2, \; \forall g \in \mathcal{G}.
	\]
	Define
	\[
	\textbf{Z} \coloneqq \sup_{g \in \mathcal{G}} \left| \frac{1}{n}\sum_{ i = 1}^n g(Z_i) - \mathbb{E}[g(Z_i)]  \right|.
	\]
	Then for any $z>0$
	\[
	P\left(\textbf{Z} \ge \mathbb{E}[\textbf{Z}] + z \sqrt{2(\tau_n^2 + 2 \eta_n\mathbb{E}[\textbf{Z}])} + \frac{2z^2\eta_n}{3} \right) \le \exp(-nz^2).
	\]
\end{Satz}

\begin{Rem}
	Looking at the original paper of \cite{Bousquet:2002}, their result looks quite different at first. To see that the above falls into their framework, set the variables in \cite{Bousquet:2002} as follows
	\begin{align*}
	f(Z_i) &= (g(Z_i) - \mathbb{E}[g(Z_i)])/(2\eta_n), &\tilde{Z}_k = \sup_f \vert \sum_{ i \neq k} f(Z_i) \vert, \\
	f_k &= \arg\sup_f \vert \sum_{ i \neq k} f(Z_i) \vert, &\tilde{Z}_k' = \vert \sum_{ i = 1}^n f_k(Z_i) \vert - \tilde{Z}_k \\
	\tilde{Z} &= \frac{2\eta_n}{n}\textbf{Z}.
	\end{align*}
	Now apply Theorem 2.1 in \cite{Bousquet:2002}, choosing for their $(Z, Z_1, \dots, Z_n)$ the above defined $(\tilde{Z}, \tilde{Z}_1, \dots, \tilde{Z}_n)$, for their $(Z_1', \dots, Z_n')$ the above defined $(\tilde{Z}_1', \dots, \tilde{Z}_n')$ and setting $u = 1$ and $\sigma^2 = \frac{\tau_n^2}{4\eta_n^2}$ in their theorem: The result is exactly Theorem \ref{Thm: concentration theorem Bousquet} above.
\end{Rem}
Finally we have a Lemma derived from Hoeffding's inequality. The proof can be found in \cite{vandegeer2011}, Lemma 14.14 (here we use the special case of their Lemma for $m=1$).
\begin{Lem}\label{Lem: 14.14}
	Let $\mathcal{G} = \{g_1, \dots, g_p \}$ be a set of real valued functions on $\mathcal{Z}$ satisfying for all $i= 1, \dots, n$ and all $j=1, \dots, p$
	\[
	\mathbb{E}[g_j(Z_i)] = 0, \; \vert g_j(Z_i) \vert \le c_{ij}
	\]
	for some positive constants $c_{ij}$.
	Then
	\[
	\mathbb{E}\left[ \max_{1 \le j \le p} \left| \sum_{ i = 1}^n g_j(Z_i) \right| \right] \le \left[ 2\log(2p)\right]^{1/2} \max_{1 \le j \le p} \left[ \sum_{ i = 1}^n c_{ij}^2   \right]^{1/2}.
	\]
\end{Lem}

\subsubsection{The expectation of \texorpdfstring{$Z_M$}{Z\_M}}

Recall the definition of $Z_M$ 
\[
Z_M \coloneqq \sup_{\substack{\bar{\theta} \in \bar{\Theta}_{\text{loc}}, \\  \Vert \bar{\theta} - \bar{\theta}^* \Vert_1 \le M}} \vert \bar{v}_n(\bar{\theta}) - \bar{v}_n(\bar{\theta}^*) \vert,
\]
where $\bar{v}_n$ denotes the re-parametrized empirical process. 
Recall, that there is a constant $c \in \R$ such that uniformly $\vert Z_{ij,k} \vert \le c, 1 \le i < j \le n, k = 1, \dots, p$.

\begin{Lem}\label{Lem: 14.20}\label{Lem: 14.20 covariates} For any $M > 0$ we have in \sbmc
	\[
	\mathbb{E}[Z_M] \le 8 M(1 \vee c) \sqrt{\frac{2\log(2(n+p+1))}{\binom{n}{2}}}
	\]
	 and in the S$\beta$M without covariates
	\[
	\mathbb{E}[Z_M] \le 8 M \sqrt{\frac{\log(2(n+1))}{\binom{n}{2}}}.
	\]
\end{Lem}

\begin{proof}
	We only give the proof for the S$\beta$RM. The proof for the case without covariates is exactly the same with the corresponding parts set to zero.
	Let $\epsilon_{ij}, i < j,$ be a Rademacher sequence independent of $A_{ij}, Z_{ij}, i < j$. We first want to use the symmetrization Theorem \ref{Thm: symmetrization theorem}: For the random variables $Z_1, ..., $ we choose $T_{ij} = (A_{ij}, \bar{X}_{ij}^T,1, Z_{ij}^T)^T \in \{0,1\} \times \R^{n+1+p}$. 
	For any $\bar{\theta} \in \bar{\Theta}_{\text{loc}}$ we consider the functions
	\[
		g_{\bar{\theta}}(T_{ij}) = \frac{1}{\binom{n}{2}} \left\{-A_{ij} \bar D_{ij}^T(\bar \theta - \bar \theta^*) + \log(1 + \exp(\bar D_{ij}^T\bar \theta)) - \log(1 + \exp(\bar D_{ij}^T \bar \theta^*)) \right\}
	\]
	and the function set $\mathcal{G} = \mathcal{G}(M) \coloneqq \{ g_{\bar{\theta}} : \bar{\theta} \in \bar{ \Theta}_{\text{loc}}, \Vert \bar{\theta} - \bar{\theta}^* \Vert_1 \le M \}$.
	Note, that
	\[
	\bar{v}_n(\bar{\theta}) - \bar{v}_n(\bar{\theta}^*) = \sum_{ i < j} \{g_{\bar{\theta}}(T_{ij}) - \mathbb{E}[g_{\bar{\theta}}(T_{ij}) ]\}.
	\]	
	Then, the symmetrization theorem gives us
	\begin{align*}
	\E[Z_M ] &= \mathbb{E}\left[ \sup_{g_{\bar{\theta}} \in \mathcal{G}} \left\vert \sum_{ i < j} g_{\bar{\theta}}(T_{ij}) - \mathbb{E}[g_{\bar{\theta}}(T_{ij}) ] \right\vert  \right] \\
	&\le 2 \mathbb{E}\left[ \sup_{g_{\bar{\theta}} \in \mathcal{G}} \left\vert \sum_{ i < j} \epsilon_{ij} g_{\bar{\theta}}(T_{ij}) \right\vert   \right].
	\end{align*}
	Next, we want to apply Theorem \ref{Thm: contraction theorem Ledoux Talagrand}. Denote $T = (T_{ij})_{i<j}$ and let $\mathbb{E}_{T}$ be the conditional expectation given $T$. We need the conditional expectation at this point, because Theorem \ref{Thm: contraction theorem Ledoux Talagrand} requires non-random arguments in the functions. This does not hinder us, as later we will simply take iterated expectations, canceling out the conditional expectation, see below. For the functions $g_i$ in Theorem \ref{Thm: contraction theorem Ledoux Talagrand} we choose
	\[
		g_{ij}(x) = \frac{1}{2} \{-A_{ij}x + \log(1 + \exp(x))\}
	\]
	Note, that $ \log(1 + \exp(x))$ has derivative bounded by one and thus is Lipschitz continuous with constant one by the Mean Value Theorem. Thus, all $g_{ij}$ are also Lipschitz continuous with constant $1$:
	\[
	\vert g_{ij}(x) - g_{ij}(x') \vert \le \frac{1}{2} \{ \vert A_{ij} (x - x') \vert  + \vert  \log(1 + \exp(x)) -  \log(1 + \exp(x')) \vert \} \le \vert x - x' \vert.
	\]
	For the function class $\mathcal{F}$ in Theorem $\ref{Thm: contraction theorem Ledoux Talagrand}$ we choose $\mathcal{F} = \mathcal{F}_M \coloneqq \{ f_{\bar{\theta}} : \bar{\theta} \in \bar{\Theta}_{\text{loc}}, \Vert \bar{\theta} - \bar{\theta}^* \Vert_1 \le M \}$ and pick $f^* = f_{\bar{\theta}^*}$. Then, by Theorem \ref{Thm: contraction theorem Ledoux Talagrand}
	\begin{align*}
	\mathbb{E}_T &\left[ \sup_{\substack{\bar{\theta} \in \bar{\Theta}_{\text{loc}}, \\ \Vert \bar{\theta} -\bar{\theta}^* \Vert_1 \le M}} \left\vert \frac{1}{\binom{n}{2}} \sum_{ i < j} \epsilon_{ij}(g_{ij}(f_{\bar{\theta}}((\bar{X}_{ij}^T, 1, Z_{ij}^T)^T) )- g_{ij}(f_{\bar{\theta}^*}((\bar{X}_{ij}^T, 1, Z_{ij}^T)^T))) \right\vert  \right] \\
	&\le 2 \mathbb{E}_T \left[ \sup_{\substack{\bar{\theta} \in \bar{\Theta}_{\text{loc}}, \\ \Vert \bar{\theta} -\bar{\theta}^* \Vert_1 \le M}} \left\vert \frac{1}{\binom{n}{2}} \sum_{ i < j} \epsilon_{ij}(f_{\bar{\theta}}((\bar{X}_{ij}^T, 1, Z_{ij}^T)^T) - f_{\bar{\theta}^*}((\bar{X}_{ij}^T, 1, Z_{ij}^T)^T)) \right\vert  \right].
	\end{align*}
	Recall that we can express the functions $f_{\bar{\theta}} = f_{\bar{\beta}, \mu, \gamma}$ as
	\[
	f_{\bar{\beta}, \mu, \gamma} (\, . \,)= \mu e_{n+1}(\, . \,) + \sum_{i = 1}^n \bar{\beta}_i e_i(\, . \,) + \sum_{ i = 1}^p \gamma_ie_{n+1+i}(\,.\,),
	\]
	where $e_i(\,.\,)$ is the projection on the $i$th coordinate. Consider any $\bar{\theta} = (\bar{\beta}^T, \mu, \gamma^T)^T \in \bar{\Theta}_{\text{loc}}$ with $\Vert \bar{\theta} - \bar{\theta}^* \Vert_1 \le M$. For the sake of a compact representation we use our shorthand notation $\bar{ \theta} = (\bar{\theta}_i)_{i=1}^{n+1+p}$ where the components $\theta_i$ are defined in the canonical way and we also simply write $e_k(X_{ij}, 1, Z_{ij})$ for the projection of the 
	the vector $(X_{ij}^T, 1, Z_{ij}^T)^T \in \R^{n+p+1}$ to its $k$th component, i.e. instead of $e_k((X_{ij}^T, 1, Z_{ij}^T)^T)$. Then,
	\begin{align*}
	&\left\vert \frac{1}{\binom{n}{2}} \sum_{ i < j} \epsilon_{ij}(f_{\bar{\theta}}((\bar{X}_{ij}^T, 1, Z_{ij}^T)^T) - f_{\bar{\theta}^*}((\bar{X}_{ij}^T, 1, Z_{ij}^T)^T)) \right\vert \\
	&= \left\vert \frac{1}{\binom{n}{2}} \sum_{ i < j} \epsilon_{ij} \left( \sum_{k = 1}^{n+p+1} (\bar{\theta}_k - \bar{\theta}_k^*) e_k(\bar{X}_{ij}, 1, Z_{ij})\right)\right\vert \\
	&\le \frac{1}{\binom{n}{2}} \sum_{k = 1}^{n+p+1} \left\lbrace \vert \bar{\theta}_k - \bar{\theta}_k^*\vert
	\max_{1 \le l \le n+p+1} \left| \sum_{ i < j} \epsilon_{ij} e_l(\bar{X}_{ij}, 1, Z_{ij}) \right| \right\rbrace \\
	&\le M \max_{1 \le l \le n+p+1} \left|\frac{1}{\binom{n}{2}} \sum_{ i < j} \epsilon_{ij} e_l(\bar{X}_{ij}, 1, Z_{ij}) \right|.
	\end{align*}
	Note, that the last expression no longer depends on $\bar{\theta}$. To bind the right hand side in the last expression we use Lemma \ref{Lem: 14.14}: In the language of the Lemma, choose $Z_1, \dots, Z_n$ as $T_{ij} = (\epsilon_{ij}, \bar{X}_{ij}^T, 1, Z_{ij}^T)^T$. We choose for the $p$ in the formulation of the Lemma $n+p+1$ and pick for our functions
	\[
	g_{k}(T_{ij}) = \frac{1}{\binom{n}{2}}\epsilon_{ij}e_k(\bar{X}_{ij}, 1, Z_{ij}), k=1,\dots, n+p+1.
	\]
	Note, that then $\mathbb{E}[g_{k}(T_{ij})] = 0$. We want to employ Lemma \ref{Lem: 14.14} which requires us to bound $\vert g_k(T_{ij}) \vert \le c_{ij,k}$ for all $i < j$ and $k = 1, \dots, n+1+p$.
	
	For any fixed $1 \le k \le n$ we have
	\[
	\vert g_k (T_{ij}) \vert \le 
	\begin{cases}
	\frac{\sqrt{n}}{\sqrt{2}\binom{n}{2}} = \frac{\sqrt{2}}{(n-1)\sqrt{n}}, & i \text{ or } j = k \\
	0, & \text{otherwise}.
	\end{cases}
	\]
	Note that the first case occurs exactly $(n-1)$ times for each $k$. Thus, for any $k \le n$,
	\[
	\sum_{ i < j} c_{ij,k}^2 = \left( \frac{\sqrt{2}}{(n-1)\sqrt{n}}  \right)^2 (n-1) = \frac{1}{\binom{n}{2}}.
	\]
	If $k = n+1$, $\vert g_k (T_{ij}) \vert = 1/\binom{n}{2}$ and hence
	\[
	\sum_{ i < j} c_{ij,n+1}^2 = \frac{1}{\binom{n}{2}}.
	\]
	Finally, if $k > n+1$, $\vert g_k(T_{ij}) \vert \le c /\binom{n}{2}$ and therefore,
	\[
	\sum_{ i < j} c_{ij, k}^2 \le \frac{c^2}{\binom{n}{2}}.
	\]
	In total, this means
	\[
	\max_{1 \le k \le n+1+p} \sum_{ i < j} c_{ij,k}^2 \le \frac{1 \vee c^2}{\binom{n}{2}}.
	\]
	Therefore, an application of Lemma \ref{Lem: 14.14} results in
	\begin{align*}
	\mathbb{E}\left[ \max_{1 \le l \le n+p+1} \left|\frac{1}{\binom{n}{2}} \sum_{ i < j} \epsilon_{ij} e_l(\bar{X}_{ij}, Z_{ij}) \right| \right] &\le \sqrt{2\log(2(n+1+p))} \max_{1 \le k \le n+1+p} \left[\sum_{ i < j} c_{ij,k}^2\right]^{1/2} \\
	&\le \sqrt{2\log(2(n+1+p))} \sqrt{\frac{1 \vee c^2}{\binom{n}{2}}} \\
	&= \sqrt{\frac{2\log(2(n+1+p))}{\binom{n}{2}}} (1 \vee c).
	\end{align*}
	Putting everything together, we obtain
	\begin{align*}
	\mathbb{E}[Z_M] &\le 2 \mathbb{E}\left[ \sup_{\substack{\bar{\theta} \in \bar{\Theta}_{\text{loc}}, \\ \Vert \bar{\theta} -\bar{\theta}^* \Vert_1 \le M}}  \left\vert \frac{1}{\binom{n}{2}} \sum_{ i < j} \epsilon_{ij}(-A_{ij}(f_{\bar{\theta}}(\bar{X}_{ij}, Z_{ij}) - f_{\bar{\theta}^*}  (\bar{X}_{ij}, Z_{ij}))) \right\vert  \right] \\
	&= 2 \mathbb{E}\left[ \mathbb{E}_T \left[ \sup_{\substack{\bar{\theta} \in \bar{\Theta}_{\text{loc}}, \\ \Vert \bar{\theta} -\bar{\theta}^* \Vert_1 \le M}}  \left\vert \frac{1}{\binom{n}{2}} \sum_{ i < j} \epsilon_{ij}(-A_{ij}(f_{\bar{\theta}}(\bar{X}_{ij}, Z_{ij}) - f_{\bar{\theta}^*}  (\bar{X}_{ij}, Z_{ij}))) \right\vert  \right]\right] \\
	&\le 8 \mathbb{E}\left[ \mathbb{E}_T \left[ \sup_{\substack{\bar{\theta} \in \bar{\Theta}_{\text{loc}}, \\ \Vert \bar{\theta} -\bar{\theta}^* \Vert_1 \le M}}  \left\vert \frac{1}{\binom{n}{2}} \sum_{ i < j} \epsilon_{ij}(f_{\bar{\theta}}(\bar{X}_{ij}, Z_{ij}) - f_{\bar{\theta}^*}  (\bar{X}_{ij}, Z_{ij})) \right\vert  \right]\right] \\
	&\le 8M \mathbb{E} \left[ \mathbb{E}_T \left[ \max_{1 \le l \le n+p+1} \left|\frac{1}{\binom{n}{2}} \sum_{ i < j} \epsilon_{ij} e_l(\bar{X}_{ij}, Z_{ij}) \right| \right]\right] \\
	&\le 8M\sqrt{\frac{2\log(2(n+1+p))}{\binom{n}{2}}} (1 \vee c).
	\end{align*}
	This concludes the proof.
\end{proof}

We now want to show that $Z_M$ does not deviate too far from its expectation. The proof relies on the concentration theorem due to Bousquet, Theorem \ref{Thm: concentration theorem Bousquet}.

\begin{Kor}\label{Kor: Probability bound Z_M}\label{Kor: Probability bound Z_M covariates}
	Pick any confidence level $t > 0$. 
	Let
	\[
	a_n \coloneqq \sqrt{\frac{2\log(2(n+p+1))}{\binom{n}{2}}} (1 \vee c).
	\]
	and choose $\lambda_0 = \lambda_0(t,n)$ as
	\[
		\lambda_0 = 8a_n + 2 \sqrt{\frac{t}{\binom{n}{2}}( 11 (1 \vee (c^2p) ) + 8\sqrt{2}(1 \vee c) \sqrt{n} a_n  )} + \frac{2\sqrt{2}t(1 \vee c) \sqrt{n}}{3\binom{n}{2}}	
	\]
	Then, we have the inequality
	\[
		P\left( Z_M \ge M\lambda_0  \right) \le \exp(-t).
	\]
\end{Kor}

\begin{proof}
	Again, we only give the proof for the case with covariates. The case without covariates is completely analogous by setting the corresponding parts to zero.
	We want to apply Bousquet's concentration theorem \ref{Thm: concentration theorem Bousquet}. For the random variables $Z_i$ in the formulation of the theorem we choose once more $T_{ij} = (A_{ij}, \bar{X}_{ij},1, Z_{ij}), i<j,$ and as functions we consider
	\begin{align*}
	g_{\bar{\theta}}(T_{ij}) &= -A_{ij} \bar D_{ij}^T(\bar \theta - \bar \theta^*) + \log(1 + \exp(\bar D_{ij}^T\bar \theta)) - \log(1 + \exp(\bar D_{ij}^T \bar \theta^*)), \\
	\mathcal{G} &= \mathcal{G}_M \coloneqq \{ g_{\bar{\theta}} : \bar{\theta} \in \bar{\Theta}_{\text{loc}}, \Vert \bar{\theta} -\bar{\theta}^* \Vert_1 \le M \}.
	\end{align*}
	Then, by definition we have
	\[
	Z_M = \sup_{g_{\bar{\theta}} \in \mathcal{G}} \frac{1}{\binom{n}{2}} \left\vert \sum_{ i < j} \{g_{\bar{\theta}}(T_{ij}) - \mathbb{E}[g_{\bar{\theta}}(T_{ij}) ] \} \right\vert.
	\]
	To apply Theorem \ref{Thm: concentration theorem Bousquet}, we need to bound the infinity norm of $g_{\bar{\theta}}$. Recall that we denote the distribution of $[\bar{X} | 1| Z]$ by $\bar{Q}$ and the infinity norm is defined as the $\bar{Q}$-almost sure smallest upper bound on the value of $g_{\bar{\theta}}$. We have for any $g_{\bar{\theta}} \in \mathcal{G}$, using the Lipschitz continuity of $\log(1 + \exp(x))$:
	\begin{align*}
		\vert g_{\bar{\theta}}(T_{ij}) \vert &\le \vert \bar D_{ij}^T(\bar \theta - \bar \theta^*) \vert + \vert  \log(1 + \exp(\bar D_{ij}^T\bar \theta)) - \log(1 + \exp(\bar D_{ij}^T \bar \theta^*)) \vert \\
		&\le 2 \vert \bar D_{ij}^T(\bar \theta - \bar \theta^*) \vert \\
		&\le 2 \Vert \beta - \beta^* \Vert_1 + \vert \mu - \mu^* \vert + c \Vert \gamma - \gamma^* \Vert_1.
\shortintertext{Thus,}
	\Vert g_{\bar{\theta}} \Vert_{\infty} &\le 2 \Vert \beta - \beta^* \Vert_1 + \vert \mu - \mu^* \vert + c \Vert \gamma - \gamma^* \Vert_1 \\
	&\le 2(1 \vee c) \Vert \theta - \theta^* \Vert_1 \\
	&\le \sqrt{2}(1 \vee c) \sqrt{n} M \eqqcolon \eta_n.
	\end{align*}
	For the last inequality we used that for any $\theta$ with $\Vert \bar{\theta} - \bar{\theta}^* \Vert_1 \le M$ it follows that $\Vert \theta - \theta^* \Vert_1 \le \sqrt{n}/\sqrt{2}M$, which is possibly a very generous upper bound. This does not matter, however, as the term associated with the above bound will be negligible, as we shall see.
	
	The second requirement of Theorem \ref{Thm: concentration theorem Bousquet} is that the average variance of $g_{\bar{\theta}}(T_{ij})$ has to be uniformly bounded. To that end we calculate
	\begin{align*}
		\frac{1}{\binom{n}{2}} \sum_{ i < j}\text{Var}(g_{\bar{\theta}}(T_{ij})) &= \frac{1}{\binom{n}{2}} \sum_{ i < j} \text{Var}(-A_{ij}D_{ij}^T(\theta - \theta^*)) \\
		& + \frac{1}{\binom{n}{2}} \sum_{ i < j}\text{Var}(\log(1 + \exp(\bar D_{ij}^T\bar \theta)) - \log(1 + \exp(\bar D_{ij}^T \bar \theta^*))) \\
		& +\frac{2}{\binom{n}{2}} \sum_{ i < j} \text{Cov}(-A_{ij}D_{ij}^T(\theta - \theta^*), \log(1 + \exp(\bar D_{ij}^T\bar \theta)) - \log(1 + \exp(\bar D_{ij}^T \bar \theta^*))).
	\end{align*}
	Let us look at these terms in term. For the first term, we obtain
	\begin{align*}
		\frac{1}{\binom{n}{2}} \sum_{ i < j} \text{Var}(-A_{ij}D_{ij}^T(\theta - \theta^*)) &\le \frac{1}{\binom{n}{2}} \sum_{ i < j} \E[(-A_{ij}D_{ij}^T(\theta - \theta^*))^2] \le \E \left[ \frac{1}{\binom{n}{2}} \sum_{ i < j} (D_{ij}^T(\theta - \theta^*))^2 \right].
	\end{align*}
	For the second term we get
	\begin{align*}
		\frac{1}{\binom{n}{2}} \sum_{ i < j}\text{Var}&(\log(1 + \exp(\bar D_{ij}^T\bar \theta)) - \log(1 + \exp(\bar D_{ij}^T \bar \theta^*)))\\
		 &\le \frac{1}{\binom{n}{2}} \sum_{ i < j} \E[(\log(1 + \exp(\bar D_{ij}^T\bar \theta)) - \log(1 + \exp(\bar D_{ij}^T \bar \theta^*)))^2]\\
		&\le  \E \left[ \frac{1}{\binom{n}{2}} \sum_{ i < j} (D_{ij}^T(\theta - \theta^*))^2 \right].
	\end{align*}
	The last term decomposes as
	\begin{align*}
		\frac{2}{\binom{n}{2}} \sum_{ i < j} & \text{Cov}(-A_{ij}D_{ij}^T(\theta - \theta^*), \log(1 + \exp(\bar D_{ij}^T\bar \theta)) - \log(1 + \exp(\bar D_{ij}^T \bar \theta^*))) \\
		&= \frac{2}{\binom{n}{2}} \sum_{ i < j}  \E[-A_{ij}D_{ij}^T(\theta - \theta^*) \cdot (\log(1 + \exp(\bar D_{ij}^T\bar \theta)) - \log(1 + \exp(\bar D_{ij}^T \bar \theta^*)))] \\
		&\quad - \frac{2}{\binom{n}{2}} \sum_{ i < j}  \E[-A_{ij}D_{ij}^T(\theta - \theta^*)] \cdot \E[ \log(1 + \exp(\bar D_{ij}^T\bar \theta)) - \log(1 + \exp(\bar D_{ij}^T \bar \theta^*))]
	\end{align*}
	For the first term in that decomposition we have
	\begin{align*}
		\frac{2}{\binom{n}{2}} &\sum_{ i < j}  \left\vert\E[-A_{ij}D_{ij}^T(\theta - \theta^*) \cdot (\log(1 + \exp(\bar D_{ij}^T\bar \theta)) - \log(1 + \exp(\bar D_{ij}^T \bar \theta^*)))] \right\vert \\
		&\le \frac{2}{\binom{n}{2}} \sum_{ i < j}  \E[\vert D_{ij}^T(\theta - \theta^*)\vert \cdot \vert\log(1 + \exp(\bar D_{ij}^T\bar \theta)) - \log(1 + \exp(\bar D_{ij}^T \bar \theta^*))\vert] \\
		&\le \frac{2}{\binom{n}{2}} \sum_{ i < j}  \E[\vert D_{ij}^T(\theta - \theta^*)\vert^2] 
	\end{align*}
	and for the second term, using the same arguments, we get
	\begin{align*}
		\frac{2}{\binom{n}{2}} \sum_{ i < j}  \E[-A_{ij}D_{ij}^T(\theta - \theta^*)] &\cdot \E[ \log(1 + \exp(\bar D_{ij}^T\bar \theta)) - \log(1 + \exp(\bar D_{ij}^T \bar \theta^*))] \\
		& \le \frac{2}{\binom{n}{2}} \sum_{ i < j}  \E[\vert D_{ij}^T(\theta - \theta^*)\vert]^2, 
	\end{align*}
	meaning that in total
	\begin{align*}
		\frac{2}{\binom{n}{2}} \sum_{ i < j} &\left\vert \text{Cov}(-A_{ij}D_{ij}^T(\theta - \theta^*), \log(1 + \exp(\bar D_{ij}^T\bar \theta)) - \log(1 + \exp(\bar D_{ij}^T \bar \theta^*))) \right\vert \\
		&\le \frac{2}{\binom{n}{2}} \sum_{ i < j}  \E[\vert D_{ij}^T(\theta - \theta^*)\vert^2] + \frac{2}{\binom{n}{2}} \sum_{ i < j}  \E[\vert D_{ij}^T(\theta - \theta^*)\vert]^2.
	\end{align*}
	In total, we thus get
	\begin{align}\label{Eq: variance upper bound}
		\frac{1}{\binom{n}{2}} \sum_{ i < j}\text{Var}(g_{\bar{\theta}}(T_{ij})) \le 4 \cdot \E \left[ \frac{1}{\binom{n}{2}} \sum_{ i < j} (D_{ij}^T(\theta - \theta^*))^2 \right] + \frac{2}{\binom{n}{2}} \sum_{ i < j}  \E[\vert D_{ij}^T(\theta - \theta^*)\vert]^2.
	\end{align}
	Furthermore,
	\begin{align*}
	\frac{1}{\binom{n}{2}} \sum_{ i < j} (D_{ij}^T(\theta - \theta^*))^2 &= \frac{1}{\binom{n}{2}} \sum_{ i < j} (\beta_i + \beta_j + \mu - \beta_i^* - \beta_j^* - \mu^* + (\gamma - \gamma^*)^TZ_{ij})^2 \\
	&\le \frac{4}{\binom{n}{2}} \sum_{ i < j} \left\{ (\beta_i - \beta_i^*)^2 + (\beta_j - \beta_j^*)^2 + (\mu - \mu^*)^2 + ((\gamma - \gamma^*)^TZ_{ij})^2\right\},
	\end{align*}
	where the inequality follows from the Cauchy-Schwarz inequality.
	Recall that for any $x \in \R^p, \Vert x \Vert_2 \le \Vert x \Vert_1 \le \sqrt{p}\Vert x \Vert_2$ and note that
	\[
	\vert(\gamma - \gamma^*)^TZ_{ij}\vert \le c \Vert \gamma - \gamma^* \Vert_1 \le c\sqrt{p}\Vert \gamma  - \gamma^* \Vert_2.
	\]	
	Then, from the above
	\begin{align}\label{Eq: bound on sum D_ij theta -theta*}
	\begin{split}
		\frac{1}{\binom{n}{2}} \sum_{ i < j} (D_{ij}^T(\theta - \theta^*))^2 &\le \frac{4}{\binom{n}{2}} \sum_{ i < j} \left\{ (\beta_i - \beta_i^*)^2 + (\beta_j - \beta_j^*)^2 + (\mu - \mu^*)^2 + c^2 p \Vert \gamma - \gamma^* \Vert_2^2 \right\} \\
		&= 4 \left( (\mu - \mu^*)^2 + c^2p \Vert \gamma - \gamma^* \Vert_2^2 + \frac{1}{\binom{n}{2}}   \sum_{ i < j} \left\{ (\beta_i - \beta_i^*)^2 + (\beta_j - \beta_j^*)^2 \right\} \right) \\
		&= 4 \left( (\mu - \mu^*)^2 + c^2p \Vert \gamma - \gamma^* \Vert_2^2 + \frac{1}{\binom{n}{2}}   (n-1) \Vert \beta - \beta^* \Vert_2^2 \right) \\
		&= 4 \left( (\mu - \mu^*)^2 + c^2p \Vert \gamma - \gamma^* \Vert_2^2 + \left\Vert \frac{\sqrt{2}}{\sqrt{n}} (\beta - \beta^*) \right\Vert_2^2 \right) \\
		&= 4 \left( (\mu - \mu^*)^2 + c^2p \Vert \gamma - \gamma^* \Vert_2^2 + \Vert  \bar{\beta} - \bar{\beta}^* \Vert_2^2 \right) \\
		&\le 4(1 \vee (c^2p)) \Vert \bar{\theta} - \bar{\theta}^* \Vert_2^2 \\
		&\le 4(1 \vee (c^2p)) \Vert \bar{\theta} - \bar{\theta}^* \Vert_1^2 \\
		&\le 4(1 \vee (c^2p)) M^2.
	\end{split}
	\end{align}
	Notice that for the second term in \eqref{Eq: variance upper bound} we have
	\begin{align*}
		\frac{2}{\binom{n}{2}} \sum_{ i < j}  \E[\vert D_{ij}^T(\theta - \theta^*)\vert]^2 &= \frac{2}{\binom{n}{2}} \sum_{ i < j} (\beta_i - \beta_i^* + \beta_j - \beta_j^* + \mu - \mu^* + (\gamma - \gamma^*)^T\E[Z_{ij}])^2 \\
		&= \frac{2}{\binom{n}{2}} \sum_{ i < j} (\beta_i - \beta_i^* + \beta_j - \beta_j^* + \mu - \mu^*)^2 \\
		&\le \frac{6}{\binom{n}{2}} \sum_{ i < j} \left\{ (\beta_i - \beta_i^*)^2 + (\beta_j - \beta_j^*)^2 + (\mu - \mu^*)^2 \right\},
	\end{align*}
	so that we may use the same steps as in \eqref{Eq: bound on sum D_ij theta -theta*} to conclude that
	\begin{equation*}
			\frac{2}{\binom{n}{2}} \sum_{ i < j}  \E[\vert D_{ij}^T(\theta - \theta^*)\vert]^2 \le 6 M^2 \le 6 (1 \vee (c^2p))M^2.
	\end{equation*}
	Such that in total,
	\begin{align*}
		\frac{1}{\binom{n}{2}} \sum_{ i < j}\text{Var}(g_{\bar{\theta}}(T_{ij})) \le 22(1 \vee (c^2p)) M^2 \eqqcolon \tau_n^2.
	\end{align*}
	Applying Bousquet's concentration theorem \ref{Thm: concentration theorem Bousquet} with $\eta_n, \tau_n$ defined above, we obtain for all $z > 0$
	\begin{align}\label{Eq: prob inequality I covariates}
	\begin{split}
	\exp\left(-\binom{n}{2}z^2\right) &\ge P\left( Z_M \ge \mathbb{E}[Z_M] + z \sqrt{2(\tau_n^2 + 2\eta_n \mathbb{E}[Z_M])} + \frac{2z^2\eta_n}{3} \right) \\
	&= P\Biggl( Z_M \ge \mathbb{E}[Z_M] + z \sqrt{2(22(1 \vee (c^2p)) M^2 + 2\sqrt{2}(1 \vee c) \sqrt{n} M  \mathbb{E}[Z_M])} \\
	&\quad \quad \quad + \frac{2\sqrt{2}z^2(1 \vee c) \sqrt{n} M }{3} \Biggr).
	\end{split}
	\end{align}
	From Lemma \ref{Lem: 14.20 covariates}, we know
	\[
	\mathbb{E}[Z_M] \le 8 M \sqrt{\frac{2\log(2(n+p+1))}{\binom{n}{2}}}(1 \vee c) = 8Ma_n.
	\]
	Using this, we obtain from (\ref{Eq: prob inequality I covariates})
	\begin{align*}
	\exp\left(-\binom{n}{2}z^2\right)& \ge P\Biggl( Z_M \ge 8Ma_n + z \sqrt{2(22(1 \vee (c^2p)) M^2 + 16\sqrt{2}(1 \vee c) \sqrt{n} M^2 a_n)} \\
	&\quad \quad \quad + \frac{2\sqrt{2}z^2(1 \vee c) \sqrt{n} M }{3} \Biggl)  \\
	&= P\left( Z_M \ge M \left( 8a_n + 2z \sqrt{11(1 \vee (c^2p) ) + 8\sqrt{2}(1 \vee c) \sqrt{n}a_n } + \frac{2\sqrt{2}z^2(1 \vee c) \sqrt{n}}{3} \right)\right).
	\end{align*}
	Now, pick $z = \sqrt{t/\binom{n}{2}}$ to get
	\begin{align*}
		\exp(-t) &\ge \\
		&P\left( Z_M \ge M \left( 8a_n + 2 \sqrt{\frac{t}{\binom{n}{2}}( 11 (1 \vee (c^2p) ) + 8\sqrt{2}(1 \vee c) \sqrt{n} a_n  )} + \frac{2\sqrt{2}t(1 \vee c) \sqrt{n}}{3\binom{n}{2}} \right) \right),
	\end{align*}
	which is the claim.
\end{proof}

\subsection{Proofs of Theorem \ref{Cor: no approximation error} and Theorem \ref{Thm: consistency}}\label{Sec: Proof of main results}
\begin{proof}[Proof of Theorem \ref{Thm: consistency}]
	The proof follows immediately from Theorem \ref{Thm: Consistency covariates} and Corollary \ref{Kor: Probability bound Z_M covariates}.
\end{proof}
\begin{proof}[Proof of Theorem \ref{Cor: no approximation error}]
We are in the case where no approximation error is committed, that is in the case $r_{n,0} \le r_n$. In that case $\theta^* = \theta_0$ and hence $\mathcal{E}(\theta^*) = 0$. Let $\rho_n$ be the lower bound on the link probabilities corresponding to $r_n$. $K_n$ then simplifies to
	\begin{align} \label{K_n inequality}
	K_n &= 2 \frac{(1 + \exp(r_{n}))^2}{\exp(r_{n})} = 2 \frac{\left( 1 + \exp\left( - \text{logit}(\rho_{n})  \right) \right)^2}{\exp\left( - \text{logit}(\rho_{n,0}) \right)} \le \frac{4}{\rho_{n}}.
	\end{align}
	Thus, under the conditions of Theorem \ref{Cor: no approximation error}, we have with high probability
	\begin{equation*}
	\mathcal{E}(\hat{\theta}) + \bar{ \lambda} \left( \frac{\sqrt{2}}{\sqrt{n}} \Vert \hat{\beta} - \beta_{0}  \Vert_1 + \vert \hat{\mu} - \mu_{0} \vert + \Vert \hat{\gamma} - \gamma_{0} \Vert_1  \right) \le C  \frac{s_{0, +}\bar{ \lambda}^2}{\rho_{n}}.
	\end{equation*}
	with constant $C = 128/c_{\min}$.
\end{proof}

\subsection{Proof of Proposition \ref{Prop: q error bound}}

Similarly, we don't assume $r_n \ge r_{n,0}$ (i.e. $\theta^*=\theta_0$) in the beginning. To derive the $\ell_q$ ($1 < q \le 2$) error bounds for $\hat{\theta}$, we need a stronger compatibility condition. Here we use the minimal adaptive restricted eigenvalue condition as in \cite{vandegeer2011}, Section 6.8 and rewrite it in our notation. 

Let $S$ be an index set with cardinality $s$ and $N \geq s$ be an integer. We say that the adaptive $(L, S, N)$-restricted eigenvalue condition is satisfied, with constant $\phi_{\text {adap }}(L, S, N)>0$, if for all $\mathcal{N} \supset S$, with $|\mathcal{N}|=N$, and all $\theta \in \mathcal{R}_{\text {adap }}(L, S, \mathcal{N})$, it holds that
$$
\Vert \theta_{\mathcal{N}}\Vert_2 \leq \frac{\Vert f_\theta\Vert_2}{\phi_{\text {adap}}(L, S, N)},
$$
where the restricted set of $\theta$ is defined as
$$
\mathcal{R}_{\text {adap}}(L, S, \mathcal{N}):=\left\{\Vert\theta_{S^c}\Vert_1 \leq L \sqrt{s}\Vert \theta_S\Vert_2,\Vert \theta_{\mathcal{N}^c}\Vert_{\infty} \leq \min _{j \in \mathcal{N} \backslash S}\vert \theta_j\vert\right\} .
$$
The minimal adaptive restricted eigenvalue is
$$
\phi_{\min}^2(L, S, N)=\min _{\mathcal{N} \supset S,|\mathcal{N}|=N} \phi_{\text {adap}}^2(L, \mathcal{N}, N).
$$

As in Proposition \ref{Prop: compatibility condition Sigma}, we first verify that $\Sigma_{A}$ satisfies the adaptive $(3, S, N)$-restricted eigenvalue condition for any $S$ and $|\mathcal{N}|=N \ge |S|$.
Notice that the compatibility condition is equivalent to the condition that
$$
\min _{\substack{\theta \in \mathcal{R}_{\text {adap}}(3, S, \mathcal{N}) \\\theta \in \R^{n+1+p} \backslash\{0\}}} \frac{\theta^T \Sigma \theta}{\Vert \theta_{\mathcal{N}}\Vert_2^2} > 0
$$
for any $|\mathcal{N}|=N \geq |S|$.
$$
\begin{aligned}
\frac{\theta^T \Sigma_A \theta}{\Vert \theta_{\mathcal{N}}\Vert_2^2} & \geq \frac{\frac{n-2}{n-1} \beta^T \beta+\mu^2+\frac{1}{\binom{n}{2}} \gamma^T \mathbb{E}\left[Z^T Z\right] \gamma}{\Vert\theta_{\mathcal{N}}\Vert_2^2} \\
& \geq \frac{\frac{n-2}{n-1}\|\beta\|_2^2+\mu^2+\frac{1}{\binom{n}{2}} \gamma^T \mathbb{E}\left[Z^T Z\right] \gamma}{\|\beta\|_2^2+\mu^2+\|\gamma\|_2^2}, \quad  (\text{since }\Vert\theta_{\mathcal{N}}\Vert_2^2 \leq  \Vert \theta\Vert_2^2) \\
& \geq \frac{\frac{n-2}{n-1}\left(\Vert \beta\Vert_2^2+\mu^2\right)+\frac{1}{\binom{n}{2}} \gamma^T \mathbb{E}\left[Z^T Z\right] \gamma}{\Vert\beta\Vert_2^2+\mu^2+\Vert \gamma\Vert_2^2}, \quad \text { since } 1 \geq(n-2) /(n-1) \\
& =\frac{n-2}{n-1} \cdot \frac{\|\beta\|_2^2+\mu^2+\frac{n-1}{n-2} \frac{1}{\binom{n}{2}} \gamma^T \mathbb{E}\left[Z^T Z\right] \gamma}{\|\beta\|_2^2+\mu^2+\|\gamma\|_2^2}  \\
& \geq c_{\min }>0.
\end{aligned}
$$
Therefore, for $\Sigma_{A}$, 
$$
\phi_{\min }^2(3, S_{+}^{*}, 2s_{+}^{*})=\min _{\mathcal{N} \supset S_{+}^{*},|\mathcal{N}|=2s_{+}^{*}} \phi_{\text {adap}}^2(3, \mathcal{N}, 2s_{+}^{*}) \ge c_{\min}> 0.
$$

We then show that for $\Sigma$, we still have $\phi_{\min }^2(3, S_{+}^{*}, 2s_{+}^{*})>0$. Let $\delta=\max _{i j}\left|(\Sigma_A)_{i j}-\Sigma_{i j}\right|$, for any $\mathcal{N} \supset S_{+}^{*},|\mathcal{N}|=2s_{+}^{*}$ and any $\theta$ satisfies $\left\|\theta_{\mathcal{N}^{c}}\right\|_1 \leq 3\sqrt{2s_{+}^{*}}\left\|\theta_{\mathcal{N}}\right\|_2$, similarly as Lemma \ref{Lem: Lemma 6 in KockTang}, we have
$$
\begin{aligned}
\left|\theta^T \Sigma_A \theta-\theta^T \Sigma \theta\right| & =\left|\theta^T(\Sigma_A-\Sigma) \theta\right| \leq\|\theta\|_1\|(\Sigma_A-\Sigma) \theta\|_{\infty} \leq \delta\|\theta\|_1^2 \\
& =\delta\left(\left\|\theta_{\mathcal{N}}\right\|_1+\left\|\theta_{\mathcal{N}^{c}}\right\|_1\right)^2 \leq \delta\left(\sqrt{2s_{+}^{*}}\left\|\theta_{\mathcal{N}}\right\|_2+3\sqrt{2s_{+}^{*}}\Vert\theta_{\mathcal{N}}\Vert_2\right)^2 \\
& \leq 32s_{+}^{*} \delta\Vert\theta_\mathcal{N}\Vert_2^2.
\end{aligned}
$$
Then, 
$$
\frac{\theta^T \Sigma \theta}{\Vert \theta_{\mathcal{N}}\Vert_2^2} \geq \frac{\theta^T \Sigma_A \theta}{\Vert\theta_{\mathcal{N}}\Vert_2^2}-32 \delta s_{+}^* \geq c_{\min}-32 \delta s_{+}^* \geq \frac{c_{\min}}{2} >0
$$
since $\delta < \frac{c_{\min}}{64 s_{+}^*}$. Therefore $\phi_{\min }^2(3, S_{+}^{*}, 2s_{+}^{*})\ge \frac{c_{\min}}{2} > 0$ for $\Sigma$.

Next, let $\mathcal{N} = S_{+}^{*} \cup \{\text{the set indices of the largest } s_{+}^{*} \text{ elements of }\hat{\bar{\theta}}_{j} \text{ for } j \in S_{+}^{*c} \}$, where $S_{+}^{*c} = \{1,2,3, ..., n+1+p\} \backslash S^{*}_{+}$. Let $\mathcal{N}_{1} = \mathcal{N} \cap \{1,2,3, ..., n\}$ and $\mathcal{N}_{1}^{c} =  \{1,2,3, ..., n\} \backslash \mathcal{N}_{1}$. $\mathcal{N}^{c} =  \{1,2,3, ..., n+1+p\} \backslash \mathcal{N}$. Since $S_{+}^{*} =S^{*} \cup \{n+1, ... ,n+1+p\}$ and $S_{+}^{*} \subset \mathcal{N}$. Therefore $\mathcal{N}^{c}  = \mathcal{N}_{1}^{c}$, $\mathcal{N}_{1} \cup  \{n+1, ... ,n+1+p\} = \mathcal{N}$ and $\beta^{*} = 0$ on $\mathcal{N}_{1}^{c}$. Similarly as (\ref{Eq: 2.14}),
$$
\begin{aligned}
\hat{\mathcal{E}}+\bar{\lambda}\Vert\hat{\bar{\beta}}_{\mathcal{N}_{1}^{c}}\Vert_1 & \leq \epsilon^*+\bar{\lambda}(\Vert\bar{\beta}^*\Vert_1-\Vert\hat{\bar{\beta}}_{\mathcal{N}_{1}}\Vert_1)+\mathcal{E}^* \\
& \leq \epsilon^*+\bar{\lambda}(\Vert(\bar{\beta}^*-\hat{\bar{\beta}})_{\mathcal{N}_{1}}\Vert_1)+\mathcal{E}^* \\
& \leq \epsilon^*+\bar{\lambda}(\Vert(\bar{\beta}^*-\hat{\bar{\beta}})_{\mathcal{N}_{1}}\Vert_1+\Vert(\mu^*, \gamma^{* T})^T-(\hat{\mu}, \hat{\gamma}^T)^T\Vert_1)+\mathcal{E}^* \\
& =\epsilon^*+\bar{\lambda}\Vert(\hat{\bar{\theta}}-\bar{\theta}^*)_{\mathcal{N}}\Vert_1+\mathcal{E}^* \\
& \leq 2 \epsilon^*+\bar{\lambda}\Vert(\hat{\bar{\theta}}-\bar{\theta}^*)_{\mathcal{N}}\Vert_1\\
& \leq 2 \epsilon^*+ \sqrt{2s_{+}^*}\bar{\lambda}\Vert(\hat{\bar{\theta}}-\bar{\theta}^*)_{\mathcal{N}}\Vert_2.
\end{aligned}
$$
And since $\mathcal{N}^{c}  = \mathcal{N}_{1}^{c}$, $\Vert(\hat{\bar{\theta}}-\bar{\theta}^*)_{\mathcal{N}^{c}}\Vert_1=\Vert \hat{\bar{\beta}}_{\mathcal{N}_{1}^{c}}\Vert_1$, we have
$$
\hat{\mathcal{E}}+\bar{\lambda}\Vert(\hat{\bar{\theta}}-\bar{\theta}^*)_{\mathcal{N}^{c}}\Vert_1 \leq 2 \epsilon^*+ \sqrt{2s_{+}^*}\bar{\lambda}\Vert(\hat{\bar{\theta}}-\bar{\theta}^*)_{\mathcal{N}}\Vert_2.
$$
Then we just replace $\Vert(\hat{\bar{\theta}}-\bar{\theta}^*)_{S_{+}^*}\Vert_1$ by $\sqrt{2s_{+}^*}\Vert(\hat{\bar{\theta}}-\bar{\theta}^*)_{\mathcal{N}}\Vert_2$ in the proof of Theorem \ref{Thm: Consistency covariates} to get
$$
\bar{\lambda}\Vert(\hat{\bar{\theta}}-\bar{\theta}^*)_{\mathcal{N}^{c}}\Vert_1+\bar{\lambda}\sqrt{2s_{+}^*}\Vert(\hat{\bar{\theta}}-\bar{\theta}^*)_{\mathcal{N}}\Vert_2\leq 4e^*,
$$
where
$$
2e^* =3 \mathcal{E}(\theta^*)+2 H_n\left(\frac{8 \sqrt{ s_{+}^*} \bar{\lambda}}{\sqrt{c_{\min }}}\right).
$$
Then following the proof of Lemma 6.11 in \cite{vandegeer2011}, we get 
$$
\Vert(\hat{\bar{\theta}}-\bar{\theta}^*)\Vert^{q}_{q} \leq (4^q+2^{q+1}) (s_{+}^*)^{-(q-1)}\left(e^* / \bar{\lambda}\right)^q .
$$

\begin{proof}[Proof of Proposition \ref{Prop: q error bound}]
We are in the case where $r_n \ge r_{n,0}$ and $\theta^{*} = \theta_0$. Then $\mathcal{E}(\theta^*) = 0$ and $e^*$ simplifies to be 
$$
e^* = H_n\left(\frac{8 \sqrt{ s_{0, +}} \bar{\lambda}}{\sqrt{c_{\min }}}\right)=\frac{K_n}{4} \cdot \left(\frac{8 \sqrt{ s_{0, +}} \bar{\lambda}}{\sqrt{c_{\min }}}\right)^{2} \leq \frac{64 s_{0, +}\bar{\lambda}^{2}}{\rho_nc_{\min }},
$$
where the last inequality follows from (\ref{K_n inequality}).
\end{proof}

\section{Proof of Theorem \ref{Thm: inference}}\label{proof_thm3}

\subsection{Inverting population and sample Gram matrices}\label{Sec: inverting Gram matrices}\label{sample and population matrices}

Our strategy for proving Theorem \ref{Thm: inference} will be inverting the KKT conditions, similar to \cite{vandegeer2014}. Since the estimation in (\ref{Eq: Penalized llhd with covariates}) is a convex optimization problem, by subdifferential calculus, we know $0$ has to be contained in the subdifferential of $\frac{1}{\binom{n}{2}}\mathcal{L}(\theta) + \lambda \Vert \beta \Vert_1$ at $\hat{\theta}$. That is, there exists some $v \in \R^{n+1+p}$ such that
\begin{equation}\label{Eq: subdifferential first order equation}
	0 = \frac{1}{\binom{n}{2}}\nabla \left.\mathcal{L}(\theta)\right|_{\theta = \hat{\theta}} + \lambda v,
\end{equation}
where $\nabla\left.\mathcal{L}(\theta)\right|_{\theta = \hat{\theta}}$ is the gradient of $\mathcal{L}(\theta)$ evaluated at $\hat{\theta}$ and for $i=1, \dots, n, v_i = 1$ if $\hat{\beta}_i > 0$ and $v_i \in [-1,1]$ if $\hat{\beta}_i = 0$, and $v_i = 0$ for $i = n+1, \dots, n+ 1+p$.

Recall that we use $\vartheta = (\mu, \gamma^T)^T$ to refer to the unpenalized parameter subvector of $\theta$. 
Thus, denoting $\nabla_\vartheta \left.\mathcal{L}(\theta)\right|_{\theta = \hat{\theta}} \in \R^{p+1}$ the gradient of $\mathcal{L}$ with respect to the unpenalized parameters $(\mu, \gamma^T)^T$ only, evaluated at $\hat{\theta}$, we have
\begin{equation}\label{Eq: KKT gamma}
	0 = \nabla_\vartheta \left.\mathcal{L}(\theta)\right|_{\theta = \hat{\theta}}.
\end{equation}
Denote by $H(\hat{\theta}) \coloneqq \left.H_{\vartheta \times \vartheta}(\theta)\right|_{\theta = \hat{\theta}}$ the Hessian of $\frac{1}{\binom{n}{2}}\mathcal{L}(\theta)$ with respect to $\vartheta$ only, evaluated at $\hat{\theta}$. Denote $p_{ij}(\theta) = \frac{\exp(\beta_i + \beta_j + \mu + \gamma^T Z_{ij})}{1 + \exp(\beta_i + \beta_j + \mu + \gamma^T Z_{ij})}$. 
Now, consider the entries of $H(\hat{\theta})$.  For all $k,l = 1, \dots, (p+1)$,
\begin{align*}
	H(\hat{\theta})_{k,l} = \frac{1}{\binom{n}{2}}\partial_{\vartheta_k\vartheta_l}\mathcal{L}(\hat{\theta}) =  \frac{1}{\binom{n}{2}} \sum_{i < j} D_{ij,n+k}D_{ij,n+l} p_{ij}(\hat{\theta})(1-p_{ij}(\hat{\theta})),
\end{align*}
where $D_{ij}^T$ is the $(i,j)$-th row of the design matrix $D$, i.e.~in particular $D_{ij,n+k} = 1$ if $k=1$ and $D_{ij,n+k} = Z_{ij,k-1}$ for $k=2, \dots, (p+1)$.
In particular, we have the following matrix representation of $H(\hat{\theta})$. Let $D_\vartheta = [\textbf{1}|Z]$ be the part of the design matrix $D$ corresponding to $\vartheta$ with rows $D_{\vartheta, ij}^T = (1, Z_{ij}^T), i < j$. Also let $\hat{W} = \text{diag}(\sqrt{p_{ij}(\hat{\theta})(1-p_{ij}(\hat{\theta}))}, i < j)$. Then we have
\[
H(\hat{\theta}) = \frac{1}{\binom{n}{2}} D_\vartheta^T \hat{W}^2 D_\vartheta.
\]
Let $W_0 = \text{diag}(\sqrt{p_{ij}(\theta_0)(1-p_{ij}(\theta_0))}, i < j)$ and consider the corresponding population version: 
\[
\mathbb{E}[H(\theta_0)] = \frac{1}{\binom{n}{2}} \mathbb{E}[ D_\vartheta^T W_0^2 D_\vartheta].
\]
Recall that we use the commonly used notation $\hat{\Sigma}_\vartheta = H(\hat \theta) = \frac{1}{\binom{n}{2}} D_\vartheta^T \hat{W}^2D_\vartheta$ and $\Sigma_\vartheta = \mathbb{E}[H(\theta_0)] = \frac{1}{\binom{n}{2}} \mathbb{E}[ D_\vartheta^T W_0^2 D_\vartheta]$ and
$
\hat \Theta_\vartheta \coloneqq \hat \Sigma_\vartheta^{-1}, \Theta_\vartheta \coloneqq \Sigma_\vartheta^{-1}.
$

We will need to invert $\hat \Sigma_\vartheta$ and $\Sigma_\vartheta$ and show that these inverses are close to each other in an appropriate sense. It is commonly assumed in LASSO theory (cf.~\cite{vandegeer2014}) that the minimum eigenvalues of these matrices stay bounded away from zero. In our case, however, such an assumption is invalid.

Indeed, since $\rho_n \le 1/2$, we find that for all $i<j$, $p_{ij}(\theta_0)(1-p_{ij}(\theta_0)) \ge 1/2 \cdot \rho_n$. Also, recall that by Assumption \ref{Assum: minimum EW}, the minimum eigenvalue  $\lambda_{\text{min}}$ of $\mathbb{E}[Z^TZ/\binom{n}{2}]$ stays uniformly bounded away from zero for all $n$. Then, for any $n$ and $v \in \R^{p+1} \backslash \{0\}$ with components $v = (v_1, v_R^T)^T, v_R \in \R^p$, we have 
\begin{align*}
	v^T \Sigma_\vartheta v &\ge \frac{1}{2} \rho_n v^T \frac{1}{\binom{n}{2}} \mathbb{E}[ D_\vartheta^T D_\vartheta] v = \frac{1}{2} \rho_n v^T \begin{pmatrix}
		1 & \textbf{0} \\
		\textbf{0} & \frac{1}{\binom{n}{2}} \E[ Z^TZ]
	\end{pmatrix} v \\
	& = \frac{1}{2} \rho_n \left(v_1^2 + v_R^T \frac{1}{\binom{n}{2}} \E[Z^TZ] v_R  \right) \\
	&\ge  \frac{1}{2} \rho_n (v_1^2 + \lambda_{\text{min}}\Vert v_R \Vert_2^2) \ge \frac{1}{2} \rho_n (1 \wedge \lambda_{\min}) \Vert v \Vert_2^2 > 0.
\end{align*}
Hence, for finite $n$ all eigenvalues of $\Sigma_\vartheta$ are strictly positive and consequently this matrix is invertible. Using similar techniques as in the proof of Proposition \ref{Prop: compatibility condition Sigma} we can now show that with high probability the minimum eigenvalue of $D_\vartheta^TD_\vartheta / \binom{n}{2}$ is also strictly larger than zero and thus for any $v \in \R^{p+1} \backslash\{0\}$ and any finite $n$,
\[
\frac{1}{\binom{n}{2}}v^T D_\vartheta^T \hat{W}^2 D_\vartheta v \ge C \rho_n \lambda_{\min}\left(\frac{1}{\binom{n}{2}}Z^T Z\right) \Vert v \Vert_2^2 >0.
\]
Thus, for every finite $n$, $\hat \Sigma_\vartheta$ is invertible with high probability. However, the lower bound on the right-hand side tends to zero with increasing $n$.

Recall that by Assumption \ref{Assum: minimum EW}, the minimum eigenvalue  $\lambda_{\text{min}}$ of $\frac{1}{\binom{n}{2}} \mathbb{E}[Z^TZ]$ stays uniformly bounded away from zero for all $n$. Consequently the minimum eigenvalue of $\frac{1}{\binom{n}{2}} \mathbb{E}[D_\vartheta^TD_\vartheta]$ is lower bounded by $1 \wedge \lambda_{\min}$ which is bounded away from zero uniformly for all $n$.
We now show that under Assumption \ref{Assum: minimum EW}, with high probability the minimum eigenvalue of $\frac{1}{\binom{n}{2}}D_\vartheta^T D_\vartheta$ is bounded away from zero. 
More precisely, recall the definition of $\kappa(A, m)$ for square matrices $A$ and dimensions $m$. We want to consider the expression $\kappa^2\left(\frac{1}{\binom{n}{2}} \mathbb{E}[D_\vartheta^TD_\vartheta], p+1\right)$ which simplifies to
\[
\kappa^2\left(\frac{1}{\binom{n}{2}} \mathbb{E}[D_\vartheta^TD_\vartheta], p+1 \right) \coloneqq \min_{v \in \R^{p+1} \backslash\{0\}} \frac{v^T\frac{1}{\binom{n}{2}} \mathbb{E}[D_\vartheta^TD_\vartheta] v}{\frac{1}{p+1} \Vert v \Vert_1^2} \ge (1 \wedge \lambda_{\min}).
\]
and compare it to $\kappa^2\left(\frac{1}{\binom{n}{2}}D_\vartheta^TD_\vartheta,p+1\right)$.  By Assumption \ref{Assum: minimum EW} and the argument above, we have
\[
\kappa^2\left(\frac{1}{\binom{n}{2}} \mathbb{E}[D_\vartheta^TD_\vartheta], p+1\right) \ge C > 0
\]
for a constant $C$ independent of $n$.
With $\delta = \max_{kl} \left\vert  \left(\frac{1}{\binom{n}{2}}D_\vartheta^TD_\vartheta\right)_{kl} - \left( \frac{1}{\binom{n}{2}}\E[D_\vartheta^TD_\vartheta]\right)_{kl} \right\vert$, by Lemma \ref{Lem: Lemma 6 in KockTang}, we have
\[
\kappa^2\left(\frac{1}{\binom{n}{2}}D_\vartheta^TD_\vartheta, p+1\right) \ge \kappa^2\left(\frac{1}{\binom{n}{2}}\E[D_\vartheta^TD_\vartheta], p+1\right) - 16\delta (p+1).
\]
By looking at the proof of Lemma \ref{Lem: Lemma 6 in KockTang}, we see that in this particular case we do not even need the factor $16(p+1)$ on the right hand side above, but this does not matter anyways, so we keep it. We notice that

\begin{Lem}\label{Lem: Sigma is invertible whp}
	\[
		\delta = \max_{kl} \left\vert  \left(\frac{1}{\binom{n}{2}}D_\vartheta^TD_\vartheta\right)_{kl} - \left( \frac{1}{\binom{n}{2}}\E[D_\vartheta^TD_\vartheta]\right)_{kl} \right\vert = O_P\left( \binom{n}{2}^{-1/2}  \right).
	\]
\end{Lem}

\begin{proof}
	To make referencing submatrices of $1/\binom{n}{2}D_\vartheta^TD_\vartheta$ and its expectation easier, write
	\begin{equation*}
		B \coloneqq \frac{1}{\binom{n}{2}}D_\vartheta^TD_\vartheta = \frac{1}{\binom{n}{2}} \begin{bmatrix}
			\underbrace{\textbf{1}^T\textbf{1} }_{\text{\textcircled{5}}}& \underbrace{\textbf{1}^TZ}_{\text{\textcircled{6}}} \\
			\underbrace{ Z^T\textbf{1} }_{\text{\textcircled{8}}}& \underbrace{Z^TZ}_{\text{\textcircled{9}}}
		\end{bmatrix}, \quad A \coloneqq \frac{1}{\binom{n}{2}} \E [D_\vartheta^TD_\vartheta] = \frac{1}{\binom{n}{2}} \begin{bmatrix}
			\underbrace{\textbf{1}^T\textbf{1} }_{\text{\textcircled{5}}}& \underbrace{\textbf{0}}_{\text{\textcircled{6}}} \\
			\underbrace{ \textbf{0} }_{\text{\textcircled{8}}}& \underbrace{\E [Z^TZ]}_{\text{\textcircled{9}}}
		\end{bmatrix}
	\end{equation*}
	where we have chosen our numbering to be consistent with the notation used in the proof of Proposition \ref{Prop: compatibility condition Sigma}. The matrices $A$ and $B$ are equal in block \textcircled{5}.
	For $i,j$ corresponding to the blocks \textcircled{6} and \textcircled{8}, $B_{ij} - A_{ij} = B_{ij}$ is the sum of all the entries of some column $Z_k$ of the matrix $Z$ for an appropriate $k$. That is, there is a $1 \le k \le p$ such that
	\[
	B_{ij} - A_{ij} = \frac{1}{\binom{n}{2}} Z_k^T\textbf{1} = \frac{1}{\binom{n}{2}} \sum_{ s < t} Z_{k,st}.
	\]
	Note, that thus by model assumption $\mathbb{E}[B_{ij} - A_{ij}] = 0$. We know that for each $k,s,t: Z_{k,st} \in [-c,c]$. Hence, by Hoeffding's inequality, for all $\eta > 0$,
	\[
	P\left( \vert 	B_{ij} - A_{ij}  \vert \ge \eta  \right) = P\left( \left\vert \sum_{ s < t} Z_{k,st} \right\vert \ge \binom{n}{2} \eta  \right) \le 2 \exp\left( - \frac{2\binom{n}{2}^2\eta^2}{\sum_{i<j} (2c)^2}  \right) = 2\exp\left( - \binom{n}{2}\frac{\eta^2}{2c^2}  \right).
	\]
	For $i,j$ from block \textcircled{9}, a typical element has the form
	\[
	B_{ij} - A_{ij} = \frac{1}{\binom{n}{2}} \sum_{ s < t} \left\{ Z_{k,st}Z_{l,st} - \mathbb{E}[Z_{k,st}Z_{l,st}] \right\},
	\]
	for appropriate $k,l$. In other words, $B_{ij} - A_{ij}$ is the inner product of two columns of $Z$, minus their expectation, scaled by $1/\binom{n}{2}$. Since $Z_{k,st}Z_{l,st} \in [-c^2,c^2]$ for all $k,l,s,t$, we have that for all $k,l,s,t$: $Z_{k,st}Z_{l,st} - \mathbb{E}[Z_{k,st}Z_{l,st}] \in [-2c^2,2c^2]$. Thus, by Hoeffding's inequality, for all $\eta > 0$,
	\[
	P\left( \vert B_{ij} - A_{ij}  \vert \ge \eta  \right) = P\left( \left\vert \sum_{ s < t} \{ Z_{k,st}Z_{l,st} - \mathbb{E}[Z_{k,st}Z_{l,st}]\}  \right\vert \ge  \binom{n}{2} \eta \right) \le 2 \exp\left( - \binom{n}{2} \frac{\eta^2}{8c^4}  \right).
	\]
	Thus, with $\tilde{ c} = c^2 \vee (2c^4)$, we have for any entry in blocks \textcircled{6}, \textcircled{8}, \textcircled{9}, that for any $\eta > 0$,
	\[
	P\left( \vert B_{ij} - A_{ij}  \vert \ge \eta \right) \le 2 \exp\left( - \binom{n}{2} \frac{\eta^2}{2\tilde{ c}}  \right).
	\]
	The claim will follow from a union bound: Because block \textcircled{6} is the transpose of block \textcircled{8}, it is sufficient to control one of them. By symmetry of block \textcircled{9} it suffices to control the upper triangular half, including the diagonal, of block \textcircled{9}. Thus, we only need to control the entries $B_{ij} - A_{ij}$ for $i,j$ in the following index set
	\begin{align*}
		\mathcal{A} &= \{( i,j ): i,j \text{ belong to block \textcircled{8} or the upper triangular half or the diagonal of block \textcircled{9}} \}.
	\end{align*}
	Keep in mind that block \textcircled{8} has $p$ elements, while the upper triangular part of block \textcircled{9} plus its diagonal has $\binom{p}{2} + p = \binom{p+1}{2}$ elements.
	Thus, for any $\eta  > 0$,
	\begin{align*}
		P\left( \max_{ij} \vert B_{ij} - A_{ij} \vert \ge \eta  \right)
		& \le \sum_{( i,j ) \in \mathcal{A}} P\left(  \vert B_{ij} - A_{ij} \vert \ge \eta  \right) \\
		& \le 2p \exp\left( - \binom{n}{2}\frac{\eta^2}{2c^2}  \right) + 2\binom{p+1}{2}\exp\left( - \binom{n}{2} \frac{\eta^2}{8c^4}  \right) \\
		&\le 2\left(p+\binom{p+1}{2}\right) \exp\left( - \binom{n}{2} \frac{\eta^2}{2\tilde{ c}}  \right) \\
		&= p(p+3) \exp\left( - \binom{n}{2} \frac{\eta^2}{2\tilde{ c}}  \right).
	\end{align*}
	This proves the claim.
\end{proof}

Thus, for $n$ large enough, we have with high probability $\delta \le \frac{(1 \wedge \lambda_{\min})}{32(p+1)}$. Then, by Lemma \ref{Lem: Lemma 6 in KockTang}, with high probability and uniformly in $n$,
\[
\kappa^2\left(\frac{1}{\binom{n}{2}}D_\vartheta^TD_\vartheta, p+1\right) \ge \kappa^2\left(\frac{1}{\binom{n}{2}}\E[D_\vartheta^TD_\vartheta], p+1\right) - 16\delta (p+1) \ge \frac{(1 \wedge \lambda_{\min})}{2} \ge C > 0.
\]
Yet, if $\kappa^2\left(\frac{1}{\binom{n}{2}} D_\vartheta^TD_\vartheta, p+1\right) \ge C > 0$ uniformly in $n$, then for any $v \neq 0, v^T\frac{1}{\binom{n}{2}} D_\vartheta^TD_\vartheta v \ge C \Vert v \Vert_2^2$. But we also know that the minimum eigenvalue of $\frac{1}{\binom{n}{2}} D_\vartheta^TD_\vartheta$ is the largest possible $C$ such that this bound holds (it is actually tight with equality for the eigenvectors corresponding to the minimum eigenvalue). Therefore, with high probability, the minimum eigenvalue of $\frac{1}{\binom{n}{2}}D_\vartheta^TD_\vartheta$ stays uniformly bounded away from zero.
Thus, for any $v \in \R^{p+1}\backslash\{0\}$ and any finite $n$:
\[
\frac{1}{\binom{n}{2}}v^T D_\vartheta^T \hat{W}^2 D_\vartheta v \ge \min_{i<j} \{ p_{ij}(\hat{\theta}) (1 - p_{ij}(\hat{\theta}))\} \left(	v^T \frac{1}{\binom{n}{2}} D_\vartheta^T D_\vartheta v\right) \ge C \rho_n \Vert v \Vert_2^2 >0.
\]
Thus, $\lambda_{\min}\left(\frac{1}{\binom{n}{2}}D_\vartheta^T \hat{W}^2 D_\vartheta \right) \ge C \rho_n\lambda_{\min}
\left(\frac{1}{\binom{n}{2}}D_\vartheta^T D_\vartheta  \right) > 0$. That means, for every finite $n$, $\frac{1}{\binom{n}{2}} D_\vartheta^T \hat{W}^2 D_\vartheta$ is invertible with high probability.

\subsection{Goal and approach}

\textbf{Goal:} We want to show that for $k=1, \dots, p+1$,
\[
\sqrt{\binom{n}{2}}\frac{\hat \vartheta_k - \vartheta_{0,k}}{\sqrt{\hat \Theta_{\vartheta,k,k}}} \rightarrow \mathcal{N}(0,1).
\]

\noindent\textbf{Approach:} Recall the definition of the "one-sample-version" of $\mathcal{L}$, i.e. $l_{\theta}: \{0,1\} \times \R^{n+1 + p} \rightarrow \R$,
for $\theta = (\beta^T, \mu, \gamma^T)^T \in \Theta$,
\begin{equation*}
l_{\theta}(y,x) \coloneqq -y \theta^Tx + \log(1 + \exp(\theta^Tx)).
\end{equation*}
Then, the negative log-likelihood is given by
\[
\mathcal{L}(\theta) = \sum_{ i < j} l_\theta(A_{ij}, (X_{ij}^T, 1, Z_{ij}^T)^T)
\]
and
\begin{align*}
\nabla\mathcal{L}(\theta) = \sum_{ i < j} \nabla l_\theta(A_{ij}, (X_{ij}^T, 1, Z_{ij}^T)^T),  \quad H\mathcal{L}(\theta)= \sum_{ i < j} Hl_\theta(A_{ij}, (X_{ij}^T, 1, Z_{ij}^T)^T),
\end{align*}
where $H$ denotes the Hessian with respect to $\theta$. 
Consider $l_\theta$ as a function in $\theta^T x$ and introduce:
\begin{equation}\label{Eq: Def l (a)}
l(y,a) \coloneqq -ya + \log(1+\exp(a)),
\end{equation}
with second derivative: $\ddot l (y,a) = \partial_{a^2}l(y,a) = \frac{\exp(a)}{(1+\exp(a))^2}$. Note, that $\partial_{a^2}l(y,a)$ is Lipschitz continuous (it has bounded derivative $\vert \partial_{a^3}l(y,a) \vert \le 1/(6\sqrt{3})$; Lipschitz continuity then follows by the Mean Value Theorem). Doing a first order Taylor expansion in $a$ of $\dot l (y,a) = \partial_al(y,a)$ in the point $(A_{ij}, D_{ij}^T\theta_0)$ evaluated at $(A_{ij}, D_{ij}^T\hat \theta)$, we get
\begin{equation}\label{Eq: Taylor of l}
\partial_a l(A_{ij}, D_{ij}^T\hat \theta) = \partial_a l(A_{ij}, D_{ij}^T\theta_0) + \partial_{a^2} l(A_{ij}, \alpha) D_{ij}^T(\hat \theta - \theta_0),
\end{equation}
for an $\alpha$ between $D_{ij}^T\hat \theta$ and $D_{ij}^T \theta_0$. By Lipschitz continuity of $\partial_{a^2}l$, we also find
\begin{align}\label{Eq: Lipschitz l a^2}
\begin{split}
\vert \partial_{a^2}l(A_{ij}, \alpha) D_{ij}^T(\hat \theta - \theta_0) - \partial_{a^2}l(A_{ij}, D_{ij}^T\hat \theta) D_{ij}^T(\hat \theta - \theta_0) \vert &\le \vert \alpha - D_{ij}^T\hat \theta \vert \vert D_{ij}^T(\hat \theta - \theta_0) \vert \\
&\le \vert D_{ij}^T(\hat \theta - \theta_0) \vert^2,
\end{split}
\end{align}
where the last inequality follows, because $\alpha$ is between $D_{ij}^T\hat \theta$ and $D_{ij}^T \theta_0$.

Consider the vector $P_n\nabla l_{\hat \theta}$: By equation (\ref{Eq: Taylor of l}), with $\alpha_{ij}$ between $D_{ij}^T\hat \theta$ and $D_{ij}^T \theta_0$,
\begin{align*}
P_n \nabla l_{\hat \theta} &= \frac{1}{\binom{n}{2}} \sum_{ i < j} \left( \partial_{\theta_k} l (A_{ij}, D_{ij}^T \hat \theta) \right)_{k=1, \dots, n+1+p}, \quad \text{ as a }(n+1+p)\times 1\text{-vector} \\
&= \frac{1}{\binom{n}{2}} \sum_{ i < j} \dot l (A_{ij}, D_{ij}^T \hat \theta) D_{ij} \\
&= \frac{1}{\binom{n}{2}} \sum_{ i < j} (\dot l (A_{ij}, D_{ij}^T \theta_0) + \ddot l(A_{ij}, \alpha_{ij}) D_{ij}^T (\hat \theta - \theta_0)) D_{ij} \\
\shortintertext{which by (\ref{Eq: Lipschitz l a^2}) gives}
& = P_n\nabla l_{\theta_0} + \frac{1}{\binom{n}{2}} \sum_{ i < j} D_{ij} \left\{\ddot l(A_{ij}, D_{ij}^T\hat \theta) D_{ij}^T(\hat \theta - \theta_0) + O(\vert D_{ij}^T(\hat \theta - \theta_0) \vert^2)\right\}. \\
\shortintertext{Noticing that $\ddot l(A_{ij}, D_{ij}^T\hat \theta) = p_{ij}(\hat \theta) (1 - p_{ij}(\hat{\theta}))$ and we thus have $\sum_{ i < j}\ddot l(A_{ij}, D_{ij}^T\hat \theta)D_{ij} D_{ij}^T(\hat \theta - \theta_0)= D^T \hat W^2 D (\hat \theta - \theta_0)$:}
&= P_n\nabla l_{\theta_0} + P_n Hl_{\hat{\theta}}(\hat \theta - \theta_0)  + O\left(\frac{1}{\binom{n}{2}}\sum_{i<j} D_{ij} \vert D_{ij}^T(\hat \theta - \theta_0) \vert^2 \right) \\
&= P_n\nabla l_{\theta_0} + \frac{1}{\binom{n}{2}} D^T \hat W^2 D (\hat{ \theta} - \theta_0) + O\left(\frac{1}{\binom{n}{2}}\sum_{i<j} D_{ij} \vert D_{ij}^T(\hat \theta - \theta_0) \vert^2 \right),
\end{align*}
where the $O$ notation is to be understood componentwise.
Above, we have equality of two $((n+1+p) \times 1)$-vectors. We are only interested in the portion relating to $\vartheta = (\mu, \gamma^T)^T$, that is, in the last $p+1$ entries. Introduce the $((n+1+p) \times (n+1+p))$-matrix
\begin{equation*}
M = \begin{pmatrix}
\textbf{0} & \textbf{0} \\
\textbf{0} & \hat{ \Theta}_\vartheta
\end{pmatrix},
\end{equation*}
where $\textbf{0}$ are zero-matrices of appropriate dimensions. Multiplying the above with $M$ on both sides gives:
\begin{equation}\label{Eq: Taylor P_n}
MP_n \nabla l_{\hat \theta} = M P_n \nabla l_{\theta_0} + M \frac{1}{\binom{n}{2}} D^T \hat W^2 D (\hat{ \theta} - \theta_0) + M O\left(\frac{1}{\binom{n}{2}}\sum_{i<j} D_{ij} \vert D_{ij}^T(\hat \theta - \theta_0) \vert^2 \right).
\end{equation}
Let us consider these terms in turn: Multiplication by $M$ means that the first $n$ entries of any of the vectors above are zero. Hence we only need to consider the last $p+1$ entries.
The left-hand side of (\ref{Eq: Taylor P_n}) is equal to zero by (\ref{Eq: KKT gamma}). The last $p+1$ entries of the first term on the right-hand side are $\hat \Theta_\vartheta P_n \nabla_\vartheta l_{\theta_0}$. For the second term on the right hand side, notice that
\[
\frac{1}{\binom{n}{2}} D^T \hat W^2 D = \frac{1}{\binom{n}{2}}\begin{bmatrix}
X^T\hat W^2X & X^T\hat W^2\textbf{1} & X^T\hat W^2Z \\
\textbf{1}^T\hat W^2X & \textbf{1}^T\hat W^2\textbf{1} & \textbf{1}^T\hat W^2Z \\
Z^T\hat W^2X & Z^T\hat W^2\textbf{1} & Z^T\hat W^2Z
\end{bmatrix}.
\]
$\hat \Theta_\vartheta$ is the exact inverse of $\hat \Sigma_\vartheta$ which is the lower-right $(p+1) \times (p+1)$ block of above matrix. Thus,
\[
M \frac{1}{\binom{n}{2}} D^T \hat W^2 D = \begin{bmatrix}
\textbf{0} &\textbf{0} \\
\hat \Theta_\vartheta \frac{1}{\binom{n}{2}}D_\vartheta^T\hat W^2X & I_{(p+1) \times (p+1)}
\end{bmatrix}.
\]
Then, for the last $p+1$ entries of $M \frac{1}{\binom{n}{2}} D^T \hat W^2 D (\hat{ \theta} - \theta_0)$,
\[
\left(M \frac{1}{\binom{n}{2}} D^T \hat W^2 D (\hat{ \theta} - \theta_0) \right)_{\text{last }p+1\text{ entries}}= \hat{ \Theta}_\vartheta \frac{1}{\binom{n}{2}}D_\vartheta^T\hat W^2X (\hat{ \beta} - \beta_0) + \begin{pmatrix}
	\hat \mu - \mu_0 \\
	\hat \gamma - \gamma_0
\end{pmatrix}.
\]
Thus, (\ref{Eq: Taylor P_n}) implies
\begin{equation*}
0 = \hat \Theta_\vartheta P_n \nabla_\gamma l_{\theta_0} +
\hat{ \Theta}_\vartheta \frac{1}{\binom{n}{2}}D_\vartheta^T\hat W^2X (\hat{ \beta} - \beta_0)
+ \begin{pmatrix}
\hat \mu - \mu_0 \\
\hat \gamma - \gamma_0
\end{pmatrix}
+ O\left( \hat \Theta_\vartheta \frac{1}{\binom{n}{2}}\sum_{i<j} \begin{pmatrix}
	1 \\
	Z_{ij}
\end{pmatrix} \vert D_{ij}^T(\hat \theta - \theta_0) \vert^2 \right),
\end{equation*}
which is equivalent to
\begin{equation}\label{Eq: inference equation}
\begin{pmatrix}
\hat \mu - \mu_0 \\
\hat \gamma - \gamma_0
\end{pmatrix} = - \hat \Theta_\vartheta P_n \nabla_\vartheta l_{\theta_0} -
\hat{ \Theta}_\vartheta \frac{1}{\binom{n}{2}}D_\vartheta^T\hat W^2X (\hat{ \beta} - \beta_0)
+ O\left( \hat \Theta_\vartheta \frac{1}{\binom{n}{2}}\sum_{i<j} \begin{pmatrix}
1 \\
Z_{ij}
\end{pmatrix} \vert D_{ij}^T(\hat \theta - \theta_0) \vert^2 \right).
\end{equation}
Our goal is now to show that for each component $k = 1, \dots, p+1$,
\[
	\sqrt{\binom{n}{2}}\frac{\hat \vartheta_k - \vartheta_{0,k}}{\sqrt{\hat \Theta_{\vartheta,k,k}}} \overset{d}{\longrightarrow} \mathcal{N}(0,1).
\]
as described in the \textbf{Goal} section.
To that end, by equation (\ref{Eq: inference equation}), we now need to solve the following three problems: Writing $\hat{ \Theta}_{\vartheta,k}$ for the $k$th row of $\hat{ \Theta}_\vartheta$,
\begin{enumerate}
	\item $\sqrt{\binom{n}{2}} \frac{ \hat \Theta_{\vartheta,k} P_n \nabla_\vartheta l_{\theta_0}}{\sqrt{\hat \Theta_{\vartheta,k,k}}} \overset{d}{\longrightarrow}  \mathcal{N}(0,1)$,
	\item 
	$
	\frac{1}{\sqrt{\hat \Theta_{\vartheta,k,k}}} \hat{ \Theta}_{\vartheta,k} \frac{1}{\binom{n}{2}}D_\vartheta^T\hat W^2X (\hat{ \beta} - \beta_0) = o_P\left( \binom{n}{2}^{-1/2} \right),
	$
	\item $
	O\left( 	\frac{1}{\sqrt{\hat \Theta_{\vartheta,k,k}}} \hat \Theta_{\vartheta,k} \frac{1}{\binom{n}{2}}\sum_{i<j} \begin{pmatrix}
	1 \\
	Z_{ij}
	\end{pmatrix} \vert D_{ij}^T(\hat \theta - \theta_0) \vert^2 \right) = o_P\left( \binom{n}{2}^{-1/2} \right).
	$
\end{enumerate}

\subsection{Bounding inverses}\label{Sec: Bounding inverses}

The problems (1) - (3) above suggest that it will be essential to bound the norm and the distance of $\hat \Theta_\vartheta$ and $\Theta_\vartheta$ in an appropriate manner. Notice that for any invertible matrices $A, B \in \R^{m \times m}$ we have
\[
A^{-1} - B^{-1} = A^{-1}(B - A)B^{-1}.
\]
Thus, for any sub-multiplicative matrix norm $\Vert \; . \; \Vert$, we get
\begin{equation}\label{Eq: difference between inverse matrices}
\Vert A^{-1} - B^{-1} \Vert \le  \Vert A^{-1} \Vert \Vert B^{-1} \Vert \Vert B - A \Vert.
\end{equation}
We are particularly interested in the matrix $\infty$-norm, defined as
\[
\Vert A \Vert_\infty \coloneqq \sup \left\{ \frac{\Vert Ax \Vert_\infty}{\Vert x \Vert_\infty}, x \neq 0  \right\} =  \sup \left\{ \Vert Ax \Vert_\infty, \Vert x \Vert_\infty = 1 \right\}  = \max_{1 \le i \le m} \sum_{j = 1}^m \vert A_{i,j} \vert,
\]
i.e. $\Vert A \Vert_\infty$ is the maximal row $\ell_1$-norm of $A$. It is well-known, that any such matrix norm induced by a vector norm is sub-multiplicative ($\Vert AB \Vert_\infty \le \Vert A \Vert_\infty \Vert B \Vert_\infty$) and consistent with the inducing vector norm ($\Vert Ax \Vert_\infty \le \Vert A \Vert_\infty \Vert x \Vert_\infty$ for any vector $x$ of appropriate dimension). We first want to bound the matrix $\infty$-norm in terms of the largest eigenvalue.

\begin{Lem}\label{Lem: bound matrix norm by eval}
	For any symmetric, positive semi-definite $(m \times m)$-matrix $A$ with maximal eigenvalue $\lambda > 0$, we have $\Vert A \Vert_\infty \le \sqrt{m} \lambda$.
\end{Lem}
\begin{proof}
	\begin{align*}
	\Vert A \Vert_\infty &=  \sup \left\{ \Vert Ax \Vert_\infty, \Vert x \Vert_\infty = 1 \right\}  \\
	&\le \sup \left\{ \Vert Ax \Vert_2, \Vert x \Vert_\infty = 1 \right\}, \quad \Vert Ax \Vert_\infty \le \Vert Ax \Vert_2 \\
	&= \sup \left\{ \frac{\Vert Ax \Vert_2}{\Vert x \Vert_2} \Vert x \Vert_2, \Vert x \Vert_\infty = 1   \right\} \\
	&\le \sqrt{m} \sup \left\{ \frac{\Vert Ax \Vert_2}{\Vert x \Vert_2}, \Vert x \Vert_\infty = 1   \right\}, \quad \text{ if } \Vert x \Vert_\infty = 1, \text{ then } \Vert x \Vert_2 \le \sqrt{m}, \\
	&\le  \sqrt{m} \sup \left\{ \frac{\Vert Ax \Vert_2}{\Vert x \Vert_2}, x \neq 0  \right\} \\
	&= \sqrt{m} \Vert A \Vert_2 = \sqrt{m}\lambda,
	\end{align*}
	where $\Vert A \Vert_2$ is the spectral norm of the matrix $A$ and we have used that for symmetric matrices, the spectral norm is equal to the modulus of the largest eigenvalue of $A$.
\end{proof}
Also, recall that the inverse of a symmetric matrix $A$ is itself symmetric:
\[
I = A A^{-1} = A^T A^{-1} \overset{\text{transpose}}{\Rightarrow} I = (A^{-1})^TA^T \overset{\text{symmetry}}{=} (A^{-1})^TA \overset{\text{uniqueness of inverse}}{\Rightarrow} (A^{-1})^T = A^{-1}.
\]
Hence, $\hat \Theta_\vartheta$ and $\Theta_\vartheta$ are symmetric and we may apply Lemma \ref{Lem: bound matrix norm by eval}. Using that $\lambda_{\max}(\Sigma_\vartheta^{-1}) = \frac{1}{\lambda_{\min}(\Sigma_\vartheta)}$, we get
\begin{equation*}
\Vert \Theta_\vartheta \Vert_\infty \le \sqrt{p} \cdot \lambda_{\max}(\Sigma_\vartheta^{-1}) \le C \frac{1}{\rho_n},
\end{equation*}
and with high probability
\begin{equation*}
\Vert \hat \Theta_\vartheta \Vert_\infty \le \sqrt{p}\cdot \lambda_{\max}(\hat \Sigma_\vartheta^{-1}) \le C \frac{1}{\rho_n},
\end{equation*}
with some absolute constant $C$. 
Finally, by (\ref{Eq: difference between inverse matrices}),
\begin{equation*}
\Vert \hat \Theta_\vartheta - \Theta_\vartheta \Vert_\infty \le \Vert \hat \Theta_\vartheta \Vert_\infty \Vert \Theta_\vartheta \Vert_\infty \Vert \hat \Sigma_\vartheta - \Sigma_\vartheta \Vert_\infty \le \frac{C}{\rho_n^2} \Vert \hat \Sigma_\vartheta - \Sigma_\vartheta \Vert_\infty.
\end{equation*}
It remains to control $\Vert \hat \Sigma_\vartheta - \Sigma_\vartheta \Vert_\infty$. We have
\begin{align*}
\hat \Sigma_\vartheta - \Sigma_\vartheta  &= \frac{1}{\binom{n}{2}} \left( D_\vartheta^T\hat W^2 D_\vartheta - \E[D_\vartheta^T W_0^2 D_\vartheta]  \right) \\
&= \underbrace{\frac{1}{\binom{n}{2}} \left( D_\vartheta^T (\hat W^2 - W_0^2) D_\vartheta   \right)}_{({I})} + \underbrace{\frac{1}{\binom{n}{2}} \left( D_\vartheta^TW_0^2 D_\vartheta - \E[D_\vartheta^TW_0^2 D_\vartheta]  \right)}_{(II)}.
\end{align*}
Recall that $\hat w_{ij}^2 = p_{ij}(\hat \theta) (1 - p_{ij}(\hat \theta)) = \frac{\exp(D_{ij}^T\hat \theta)}{(1 + \exp(D_{ij}^T\hat{ \theta}))^2} = \partial_{a^2}l(A_{ij}, D_{ij}^T\hat{ \theta})$, with the function $l$ defined in (\ref{Eq: Def l (a)}). Also recall that $\partial_{a^2}l$ is Lipschitz with constant one, by the Mean Value Theorem and the fact that it has derivative $\partial_{a^3}l$ bounded by one. Thus, considering the $(k,l)$-th element of $(I)$ above, we get:
\begin{align*}
\left| \frac{1}{\binom{n}{2}} \left( D_\vartheta^T (\hat W^2 - W_0^2) D_\vartheta   \right)_{kl} \right| &= \left| \frac{1}{\binom{n}{2}} \sum_{ i < j} D_{ij,n+k}D_{ij,n+l} (\hat w_{ij}^2 - w_{0,ij}^2)\right| \\
&\le C \frac{1}{\binom{n}{2}} \sum_{ i < j} \vert \hat w_{ij}^2 - w_{0,ij}^2\vert, \quad \text{ by unifrom boundedness of } Z_{ij} \\
&\le C \frac{1}{\binom{n}{2}} \sum_{ i < j} \vert D_{ij}^T(\hat \theta - \theta_0) \vert, \quad \text{ by Lipschitz continuity} \\
&\le \frac{C}{\binom{n}{2}} \sum_{ i < j} \left\{  \vert \hat \beta_i - \beta_{0,i} \vert + \vert \hat \beta_j - \beta_{0,j} \vert + \vert \hat \mu - \mu_0 \vert + \vert Z_{ij}^T(\hat \gamma - \gamma_0) \vert \right\} \\
&\le \frac{C}{\binom{n}{2}} \underbrace{ \left\{\sum_{ i < j}  \vert \hat \beta_i - \beta_{0,i} \vert + \vert \hat \beta_j - \beta_{0,j} \vert \right\}}_{= (n-1) \Vert \hat \beta - \beta_0 \Vert_1}  + C \vert \hat \mu - \mu_0 \vert  + C \Vert \hat \gamma - \gamma_0 \Vert_1 \\
&\le C \left\{  \frac{1}{n} \Vert \hat \beta - \beta_0 \Vert_1 + \vert \hat \mu - \mu_0 \vert + \Vert \hat \gamma - \gamma_0 \Vert_1  \right\} \\
&= O_P\left(s_{0,+} \sqrt{\frac{\log(n)}{\binom{n}{2}}} \rho_n^{-1}  \right), \quad \text{ under the conditions of Theorem \ref{Cor: no approximation error}}.
\end{align*} 
Since the dimension of $(I)$ is $(p+1) \times (p+1)$ and thus remains fixed, any row of $(I)$ has $\ell_1$ norm of order $O_P\left(s_{0,+}\sqrt{\frac{\log(n)}{\binom{n}{2}}} \rho_n^{-1}  \right)$ and thus
\[
\Vert(I) \Vert_\infty = O_P\left(s_{0,+} \sqrt{\frac{\log(n)}{\binom{n}{2}}} \rho_n^{-1}  \right).
\]
Taking a look at the $(k,l)$-th element in $(II)$:
\begin{align*}
\left\vert \frac{1}{\binom{n}{2}} \left( D_\vartheta^TW_0^2 D_\vartheta - \E[D_\vartheta^TW_0^2 D_\vartheta]  \right)_{kl} \right\vert = \left\vert \frac{1}{\binom{n}{2}} \sum_{ i < j} \left\{D_{ij,n+k}D_{ij,n+l}w_{0,ij}^2 - \E[D_{ij,n+k}D_{ij,n+l} w_{0,ij}^2] \right\}\right\vert.
\end{align*}
Note that the random variables $D_{ij,n+k}D_{ij,n+l}w_{0,ij}^2$ are bounded uniformly in $i,j,k,l$. Thus, by Hoeffding's inequality, for any $t \ge 0$,
\begin{align*}
P\left( \left\vert \frac{1}{\binom{n}{2}} \sum_{ i < j} \left\{D_{ij,n+k}D_{ij,n+l}w_{0,ij}^2 - \E[D_{ij,n+k}D_{ij,n+l} w_{0,ij}^2] \right\}\right\vert \ge t  \right) \le 2 \exp\left( - C \binom{n}{2} t^2 \right).
\end{align*}
This means, $\left\vert \frac{1}{\binom{n}{2}} \left( D_\vartheta^TW_0^2 D_\vartheta - \E[D_\vartheta^TW_0^2 D_\vartheta]  \right)_{kl} \right\vert = O_P\left( \binom{n}{2}^{-1/2}  \right)$. Again, since the dimension $p+1$ is fixed, we get by a simple union bound
\[
\Vert (II) \Vert_\infty = O_P\left(  \binom{n}{2}^{-1/2} \right).
\]
In total, we thus get
\begin{align*}
\Vert \hat \Sigma_\vartheta - \Sigma_\vartheta \Vert_\infty 
&=O_P\left(s_{0,+}  \sqrt{\frac{\log(n)}{\binom{n}{2}}} \rho_n^{-1}  + \frac{1}{\sqrt{\binom{n}{2}}}  \right) = O_P\left(s_{0,+}  \sqrt{\frac{\log(n)}{\binom{n}{2}}} \rho_n^{-1} \right).
\end{align*}
We can now obtain a rate for $	\Vert \hat \Theta_\vartheta - \Theta_\vartheta \Vert_\infty$.
\begin{align*}
\Vert \hat \Theta_\vartheta - \Theta_\vartheta \Vert_\infty \le \frac{C}{\rho_n^2} \Vert \hat \Sigma_\vartheta - \Sigma_\vartheta \Vert_\infty = O_P\left( s_{0,+}  \sqrt{\frac{\log(n)}{\binom{n}{2}}} \rho_n^{-3} \right).
\end{align*}
By Assumption \ref{Assum: new rate of s and rho_n}, we have $s_{0,+} \frac{\sqrt{\log(n)}}{\sqrt{n}\rho_n^2} \rightarrow 0, n \rightarrow \infty$, which in particular also implies that the above is $o_P(1)$. Notice in particular, that we have now managed to get for $k = 1, \dots, p+1,$
\begin{itemize}
	\item $\Vert \hat \Theta_{\vartheta,k}  - \Theta_{\vartheta,k} \Vert_1 = o_P(1)$,
	\item $\hat \Theta_{\vartheta,k,k} = \Theta_{\vartheta,k,k} + o_p(1)$.
\end{itemize}

\subsection{Problem 1}\label{C.4}

We can now take a look at the problems (1) - (3) outlined above. For problem (1), we want to show:
\[
\sqrt{\binom{n}{2}} \frac{ \hat \Theta_{\vartheta,k} P_n \nabla_\vartheta l_{\theta_0}}{\sqrt{\hat \Theta_{\vartheta,k,k}}} \rightarrow \mathcal{N}(0,1).
\]
\textbf{Step 1:} Show that
\begin{equation}\label{Eq: root n}
\hat{ \Theta}_{\vartheta,k} P_n \nabla_\vartheta l_{\theta_0} = \Theta_{\vartheta,k} P_n \nabla_\vartheta l_{\theta_0}+ o_P\left( \binom{n}{2}^{-1/2}\right).
\end{equation}
We have
\begin{align*}
\vert (\hat \Theta_{\vartheta,k} - \Theta_{\vartheta,k} ) P_n \nabla_\vartheta l_{\theta_0} \vert &\le \Vert \hat \Theta_{\vartheta,k} - \Theta_{\vartheta,k} \Vert_1 \left\Vert \frac{1}{\binom{n}{2}} \sum_{ i < j} \begin{pmatrix}
	1 \\
	Z_{ij}
\end{pmatrix} (p_{ij}(\theta_0) - A_{ij}) \right\Vert_\infty \\
&\le \Vert \hat \Theta_{\vartheta} - \Theta_{\vartheta} \Vert_\infty \left\Vert \frac{1}{\binom{n}{2}} \sum_{ i < j} D_{\vartheta, ij} (p_{ij}(\theta_0) - A_{ij}) \right\Vert_\infty.
\end{align*}
Consider the vector $\sum_{ i < j} D_{\vartheta, ij} (p_{ij}(\theta_0) - A_{ij}) \in \R^{p+1}$. The $k$th component of it has the form $\sum_{ i < j} (p_{ij}(\theta_0) - A_{ij})$ for $k=1$ and $\sum_{ i < j} Z_{ij,k-1} (p_{ij}(\theta_0) - A_{ij}), k = 2, \dots, p+1$. Notice that these components are all centered:
\[
\mathbb{E}[ D_{\vartheta, ij,k} (p_{ij}(\theta_0) - A_{ij})] = \mathbb{E}[D_{\vartheta,ij,k} \mathbb{E}[ (p_{ij}(\theta_0) - A_{ij})| Z_{ij}]] = \mathbb{E}[D_{\vartheta, ij, k} \cdot0] = 0,
\]
as well as $\vert D_{\vartheta, ij,k} (p_{ij}(\theta_0) - A_{ij}) \vert \le c$, where $c > 1$ is a universal constant bounding $\vert Z_{ij,k} \vert$ for all $i,j,k$.
Thus, by Hoeffding's inequality, for any $t > 0$,
\begin{align*}
P\left( \left|   \frac{1}{\binom{n}{2}} \sum_{ i < j} D_{\vartheta, ij, k} (p_{ij}(\theta_0) - A_{ij})  \right| \ge t   \right) \le 2 \exp \left(- 2 \frac{\binom{n}{2} t^2 }{c^2}  \right)
\end{align*}
and thus, 
\[
\frac{1}{\binom{n}{2}} \sum_{ i < j} D_{\vartheta, ij} (p_{ij}(\theta_0) - A_{ij}) = O_P\left( \binom{n}{2}^{-1/2} \right).
\]
Since we have $\Vert \hat \Theta_\vartheta - \Theta_\vartheta \Vert_\infty = o_P(1)$, by Section \ref{Sec: Bounding inverses}, step 1 is now concluded. 

\noindent \textbf{Step 2:} Show that

\begin{equation*}
\hat \Theta_{\vartheta,k,k} = \Theta_{\vartheta,k,k} + o_P(1).
\end{equation*}
Since $\Vert \hat \Theta_\vartheta - \Theta_\vartheta \Vert_\infty = o_P(1)$, by Section \ref{Sec: Bounding inverses}, for all $k$
\[
\vert \hat \Theta_{\vartheta,k,k} - \Theta_{\vartheta,k,k} \vert \le \Vert \hat \Theta_\vartheta - \Theta_\vartheta \Vert_\infty = o_P(1)
\]
and step 2 is concluded.

\noindent\textbf{Step 3:} Show that
\[
\left\vert \frac{1}{\Theta_{\vartheta,k,k}} \right\vert \le C < \infty,
\]
for some universal constant $C > 0$. Then, we may conclude from step 1 and step 2 that
\[
\sqrt{\binom{n}{2}} \frac{ \hat \Theta_{\vartheta,k} P_n \nabla_\vartheta l_{\theta_0}}{\sqrt{\hat \Theta_{\vartheta,k,k}}} = \sqrt{\binom{n}{2}}\frac{\Theta_{\vartheta,k} P_n \nabla_\vartheta l_{\theta_0}}{\sqrt{\Theta_{\vartheta,k,k}}} + o_P(1). 
\]
To prove step 3, notice that $\Theta_\vartheta$ is symmetric and hence has only real eigenvalues. Therefore it is unitarily diagonalizable and for any $x \in \R^{p+1}$, we have $x^T \Theta_\vartheta x \ge \lambda_{\min}( \Theta_\vartheta) \Vert x \Vert_2^2$. We also know that
\[
	\lambda_{\min}({  \Theta_\vartheta}) = \frac{1}{\lambda_{\max}( \Sigma_\vartheta)}.
\]
Under Assumption \ref{Assum: maximum EW} we can now deduce an upper bound on the maximum eigenvalue of $\Sigma_\vartheta$: For any $x \in \R^p$,
\[
x^T\Sigma_\vartheta x = x^T \frac{1}{\binom{n}{2}}E[D_\vartheta^T W_0^2 D_\vartheta] x \le x^T \frac{1}{\binom{n}{2}}E[D_\vartheta^T D_\vartheta] x \le (1 \vee \lambda_{\max}) \Vert x \Vert_2^2,
\]
where used that any entry in $W_0^2$ is bounded above by one. Since $x^T\Sigma_\vartheta x \le \lambda_{\max}(\Sigma_\vartheta) \Vert x \Vert_2^2$ and since this bound is tight, we can conclude by Assumption \ref{Assum: maximum EW} that $\lambda_{\max}(\Sigma_\vartheta) \le (1 \vee \lambda_{\max}) \le C < \infty$ for some universal constant $C > 0$.

In particular, since ${ \Theta}_{\vartheta,k,k} = e_k^T \Theta_\vartheta e_k$, we get
\[
{ \Theta}_{\vartheta,k,k} \ge \lambda_{\min}({ \Theta_\vartheta}) \Vert e_k \Vert_2^2 = \frac{1}{\lambda_{\max} ( \Sigma_\vartheta)} \ge C > 0,
\]
uniformly for all $n$. Consequently,
\[
0 <  \frac{1}{\Theta_{\vartheta,k,k}}  \le C < \infty.
\]
Step 3 is thus concluded.

\noindent \textbf{Step 4:} Finally, show that
\begin{equation*}
\sqrt{\binom{n}{2}}\frac{\Theta_{\vartheta,k} P_n \nabla_\vartheta l_{\theta_0}}{\sqrt{\Theta_{\vartheta,k,k}}} \overset{d}{\longrightarrow} \mathcal{N}(0,1),
\end{equation*}
Such that by all the above
\begin{equation*}
\sqrt{\binom{n}{2}} \frac{ \hat \Theta_{\vartheta,k} P_n \nabla_\vartheta l_{\theta_0}}{\sqrt{\hat \Theta_{\vartheta,k,k}}} \overset{d}{\longrightarrow} \mathcal{N}(0,1).
\end{equation*}
For brevity, we write $p_{ij}$ for the true link probabilities $p_{ij}(\theta_0)$. Also keep in mind that $\Theta_{\vartheta,k}$ denotes the $k$th \textit{row} of $\Theta_\vartheta$, while $D_{\vartheta, ij}$ denote $((p+1) \times 1)$-\textit{column} vectors.
We want to apply the Lindeberg-Feller Central Limit Theorem. The random variables we study are the summands in 
\[
\sqrt{\binom{n}{2}}  \Theta_{\vartheta,k} P_n \nabla_\vartheta l_{\theta_0} = \sum_{ i < j} \left\{ \frac{1}{\sqrt{\binom{n}{2}}} \Theta_{\vartheta,k} D_{\vartheta, ij} (p_{ij} - A_{ij}) \right\}.
\]
First, notice that these random variables are centered:
\begin{align*}
	\mathbb{E}\left[ \frac{1}{\sqrt{\binom{n}{2}}} \Theta_{\vartheta,k}D_{\vartheta, ij} (p_{ij} - A_{ij}) \right] &= \mathbb{E}\left[ \frac{1}{\sqrt{\binom{n}{2}}} \Theta_{\vartheta,k}D_{\vartheta, ij} \mathbb{E}[p_{ij} - A_{ij} | Z_{ij}]  \right] \\
	&= \mathbb{E}\left[ \frac{1}{\sqrt{\binom{n}{2}}} \Theta_{\vartheta,k}D_{\vartheta, ij}  \cdot 0 \right] = 0.
\end{align*}
For the Lindeberg-Feller CLT we need to sum up the variances of these random variables. 
We claim that
\[
\sum_{ i < j} \text{Var}\left( \frac{1}{\sqrt{\binom{n}{2}}} \Theta_{\vartheta,k}D_{\vartheta, ij} (p_{ij} - A_{ij})   \right) = \Theta_{\vartheta,k,k}.
\]
Indeed, consider the vector-valued random variable $\sum_{ i < j} \left\{ \frac{1}{\sqrt{\binom{n}{2}}} D_{\vartheta, ij} (p_{ij} - A_{ij}) \right\} \in \R^{p+1}$. It has covariance matrix
\begin{align*}
\mathbb{E}&\left[ \sum_{ i < j} \left\{ \frac{1}{\sqrt{\binom{n}{2}}} D_{\vartheta, ij} (p_{ij} - A_{ij}) \right\} \sum_{ i < j} \left\{ \frac{1}{\sqrt{\binom{n}{2}}} D_{\vartheta, ij} (p_{ij} - A_{ij}) \right\}^T   \right] \\
&= \mathbb{E}\left[ \sum_{ i < j} \frac{1}{\sqrt{\binom{n}{2}}} D_{\vartheta, ij} (p_{ij} - A_{ij}) \frac{1}{\sqrt{\binom{n}{2}}} D_{\vartheta, ij}^T (p_{ij} - A_{ij})   \right], \quad \text{by independence accross } i,j \\
&= \frac{1}{\binom{n}{2}} \sum_{i<j} \left[  \mathbb{E}[D_{\vartheta, ij,k}D_{\vartheta, ij,l}  (p_{ij} - A_{ij})^2] \right]_{k,l = 1, \dots, p+1}, \quad \text{ as a } ((p+1) \times (p+1))\text{-matrix} \\
&= \frac{1}{\binom{n}{2}} \mathbb{E}[D_\vartheta^TW_0^2D_\vartheta] \\
&= \Sigma_\vartheta.
\end{align*}
Thus, by independence across $i,j$,
\begin{align*}
	\sum_{ i < j} \text{Var}\left( \frac{1}{\sqrt{\binom{n}{2}}} \Theta_{\vartheta,k}D_{\vartheta,ij} (p_{ij} - A_{ij})   \right) &= \text{Var}\left( \Theta_{\vartheta,k}  \sum_{ i < j}  \frac{1}{\sqrt{\binom{n}{2}}} D_{\vartheta, ij} (p_{ij} - A_{ij})  \right) \\
	&= \Theta_{\vartheta,k} \Sigma_\vartheta \Theta_{\vartheta,k}^T = \Theta_{\vartheta,k,k},
\end{align*}
where for the last equality we have used that $\Theta_\vartheta$ is the inverse of $\Sigma_\vartheta$ and thus, $\Sigma_\vartheta\Theta_{\vartheta,k}^T= e_k$.
Now, we need to show that the Lindeberg condition holds. That is, we want that for any $\epsilon > 0$,
\begin{equation}\label{Eq: Lindeberg condition}
\lim_{n\rightarrow \infty} \frac{1}{\Theta_{\vartheta,k,k}} \sum_{ i < j} \mathbb{E}\left[ \left\{  \frac{1}{\sqrt{\binom{n}{2}}} \Theta_{\vartheta,k}D_{\vartheta, ij} (p_{ij} - A_{ij}) \right\}^2 \mathbbm{1}\left( \vert \Theta_{\vartheta,k}D_{\vartheta, ij} (p_{ij} - A_{ij}) \vert > \epsilon \sqrt{\binom{n}{2}\Theta_{\vartheta,k,k} } \right) \right] = 0.
\end{equation}
We have
\[
\vert \Theta_{\vartheta,k}D_{\vartheta, ij} (p_{ij} - A_{ij}) \vert \le p \cdot c \cdot \Vert \Theta_{\vartheta,k} \Vert_1 \le C \Vert \Theta_\vartheta \Vert_\infty \le C \rho_n^{-1}.
\]
At the same time, we know from step 3 that $\Theta_{Z,k,k} \ge C > 0$ for some universal $C$. Then, as long as $\rho_n^{-1}$ goes to infinity at a rate slower than $n$, which is enforced by Assumption \ref{Assum: new rate of s and rho_n}, we must have for $n$ large enough
\[
\vert \Theta_{\vartheta,k}D_{\vartheta, ij} (p_{ij} - A_{ij}) \vert < \epsilon \sqrt{\binom{n}{2}\Theta_{\vartheta,k,k} }
\]
uniformly in $i,j$. Thus, the indicator function and therefore each summand in (\ref{Eq: Lindeberg condition}) is equal to zero for $n$ large enough. Hence, (\ref{Eq: Lindeberg condition}) holds. Then, by the Lindeberg-Feller CLT,
\[
\sqrt{\binom{n}{2}}\frac{\Theta_{\vartheta,k} P_n \nabla_\vartheta l_{\theta_0}}{\sqrt{\Theta_{\vartheta,k,k}}} \overset{d}{\longrightarrow} \mathcal{N}(0,1).
\]
Now, by the steps 1-4 and Slutzky's Theorem
\begin{align*}
	\sqrt{\binom{n}{2}} \frac{ \hat \Theta_{\vartheta,k} P_n \nabla_\vartheta l_{\theta_0}}{\sqrt{\hat \Theta_{\vartheta,k,k}}} &= \sqrt{\binom{n}{2}} \frac{ ( \Theta_{\vartheta,k} + o_P(1))P_n \nabla_\vartheta l_{\theta_0}}{\sqrt{ (\Theta_{\vartheta,k,k} + o_P(1)) }} \\
	&= \sqrt{\binom{n}{2}} \frac{ \Theta_{\vartheta,k} P_n \nabla_\vartheta l_{\theta_0}}{\sqrt{ (\Theta_{\vartheta,k,k} + o_P(1)) }} + \sqrt{\binom{n}{2}} \frac{ o_P(1)P_n \nabla_\vartheta l_{\theta_0}}{\sqrt{ (\Theta_{\vartheta,k,k} + o_P(1)) }} \\
	&\overset{d}{\longrightarrow} \mathcal{N}(0,1).
\end{align*}
This concludes solving problem 1. 

\subsection{Problem 2}\label{Sec: Problem 2, beta}
For problem 2 we must show
\[
	\frac{1}{\sqrt{\hat \Theta_{\vartheta,k,k}}} \hat{ \Theta}_{\vartheta,k} \frac{1}{\binom{n}{2}}D_\vartheta^T\hat W^2X (\hat{ \beta} - \beta_0) = o_P\left( \binom{n}{2}^{-1/2} \right).
\]
Since we have $\Vert \hat \Theta_\vartheta - \Theta_\vartheta \Vert_\infty = o_P(1)$, we do not need to worry about $\frac{1}{\sqrt{\hat \Theta_{Z,k,k}}}$, because $\hat{ \Theta}_{Z,k,k} = \Theta_{Z,k,k} + o_P(1)$ and $\frac{1}{\sqrt{ \Theta_{Z,k,k}}} \le C < \infty$, i.e. $\frac{1}{\sqrt{ \hat \Theta_{Z,k,k}}} = O_P(1)$ . By Theorem \ref{Cor: no approximation error} we also have a high-probability error bound on $\Vert \hat \beta - \beta_0 \Vert_1$. The problem will be bounding the corresponding matrix norms.
\[
\left\vert \hat{ \Theta}_{\vartheta,k} \frac{1}{\binom{n}{2}}D_\vartheta^T\hat W^2X (\hat{ \beta} - \beta_0) \right\vert \le \left\Vert \frac{1}{\binom{n}{2}} X^T\hat W^2 D_\vartheta \hat{ \Theta}_{\vartheta,k}^T\right\Vert_\infty \Vert \hat \beta - \beta_0 \Vert_1.
\]
Notice that in the display above we have the vector $\ell_\infty$-norm.
Also,
\[
\left\Vert\frac{1}{\binom{n}{2}} X^T\hat W^2 D_\vartheta \hat{ \Theta}_{\vartheta,k}^T \right\Vert_\infty \le \Vert \hat{ \Theta}_{\vartheta,k}^T \Vert_\infty \left\Vert \frac{1}{\binom{n}{2}}X^T\hat W^2D_\vartheta \right\Vert_\infty.
\]
Here we used the compatibility of the matrix $\ell_\infty$-norm with the vector $\ell_\infty$-norm. The first term is the vector norm, the second the matrix norm.
We know,
\[
\Vert \hat{ \Theta}_{\vartheta,k}^T \Vert_\infty \le \Vert \hat \Theta_\vartheta \Vert_\infty \le C \rho_n^{-1},
\]
where on the left hand side we have the vector norm and in the middle display the matrix norm.
Finally, $\frac{1}{\binom{n}{2}}X^T\hat W^2D_\vartheta$ is a $(n \times (p+1))$-matrix. The $(k,l)$-th element looks like
\[
\left\vert \frac{1}{\binom{n}{2}} \sum_{ i = 1, i \neq l}^n D_{\vartheta, il,k} \hat{ w}_{il}^2 \right\vert \le \frac{1}{\binom{n}{2}} \cdot (n - 1)  \cdot c = \frac{C}{n}.
\]
Thus, the $\ell_1$-norm of any row of $\frac{1}{\binom{n}{2}}X^T\hat W^2D_\vartheta$ is bounded by $C/n$ and thus
\[
\left\Vert \frac{1}{\binom{n}{2}}X^T\hat W^2D_\vartheta \right\Vert_\infty \le \frac{C}{n}.
\]
Recall that $\Vert \hat \beta - \beta_0 \Vert_1 = O_P\left(s_{0,+}  \frac{\sqrt{\log(n)}}{\sqrt{n}} \rho_n^{-1}  \right)$ by Theorem \ref{Cor: no approximation error}. Then,
\begin{align*}
\left\vert \hat{ \Theta}_{\vartheta,k} \frac{1}{\binom{n}{2}}X^T\hat W^2D_\vartheta (\hat{ \beta} - \beta_0) \right\vert &\le \Vert \hat{ \Theta}_{\vartheta,k}^T \Vert_\infty \left\Vert \frac{1}{\binom{n}{2}}D_\vartheta^T\hat W^2X \right\Vert_\infty \Vert \hat \beta - \beta_0 \Vert_1 \\
&=O_P\left( \frac{s_{0,+} }{\rho_n^2 \cdot n} \cdot  \frac{\sqrt{\log(n)}}{\sqrt{n}}  \right).
\end{align*}
Multiplying by $\sqrt{\binom{n}{2}} = O(n)$, gives
\begin{align*}
\sqrt{\binom{n}{2}}\left\vert \hat{ \Theta}_{\vartheta,k} \frac{1}{\binom{n}{2}}D_\vartheta^T\hat W^2X (\hat{ \beta} - \beta_0) \right\vert &= O_P\left(  \frac{s_{0,+} }{\rho_n^2 } \cdot  \frac{\sqrt{\log(n)}}{\sqrt{n}}   \right),
\end{align*}
which is $o_P(1)$ under Assumption \ref{Assum: new rate of s and rho_n}.

\subsection{Problem 3}

Finally, we must show
\[
	O\left( 	\frac{1}{\sqrt{\hat \Theta_{\vartheta,k,k}}} \hat \Theta_{\vartheta,k} \frac{1}{\binom{n}{2}}\sum_{i<j} \begin{pmatrix}
	1 \\
	Z_{ij}
	\end{pmatrix} \vert D_{ij}^T(\hat \theta - \theta_0) \vert^2 \right) = o_P\left( \binom{n}{2}^{-1/2} \right).
\]
Again, since $\hat{ \Theta}_{\vartheta,k,k} = \Theta_{\vartheta,k,k} + o_P(1)$ and $\Theta_{\vartheta,k,k} \ge C > 0$ uniformly in $n$, we do not need to worry about the factor $	\frac{1}{\sqrt{\hat \Theta_{\vartheta,k,k}}}$ and it remains to show
\[
O\left( \hat \Theta_{\vartheta,k} \frac{1}{\binom{n}{2}}\sum_{i<j} D_{\vartheta, ij} \vert D_{ij}^{T}(\hat \theta - \theta_0) \vert^2 \right) =  o_P\left( \binom{n}{2}^{-1/2} \right).
\]
We have
\begin{align*}
\left\vert  \hat \Theta_{\vartheta,k} \frac{1}{\binom{n}{2}}\sum_{i<j} D_{\vartheta, ij} \vert D_{ij}^{T}(\hat \theta - \theta_0) \vert^2  \right\vert &\le   \frac{1}{\binom{n}{2}}\sum_{i<j}  \vert \hat \Theta_{\vartheta,k}D_{\vartheta, ij} \vert \vert D_{ij}^T(\hat \theta - \theta_0) \vert^2 \\
&\le c \Vert \hat \Theta_{\vartheta,k} \Vert_1 \frac{1}{\binom{n}{2}}\sum_{i<j} \vert D_{ij}^T(\hat \theta - \theta_0) \vert^2 \\
&\le C \frac{1}{\rho_n} \frac{1}{\binom{n}{2}}\sum_{i<j} \vert D_{ij}^T(\hat \theta - \theta_0) \vert^2,
\end{align*}
where for the last inequality we have used that $\Vert \hat \Theta_{\vartheta,k} \Vert_1 \le \Vert \hat \Theta_\vartheta \Vert_\infty \le C \frac{1}{\rho_n}$. 
Now remember from \eqref{Eq: bound on sum D_ij theta -theta*} that
\[
	\frac{1}{\binom{n}{2}}\sum_{i<j} \vert D_{ij}^T(\hat \theta - \theta_0) \vert^2 \le C \Vert \hat{ \bar \theta} - \bar \theta_0 \Vert_1^2,
\]
where we make use of the fact that $\theta^* = \theta_0$ if there is no approximation error (as assumed by Theorem \ref{Thm: inference}) and that $\bar D \bar \theta = D \theta$.
From Theorem \ref{Cor: no approximation error} we know that under the assumptions of Theorem \ref{Thm: inference}, $\Vert \hat{ \bar \theta} - \bar \theta_0 \Vert_1 = O_P\left(s_{0,+}  \sqrt{\frac{\log(n)}{\binom{n}{2}}} \rho_n^{-1} \right)$.
Thus,
\[
\sqrt{\binom{n}{2}}\left\vert  \hat \Theta_{\vartheta, k} \frac{1}{\binom{n}{2}}\sum_{i<j} D_{\vartheta, ij} \vert D_{ij}^T(\hat \theta - \theta_0) \vert^2  \right\vert = O_P\left( (s_{0,+} )^2 \frac{\log(n)}{\sqrt{\binom{n}{2}}} \rho_n^{-3} \right).
\]
We see that this is $o_P(1)$ by applying Assumption \ref{Assum: new rate of s and rho_n} twice. Problem 3 is solved.

\begin{proof}[Proof of Theorem \ref{Thm: inference}]
	Theorem \ref{Thm: inference} now follows from the solved problems (1) - (3).
\end{proof}

\section{Proof of Theorem \ref{beta_AN}}

To be consistent with the notation used in Section \ref{proof_thm3}, let $\hat{S}_\beta = \binom{n}{2}\hat{U}_\beta$ and recall that $P_nl_{\theta} =  \mathcal{L}(\theta)/\binom{n}{2}$. Then $\hat{b} =\hat{\beta}-\hat{U}_\beta \nabla_{\beta} \mathcal{L}(\hat{\theta})= \hat{\beta}-\hat{S}_\beta P_n \nabla_{\beta} l_{\hat{\theta}}$.

\begin{proof} Using Taylor expansion, we have
\begin{align*}
P_n \nabla_{\beta} l_{\hat{\theta}} &=P_n \nabla_{\beta} l_{\theta_0}+\frac{1}{\binom{n}{2}} X^T \hat{W}^2 D\big(\hat{\theta}-\theta_0\big)+O\left(\frac{1}{\binom{n}{2}} \sum_{i<j} X_{i j}l^{\prime\prime\prime} (A_{ij},D_{i j}\tilde{\theta})\big|D_{i j}^T\big(\hat{\theta}-\theta_0\big)\big|^2\right).
\end{align*}
Noticing that 
$$
X^T \hat{W}^2 D\big(\hat{\theta}-\theta_0\big) = X^{T} \hat{W}^2X(\hat{\beta}-\beta_0)+X^{T} \hat{W}^2(\mathbf{1} \text{ } Z)(\hat{\vartheta}-\vartheta_0),
$$
 we get
\begin{align*}
\hat{b} - \beta_0 =& \hat{\beta}-\hat{S}_\beta P_n \nabla_{\beta} l_{\hat{\theta}} -\beta_0 \\
=& \hat{\beta}-\beta_0 -\hat{S}_\beta P_n \nabla_{\beta} l_{\theta_0} -\hat{S}_\beta\left(P_n \nabla_{\beta} l_{\hat{\theta}}- P_n \nabla_{\beta} l_{\theta_0}\right) \\
=&-\hat{S}_\beta P_n \nabla_{\beta} l_{\theta_0}+ \bigg(I_{n\times n}- \frac{1}{\binom{n}{2}}\hat{S}_\beta \hat{V}_{\beta}\bigg) (\hat{\beta}-\beta_0) -\frac{1}{\binom{n}{2}}\hat{S}_\beta X^{T} \hat{W}^2(\mathbf{1} \text{ } Z)(\hat{\vartheta}-\vartheta_0)\\
&+O\left(\frac{1}{\binom{n}{2}} \hat{S}_\beta\sum_{i<j} X_{i j}l^{\prime\prime\prime} (A_{ij},D_{i j}\tilde{\theta}) \left|D_{i j}^T\left(\hat{\theta}-\theta_0\right)\right|^2\right).
\end{align*}

Our goal is now to solve the following four problems: 

\begin{enumerate}
	\item 
$\sqrt{\hat{V}_{\beta,k,k}}\hat{S}_{\beta,k} P_n \nabla_{\beta} l_{\theta_0}  \overset{d}{\longrightarrow} N(0,1)$, where $\hat{S}_{\beta,k}$ is the $k$-th row of $\hat{S}_{\beta}$;
	\item 
	$\sqrt{\hat{V}_{\beta,k,k}}\big(e_{k}- \binom{n}{2}^{-1}\hat{S}_{\beta,k} \hat{V}_{\beta}\big)(\hat{\beta}-\beta_0) = o_{P}(1)$;
	\item 
	$\sqrt{\hat{V}_{\beta,k,k}}\binom{n}{2}^{-1}\hat{S}_{\beta,k} X^{T} \hat{W}^2(\mathbf{1} \text{ } Z)(\hat{\vartheta}-\vartheta_0)= o_{P}(1)$;
	\item $
O\bigg(\sqrt{\hat{V}_{\beta,k,k}}\binom{n}{2}^{-1} \hat{S}_{\beta,k}\sum_{i<j} X_{i j} l^{\prime\prime\prime} (A_{ij},D_{i j}\tilde{\theta})\big|D_{i j}^T(\hat{\theta}-\theta_0)\big|^2\bigg)= o_{P}(1)
	$.
\end{enumerate}

For Problem 1, 
$$
\sqrt{\hat{V}_{\beta,k,k}}\hat{S}_{\beta,k} P_n \nabla_{\beta} l_{\theta_0} = \frac{1}{\sqrt{\hat{V}_{\beta,k,k}}} \sum_{i\neq k} (A_{ik} -p_{ik}(\theta_0)),
$$ 
which is asymptotically normal using a standard argument of the central limit theorem as in Section \ref{C.4}.

For Problem 2,
\begin{align*}
\sqrt{\hat{V}_{\beta,k,k}}\bigg(e_{k}- \frac{1}{\binom{n}{2}}\hat{S}_{\beta,k} \hat{V}_{\beta}\bigg)(\hat{\beta}-\beta_0)  &\leq \sqrt{\hat{V}_{\beta,k,k}}\bigg\Vert e_{k}- \frac{1}{\binom{n}{2}}\hat{S}_{\beta,k} \hat{V}_{\beta}\bigg\Vert_{\infty}\Vert \hat{\beta}-\beta_0 \Vert_1 \\
&\leq  \sqrt{\hat{V}_{\beta,k,k}} \frac{1}{\hat{V}_{\beta,k,k}}\Vert \hat{\beta}-\beta_0 \Vert_1\\
&\leq  \frac{1}{\sqrt{C(n-1)\rho_n}}\frac{s_{0,+}\sqrt{\log n}}{\sqrt{n}\rho_n}\\
& = o_{P}(1),
\end{align*}
where for the last inequality, we use $\hat{V}_{\beta,k,k} = \sum_{i\neq k}p_{ik}(\hat{\theta})(1-p_{ik}(\hat{\theta})) \geq C(n-1)\rho_n$ and $\Vert \hat{\beta}-\beta_0 \Vert_1 =O_{P}\bigg(\frac{s_{0,+}\sqrt{\log n}}{\sqrt{n}\rho_n}\bigg)$ by Theorem \ref{Cor: no approximation error}.

For Problem 3,
\begin{align*}
\sqrt{\hat{V}_{\beta,k,k}}\frac{1}{\binom{n}{2}}\hat{S}_{\beta,k} X^{T} \hat{W}^2(\mathbf{1} \text{ } Z)(\hat{\vartheta}-\vartheta_0) &= \sqrt{\hat{V}_{\beta,k,k}}\frac{1}{\hat{V}_{\beta,k,k}} \sum_{i\neq k}\hat{w}_{ik}^2 (1 \text{ } Z^{T}_{ik})(\hat{\vartheta}-\vartheta_0)\\
&\leq C\sqrt{\hat{V}_{\beta,k,k}} \Vert \hat{\vartheta}-\vartheta_0 \Vert_1 \\
&\leq  C\sqrt{n}\frac{s_{0,+}\sqrt{\log n}}{n\rho_n}\\
&= o_{P}(1).
\end{align*}

For Problem 4, notice that
\begin{align*}
\vert l^{\prime\prime\prime} (A_{ik},D_{i k}\tilde{\theta})\vert = \bigg\vert\frac{\exp (D_{ik}^T\tilde{\theta})(1-\exp(D_{ik}^T\tilde{\theta}))}{(1+\exp (D_{ik}^T\tilde{\theta}))^3} \bigg\vert \leq \bigg\vert\frac{\exp (D_{ik}^T\tilde{\theta})}{(1+\exp (D_{ik}^T\tilde{\theta}))^2}\bigg\vert  \leq C\hat{V}_{\beta,i,k}
\end{align*}
for any $i, k,$ since $\tilde{\theta}$ lies between $\hat{\theta}$ and $\theta_0$ and $\hat{\theta}\overset{P}{\longrightarrow} \theta_0$. Then we have
\begin{align*}
\sqrt{\hat{V}_{\beta,k,k}}\binom{n}{2}^{-1} \hat{S}_{\beta,k}\sum_{i<j} X_{i j}l^{\prime\prime\prime} (A_{ij},D_{i j}\tilde{\theta}) \big|D_{i j}^T(\hat{\theta}-\theta_0)\big|^2 & \leq \sqrt{\hat{V}_{\beta,k,k}}\frac{1}{\hat{V}_{\beta,k,k}}\sum_{i\neq k} C\hat{V}_{\beta,i,k} \frac{s_{0,+}\log n}{n\rho^{2}_n} \\
&\leq\frac{s_{0,+}\log n}{\sqrt{n}\rho^{2}_n} \\
& =   o_{P}(1),
\end{align*}
where for the first inequality, we use Proposition \ref{Prop: q error bound}. For the second inequality, we use $\hat{V}_{\beta,k,k} = \sum_{i \neq k}\hat{V}_{\beta, i,k}$ and $\hat{V}_{\beta,k,k} \leq n$. Assumption \ref{inference_beta} implies the last equation directly.  
\end{proof}

\section{Proofs of Section \ref{subsection: generalized ER model}}\label{Sec: Proofs for ER-C}

We first prove the consistency of the MLE $\hat \theta = (\hat \mu^\dagger, \hat \gamma^T)^T$ and then its asymptotic normality.

\subsection{Consistency of \texorpdfstring{$(\hat \mu^\dagger, \hat \gamma)$}{(mu, gamma)}}

We want to find a limit for an appropriately scaled version of $\mathcal{L}^\dagger$. To that end, we first prove a concentration result of $d_+$ around its expectation. Consider
\begin{align*}
\E[d_+] &= \E[\E[d_+ | Z]] = \sum_{ i < j} \E\left[ \frac{n^{-\xi} \exp(\mu_0^\dagger)  \exp(\gamma_0^TZ_{ij})}{1 + n^{-\xi} \exp(\mu_0^\dagger)  \exp(\gamma_0^TZ_{ij}) } \right] \\
&= n^{-\xi}\exp(\mu_0^\dagger) \sum_{ i < j} \E\left[ \frac{ \exp(\gamma_0^TZ_{ij})}{1 + n^{-\xi} \exp(\mu_0^\dagger)  \exp(\gamma_0^TZ_{ij}) } \right] \\
&= 	n^{-\xi}\exp(\mu_0^\dagger) \binom{n}{2} \E\left[ \frac{ \exp(\gamma_0^TZ_{12})}{1 + n^{-\xi} \exp(\mu_0^\dagger)  \exp(\gamma_0^TZ_{12}) } \right], \quad \text{since $Z_{ij}$ are i.i.d.} \\
&=  \frac{n^{2 - \xi}}{2} \exp(\mu_0^\dagger) \E\left[ \frac{ \exp(\gamma_0^TZ_{12})}{1 + n^{-\xi} \exp(\mu_0^\dagger)  \exp(\gamma_0^TZ_{12}) } \right] + o(n^{2-\xi}).
\end{align*}
By the law of total variance, we may write the variance of $d_+$ as
\[
	\text{Var}(d_+) = \E[\text{Var}(d_+ \vert Z)] + \text{Var}(\E[d_+ \vert Z]).
\]
We have,
\begin{align*}
	\text{Var}(\E[d_+ \vert Z]) = \text{Var}\left(\sum_{i < j} p_{ij} \right) = \sum_{ i < j} n^{-2\xi} \text{Var}\left( \frac{\exp(\mu_0^\dagger + \gamma_0^TZ_{ij})}{1 + n^{-\xi} \exp(\mu_0^\dagger)  \exp(\gamma_0^TZ_{ij})}\right) = O\left( n^{2 - 2 \xi} \right).
\end{align*}
Also, by independence of the $A_{ij}$ given $Z$,
\begin{align*}
	\text{Var}(d_+ \vert Z)&= \sum_{ i < j}\text{Var}(A_{ij} \vert Z) = \sum_{ i < j} p_{ij}(1-p_{ij}) = O(n^{2-\xi}).
\end{align*}
Therefore,
\[
	\text{Var}(d_+) = O\left( n^{2 - 2 \xi} \right) + O(n^{2-\xi}) = O(n^{2-\xi}).
\]
By Chebychev's inequality, for any $t > 0$,
\[
	P(\vert d_+ - \E[d_+] \vert \ge t) \le \frac{\text{Var}(d_+)}{t^2}.
\]
Letting $\epsilon > 0$ and picking $t = n^{2-\xi}\epsilon$, we obtain
\[
	P( n^{-2+\xi} \vert d_+ - \E[d_+] \vert \ge \epsilon) \le \frac{O(n^{2-\xi})}{n^{4-2\xi}} = \frac{O(1)}{n^{2-\xi}} \rightarrow 0, \quad n \rightarrow \infty,
\]
since $\xi \in [0,2)$. This implies
\[
	d_+ = \E[d_+] + o_P(n^{2-\xi}) = \frac{n^{2 - \xi}}{2} \exp(\mu_0^\dagger) \E\left[ \frac{ \exp(\gamma_0^TZ_{12})}{1 + n^{-\xi} \exp(\mu_0^\dagger)  \exp(\gamma_0^TZ_{12}) } \right] + o_P(n^{2 - \xi} ).
\]
In particular, this implies
\begin{equation}\label{Eq: limit of d_+}
2 n^{-2+\xi} d_+ \overset{P}{\rightarrow} \exp(\mu_0^\dagger) \E\left[ \exp(\gamma_0^TZ_{12})\right], \quad n \rightarrow \infty.
\end{equation}

Next, we deal with the second term in $\mathcal{L}^\dagger$: 
\begin{align*}
\E\left[ \sum_{i < j} (\gamma^TZ_{ij})A_{ij}\right] &= \sum_{i < j} \E\left[ (\gamma^TZ_{ij})\E [A_{ij} | Z_{ij}]\right] = \sum_{i < j} \E\left[ (\gamma^TZ_{ij}) p_{ij}\right] \\
&= \sum_{ i < j} n^{-\xi} \exp(\mu_0^\dagger)  \E\left[(\gamma^TZ_{ij}) \frac{ \exp(\gamma_0^TZ_{ij})}{1 + n^{-\xi} \exp(\mu_0^\dagger)  \exp(\gamma_0^TZ_{ij}) } \right] \\
&= n^{-\xi} \exp(\mu_0^\dagger) \binom{n}{2} \E\left[(\gamma^TZ_{12}) \frac{ \exp(\gamma_0^TZ_{12})}{1 + n^{-\xi} \exp(\mu_0^\dagger)  \exp(\gamma_0^TZ_{12}) } \right], \quad \text{$Z_{ij}$ \text{are i.i.d.}} \\
&\eqqcolon n^{-\xi} \exp(\mu_0^\dagger) \binom{n}{2} \bar{\alpha}_n,
\end{align*}
where we suppress the dependence of $\bar \alpha_n$ on $\gamma$ in our notation. Pay special attention to the distinction between the generic $\gamma$ and the true parameter $\gamma_0$ here. The last equality in the previous display can be written as
\[
\E\left[ \sum_{i < j} (\gamma^TZ_{ij})A_{ij}\right] = \frac{n^{2 -\xi}}{2} \exp(\mu_0^\dagger) \bar{\alpha}_n + o(n^{2-\xi}).
\]
We use the law of total variance once more to bound $\text{Var}(\sum_{i < j} (\gamma^TZ_{ij})A_{ij})$. For any $i,j$,
\[
	\text{Var}((\gamma^TZ_{ij})A_{ij}) = \E[\text{Var}((\gamma^TZ_{ij})A_{ij} \vert Z) ] + \text{Var}(\E[(\gamma^TZ_{ij})A_{ij}\vert Z]).
\]
We have,
\[
	\text{Var}(\E[(\gamma^TZ_{ij})A_{ij}\vert Z]) = \text{Var}((\gamma^TZ_{ij}) p_{ij}  ) \le \E[\left( (\gamma^TZ_{ij}) p_{ij} \right)^2] \le C n^{-2\xi}
\]
and
\[
	\text{Var}((\gamma^TZ_{ij})A_{ij} \vert Z) = (\gamma^TZ_{ij})^2 p_{ij} (1 - p_{ij}) \le C n^{-\xi},
\]
where in both instances we may choose some constant $C > 0$ independent of $i,j$ and $n$.
Thus,
\[
	\text{Var}\left(\sum_{i < j} (\gamma^TZ_{ij})A_{ij}\right) \le \sum_{i < j} C (n^{-2\xi} + n^{-\xi}) = O(n^{2-\xi}).
\]
Using Chebyshev's inequality, we obtain for any $t > 0$,
\[
	P\left( \left\vert \sum_{i < j} (\gamma^TZ_{ij})A_{ij} -  \E\left[ \sum_{i < j} (\gamma^TZ_{ij})A_{ij}\right] \right\vert \ge t  \right) \le \frac{\text{Var}\left(\sum_{i < j} (\gamma^TZ_{ij})A_{ij}\right)}{t^2}.
\]
Letting $\epsilon > 0$ and picking $t = n^{2-\xi}\epsilon$, we obtain
\[
P\left( n^{-2 + \xi} \left\vert \sum_{i < j} (\gamma^TZ_{ij})A_{ij} -  \E\left[ \sum_{i < j} (\gamma^TZ_{ij})A_{ij}\right] \right\vert \ge \epsilon  \right) \le \frac{O(n^{2-\xi})}{n^{2-\xi} \cdot n^{2-\xi}} \rightarrow 0.
\]
This implies
\[
\sum_{i < j} (\gamma^TZ_{ij})A_{ij} = \E\left[ \sum_{i < j} (\gamma^TZ_{ij})A_{ij}\right] + o_P(n^{2-\xi}) = \frac{n^{2 -\xi}}{2} \exp(\mu_0^\dagger) \bar{\alpha}_n + o_P(n^{2-\xi}).
\]
Since $\bar{\alpha}_n \rightarrow \E[(\gamma^TZ_{12}) \exp(\gamma_0^TZ_{12})]$ almost surely, we end up with
\begin{equation}\label{Eq: limit gamma term}
2n^{-2+\xi} \sum_{i < j} (\gamma^TZ_{ij})A_{ij} \overset{P}{\rightarrow} \exp(\mu_0^\dagger) \E[(\gamma^TZ_{12}) \exp(\gamma_0^TZ_{12})], \quad n \rightarrow \infty.
\end{equation}

It remains to analyze the last term in $\mathcal{L}^\dagger$, i.e.~term $\sum_{i < j} \log\left( 1 + n^{-\xi} \exp(\mu^\dagger + \gamma^TZ_{ij}) \right)$. Since $\log(1 + x) \le x$ for $x > -1$:
\begin{align*}
\sum_{i < j} \log\left( 1 + n^{-\xi} \exp(\mu^\dagger + \gamma^TZ_{ij}) \right) &\le n^{-\xi} \exp(\mu^\dagger) \sum_{i<j} \exp(Z_{ij}^T\gamma) \\
&= n^{-\xi} \exp(\mu^\dagger)\binom{n}{2} \underbrace{ \frac{1}{\binom{n}{2}} \sum_{i<j} \exp(Z_{ij}^T\gamma)}_{\eqqcolon \alpha_n} \\
&= \frac{n^{2-\xi}}{2} \exp(\mu^\dagger) \alpha_n + o(n^{2-\xi}).
\end{align*}
On the other hand, we also have $x/(1+x) \le \log(1+x)$ for all $x > -1$. Also recall that $\vert \gamma^TZ_{ij} \vert \le \kappa$ almost surely. Thus,
\begin{align*}
\sum_{i < j} \log\left( 1 + n^{-\xi} \exp(\mu^\dagger + \gamma^TZ_{ij}) \right) &\ge n^{-\xi}\exp(\mu^\dagger) \sum_{ i < j}  \frac{ \exp(\gamma^TZ_{ij})}{1 + n^{-\xi} \exp(\mu^\dagger)  \exp(\gamma^TZ_{ij}) } \\
&\ge n^{-\xi}\exp(\mu^\dagger) \frac{1}{1 + n^{-\xi} \exp(\mu^\dagger + \kappa)} \sum_{ i < j} \exp(\gamma^TZ_{ij}) \\
&= n^{-\xi}\exp(\mu^\dagger) \frac{1}{1 + n^{-\xi} \exp(\mu^\dagger + \kappa)} \binom{n}{2}\alpha_n \\
&= \frac{n^{2 -\xi}}{2}\exp(\mu^\dagger) \frac{1}{1 + n^{-\xi} \exp(\mu^\dagger + \kappa)}\alpha_n + o(n^{2-\xi}).
\end{align*}
Notice that since the $Z_{ij}$ are i.i.d. and since $\gamma^TZ_{ij}$ is uniformly bounded,
\[
\alpha_n \overset{a.s.}{\rightarrow} \E[\exp(\gamma^TZ_{12})].
\]
We now have found an upper and a lower bound on $\sum_{i < j} \log\left( 1 + n^{-\xi} \exp(\mu^\dagger + \gamma^TZ_{ij}) \right)$. Multiplying both sides with $2 n^{-2+\xi}$ and taking the limit $n \rightarrow \infty$, we see that both the lower as well as the upper bound converge to $\exp(\mu^\dagger) \E[\exp(\gamma^TZ_{12})]$. But then this already must be the limit for $2 n^{-2+\xi} \sum_{i < j} \log\left( 1 + n^{-\xi} \exp(\mu^\dagger + \gamma^TZ_{ij}) \right)$:
\begin{equation}\label{Eq: limit of log term}
2 n^{-2+\xi} \sum_{i < j} \log\left( 1 + n^{-\xi} \exp(\mu^\dagger + \gamma^TZ_{ij}) \right) \overset{P}{\rightarrow} \exp(\mu^\dagger) \E[\exp(\gamma^TZ_{12})], \quad n \rightarrow \infty.
\end{equation}
Putting equations \eqref{Eq: limit of d_+}, \eqref{Eq: limit gamma term} and \eqref{Eq: limit of log term} together, we obtain that for any $(\mu^\dagger, \gamma) \in [-M,M] \times \Gamma$:
\begin{align}\label{Eq: limit of neg llhd}
	\begin{split}
		2&n^{-2+\xi}\mathcal{L}^\dagger(\mu^\dagger, \gamma) \\
		&\overset{P}{\rightarrow} - \mu^\dagger \exp(\mu_0^\dagger) \E[\exp(\gamma_0^TZ_{12})] - \exp(\mu_0^\dagger)\E[\gamma^TZ_{12} \exp(\gamma_0^TZ_{12})] + \exp(\mu^\dagger) \E[\exp(\gamma^TZ_{12})],
	\end{split}
\end{align}
as $n\rightarrow \infty$. We thus define this limiting function as $M: \R^{p+1} \rightarrow \R$,
\[
M(\mu^\dagger, \gamma) \coloneqq - \mu^\dagger \exp(\mu_0^\dagger) \E[\exp(\gamma_0^TZ_{12})] - \exp(\mu_0^\dagger)\E[\gamma^TZ_{12}\cdot \exp(\gamma_0^TZ_{12})] + \exp(\mu^\dagger) \E[\exp(\gamma^TZ_{12})].
\]
We want to employ Theorem 5.7 in \cite{Vaart:1998}. To that end, we must show that this convergence is uniform in probability, that is, we must show that
\begin{equation}\label{Eq: uniform consistency}
\sup_{\theta} \vert 2n^{-2+\xi}\mathcal{L}^\dagger(\theta) - M(\theta) \vert = o_P(1),
\end{equation}
with the supremum taken over all $\theta \in [-M,M] \times \Gamma$.

To shorten notation, introduce $M_n(\theta) \coloneqq 2n^{-2+\xi}\mathcal{L}^\dagger(\theta)$.
Since we already have pointwise convergence in probability of $M_n$ to $M$, it will be suffice to show that for any $\epsilon > 0$
\begin{equation}\label{Eq: lim delta limsup P(Delta)}
\lim_{\delta \downarrow 0} \limsup_{n\rightarrow \infty} P \left( \sup_{\Vert \theta_1 - \theta_2 \Vert_2 \le \delta} \vert M_n(\theta_1) - M_n(\theta_2) \vert \ge \epsilon \right) = 0.
\end{equation}
Property \eqref{Eq: uniform consistency} then follows from the pointwise convergence, the continuity of $M$ and the compactness of the parameter space $[-M,M] \times \Gamma$.
To ease notation further, define
\[
\Delta_\delta^n \coloneqq \sup_{\Vert \theta_1 - \theta_2 \Vert_2 \le \delta} \vert M_n(\theta_1) - M_n(\theta_2) \vert.
\]
Let $\epsilon, \eta > 0$. We have to show that there exists a $\delta > 0$ such that
\begin{equation}\label{Eq: limsup of P(Delta)}
\limsup_{n\rightarrow \infty}P(\Delta_\delta^{n} \ge \epsilon) \le \eta.
\end{equation}
Consider the following representation of $\mathcal{L}^\dagger(\theta)$:
\begin{align*}
\mathcal{L}^\dagger (\theta) &= -d_+\mu^\dagger - \sum_{i < j} (\gamma^TZ_{ij})A_{ij} + \sum_{i < j} \log\left( 1 + n^{-\xi} \exp(\mu^\dagger + \gamma^TZ_{ij}) \right) \\
&= \sum_{ i < j} -(\mu^\dagger + \gamma^TZ_{ij}) A_{ij} + \log\left( 1 + n^{-\xi} \exp(\mu^\dagger + \gamma^TZ_{ij}) \right) \\
&= \sum_{ i < j} \underbrace{-D_{ij}^T\theta A_{ij} + \log\left( 1 + n^{-\xi} \exp(D_{ij}^T\theta) \right) }_{\eqqcolon l_{ij}(\theta)}.
\end{align*}
Now, for any $\delta > 0$ and any $\theta_1, \theta_2$ with $\Vert \theta_1 - \theta_2 \Vert_2 < \delta$ and any $i < j$, we obtain:
\begin{align*}
\E \vert l_{ij}(\theta_1) - l_{ij}(\theta_2) \vert &= \E \left\vert -D_{ij}^T(\theta_1 - \theta_2) A_{ij} + \log\left( 1 + n^{-\xi} \exp(D_{ij}^T\theta_1) \right) - \log\left( 1 + n^{-\xi} \exp(D_{ij}^T\theta_2) \right) \right\vert. \\
\shortintertext{Hence, by the Mean Value Theorem with $\alpha$ between $D_{ij}\theta_1$ and $D_{ij}\theta_2$:}
\E \vert l_{ij}(\theta_1) - l_{ij}(\theta_2) \vert &\le \E  \left[ \vert D_{ij}^T(\theta_1 - \theta_2) \vert A_{ij} \right] + \frac{n^{-\xi} \exp(\alpha)}{1 + n^{-\xi}\exp(\alpha)} \E \left\vert D_{ij}^T(\theta_1 - \theta_2) \right\vert \\
&\le C \Vert \theta_1 - \theta_2 \Vert_2 \E[p_{ij}]  + C n^{-\xi} \Vert \theta_1 - \theta_2 \Vert_2\\
&\le C \Vert \theta_1 - \theta_2 \Vert_2 \left(\E\left[ n^{-\xi} \frac{\exp(D_{ij}^T\theta_0)}{1 + n^{-\xi}\exp(D_{ij}^T\theta_0)}\right] + n ^{-\xi} \right) \\
&\le C n^{-\xi} \Vert \theta_1 - \theta_2 \Vert_2 \\
&\le C n^{-\xi}\delta,
\end{align*}
where $C > 0$ denotes some generic constant that may change between displays. By the compactness of our parameter space and the resulting uniform boundedness of $\vert D_{ij}(\theta_1 - \theta_2)\vert $, we may in particular choose this $C$ independent of $n, i$ and $j$. Then, almost surely,
\[
\E \vert \mathcal{L}^\dagger(\theta_1) - \mathcal{L}^\dagger(\theta_2) \vert \le C \binom{n}{2} n^{-\xi} \delta 
\]
and thus, almost surely,
\[
\E \Delta_\delta^n \le C n^{-2+ \xi} n^{-\xi} \binom{n}{2} \delta \le C \delta.
\]
Thus, we can choose a $\delta > 0$ independent of $n$, such that $\E \Delta_\delta^n \le \epsilon \eta$. But then an application of Markov's inequality yields for all $n$ large enough
\[
P(\Delta_\delta^n \ge \epsilon) \le \eta.
\]
It follows \eqref{Eq: limsup of P(Delta)}, which implies \eqref{Eq: lim delta limsup P(Delta)}, which yields \eqref{Eq: uniform consistency}.

The second condition of Theorem 5.7 in \cite{Vaart:1998} requires that the true parameter be a well-separated extrema of $M$. That is, we must show: For any fixed $\epsilon > 0$,
\begin{equation}\label{Eq: identifiability}
\sup_{\theta: d(\theta, \theta_0) \ge \epsilon} M(\theta) > M(\theta_0).
\end{equation}
Consider the first partial derivatives of $M$:
\begin{align*}
\partial_{\mu^\dagger} M(\mu^\dagger, \gamma) &= - \exp(\mu_0^\dagger) \E[\exp(\gamma_0^TZ_{12})] + \exp(\mu^\dagger) \E[\exp(\gamma^TZ_{12})], \\
\partial_{\gamma_k} M(\mu^\dagger, \gamma) &= -\exp(\mu_0^\dagger) \E[Z_{12,k} \exp(\gamma_0^TZ_{12})] + \exp(\mu^\dagger) \E[Z_{12,k}\exp(\gamma^TZ_{12})].
\end{align*}
Clearly, by Assumption \ref{Assum: MLE in interior} the true parameter is a critical point of $M$, i.e. the first partial derivatives of $M$ evaluated at $\theta_0 = (\mu_0^\dagger, \gamma_0^T)^T$ are zero:
\[
\nabla M (\theta_0) = 0.
\]
Consider the Hessian $HM(\mu^\dagger, \gamma)$ of $M$ at the point $(\mu^\dagger, \gamma)$:
\begin{align*}
\frac{\partial^2}{\partial (\mu^\dagger)^2} M(\mu^\dagger, \gamma) &= \exp(\mu^\dagger) \E[\exp(\gamma^TZ_{12})], \\
\frac{\partial^2}{\partial \mu^\dagger \gamma_k} M(\mu^\dagger, \gamma) &= \exp(\mu^\dagger) \E[Z_{12,k}\exp(\gamma^TZ_{12})], \\
\frac{\partial^2}{\partial \gamma_k^2} M(\mu^\dagger, \gamma) &= \exp(\mu^\dagger) \E[Z_{12,k}^2\exp(\gamma^TZ_{12})], \\
\frac{\partial^2}{\partial \gamma_k \gamma_l} M(\mu^\dagger, \gamma) &= \exp(\mu^\dagger) \E[Z_{12,k}Z_{12,l}\exp(\gamma^TZ_{12})].
\end{align*}
We thus see that $HM(\mu^\dagger, \gamma)$ allows a matrix representation as
\[
HM(\mu^\dagger, \gamma) = \exp(\mu^\dagger) \E \left[ \exp(\gamma^TZ_{12}) \begin{bmatrix}
1 & Z_{12}^T \\
Z_{12} & Z_{12}Z_{12}^T
\end{bmatrix} \right] \in \R^{(p+1) \times (p+1)}.
\]
By the compactness of our parameter space and the boundedness of $Z_{12}$, we now obtain for any $v \in  \R^{p+1}$:
\begin{align*}
v^THM(\mu^\dagger, \gamma) v &= \exp(\mu^\dagger) \E \left[ \exp(\gamma^TZ_{12}) v^TD_{12}D_{12}^Tv  \right] 
\ge C \E\left[  v^T D_{12}D_{12}^T v  \right] \\
&= C v^T \E\left[  \begin{bmatrix}
1 & \textbf{0} \\
\textbf{0} & Z_{12}Z_{12}^T
\end{bmatrix}  \right] v 
\ge C \Vert v \Vert_2^2,
\end{align*}
where for the last inequality we have used that the matrix is strictly positive definite by Assumption \ref{Assum: min eval}. That means, $HM(\mu^\dagger, \gamma)$ is strictly positive definite on the entire parameter space $[-M,M] \times \Gamma$. Hence, $M$ is strictly convex and its minimum $\theta_0$ already must be a global minimum. Now, since our parameter space is compact, $M$ is continuous and $\theta_0$ is a global maximum, it is easy to see that \eqref{Eq: identifiability} must hold.

Finally, since \eqref{Eq: uniform consistency} and \eqref{Eq: identifiability} hold, we have consistency as
$
\hat \theta \overset{P}{\rightarrow} \theta_0
$ \citep[Theoem 5.7]{Vaart:1998}.

\subsection{Asymptotic normality}

The proof of asymptotic normality in spirit follows to some extent the proof of Theorem \ref{Thm: inference}.
By Assumption \ref{Assum: MLE in interior}, the MLE $\hat \theta$ fulfills the first order estimating equations:
\begin{equation*}
0 = \nabla \mathcal{L}^\dagger(\hat \theta),
\end{equation*}
which, when looking at the individual components, means that
\begin{align*}
0 &= \partial_{\mu^\dagger}\mathcal{L}^\dagger(\hat \theta) = - d_+ + n^{-\xi} \exp(\hat \mu^\dagger) \sum_{i<j}  \frac{\exp(\hat \gamma^TZ_{ij})}{1 + n^{-\xi}\exp(\hat \mu^\dagger + \hat\gamma^TZ_{ij})}, \\
0 &= \partial_{\gamma_k}\mathcal{L}^\dagger(\hat \theta) = \sum_{i < j} Z_{ij,k}A_{ij} + n^{-\xi} \exp(\hat \mu^\dagger) \sum_{i<j}  \frac{Z_{ij,k}\exp(\hat \gamma^TZ_{ij})}{1 + n^{-\xi}\exp(\hat \mu^\dagger + \hat\gamma^TZ_{ij})}, \quad k = 1, \dots, p.
\end{align*}

We want to make use of a Taylor expansion. Define the functions $l_n(y,a): \{0,1\} \times \R \rightarrow \R$,
\[
	l_n(y, a) = -ya + \log(1 + n^{-\xi} \exp(a)).
\]
In particular,
\[
\mathcal{L}^\dagger(\mu^\dagger, \gamma) = \sum_{ i < j} l_n(A_{ij}, (\mu^\dagger, \gamma^T)^TD_{ij}).
\]
The $l_n$ have the following derivatives:
\begin{align*}
\dot l_n(y,a) \coloneqq \partial_a l_n(y,a) = -y + n^{-\xi} \frac{\exp(a)}{1 + n^{-\xi}\exp(a)}, \\
\ddot l_n(y,a) \coloneqq \partial^2_{a^2} l_n(y,a) = n^{-\xi} \frac{\exp(a)}{(1 + n^{-\xi}\exp(a))^2}, \\
\partial^3_{a^3} l_n(y,a) = n^{-\xi} \frac{\exp(a)}{(1 + n^{-\xi}\exp(a))^2} \cdot \frac{1 - n^{-\xi}\exp(a)}{1 + n^{-\xi}\exp(a)}.
\end{align*}
Note that $\vert \partial^3_{a^3} l_n(y,a) \vert \le C n^{-\xi}$ and hence $\ddot l_n(y,a)$ is Lipschitz continuous in $a$ with constant $Cn^{-\xi}$ by the Mean-Value Theorem. Doing a first order Taylor expansion in $a$ of $\dot l_n (y,a) = \partial_al(y,a)$ in the point $a_0 = (A_{ij}, D_{ij}^T\theta_0)$ evaluated at $a = (A_{ij}, D_{ij}^T\hat \theta)$, we get
\begin{equation}\label{Eq: Taylor of l ER-C}
\partial_a l(A_{ij}, D_{ij}\hat \theta) = \partial_a l(A_{ij}, D_{ij}^T\theta_0) + \partial_{a^2} l(A_{ij}, \alpha) D_{ij}^T(\hat \theta - \theta_0),
\end{equation}
for an $\alpha$ between $D_{ij}^T\hat \theta$ and $D_{ij}^T \theta_0$.

Consider the vector $1/\binom{n}{2}\nabla \mathcal{L}^\dagger(\hat \theta)$: By equation (\ref{Eq: Taylor of l ER-C}), with $\alpha_{ij}$ between $D_{ij}^T\hat \theta$ and $D_{ij}^T \theta_0$,
\begin{align*}
0 = \frac{1}{\binom{n}{2}} \nabla \mathcal{L}^\dagger(\hat \theta)&= \frac{1}{\binom{n}{2}} \sum_{ i < j} \left( \partial_{\theta_k} l (A_{ij}, D_{ij}^T \hat \theta) \right)_{k=1,\dots, p+1}, \quad \text{ as a }(p+1)\times 1\text{-vector} \\
&= \frac{1}{\binom{n}{2}} \sum_{ i < j} \dot l (A_{ij}, D_{ij}^T \hat \theta) D_{ij}, \quad \text{by the chain rule} \\
&= \frac{1}{\binom{n}{2}} \sum_{ i < j} (\dot l (A_{ij}, D_{ij}^T \theta_0) + \ddot l(A_{ij}, \alpha_{ij}) D_{ij}^T (\hat \theta - \theta_0)) D_{ij}, \quad \text{by \eqref{Eq: Taylor of l ER-C}} \\
&= \frac{1}{\binom{n}{2}}  \nabla \mathcal{L}^\dagger(\theta_0) + \frac{1}{\binom{n}{2}} \sum_{ i < j} \ddot l(A_{ij}, \alpha_{ij}) D_{ij}D_{ij}^T (\hat \theta - \theta_0).
\end{align*}
Proving Theorem \ref{Thm: asymptotic normality ER-C} now breaks down into three problems.

\subsubsection{Problem 1}
First, we show that under appropriate scaling $\frac{1}{\binom{n}{2}}  \nabla \mathcal{L}^\dagger(\theta_0)$ is asymptotically normal. We may write the components of $\nabla \mathcal{L}^\dagger(\theta_0)$ more compactly as
\[
\nabla \mathcal{L}^\dagger(\theta_0)_k = \sum_{i < j} D_{ij,k} (p_{ij} - A_{ij}),
\]
where $D_{ij,k}$ is the $k$th component of the $(i,j)$-th row of $D$, i.e. $D_{ij,k} = 1$, if $k = 1$ and $D_{ij,k} = Z_{ij,k-1}$, if $k=2, \dots, p+1$ and 
\[
p_{ij} = \E[A_{ij} | Z_{ij}] = n^{-\xi} \cdot \frac{\exp(\mu_0^\dagger +  \gamma_0^TZ_{ij})}{1 + n^{-\xi}\exp(\mu_0^\dagger + \gamma_0^TZ_{ij})}.
\]
Notice that all components of $\nabla \mathcal{L}^\dagger(\theta_0)$ are centered Indeed,
\[
\E[\nabla \mathcal{L}^\dagger(\theta_0)_k] = \sum_{i < j} \E [D_{ij,k} (p_{ij} - A_{ij})] = \sum_{i < j} \E [D_{ij,k} \E[(p_{ij} - A_{ij}) | Z_{ij}]] = \sum_{i < j} \E [D_{ij,k} \cdot 0] = 0.
\]
We want to apply the Lindeberg-Feller Central Limit Theorem to the term
\[
\sqrt{\binom{n}{2}} n^{\xi/2} \cdot \frac{1}{\binom{n}{2}} \nabla \mathcal{L}^\dagger(\theta_0) = \sum_{i < j} D_{ij} (p_{ij} - A_{ij}) \cdot \sqrt{\frac{n^\xi}{\binom{n}{2}}}.
\]
To that end, define the triangular array $Y_{n,ij} = D_{ij} (p_{ij} - A_{ij}) \cdot \sqrt{\frac{n^\xi}{\binom{n}{2}}}, 1 \le i < j \le n, n \in \mathbb{N}$. Since the $Y_{n,ij}$ are centered, their covariance matrix is given by
\begin{align*}
\text{Cov}(Y_{n,ij}) &= \E[Y_{n,ij}Y_{n,ij}^T] = \E\left[ D_{ij}D_{ij}^T (p_{ij} - A_{ij})^2 \cdot \frac{n^\xi}{\binom{n}{2}}  \right] = \E\left[ D_{ij}D_{ij}^T p_{ij} (1 - p_{ij})\cdot \frac{n^\xi}{\binom{n}{2}}  \right],
\end{align*}
where for the last equality we have used that $\E[(p_{ij} - A_{ij})^2 | Z_{ij}] = p_{ij}(1-p_{ij})$. In analogy to the case with non-zero $\beta$, we write $W_0^2 = \text{diag}(p_{ij}(1-p_{ij}), i < j) \in \R^{\binom{n}{2} \times \binom{n}{2}}$. Then, we get for the sum of covariance matrices
\[
\sum_{ i < j} \text{Cov}(Y_{n,ij}) = \sum_{ i < j} \E\left[ D_{ij}D_{ij}^T p_{ij} (1 - p_{ij})\cdot \frac{n^\xi}{\binom{n}{2}}  \right] =  \frac{n^\xi}{\binom{n}{2}} \E[D^TW_0^2D] \eqqcolon \Sigma^{(n)}.
\]
For any pair $i < j$, we have $p_{ij}(1-p_{ij}) = n^{-\xi} \exp(\mu_0^\dagger) \frac{\exp(\gamma_0^TZ_{ij})}{(1 + n^{-\xi} \exp(\mu_0^\dagger + \gamma_0^TZ_{ij}))^2}$. Hence, $n^{\xi}p_{ij}(1-p_{ij}) \rightarrow \exp(\mu_0^\dagger + \gamma_0^TZ_{ij})$ as $n \rightarrow \infty$. Consider the $(k,l)$-th entry of $\Sigma^{(n)}$:
\begin{align*}
\Sigma^{(n)}_{k,l} &= \frac{1}{\binom{n}{2}}\sum_{ i < j} \E\left[(D_{ij}D_{ij}^T)_{k,l} \exp(\mu_0^\dagger) \frac{\exp(\gamma_0^TZ_{ij})}{(1 + n^{-\xi} \exp(\mu_0^\dagger + \gamma_0^TZ_{ij}))^2}\right] \\
&= \E\left[  (D_{12}D_{12}^T)_{k,l} \exp(\mu_0^\dagger) \frac{\exp(\gamma_0^TZ_{12})}{(1 + n^{-\xi} \exp(\mu_0^\dagger + \gamma_0^TZ_{12}))^2} \right], \quad Z_{ij} \text{ i.i.d.} \\
&\overset{n \rightarrow \infty}{\longrightarrow} \E\left[  (D_{12}D_{12}^T)_{k,l} \exp(\mu_0^\dagger) \exp(\gamma_0^TZ_{12}) \right] \eqqcolon \Sigma_{kl},
\end{align*}
by dominated convergence.
Hence, with $\Sigma = (\Sigma_{kl})_{k,l} \in \R^{(p+1) \times p+1}$, as $n \rightarrow \infty$,
\[
\sum_{ i < j} \text{Cov}(Y_{n,ij}) \rightarrow \Sigma,
\]
where convergence is to be understood componentwise.
We claim that $\Sigma$ is strictly positive definite. Indeed, since $\mu_0^\dagger + \gamma_0^TZ_{12}$ lies in some compact set there is a constant $C > 0$ such that $\exp(\mu_0^\dagger) \exp(\gamma_0^TZ_{12}) > C > 0$ almost surely. Then, for any vector $v = (v_1, v_R^T)^T \in \R^{p+1}, v_1 \in \R$,
\[
v^T\Sigma v = \E[(D_{12}^Tv)^2 \exp(\mu_0^\dagger) \exp(\gamma_0^TZ_{12})] > C v^T \E[D_{12}D_{12}^T] v.
\]
Yet, by Assumption \ref{Assum: min eval},
\begin{align*}
	v^T \E[D_{12}D_{12}^T] v &= v^T \E \begin{bmatrix}
		1 & Z_{12}^T \\
		Z_{12} & Z_{12}Z_{12}^T
	\end{bmatrix} v = 
	v^T \begin{bmatrix}
		1 & \textbf{0}^T \\
		\textbf{0} & \E[Z_{12}Z_{12}^T]
	\end{bmatrix} v \\
	&= v_1^2 + v_R^T  \E[Z_{12}Z_{12}^T] v_R \ge (1 \wedge \lambda_{\min}) \Vert v \Vert_2^2.
\end{align*}
Thus, for any $v \neq 0$,
\[
v^T\Sigma v \ge C \Vert v \Vert_2^2 > 0
\]
and therefore $\Sigma$ is positive definite.

Furthermore, we clearly have $\E[\Vert Y_{n,ij} \Vert^2] < C < \infty$ for any $i,j,n$. Finally, let $\epsilon > 0$. Since $\Vert D_{ij}(p_{ij} - A_{ij}) \Vert_2$ is uniformly bounded for all $i < j$, we we may find an $n_0 \in \mathbb{N}$ such that for all $n > n_0$ we have $\Vert  Y_{n,ij} \Vert < \epsilon$ for all $i < j$. This gives us that, as $n \rightarrow \infty$,
\begin{align*}
\sum_{ i < j} \E[ \Vert Y_{n,ij} \Vert^2 \mathbbm{1}( \Vert Y_{n,ij} \Vert > \epsilon) ] \rightarrow 0.
\end{align*}
Then, by the vector-valued Lindeberg-Feller Central Limit Theorem, we obtain
\begin{equation}\label{Eq: CLT gradient of L}
\sqrt{\binom{n}{2}} n^{\xi/2} \cdot \frac{1}{\binom{n}{2}} \nabla \mathcal{L}^\dagger(\theta_0) = \sum_{ i < j} Y_{n,ij} \overset{d}{\longrightarrow} \mathcal{N}(0, \Sigma).
\end{equation}

\subsubsection{Problem 2}
Next, we must find a bound on the speed of convergence of $\hat \theta - \theta_0$. Recall that we obtained the equality
\begin{equation}\label{Eq: Taylor}
0 = \frac{1}{\binom{n}{2}}  \nabla \mathcal{L}^\dagger(\theta_0) + \frac{1}{\binom{n}{2}} \sum_{ i < j} \ddot l(A_{ij}, \alpha_{ij}) D_{ij}D_{ij}^T (\hat \theta - \theta_0). 
\end{equation}
Consider the matrix
\[
\Sigma_\alpha \coloneqq \frac{1}{\binom{n}{2}} \sum_{ i < j} \ddot l(A_{ij}, \alpha_{ij}) D_{ij}D_{ij}^T = \frac{1}{\binom{n}{2}} D^T\text{diag}(\ddot l(A_{ij}, \alpha_{ij}), i < j) D.
\]
Since $\alpha_{ij}$ lies between $D_{ij}^T\hat \theta$ and $D_{ij}^T \theta_0$ and both of these points lie in some compact set, we have for some universal constant $C > 0$, independent of $i,j$,
\[
\ddot l(A_{ij}, \alpha_{ij}) \ge C n^{-\xi}.
\]
Thus, for any $v \in \R^{p+1}$,
\[
v^T \Sigma_\alpha v \ge C n^{-\xi} v^T \left( \frac{1}{\binom{n}{2}} D^TD \right) v.
\]
Completely analogously to the case with non-zero $\beta$, we can show that $\frac{1}{\binom{n}{2}} D^TD$ is positive definite with high probability by using Lemma 6 in \cite{kock_tang_2019} (cf. section \ref{Sec: inverting Gram matrices}). Therefore, with high probability, $\lambda_{\min}(\Sigma_\alpha) \ge Cn^{-\xi} > 0$. Thus,
\[
\lambda_{\max}(\Sigma_\alpha^{-1}) = \frac{1}{\lambda_{\min}(\Sigma_\alpha)} \le C n^{\xi}.
\]
From \eqref{Eq: Taylor} we now obtain
\begin{align*}
\Sigma_\alpha (\hat{\theta} - \theta_0) &= - \frac{1}{\binom{n}{2}}  \nabla \mathcal{L}^\dagger(\theta_0) \\
\shortintertext{which is equivalent to}
\hat{\theta} - \theta_0 &= - \Sigma_\alpha^{-1} \frac{1}{\binom{n}{2}}  \nabla \mathcal{L}^\dagger(\theta_0) \\
\shortintertext{which after rescaling gives}
\sqrt{\frac{\binom{n}{2}}{n^{\xi}}} (\hat{\theta} - \theta_0) &= - \sqrt{\frac{\binom{n}{2}}{n^{\xi}}} \Sigma_\alpha^{-1} \frac{1}{\binom{n}{2}}  \nabla \mathcal{L}^\dagger(\theta_0) 
= - n^{-\xi} \Sigma_\alpha^{-1} \cdot \sqrt{\binom{n}{2}} n^{\xi/2} \frac{1}{\binom{n}{2}}  \nabla \mathcal{L}^\dagger(\theta_0).
\end{align*}

From the previous section we know $\sqrt{\binom{n}{2}} n^{\xi/2} \frac{1}{\binom{n}{2}}  \nabla \mathcal{L}^\dagger(\theta_0) \overset{d}{\rightarrow} \mathcal{N}(0,\Sigma)$. Also, the maximum eigenvalue of $n^{-\xi} \Sigma_\alpha^{-1}$ is uniformly bounded by some universal constant $C < \infty$, making the right-hand side above $O_P(1)$. This means
\[
	\hat{\theta} - \theta_0 = O_P\left( \sqrt{\frac{n^{\xi}}{\binom{n}{2}}}  \right).
\]

\subsubsection{Problem 3}
Finally, we derive the desired central limit theorem for our estimator.
We claim that $n^{\xi}\Sigma_\alpha = \Sigma + o_P(1)$. To prove this, first consider the functions
\[
f_n(x) = \frac{\exp(x)}{(1 + n^{-\xi} \exp(x))^2}.
\]
For every $x$, we have pointwise convergence $f_n(x) \rightarrow f(x) \coloneqq \exp(x)$ as $n \rightarrow \infty$. Since $\hat \theta$ and $\theta_0$ lie in some compact set and since $Z_{ij}$ is uniformly bounded, the values $\alpha_{ij}$ in \eqref{Eq: Taylor} and $\mu_0^\dagger + \gamma_0^TZ_{ij}, i < j$ all lie in some compact interval $I \subset \R$ independent of $i, j$ and $n$. Also notice that $f_n(x) \le f_{n+1}(x)$ for all $n \in \mathbb{N}$ and $x \in I$. Recall that by Dini's theorem a sequence of monotonically increasing, continuous, real-valued functions that converges pointwise to some continuous limit function on a compact topological space, must already converge uniformly. Hence, $f_n$ converges uniformly to $f$ on $I$:
$
\lim_{n\rightarrow \infty} \sup_{x \in I} \vert f_n(x) - f(x) \vert = 0.
$
Furthermore, since $I$ is compact and hence bounded, $f$ has bounded derivative on $I$ and thus is Lipschitz continuous on $I$ with some finite constant $C$ by the Mean-Value Theorem:
\[
\vert f(x) - f(y) \vert \le C \vert x - y \vert, \quad \text{for all } x,y \in I.
\]
Now consider the $(k,l)$-th entry of $n^{\xi}\Sigma_\alpha - \Sigma$:
\begin{align*}
\vert (n^{\xi}\Sigma_\alpha - \Sigma)_{kl} \vert &= \left\vert \frac{1}{\binom{n}{2}} \sum_{ i < j} D_{ij,k}D_{ij,l} \frac{\exp(\alpha_{ij}) }{(1 + n^{-\xi}\exp(\alpha_{ij}))^2 }  - \E[D_{12,k}D_{12,l} \exp(\mu_0^\dagger + \gamma_0^TZ_{12})] \right\vert \\
&\le \underbrace{\left\vert \frac{1}{\binom{n}{2}} \sum_{ i < j} D_{ij,k}D_{ij,l} \left\{ \frac{\exp(\alpha_{ij}) }{(1 + n^{-\xi}\exp(\alpha_{ij}))^2 } - \exp(\mu_0^\dagger + \gamma_0^TZ_{ij}) \right\} \right\vert}_{(I)} \\
&\quad \quad + \underbrace{\left\vert \frac{1}{\binom{n}{2}} \sum_{ i < j} D_{ij,k}D_{ij,l} \exp(\mu_0^\dagger + \gamma_0^TZ_{ij}) -  \E[D_{12,k}D_{12,l} \exp(\mu_0^\dagger + \gamma_0^TZ_{12})] \right\vert}_{(II)}.
\end{align*}
By the strong law of large numbers, $(II)$ goes to zero almost surely. Let us consider $(I)$.
\begin{align*}
(I) &\le \frac{1}{\binom{n}{2}} \sum_{ i < j} \vert D_{ij,k}D_{ij,l} \vert \left\vert \frac{\exp(\alpha_{ij}) }{(1 + n^{-\xi}\exp(\alpha_{ij}))^2 } - \exp(\mu_0^\dagger + \gamma_0^TZ_{ij}) \right\vert \\
&\le C \cdot \max_{i<j}  \left\vert \frac{\exp(\alpha_{ij}) }{(1 + n^{-\xi}\exp(\alpha_{ij}))^2 } - \exp(\mu_0^\dagger + \gamma_0^TZ_{ij}) \right\vert \\
&= C \cdot \max_{i<j} \vert f_n(\alpha_{ij}) - f(\mu_0^\dagger + \gamma_0^TZ_{ij}) \vert \\
&\le C \cdot \left\{ \max_{i<j} \vert f_n(\alpha_{ij}) - f(\alpha_{ij}) \vert + \max_{i<j} \vert f(\alpha_{ij}) - f(\mu_0^\dagger + \gamma_0^TZ_{ij}) \vert  \right\} \\
&\le C \cdot \left\{ \sup_{x \in I} \vert f_n(x) - f(x) \vert + \max_{i<j} \vert \alpha_{ij} - \mu_0^\dagger + \gamma_0^TZ_{ij} \vert \right\},
\end{align*}
where we have used the Lipschitz continuity of $f$ on $I$ for the last inequality. By the uniform convergence of $f_n$ to $f$ on $I$, we know that the first term in the last line goes to zero. For the second term, recall that $\alpha_{ij}$ is a point between $D_{ij}^T\hat \theta$ and $D_{ij}^T\theta_0 = \mu_0^\dagger + \gamma_0^TZ_{ij}$. Hence,
\[
\max_{i<j} \vert \alpha_{ij} - \mu_0^\dagger + \gamma_0^TZ_{ij} \vert \le \max_{i< j} \vert (\hat \mu^\dagger - \mu_0^\dagger) + (\hat \gamma - \gamma_0)^TZ_{ij} \vert \le C \Vert \hat \theta - \theta_0 \Vert_1 \overset{P}{\rightarrow} 0,
\]
by the consistency of $\hat \theta$. Thus, $(I) \overset{P}{\rightarrow} 0$ as $n \rightarrow \infty$.

In conclusion, $\vert (n^{\xi}\Sigma_\alpha - \Sigma)_{kl} \vert \overset{P}{\rightarrow} 0$ and therefore,
\[
n^{\xi}\Sigma_\alpha = \Sigma + o_P(1),
\]
where $o_P(1)$ is to be understood as a matrix in which each component is $o_P(1)$.
Now, we get from \eqref{Eq: Taylor},
\begin{align*}
0 &= \frac{1}{\binom{n}{2}}  \nabla \mathcal{L}^\dagger(\theta_0) + \Sigma_\alpha (\hat \theta - \theta_0) \\
\shortintertext{which after multiplying with $n^{\xi}$ is equivalent to}
0 &= n^{\xi} \frac{1}{\binom{n}{2}}  \nabla \mathcal{L}^\dagger(\theta_0) + \left( \Sigma + o_P(1)\right) (\hat \theta - \theta_0). \\
\shortintertext{Rearranging gives}
\Sigma (\hat \theta - \theta_0) &= - n^{\xi} \frac{1}{\binom{n}{2}}  \nabla \mathcal{L}^\dagger(\theta_0) + o_P(1) (\hat \theta - \theta_0). \\
\shortintertext{Now, remember that $\Sigma$ is positive definite and thus invertible, to get}
(\hat \theta - \theta_0) &= - \Sigma^{-1} n^{\xi} \frac{1}{\binom{n}{2}}  \nabla \mathcal{L}^\dagger(\theta_0) + \Sigma^{-1}o_P(1) (\hat{\theta} - \theta_0).
\shortintertext{Observe that $\Sigma^{-1}$ has bounded maximum eigenvalue due to Assumption \ref{Assum: min eval} and thus $\Sigma^{-1}o_P(1) = o_P(1)$:}
(\hat \theta - \theta_0) &= - \Sigma^{-1} n^{\xi} \frac{1}{\binom{n}{2}}  \nabla \mathcal{L}^\dagger(\theta_0) + o_P(1) (\hat{\theta} - \theta_0).
\shortintertext{Finally, multiply by $\sqrt{\frac{\binom{n}{2}}{n^{\xi}}}$ and remember that 	$\hat{\theta} - \theta_0 = O_P\left( \sqrt{\frac{n^{\xi}}{\binom{n}{2}}}  \right)$}
\sqrt{\frac{\binom{n}{2}}{n^{\xi}}} (\hat \theta - \theta_0) &= - \Sigma^{-1} \sqrt{\binom{n}{2}} n^{\xi /2} \frac{1}{\binom{n}{2}}  \nabla \mathcal{L}^\dagger(\theta_0) + o_P(1).
\end{align*}
With this, due to \eqref{Eq: CLT gradient of L}, we have proven
\begin{equation}\label{Eq: CLT ER-C}
\sqrt{\frac{\binom{n}{2}}{n^{\xi}}} (\hat \theta - \theta_0) \overset{d}{\longrightarrow} \mathcal{N}(0, \Sigma^{-1}).
\end{equation}
\begin{proof}[Proof of Theorem \ref{Thm: asymptotic normality ER-C}]
		Theorem \ref{Thm: asymptotic normality ER-C} then follows from the solved problems 1 - 3 above.
\end{proof}
It remains to prove Corollary \ref{Cor: Corollary ER-C}.
\begin{proof}[Proof of Corollary \ref{Cor: Corollary ER-C}]
Notice that from \eqref{Eq: CLT ER-C} we get: For any $k = 1, \dots, (p+1)$,
\begin{equation}\label{Eq: CLT ER-C one component}
\sqrt{\frac{\binom{n}{2}}{n^{\xi}}} \cdot \frac{\hat \theta_k - \theta_{0,k}}{\sqrt{\Sigma^{-1}_{k,k}}} \overset{d}{\longrightarrow} \mathcal{N}(0,1).
\end{equation}
By the exact same arguments that we have used to show that $n^{\xi}\Sigma_\alpha = \Sigma + o_P(1)$, we can also show that
\[
n^{\xi} \hat \Sigma = \Sigma + o_P(1),
\]
where $\hat \Sigma$ is the same matrix as $\Sigma_\alpha$ with $\alpha_{ij}$ replaced by $\hat \mu_0^\dagger + \hat \gamma^TZ_{ij}$:
\[
\hat \Sigma = \frac{1}{\binom{n}{2}} D^T\text{diag}\left( \frac{n^{-\xi}\exp(\hat \mu^\dagger + \hat \gamma^TZ_{ij})}{(1 + n^{-\xi}\exp(\hat \mu^\dagger + \hat \gamma^TZ_{ij}))^2}, i < j  \right) D.
\]
By the same arguments as before, we can show that the minimum eigenvalue of $n^{\xi}\hat\Sigma$ is bounded away from zero, uniformly in $n$. This implies that the maximum eigenvalue of $(n^{\xi}\hat \Sigma)^{-1}$ is bounded by some finite constant $C$. We already know that the same property holds for $\Sigma$ and $\Sigma^{-1}$. Therefore, we have for the matrix $\infty$-norm:
\[
\Vert (n^{\xi}\hat \Sigma)^{-1} - \Sigma^{-1} \Vert_{\infty} \le \Vert (n^{\xi}\hat \Sigma)^{-1}\Vert_{\infty} \Vert \Sigma^{-1} \Vert_{\infty} \Vert n^{\xi}\hat \Sigma - \Sigma \Vert_\infty \le C  \Vert n^{\xi}\hat \Sigma - \Sigma \Vert_\infty = o_P(1).
\]
This means in particular for the diagonal elements:
\[
(n^{\xi}\hat \Sigma)^{-1}_{k,k} = n^{-\xi}\hat \Sigma^{-1}_{k,k} = \Sigma^{-1}_{k,k} + o_P(1).
\]
But then, from \eqref{Eq: CLT ER-C one component} and by Slutzky's Theorem,
\[
\sqrt{\binom{n}{2}} \cdot \frac{\hat \theta_k - \theta_{0,k}}{\sqrt{\hat\Sigma^{-1}_{k,k}}} =
\sqrt{\frac{\binom{n}{2}}{n^{\xi}}} \cdot \frac{\hat \theta_k - \theta_{0,k}}{\sqrt{n^{-\xi}\hat \Sigma^{-1}_{k,k}}} =
\sqrt{\frac{\binom{n}{2}}{n^{\xi}}} \cdot \frac{\hat \theta_k - \theta_{0,k}}{\sqrt{\Sigma^{-1}_{k,k} + o_P(1)}} \overset{d}{\longrightarrow} \mathcal{N}(0,1).
\]
\end{proof}

\subsection{Erd\H{o}s-R\'{e}nyi with diverging number of covariates}

We now extend our theoretical results in Section \ref{subsection: generalized ER model} by allowing the dimension of the covariates $p$ to go to infinity and we don't assume a  sparsity condition. We use $p_n$ instead of $p$ in the following. The asymptotic theory for a diverging number of covariates is quite different from the case where the dimension is fixed. See \cite{portnoy1984asymptotic, portnoy1985asymptotic, portnoy1988asymptotic}, \cite{fan2004nonconcave}, \cite{wang2011gee} and the references therein for more discussions. To establish consistency and asymptotic normality, we make the following assumptions.

\begin{Assum*}\label{diverging p: true theta}
The unknown parameter $\theta = (\mu^{\dagger}, \gamma^T)^T$ belongs to $\Theta := \{\theta \colon \|\theta\|_{\infty} \leq M \}$ and the true parameter $\theta_0$ lies in the interior of $\Theta$.
\end{Assum*}

\begin{Assum*}\label{diverging p: uniform bounded}
The $Z_{ij}$ are i.i.d. realizations of the same centered random variable and there exist constants $\kappa, c > 0$ such that for $D_{ij}=(1, Z_{ij}^{T})^{T}$, $\vert D_{ij}^T\theta_0 \vert \le \kappa$ for all $1 \le i < j \le n$, $\vert D_{ij,k} \vert \le c$ for all $1 \le i < j \le n, k = 1, \dots, p_n$.
\end{Assum*}

\begin{Assum*}\label{diverging p: eigenvalue}
There is a universal constant $c_{\min}> 0$ such that for all $n \in \mathbb{N}$, the minimum eigenvalue $\lambda_{\text{min}}$ and the maximum eigenvalue $\lambda_{\max}$ of $\E[D_{12}D_{12}^T]$ fulfill $0<c_{\min} \le \lambda_{\text{min}} \le \lambda_{\max} \le 1/c_{\min} < \infty$.
\end{Assum*}\label{diverging p: eigenvalue}

\begin{Satz}\label{Thm: Consistency for diverging p}
Under Assumptions \ref{diverging p: true theta}, \ref{diverging p: uniform bounded} and \ref{diverging p: eigenvalue}, if $n^{-1}p_{n}\log p_{n} = o(1)$ when $\xi = 0$; $n^{-1+\xi/2}p_{n} = o(1)$ when $\xi > 0$, then the score function $S_n(\theta) = 0$ has a root $\hat{\theta}_n$ such that
$$
\Vert \hat{\theta}_n-\theta_0 \Vert_2 =O_{P}\left(\sqrt{\frac{p_{n}n^{\xi}}{n^2}}\right).
$$
\end{Satz}

\begin{proof}
It suffices to verify the following condition as in \cite{wang2011gee}: for any $\varepsilon >0$, there exists a constant $\Delta>0$ such that for all $n$ sufficiently large,
$$
P\left(\sup_{\left\|\theta_n-\theta_0\right\|_2=\Delta \sqrt{p_{n} n^{\xi}/n^2}}(\theta_n-\theta_0)^{T} S_n(\theta_n)>0\right) \geq 1-\varepsilon .
$$
We have the following decomposition of $\left(\theta_n-\theta_0\right)^{T} S_n\left(\theta_n\right)$
$$
\begin{aligned}
(\theta_n-\theta_0)^{T} S_n(\theta_n)=& (\theta_n-\theta_0)^{T} S_n(\theta_0)+(\theta_n-\theta_0)^{T} \frac{\partial}{\partial \theta_n^{T}} S_n(\theta^*)(\theta_n-\theta_0) \\
= & (\theta_n-\theta_0)^{T} S_n(\theta_0)+(\theta_n-\theta_0)^{T} \frac{\partial}{\partial \theta_n^{T}} S_n(\theta_0)(\theta_n-\theta_0) \\
& +(\theta_n-\theta_0)^{T}\left(\frac{\partial}{\partial \theta_n^{T}} S_n(\theta^*)-\frac{\partial}{\partial \theta_n^{T}} S_n(\theta_0)\right)(\theta_n-\theta_0) \\
:= & (A_1)+(A_2)+(A_3).
\end{aligned}
$$
For $(A_1)$, since
$$
\begin{aligned}
\E\left(\| S_n(\theta_0)\|_2^2\right) & =\E\left(\sum_{i<j} D_{i j}^{T} D_{i j}\left(\frac{n^{-\xi} \exp(D_{i j}^{T} \theta_0)}{1+n^{-\xi} \exp(D_{i j}^{T} \theta_0)}-A_{i j}\right)^{2}\right) \\
& \leq  \binom{n}{2} \cdot(p_n+1)c^{2} \cdot \E\left(\frac{n^{-\xi} \exp(D_{i j}^{T} \theta_0)}{(1+n^{-\xi} \exp(D_{i j}^{T} \theta_0))^2}\right) \\
& \leq \binom{n}{2} \cdot (p_n+1)c^{2} \cdot n^{-\xi} \exp(\kappa),
\end{aligned}
$$
where we use $\vert D_{i j, k}\vert  \le c$, $\vert D_{i j}^{T} \theta_0\vert \leq \kappa$ and $1+n^{-\xi} \exp(D_{i j}^{T} \theta_0) \geq 1$.
Therefore
$$
|(A_1)| \leq \Delta O_{P}\left( \sqrt{\frac{p_n n^{\xi}}{n^{2}}}\sqrt{\binom{n}{2} \cdot (p_n+1)c^{2} \cdot n^{-\xi} \exp(\kappa)}\right)=\Delta O_{P}(p_n).
$$
For $\left(A_2\right)$,
$$
\begin{aligned}
\left(A_2\right) & =\left(\theta_n-\theta_0\right)^{T}\left(\sum_{i < j} \frac{D_{i j} n^{-\xi} \exp(D_{i j}^{T} \theta_0) D_{i j}^{T}}{(1+n^{-\xi} \exp(D_{i j}^{T} \theta_0))^2}\right)\left(\theta_n-\theta_0\right)^{T} \\
& \geq \left\|\theta_n-\theta_0\right\|_2^2 \cdot  \frac{n^{-\xi}\exp(-\kappa)}{4}\cdot  \lambda_{\min}\Bigl(\sum_{i < j} D_{i j} D_{i j}^{T}\Bigl)\\
& \geq C \cdot \left(\Delta \sqrt{\frac{p_n n^{\xi}}{n^2}}\right)^2 \cdot n^{-\xi}\cdot \binom{n}{2}\\
& =C\Delta^2 p_n,
\end{aligned}
$$
where we use $D_{i j}^{T} \theta_0 > -\kappa$, $1+n^{-\xi} \exp \left(D_{i j}^{T} \theta_0\right) \leq 2$ when $n$ is sufficiently large and 
$$
0 < c_1\le \lambda_{\min}\left(\frac{1}{\binom{n}{2}}\sum_{i < j} D_{i j} D_{i j}^{T}\right) \le \lambda_{\max}\left(\frac{1}{\binom{n}{2}}\sum_{i < j} D_{i j} D_{i j}^{T}\right) \le c_2 < \infty,
$$
when
$$
\frac{p_n\log p_n}{n} = o(1).
$$
This can be seen from a very similar discussion as in Subsection \ref{sample and population matrices} and notice that when $p_n$ is diverging, $\delta$ in Lemma \ref{Lem: Sigma is invertible whp} should be of the order $O_{P}(n^{-1}\log p_n)$ rather than $O_{P}(n^{-1})$.

\noindent Finally, for $(A_3)$,
$$
\begin{aligned}
\left|\left(A_3\right)\right|&=\left|\left(\theta_n-\theta_0\right)^{T}\left(\sum_{i<j} \frac{D_{i j} n^{-\xi} \exp (D_{i j}^{T} \theta^*) D_{i j}^{T}}{(1+n^{-\xi} \exp (D_{i j}^{T} \theta^*))^2}-\sum_{i<j} \frac{D_{i j} n^{-\xi} \exp (D_{i j}^{T} \theta_0) D_{i j}^{T}}{(1+n^{-\xi} \exp(D_{i j}^{T} \theta_0))^2}\right)\left(\theta_n-\theta_0\right)\right|\\
& =\left|\left(\theta_n-\theta_0\right)^{T}\left(\sum_{i < j} D_{i j} \cdot\left(\frac{D_{i j}^{T} n^{-\xi} \exp(D_{i j}^{T} \tilde{\theta})(n^{-\xi} \exp (D_{i j}^{T} \tilde{\theta})-1)}{(1+n^{-\xi} \exp (D_{i j}^{T} \tilde{\theta}))^3} \cdot\left(\theta^*-\theta_0\right)\right) D_{i j}^{T}\right)\left(\theta_n-\theta_0\right)\right| \\
& \leq C\left\|\theta_n-\theta_0\right\|_2^2 \cdot \lambda_{\max}\Bigl(\sum_{i < j} D_{i j} D_{i j}^{T}\Bigl) \cdot \sqrt{(p_n+1)c^2} \cdot\left\|\theta^*-\theta_0\right\|_2 \cdot n^{-\xi}  \\
& \leq C \left(\Delta\sqrt{\frac{p_n n^{\xi}}{n^2}}\right)^{2} \cdot \binom{n}{2} \cdot \sqrt{p_n} \cdot \Delta\sqrt{\frac{p_n n^{\xi}}{n^2}} \cdot n^{-\xi} \\
& =C \Delta^3 p_n \sqrt{\frac{p_n^2 n^{\xi}}{n^2}}\\
& =o_{P}(p_n),
\end{aligned}
$$
where for the inequality we use
$$
\left|\frac{n^{-\xi} \exp(D_{i j}^{T} \tilde{\theta})(n^{-\xi} \exp (D_{i j}^{T} \tilde{\theta})-1)}{(1+n^{-\xi} \exp(D_{i j}^{T} \tilde{\theta}))^3}\right| \leq Cn^{\xi}.
$$
This can be seen from the fact that
$$
\exp (D_{i j}^{T} \tilde{\theta}) = \exp(D_{i j}^{T} \theta_0) \exp (D_{i j}^{T} (\tilde{\theta}-\theta_0)) \leq  \exp (D_{i j}^{T} \theta_0) \exp \biggl(\Delta\sqrt{p_{n}}\sqrt{\frac{p_{n}n^{\xi}}{\binom{n}{2}}}\biggl)=O_{P}(1)
$$
and
$$
\left|\frac{n^{-\xi} \exp(D_{i j}^{T} \tilde{\theta})-1}{(1+n^{-\xi} \exp(D_{i j}^{T} \tilde{\theta}))^3}\right| \leq 1.
$$
Therefore, 
$$
(\theta_n-\theta_0)^{T} S_n(\theta_n)\geq (A_2) -|(A_1)|-|(A_3)|\geq C\Delta^2 p-\Delta O_{P}(p)- o_{P}(p).
$$
Hence for each $\varepsilon$, we can choose a sufficiently large $\Delta$ such that when $n$ is large enough,
$$
P\left(\sup_{\Vert\theta_n-\theta_0\Vert_2=\Delta \sqrt{p_{n} n^{\xi}/n^2}}(\theta_n-\theta_0)^{T} S_n(\theta_n)>0\right) \geq 1-\varepsilon .
$$
\end{proof}

\begin{Satz}\label{Thm: AN for diverging p}
Under Assumptions \ref{diverging p: true theta}, \ref{diverging p: uniform bounded} and \ref{diverging p: eigenvalue}, if $n^{-2}p_{n}^3(\log p_{n})^2 = o(1)$ when $\xi = 0$; $n^{-2+\xi}p_{n}^3 = o(1)$ when $\xi > 0$, then for any $u_{n} \in \R^{p_{n}+1}$ such that $\Vert u_{n} \Vert_2 = 1$, we have
$$
\sqrt{\frac{\binom{n}{2}}{n^{\xi}}}u_n^{T}\Sigma_{n}^{1/2}(\hat{\theta}_n-\theta_0) \stackrel{d}{\longrightarrow} N(0,1),
$$
where
$$
\Sigma_n = \E[D_{i j} \exp \left(D_{i j}^{T} \theta_{0}\right) D_{i j}^{T}].
$$
\end{Satz}

\begin{proof}
Denote the derivatives of $S_n(\theta)$ as
$$
H_n(\theta) = \frac{\partial S_n(\theta)}{\partial \theta^{T}} = \sum_{i<j} D_{i j}\frac{n^{-\xi} \exp (D_{i j}^{T} \theta)}{(1+n^{-\xi} \exp(D_{i j}^{T} \theta))^{2}}D_{i j}^{T}
$$
and let
$$
H(\theta_0)=\binom{n}{2} n^{-\xi} \Sigma_{n}.
$$
Then the goal is to prove
$$
u_n^{T}\left(H(\theta_0)\right)^{\frac{1}{2}}(\hat{\theta}_n-\theta_0) \stackrel{d}{\longrightarrow} N(0,1).
$$
We have
$$
\begin{aligned}
u_n^{T}\left(H(\theta_0)\right)^{-\frac{1}{2}} S(\theta_0)= & u_n^{T}\left(H(\theta_0)\right)^{-\frac{1}{2}}\left[S_n(\theta_0)-S_n(\hat{\theta}_n)\right] \\
= & u_n^{T}\left(H(\theta_0)\right)^{-\frac{1}{2}}\left[H_n(\theta^*)(\theta_0-\hat{\theta}_n)\right] \\
=& u_n^{T}\left(H(\theta_0)\right)^{-\frac{1}{2}} H(\theta_0)(\theta_0-\hat{\theta}_n) \\
&+ u_n^{T}\left(H(\theta_0)\right)^{-\frac{1}{2}}\left(H_n(\theta_0)-H(\theta_0)\right)(\theta_0-\hat{\theta}_n) \\
&+ u_n^{T}\left(H(\theta_0)\right)^{-\frac{1}{2}}\left(H_n(\theta^*)-H_n(\theta_0)\right)(\theta_0-\hat{\theta}_n),
\end{aligned}
$$
where $\theta^{*}$ lies between $\hat{\theta}_n$ and $\theta_0$. Therefore, it's sufficient to prove

\noindent \textbf{Step 1:}
$$
u_n^{T}\left(H(\theta_0)\right)^{-\frac{1}{2}} S(\theta_0) \stackrel{d}{\longrightarrow}  N(0,1),
$$

\noindent \textbf{Step 2:}
$$
u_n^{T}\left(H(\theta_0)\right)^{-\frac{1}{2}}\left(H_n(\theta_0)-H(\theta_0)\right)(\theta_0-\hat{\theta}_n) = o_{P}(1),
$$

\noindent \textbf{Step 3:}
$$
u_n^{T}\left(H(\theta_0)\right)^{-\frac{1}{2}}\left(H_n(\theta^*)-H_n(\theta_0)\right)(\theta_0-\hat{\theta}_n)= o_{P}(1).
$$

\noindent For step 1, we first calculate the asymptotic variance of $u_n^{T}\left(H(\theta_0)\right)^{-1/2}S_n(\theta_0)$.
$$
\begin{aligned}
\text{Var}(u_n^{T}\left(H(\theta_0)\right)^{-\frac{1}{2}}S_n(\theta_0))=&u_n^{T}\left(H(\theta_0)\right)^{-\frac{1}{2}}\text{Cov}(S_n(\theta_0))\left(H(\theta_0)\right)^{-\frac{1}{2}}u_n\\
=&u_n^{T}\left(\Sigma_n\right)^{-\frac{1}{2}}E\left(D_{ij}\frac{ \exp (D_{i j}^{T} \theta_0)}{(1+n^{-\xi} \exp(D_{i j}^{T} \theta_0))^{2}}D_{ij}^{T}\right)\left(\Sigma_n\right)^{-\frac{1}{2}}u_n\\
=&u_n^{T}\left(\Sigma_n\right)^{-\frac{1}{2}}\Sigma_n\left(\Sigma_n\right)^{-\frac{1}{2}}u_n\\
&+u_n^{T}\left(\Sigma_n\right)^{-\frac{1}{2}}\left(E\left(D_{ij}\frac{n^{-\xi} \exp (D_{i j}^{T} \theta_0)}{(1+n^{-\xi} \exp(D_{i j}^{T} \theta_0))^{2}}D_{ij}^{T}\right)-\Sigma_n\right)\left(\Sigma_n\right)^{-\frac{1}{2}}u_n\\
=&1+u_n^{T}\left(\Sigma_n\right)^{-\frac{1}{2}}E\left(D_{12}\left(\frac{\exp (D_{12}^{T} \theta_0)}{(1+n^{-\xi} \exp \left(D_{12}^{T} \theta_0\right))^{2}}-\exp (D_{12}^{T} \theta_0)\right)D_{12}^{T}\right)\left(\Sigma_n\right)^{-\frac{1}{2}}u_n.
\end{aligned}
$$
We have
$$
\lambda_{\min}(\Sigma_n) \geq \exp(-\kappa)c_{\min},
$$
$$
\frac{1}{(1+n^{-\xi} \exp \left(D_{12}^{T} \theta_0\right))^{2}}-1 = o(1)
$$
and
$$
\lambda_{\max}(\E(D_{12}D_{12}^{T}))\leq c_{\max}.
$$
Then
$$
u_n^{T}\left(\Sigma_n\right)^{-\frac{1}{2}}E\left(D_{12}\left(\frac{\exp (D_{12}^{T} \theta_0)}{(1+n^{-\xi} \exp \left(D_{12}^{T} \theta_0\right))^{2}}-\exp (D_{12}^{T} \theta_0)\right)D_{12}^{T}\right)\left(\Sigma_n\right)^{-\frac{1}{2}}u_n= o(1).
$$
Therefore
$$
\text{Var}(u_n^{T}\left(H(\theta_0)\right)^{-\frac{1}{2}}S_n(\theta_0)) \rightarrow 1.
$$
Next, we will show that the Lindeberg condition holds.
$$
\begin{aligned}
u_n^{T}\left(H(\theta_0)\right)^{-\frac{1}{2}}S_n(\theta_0)=&u_n^T \sqrt{\frac{n^\xi}{\binom{n}{2}}}\left(\Sigma_{n}\right)^{-\frac{1}{2}} \sum_{i < j} D_{i j}\left(\frac{n^{-\xi} \exp(D_{i j}^{T} \theta_0)}{1+n^{-\xi} \exp(D_{i j}^{T} \theta_0)}-A_{i j}\right)\\
=& \sum_{i < j} Y_{n,ij}\varepsilon_{n,ij},
\end{aligned}
$$
where 
$$
Y_{n,ij} = u_n^T \sqrt{\frac{n^\xi}{\binom{n}{2}}}\left(\Sigma_{n}\right)^{-\frac{1}{2}}D_{i j}
$$
and 
$$
\varepsilon_{n,ij} = \frac{n^{-\xi} \exp(D_{i j}^{T} \theta_0)}{1+n^{-\xi} \exp(D_{i j}^{T} \theta_0)}-A_{i j}.
$$
For $Y_{n,ij}$,
$$
\begin{aligned}
\max_{i, j}\left|Y_{n, i j}\right| &\leq \sqrt{\frac{n^\xi}{\binom{n}{2}}}\Vert u_n^{T}\Vert_2 \left(\lambda_{\min}(\Sigma_n)\right)^{-1/2}\max _{i, j}\Vert D_{i j}\Vert _2\\
& \leq  C\sqrt{\frac{n^\xi}{\binom{n}{2}}} \sqrt{(p_n+1)c} =C\sqrt{\frac{p_n n^\xi}{n^2}} =o(1).
\end{aligned}
$$
Also, note that $\varepsilon_{n,ij}$ is uniformly bounded. Then for any $\varepsilon$, there exists $n_0$ large enough such that for all $n> n_0$, $|Y_{n,ij}\varepsilon_{n,ij}|\leq \varepsilon $, therefore
$$
\sum_{i < j}  E\left((Y_{n,ij}\varepsilon_{n,ij})^{2}\mathbb{I}\{|Y_{n,ij}\varepsilon_{n,ij}|\geq \varepsilon \}\right)\longrightarrow 0.
$$ 
This gives the Lindeberg condition. Now, applying the Lindeberg-Feller central limit theorem, we complete step 1.

\noindent For step 2, denote 
$$
H_{ij, kl} = D_{i j, k} D_{i j, l} \frac{n^{-\xi} \exp(D_{i j}^{T} \theta_0)}{(1+n^{-\xi
} \exp (D_{i j}^{T} \theta_0))^2}.
$$
Since $-n^{-\xi} C \leq H_{ij, kl}  \leq n^{-\xi} C$, by Hoeffding's inequality,
$$
\begin{aligned}
 P\left(\vert H_n\left(\theta_0\right)-H\left(\theta_0\right)\vert_{k l} \geq \eta\right)  &=P\left(\left|\sum_{i< j} H_{ij, kl}-\E\left(H_{ij, kl}\right)\right| \geq \eta\right) \leq 2 \exp \left(-\frac{C \eta^2 n^{2 \xi}}{n^2}\right) 
\end{aligned}.
$$
By a union bound argument,
$$
P\left(\max _{k, l}\left|H_n\left(\theta_0\right)-H\left(\theta_0\right)\right|_{k l} \geq \eta \right) \leq 2(p_n+1)^2 \exp \left(-\frac{C \eta^2 n^{2 \xi}}{n^2}\right)
$$
and
$$
\max _{k, l}\left|H_n\left(\theta_0\right)-H\left(\theta_0\right)\right|_{k l}=O_{P}(n^{1-\xi}\log p_n).
$$
Therefore,
$$
\begin{aligned}
& u_n^{T}\left(H(\theta_0)\right)^{-\frac{1}{2}}\left(H_n(\theta_0)-H(\theta_0)\right)(\theta_0-\hat{\theta}_n) \\
\leq & \left(\binom{n}{2} n^{-\xi} \lambda_{\min}\left(\Sigma_n\right)\right)^{-\frac{1}{2}}\cdot \Vert H_n(\theta_0)-H(\theta_0)\Vert_2 \cdot  \Vert \theta_0-\hat{\theta}_n\Vert_2 \\
\leq & C\sqrt{\frac{n^{\xi}}{\binom{n}{2}}} \cdot (p_n+1)\max _{k , l}\left(\left|H_n(\theta_0)-H(\theta_0)\right|\right)_{kl}\cdot \sqrt{\frac{p_n n^{\xi}}{\binom{n}{2}}}\\
= & O_{P}\left(\sqrt{\frac{p_{n}^3}{n^{2}}}\log p_n\right)\\
= & o_{P}(1).
\end{aligned}
$$
This concludes step 2.

\noindent For step 3, we have
$$
\begin{aligned}
& u_n^{T}\left(H(\theta_0)\right)^{-\frac{1}{2}}\left(H_n(\theta^*)-H_n(\theta_0)\right)(\theta_0-\hat{\theta}_n)\\
=& u_n^{T}\left(H(\theta_0)\right)^{-\frac{1}{2}} \left(\sum_{i < j} D_{i j} \cdot\left(\frac{D_{i j}^{T} n^{-\xi} \exp(D_{i j}^{T} \tilde{\theta})(n^{-\xi} \exp(D_{i j}^{T} \tilde{\theta})-1)}{(1+n^{-\xi} \exp (D_{i j}^{T} \tilde{\theta}))^3} \cdot(\theta^*-\theta_0)\right) D_{i j}^{T}\right)(\theta_0-\hat{\theta}_n)\\
\leq &C\sqrt{\frac{n^{\xi}}{\binom{n}{2}}} \cdot  \lambda_{\max}\Bigl(\sum_{i < j} D_{i j}D_{i j}^{T}\Bigl)n^{-\xi}\sqrt{(p_n+1)c^2} \cdot \Vert \theta^{*}-\theta_0\Vert_2  \cdot \Vert \theta_0-\hat{\theta}_n\Vert_2\\
=& O_{P}\left(\sqrt{\frac{n^{\xi}}{\binom{n}{2}}}\cdot \binom{n}{2} \cdot n^{-\xi} \cdot \sqrt{p_n} \cdot \sqrt{\frac{p_n n^{\xi}}{\binom{n}{2}}} \cdot \sqrt{\frac{p_n n^{\xi}}{\binom{n}{2}}} \right)\\
=& O_{P}\left(\sqrt{\frac{p_{n}^3 n^{\xi}}{n^2}}\right)\\
= & o_{P}(1).
\end{aligned}
$$
Theorem \ref{Thm: AN for diverging p} now follows from steps 1-3 above.

\end{proof}

\begin{Prop}
Define the matrix 
$$
\hat{\Sigma}_n = \frac{1}{\binom{n}{2}} \sum_{i<j} D_{i j} \exp(D_{i j}^{T} \hat{\theta}) D_{i j}^{T}.
$$
Then under conditions of Theorem \ref{Thm: AN for diverging p} we have
$$
W_n \hat{\Sigma}_n^{-1} W_n^T-W_n \Sigma_{n}^{-1} W_n^T \stackrel{p}{\rightarrow} 0 \text {  as } n \rightarrow \infty
$$
for any $q \times p_n$ matrix $W_n$ where $q$ is any fixed integer. 
\end{Prop}

\begin{proof}
Notice that both $\Vert \hat{\Sigma}_n \Vert_2$ and $\Vert \Sigma_n \Vert_2$ are uniformly bounded away from $0$ and infinity and we have
$$
\hat{\Sigma}_n^{-1}- \Sigma_n^{-1}=\hat{\Sigma}_n^{-1}(\Sigma_n-\hat{\Sigma}_n)\Sigma_n^{-1}.
$$
Then it suffices to prove that 
$$
\Vert\hat{\Sigma}_n-\Sigma_n\Vert_2=o_{P}(1).
$$
Denote 
$$
\bar{\Sigma}_n = \frac{1}{\binom{n}{2}} \sum_{i<j} D_{i j} \exp(D_{i j}^{T} \theta_0) D_{i j}^{T}.
$$
Next, we will prove
$$
\Vert \hat{\Sigma}_n-\bar{\Sigma}_n\Vert_2=o_{P}(1)
$$
and
$$
\Vert\bar{\Sigma}_n-\Sigma_n\Vert_2=o_{P}(1).
$$
For $\Vert\hat{\Sigma}_n-\bar{\Sigma}_n\Vert_2$,
$$
\hat{\Sigma}_n-\bar{\Sigma}_n=\frac{1}{\binom{n}{2}} \sum_{i<j} D_{i j}(\exp(D_{i j}^T \hat{\theta}_n)-\exp(D_{i j}^{T} \theta_0)) D_{i j}^T.
$$
By the Mean Value Theorem,
$$
|\exp(D_{i j}^{T} \hat{\theta}_n)-\exp (D_{i j}^{T} \theta_0)|  =D_{i j}^{T} \exp (D_{i j}^{T} \tilde{\theta})(\hat{\theta}_n-\theta_0) = O_{P}\left(\sqrt{\frac{p_n^{2}n^{\xi}}{n^{2}}}\right).
$$
Therefore
$$
\Vert\hat{\Sigma}_n-\bar{\Sigma}_n\Vert_2 \leq C \cdot \frac{1}{\binom{n}{2}}\cdot \binom{n}{2} \cdot  \sqrt{\frac{p_n^2 n^{\xi}}{\binom{n}{2}}} = o_{P}(1).
$$
by noticing that when $n$ is large enough, $\lambda_{\max}(\sum_{i<j} D_{i j}D_{ij}^{T}/\binom{n}{2})\leq c_2 < \infty$.

\noindent For $\left\|\bar{\Sigma}_n-\Sigma_n\right\|_2$,
$$
\begin{aligned}
\left\|\bar{\Sigma}_n-\Sigma_n\right\|_2 & \leq (p_n+1)\max _{k, l}\left|\bar{\Sigma}_n-\Sigma_n\right|_{k l} =O_{P}\left(\sqrt{\frac{p_n^2}{n^2}} \log p_n\right) = o_{P}(1)
\end{aligned}
$$
by Hoeffding's inequality and a union bound argument. 
\end{proof}

\section{Sparse \texorpdfstring{$\beta$}{beta}-model without covariates}\label{Sec: SBetaM no covariates}

By letting $p = 0, \gamma = 0$, the results for \sbmc derived have implications for the S$\beta$M without covariates in \cite{Chen:etal:19}. In the case without covariates, the negative log-likelihood is given by
\begin{equation*}
\mathcal{L}(\beta, \mu) = - \sum_i \beta_i d_i - d_+ \mu + \sum_{i < j}\log (1 + e^{\beta_i + \beta_j + \mu})
\end{equation*}
and our design matrix is simply
$
D = \begin{bmatrix}
X | \textbf{1}
\end{bmatrix}
\in \R^{\binom{n}{2} \times (n+1)}.
$
The definitions of $\rho_{n,0}$ and $r_{n,0}$ do not change, as we can simply set $\gamma = 0$ in their original definitions. In this section we will abuse notation slightly by reusing the names from S$\beta$RM, but redefining them to have the components corresponding to $\gamma$ removed. For example, we will use $\theta = (\beta^T, \mu)^T$ for a generic parameter, $\theta_0 = (\beta_0^T, \mu_0)^T$ to denote the truth, $S^*_+ = S^* \cup \{n+1\}$ to denote the sparsity including the $\mu$ component etc. We think this is justified as it makes the connection to the respective objects in the model with covariates clearer.
Our estimator reduces to
\begin{equation*}
\hat{\theta} = (\hat{\beta}^T, \hat{\mu})^T = \argmin_{(\beta^T, \mu)^T \in \Theta_{\text{loc}}} \frac{1}{\binom{n}{2}} \mathcal{L}(\beta, \mu) + \lambda \Vert \beta \Vert_1,
\end{equation*}
where by slight abuse of notation, for this section only, we define $\Theta_{\text{loc}} = \Theta_{\text{loc}} (r_n)  \coloneqq \{ \theta = (\beta^T, \mu)^T\in \R_+^n \times \R
: \Vert D\theta \Vert_\infty \le r_n \}$, for the reduced design matrix $D$ defined above and a rate $r_n$.

We make definitions completely analogous to the case in which we observe covariates. We adapt the definitions of the excess risk $\mathcal{E}(\theta)$ in the canonical way by letting the components corresponding to $\gamma$ and  $Z_{ij}$ equal zero.
We define the best local approximation $\theta^*$ as
\[
\theta^* = \argmin_{\theta \in \Theta_{\text{loc}}} \mathcal{E}(\theta)
\]
and as before, we assume that all unpenalized parameters, i.e.~$\mu^*$ in this case, are active. Since the sparsity assumptions of our parameter only concern $\beta$, it is natural that we should need the same assumptions on $s^*_+$ as before, most notably Assumption \ref{Assum: rate of s^*}. We have the analog to Theorem \ref{Thm: consistency}.

\begin{Satz}\label{Thm: consistency no covariates}
	Assume Assumption \ref{Assum: rate of s^*}. Fix a confidence level $t$ and let
	\[
		a_n = \sqrt{\frac{\log(2(n+1))}{\binom{n}{2}}}
	\]
	and 
	\[
		\lambda_0 = 8a_n+ 2 \sqrt{\frac{t}{\binom{n}{2}}  (9  +  8\sqrt{2n} a_n  ) } + \frac{2\sqrt{2}t \sqrt{n}}{3\binom{n}{2}}.
	\]
	Let $\bar{ \lambda} = \frac{\sqrt{n}}{\sqrt{2}} \lambda \ge 8 \lambda_0$ and define $K_n$ as in \eqref{Eq: Def K_n}. Then, with probability at least $1 - \exp(-t)$ we have
	\begin{align}\label{Eq: convergence rate no covariates}
	\mathcal{E}(\hat{\theta}) + \bar{ \lambda} \left( \frac{\sqrt{2}}{\sqrt{n}} \Vert \hat{\beta} - \beta^*  \Vert_1 + \vert \hat{\mu} - \mu^* \vert \right) &\le 6 \mathcal{E}(\theta^*) + 64 s^*_+K_n\bar{ \lambda}^2.
	\end{align}
\end{Satz}
It is interesting to put this result into context by comparing it with Theorem 2 in \cite{Chen:etal:19}. The parameter space over which \cite{Chen:etal:19} are optimizing is not convex and the analogous notion of best local approximation we are using need not be well-defined in their setting. Thus, it is not possible to derive $\ell_1$-error bounds for their estimator, as we do in Theorem \ref{Thm: consistency no covariates}. Nonetheless and quite remarkably, they are able to prove an existence criterion for their $\ell_0$-constrained estimator and a high-probability, finite sample bound on its excess risk.
To compare their results to ours, we consider a special case that they discuss at length. In particular, they consider the situation in which $\mu_0 = - \xi \cdot \log(n) + O(1)$ for some $\xi \in [0,2)$ and $\beta_{0,i} = \alpha \cdot \log(n) + O(1)$ for some $\alpha \in [0,1)$ and all $i \in S_0$, where $\alpha$ and $\xi$ are such that $0 \le \xi - \alpha < 1$. It is easy to see that under these assumptions we have $\rho_{n,0} \sim n^{- \xi }$. Consider the regime in which no approximation error is committed and $\rho_n \sim \rho_{n,0}$. Then, using an analogous argument as in the proof of Theorem \ref{Cor: no approximation error}, $K_n$ is of order $\rho_{n,0}^{-1}$. Recalling Assumption \ref{Assum: rate of s^*}, we see that to obtain $\ell_1$-consistency of our estimator, we need $\xi < 1/2$, which restricts the degree of network sparsity that our estimator can handle. \cite{Chen:etal:19} need no such condition and only need to balance the global sparsity parameter $\xi$ with the local density parameter $\alpha$ to have convergence of their excess risk to zero. This illustrates that to obtain our more refined consistency result in terms of $\ell_1$-error, we understandably need to impose stricter assumptions on the permissible sparsity. We now compare the bounds on the excess risk. Note that \cite{Chen:etal:19} scale their excess risk by $\E[d_+]^{-1} \sim n^{-2 + \xi}$, rather than $\binom{n}{2}^{-1} \sim n^{-2}$ as we do. To put the excess risk on the same scale, we denote by
$
	\mathcal{E}^{\text{(r)}}(\hat \theta) = n^{\xi} \mathcal{E}(\hat \theta)
$
the excess risk rescaled to their setting.
With this notation, we see that by Theorem \ref{Thm: consistency no covariates} the error rate for the rescaled excess risk of our $\ell_1$ constrained estimator becomes
\[
	\mathcal{E}^{(r)}(\hat \theta) = O_P(s^*_+ \cdot \log(n) \cdot n^{-2 + 2\xi}),
\]
which by Assumption \ref{Assum: rate of s^*} is $o_P(\sqrt{\log(n)}\cdot n^{-3/2 + \xi})$. From \cite{Chen:etal:19}, Theorem 2, it is seen that the rate for the excess risk of their $\ell_0$ constrained estimator is
\[
O_P(\log(n) \cdot n^{-1+\xi /2}).
\]
This shows that in the regime $\xi \in [0,1/2)$ necessary for $\ell_1$-consistent parameter estimation, our estimator will always achieve a rate faster than the one in \cite{Chen:etal:19}. When we leave this regime, however, consistent estimation with respect to the $\ell_1$-norm may no longer be possible and the estimator in \cite{Chen:etal:19} can outperform our estimator.

\section{Additional Simulations}

We now provide additional simulation results for sparser networks. Specifically, 

\noindent
\textbf{Model 4}: We pick \(\beta_0 = \log(\log(n)) \cdot (2, 0.8, 1, \dots, 1, 0, \dots, 0)^T\) and set $\mu_0 = -0.5 \cdot \log(n)$;

\noindent
\textbf{Model 5}: We pick \(\beta_0 = \log(\log(n)) \cdot (2, 0.8, 1, \dots, 1, 0, \dots, 0)^T\) and set $\mu_0 = -0.75\cdot \log(n)$;

\noindent
\textbf{Model 6}: We pick \(\beta_0 = \log(\log(n)) \cdot (2, 0.8, 1, \dots, 1, 0, \dots, 0)^T\) and set $\mu_0 = -\log(n)$.

For Model 4--6, the errors for parameter estimation are shown in Figures \ref{Fig: Model4}, \ref{Fig: Model5} and \ref{Fig: Model6}. The error values are generally higher than that in Model 1--3, which is to be expected due to the much higher sparsity of the network. Also, for these very sparse cases, BIC is performing better than the heuristic. The heuristic consistently selects higher penalty values than BIC and we can see how this results in worse estimates for very sparse networks. Also, for the heuristic we choose one predefined penalty value for any network of a given size $n$, while BIC can adapt to the observed sparsity. This illustrates the point made by \cite{Yu:etal:2018}, that the penalty prescribed by mathematical theory tends to over-penalize the model. Table \ref{TableS1} presents the empirical coverage of the approximate $95\%$ confidence intervals and their median length for $\gamma_{0,1}$ and $\beta_0$ in Model 4--6 across different network sizes. It is to be noted, though, that even in this very sparse regime, the coverage is also very close to the $95\%$-level across all network sizes and all models. However, the median length of these confidence intervals increases a lot compared with Model 1--3.

\begin{figure}[!htbp]
	\centering
	\begin{subfigure}{0.32\textwidth}
		\centering
		\includegraphics[scale=0.22]{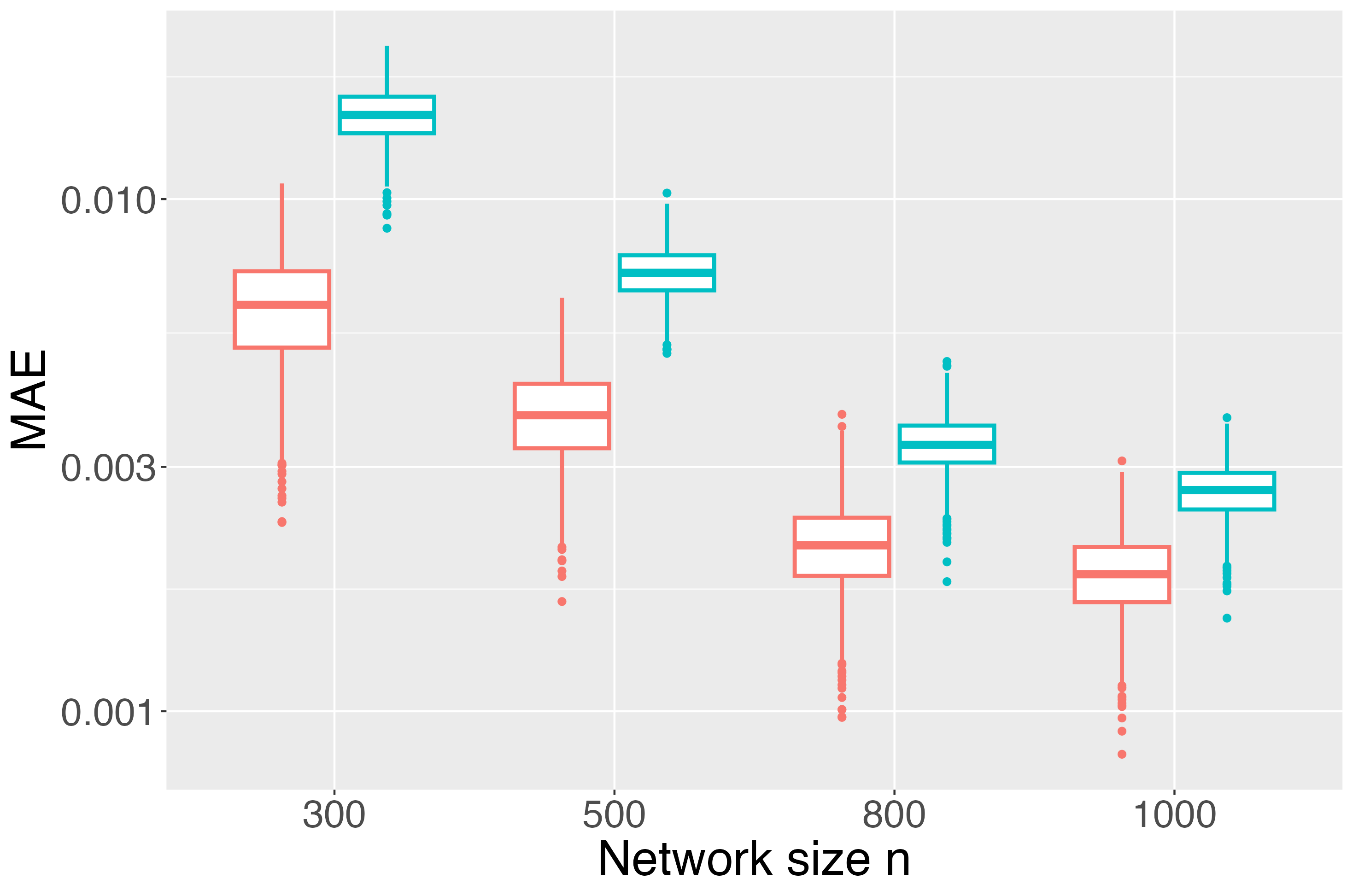}
		\caption{MAE for $\beta_0$}
		\label{Fig: beta MAE model 4}
	\end{subfigure}%
	\begin{subfigure}{0.32\textwidth}
		\centering
		\includegraphics[scale=0.22]{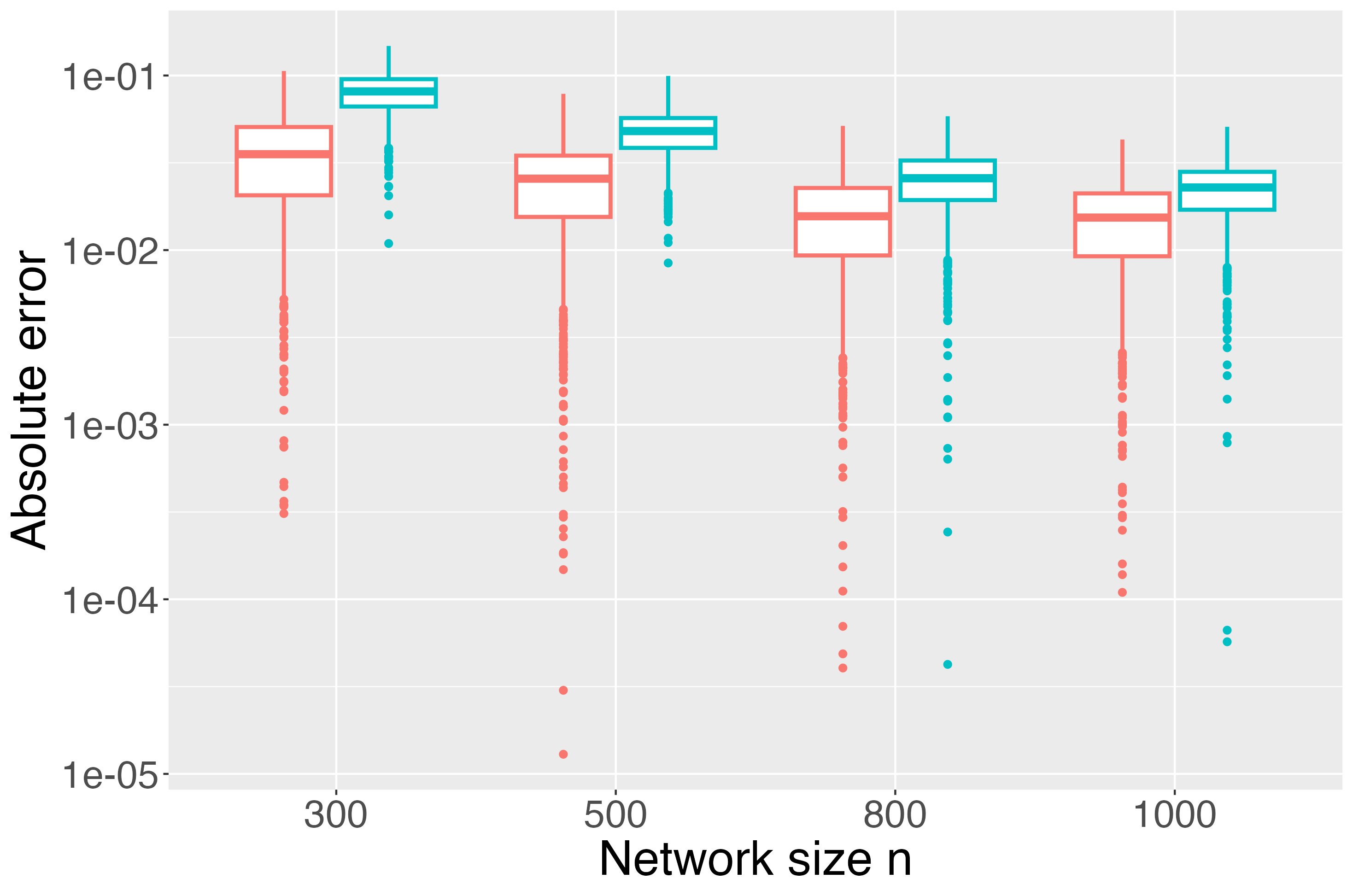}
		\caption{Absolute error for $\mu_0$.}
		\label{Fig: abs mu model 4}
	\end{subfigure}
	\begin{subfigure}{0.32\textwidth}
		\centering
		\includegraphics[scale=0.22]{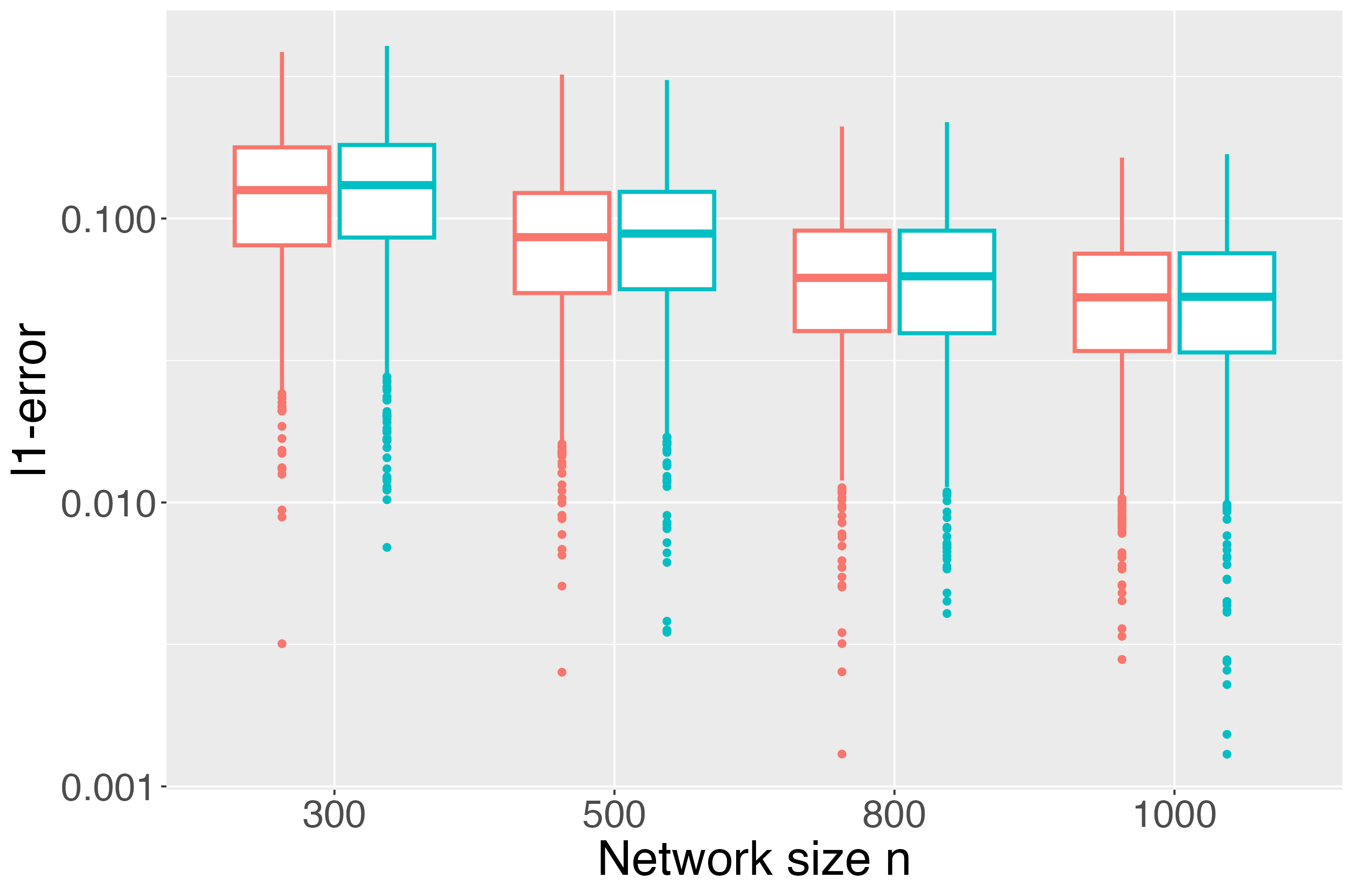}
		\caption{$\ell_1$-error for $\gamma_0$.}
		\label{Fig: l1 gamma model 4}
	\end{subfigure}
	\caption{Errors for estimating the true parameter $\theta_0$ in Model 4.}
	\label{Fig: Model4}
\end{figure}

\begin{figure}[!htbp]
	\centering
	\begin{subfigure}{0.32\textwidth}
		\centering
		\includegraphics[scale=0.22]{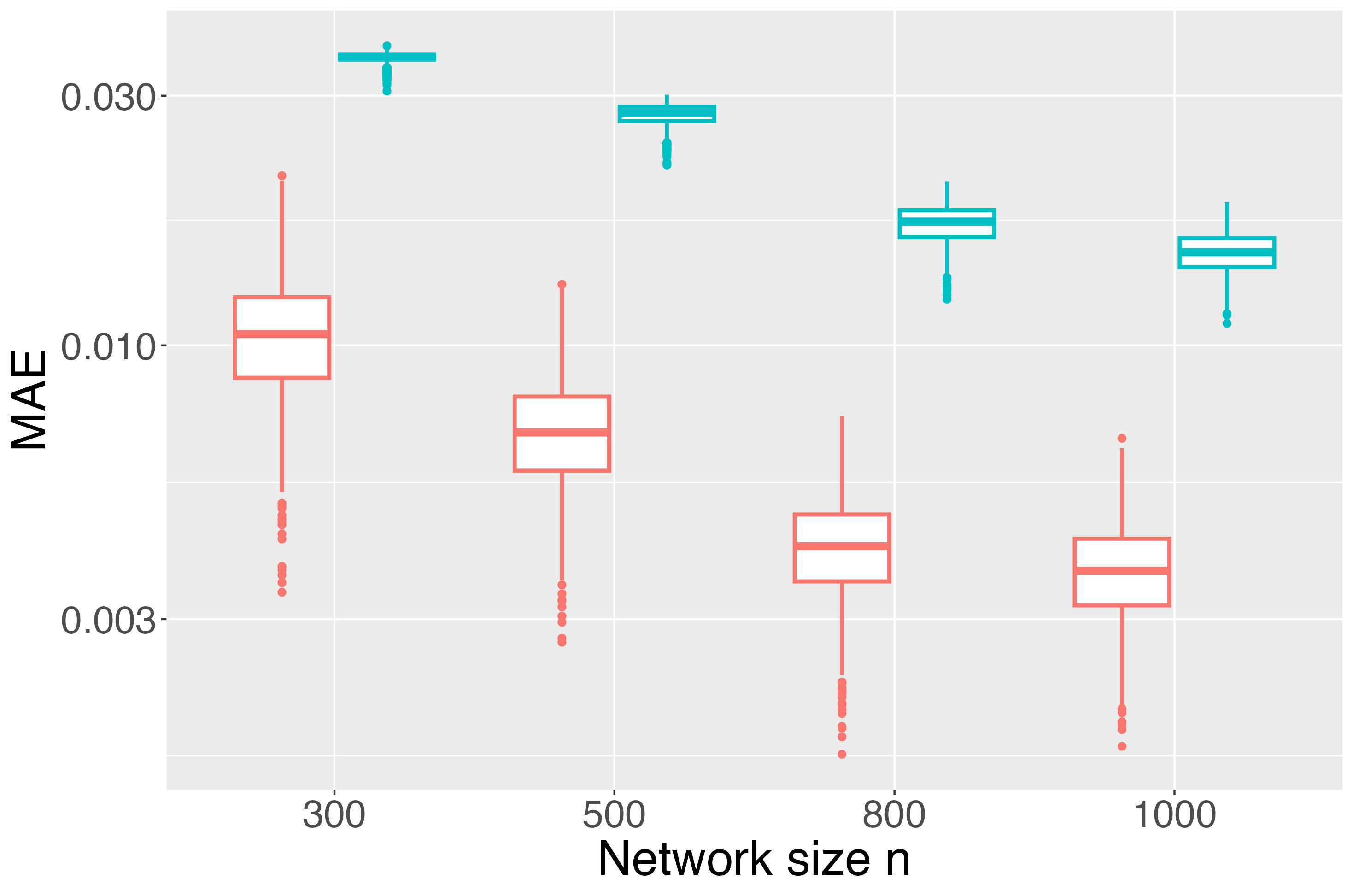}
		\caption{MAE for $\beta_0$}
		\label{Fig: beta MAE model 5}
	\end{subfigure}%
	\begin{subfigure}{0.32\textwidth}
		\centering
		\includegraphics[scale=0.22]{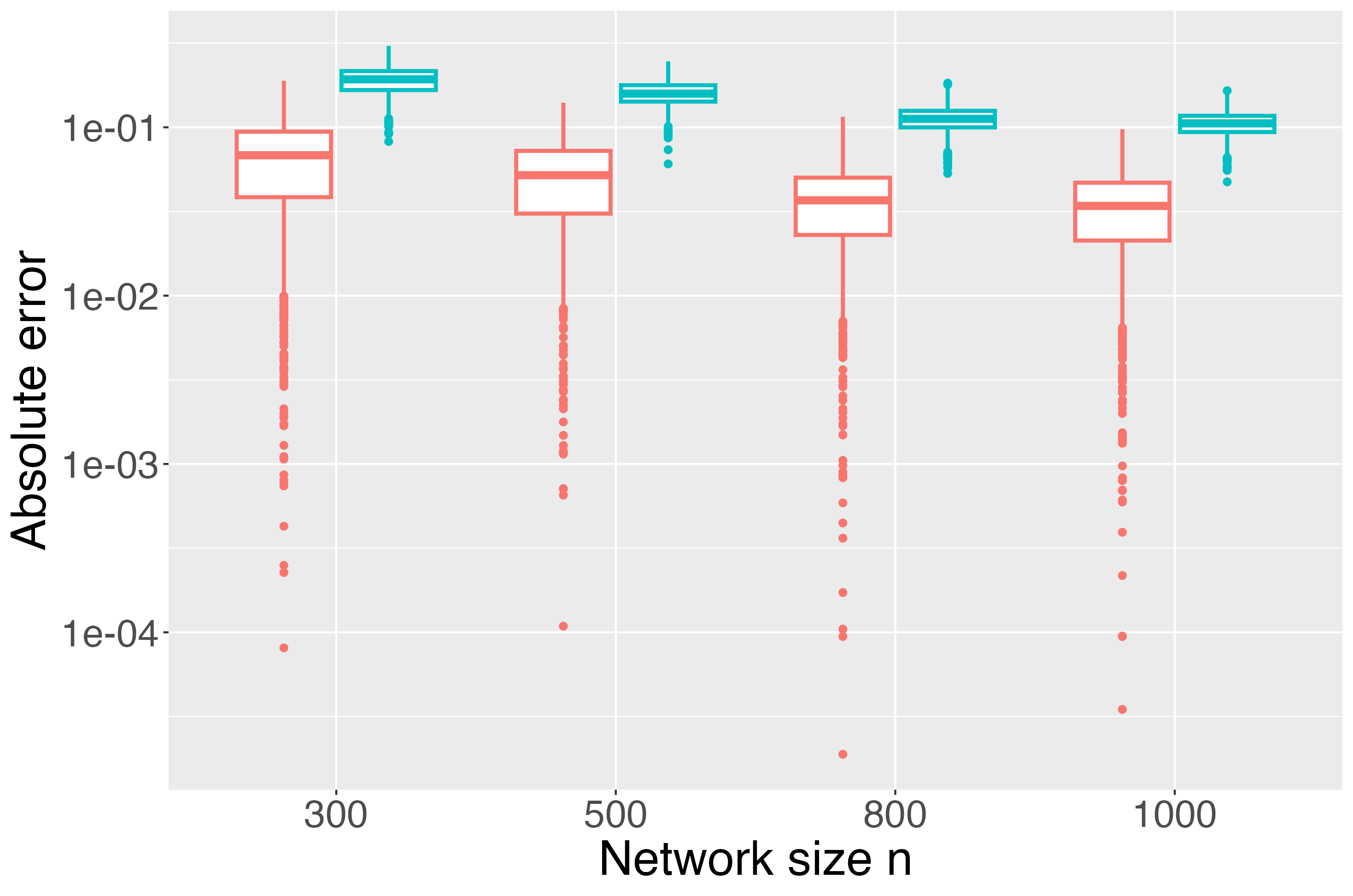}
		\caption{Absolute error for $\mu_0$.}
		\label{Fig: abs mu model 5}
	\end{subfigure}
	\begin{subfigure}{0.32\textwidth}
		\centering
		\includegraphics[scale=0.22]{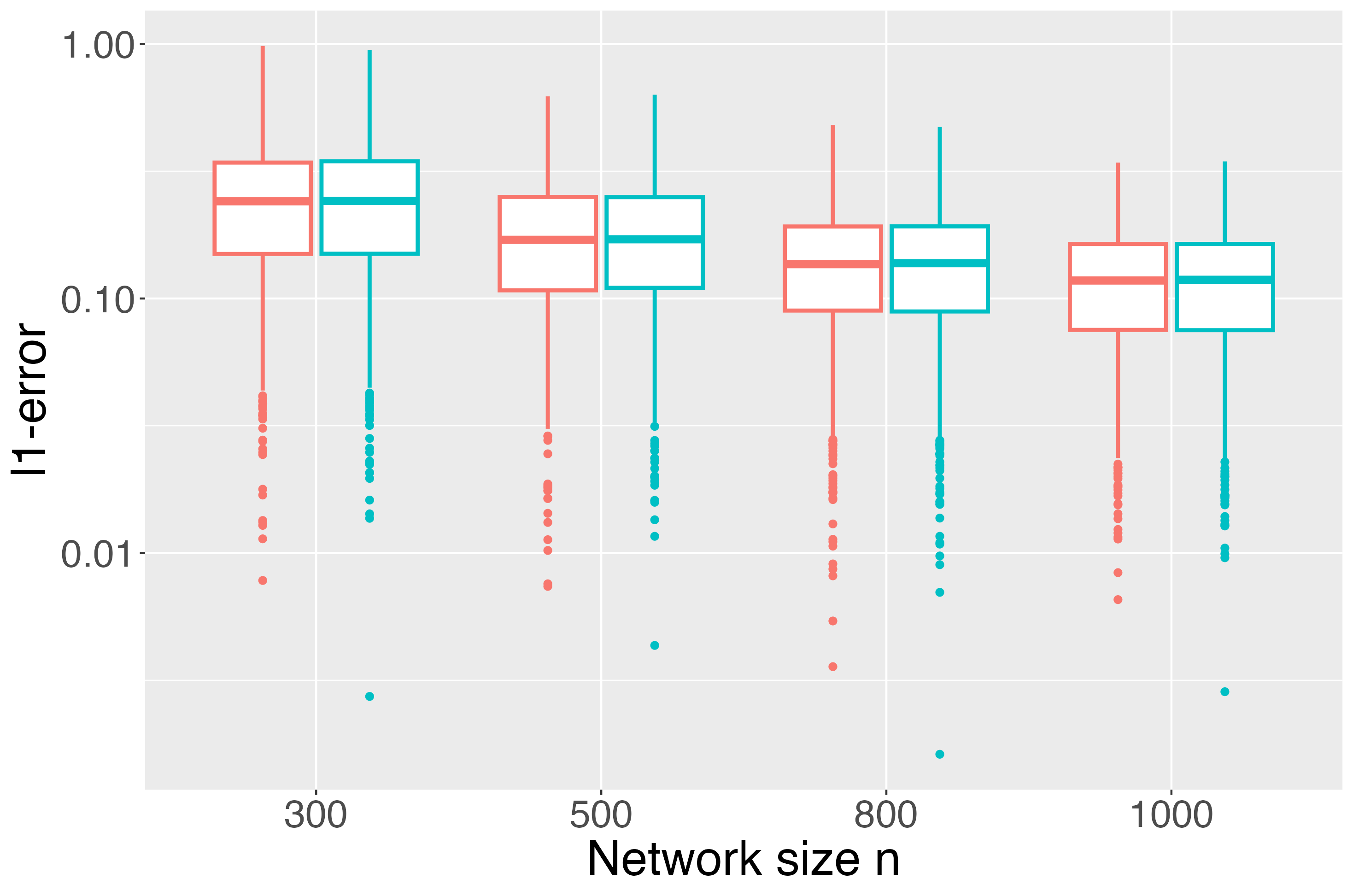}
		\caption{$\ell_1$-error for $\gamma_0$.}
		\label{Fig: l1 gamma model 5}
	\end{subfigure}
	\caption{Errors for estimating the true parameter $\theta_0$ in Model 5.}
	\label{Fig: Model5}
\end{figure}

\begin{figure}[H]
	\centering
	\begin{subfigure}{0.32\textwidth}
		\centering
		\includegraphics[scale=0.22]{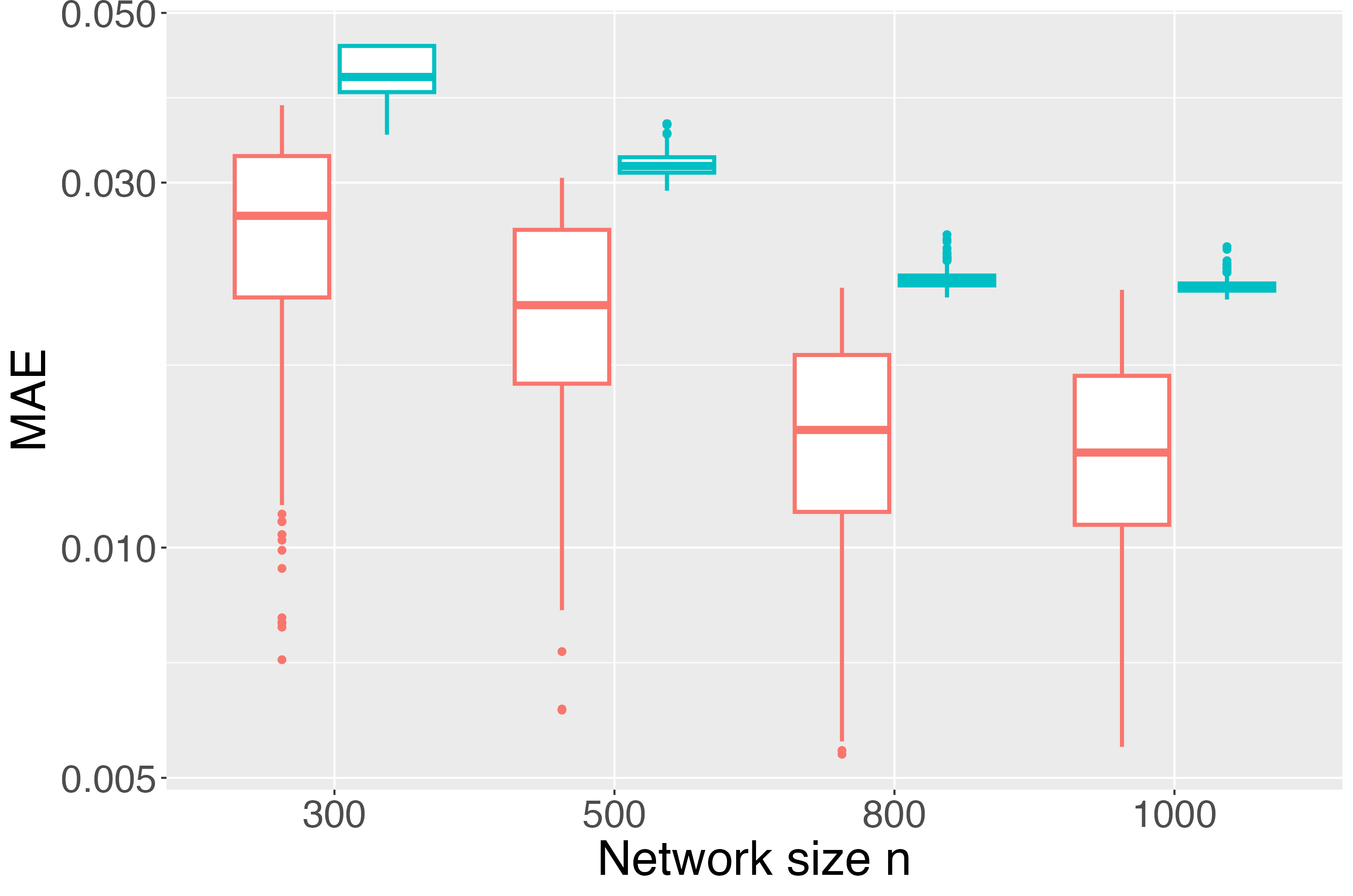}
		\caption{MAE for $\beta_0$}
		\label{Fig: beta MAE model 6}
	\end{subfigure}%
	\begin{subfigure}{0.32\textwidth}
		\centering
		\includegraphics[scale=0.22]{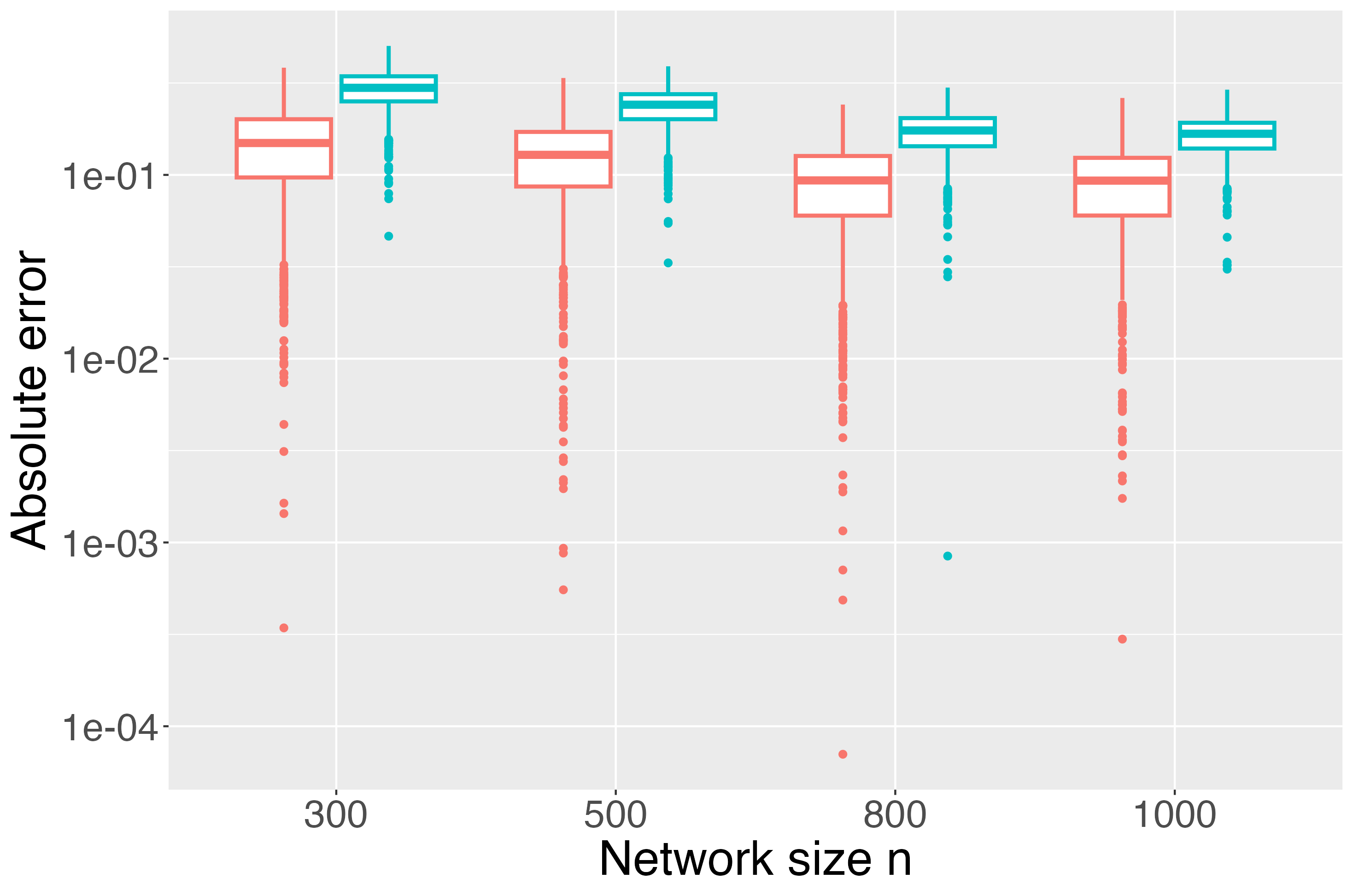}
		\caption{Absolute error for $\mu_0$.}
		\label{Fig: abs mu model 6}
	\end{subfigure}
	\begin{subfigure}{0.32\textwidth}
		\centering
		\includegraphics[scale=0.22]{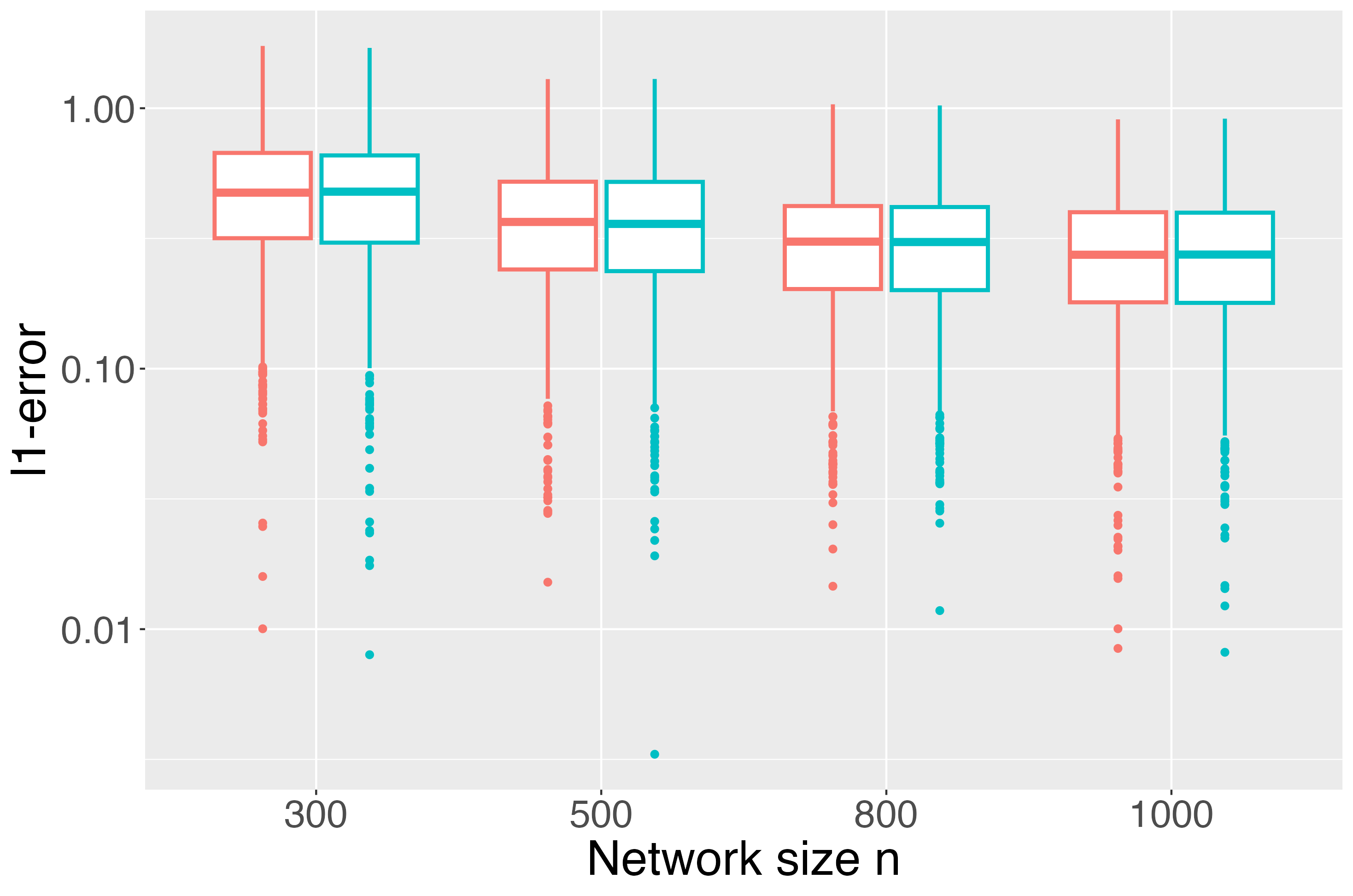}
		\caption{$\ell_1$-error for $\gamma_0$.}
		\label{Fig: l1 gamma model 6}
	\end{subfigure}
	\caption{Errors for estimating the true parameter $\theta_0$ in Model 6.}
	\label{Fig: Model6}
\end{figure}

\begin{table}[H]
\begin{center}
\begin{tabular}{crrrrrrrrr}
			\toprule
			&& \multicolumn{2}{c}{{Pre-determined $\lambda$}} & \multicolumn{2}{c}{{BIC}}& \multicolumn{2}{c}{{Pre-determined $\lambda$}} & \multicolumn{2}{c}{{BIC}}\\
			\midrule
			&\multicolumn{1}{c}{$n$} & Coverage & Width & Coverage & Width& Coverage & Width & Coverage & Width\\
			\midrule
			Model 4\\
			&300 & 0.935 & 0.333  & 0.943 & 0.335& 0.956 & 0.927 & 0.955 & 0.937 \\
			&500 & 0.950 & 0.225 & 0.959& 0.226& 0.954 & 0.813 & 0.953 & 0.818\\
			&800 & 0.946 & 0.159&0.945 & 0.159& 0.953 & 0.724 & 0.953 & 0.726\\
			&1000 & 0.946 & 0.133 & 0.947 & 0.134& 0.952 & 0.683& 0.952& 0.684\\
			\addlinespace[0.3em]
			Model 5\\
			&300 & 0.942 & 0.611 & 0.946 & 0.617& 0.952 & 1.715 & 0.965 & 1.758\\
		&500 & 0.931 & 0.449 & 0.934 & 0.451& 0.939 & 1.625 & 0.956 & 1.659\\
		&800 & 0.932 & 0.341 & 0.930 & 0.342& 0.952 & 1.561 & 0.952 & 1.585\\
		&1000 & 0.937 & 0.296 & 0.940 & 0.297& 0.950 & 1.519& 0.959& 1.541\\
		\addlinespace[0.3em]
			Model 6\\
			&300 & 0.948 &1.188  & 0.948 & 1.195& 0.955 & 3.306 & 0.961 & 3.420\\
		&500 & 0.944 & 0.941 & 0.943 & 0.944 & 0.957 & 3.394 & 0.964 & 3.479\\
		&800 & 0.947&0.766 & 0.944 & 0.767 & 0.962 & 3.497 & 0.967 & 3.560\\
		&1000 & 0.944 & 0.688 & 0.945 & 0.689& 0.962 & 3.509& 0.968& 3.565\\
			\bottomrule

\end{tabular}
\end{center}
\caption{Empirical coverage under nominal 95\% coverage and median lengths of confidence intervals for $\gamma_{0,1}$ (columns 3-6) and $\beta_0$ (last 4 columns) in sparser networks. }
\label{TableS1}
\end{table}

\bibliographystyle{agsm}
\bibliography{bib}

\end{document}